\documentclass[11pt,reqno]{amsart}
\usepackage{amsfonts,amsrefs,latexsym,amsmath,amssymb,mathrsfs,verbatim,slashed}
\usepackage{url}
\usepackage{pifont}
\usepackage{upgreek}
\usepackage{fancyhdr}
\usepackage{calligra}
\usepackage{marvosym}
\usepackage[percent]{overpic}
\usepackage{pict2e}
\usepackage[hypcap=false]{caption}
\usepackage{wasysym}
\usepackage{accents}
\usepackage{scalerel}
\usepackage[usenames,dvipsnames]{color}
\usepackage[colorlinks=true,linkcolor=Red,citecolor=Green]{hyperref}
\usepackage{color}

\topmargin - .25 in

\newtheorem{theorem}{Theorem}[section]
\newtheorem{proposition}[theorem]{Proposition}
\newtheorem{lemma}[theorem]{Lemma}
\newtheorem{corollary}[theorem]{Corollary}

\theoremstyle{definition}
\newtheorem{definition}[theorem]{Definition}
\newtheorem{remark}[theorem]{Remark}

\theoremstyle{plain}

\DeclareMathAlphabet{\mathcalligra}{T1}{calligra}{m}{n}
\DeclareFontShape{T1}{calligra}{m}{n}{<->s*[2.2]callig15}{}

\newcommand{\Tboot}{T_{(Boot)}}

\newcommand{\threePsi}{\vec{\Psi}}

\newcommand{\contr}{\diamond}

\newcommand{\Trandatasize}{\mathring{\updelta}}

\newcommand{\Psiep}{\mathring{\upalpha}}

\newcommand{\TranminusdatasizeWithFactor}{\mathring{\updelta}_{\ast}}

\newcommand{\gt}{\underline{g}}

\newcommand{\gsphere}{g \mkern-8.5mu / }
\newcommand{\ginversesphere}{\gsphere^{-1}}

\newcommand{\mytr}{{\mbox{\upshape{tr}}_{\mkern-2mu \gsphere}}}

\newcommand{\D}{\mathscr{D}}
\newcommand{\angD}{ {\nabla \mkern-14mu / \,} }

\newcommand{\angDsquared}{ {\angD_{\mkern-3mu}}^{\, 2} }

\newcommand{\angdiv}{\mbox{\upshape{div} $\mkern-17mu /$\,}}
\newcommand{\angLap}{ {\Delta \mkern-12mu / \, } }

\newcommand{\Speed}{c}

\newcommand{\Transport}{B}

\newcommand{\Ent}{s}
\newcommand{\GradEnt}{S}

\newcommand{\Densrenormalized}{\uprho}

\newcommand{\Vortrenormalized}{\varOmega}
\newcommand{\CurlofVortrenormalized}{\mathcal{C}}

\newcommand{\DivofEntrenormalized}{\mathcal{D}}

\newcommand{\Flatdiv}{\mbox{\upshape div}\mkern 1mu}
\newcommand{\Flatcurl}{\mbox{\upshape curl}\mkern 1mu}

\newcommand{\angdiff}{ {{d \mkern-9.5mu /} }}

\newcommand{\angdiffuparg}[1]{ {d \mkern-9mu /}^{#1} }

\newcommand{\Lie}{\mathcal{L}}
\newcommand{\SigmatLie}{\underline{\mathcal{L}}}
\newcommand{\angLie}{ { \mathcal{L} \mkern-10mu / } }

\newcommand{\Lgeo}{L_{(Geo)}}
\newcommand{\Lunit}{L}

\newcommand{\Radunit}{X}
\newcommand{\Rad}{\breve{X}}

\newcommand{\angk}{ { {k \mkern-10mu /} \, } }

\newcommand{\angG}{ {{\vec{G} \mkern-12mu /} \, }}
\newcommand{\angGarg}[1]{ {{\vec{G} \mkern-12mu /}_{\mkern 1mu #1} \, }}

\newcommand{\angGmixedarg}[2]{ {{\vec{G} \mkern-12mu /}_{#1}^{\ #2} \, }}

\newcommand{\NovecangGarg}[2]{ {{G \mkern-12mu /}_{\mkern 1mu #1}^{\mkern 6mu #2}}}

\newcommand{\angxi}{ { {\xi \mkern-9mu /}  \, }}

\newcommand{\angxiarg}[1]{ {{\xi \mkern-9mu /}_{#1}  \, }}

\newcommand{\Lineproject}{{\Pi \mkern-12mu / } \, }
\newcommand{\Sigmatproject}{\underline{\Pi}}

\newcommand{\vol}{\varpi}
\newcommand{\tvol}{\underline{\varpi}}
\newcommand{\conevol}{\overline{\varpi}}

\newcommand{\spherevol}{\uplambda_{{g \mkern-8.5mu /}}}

\newcommand{\upchifullmodarg}[1]{{^{(#1)} \mkern-4mu \mathscr{X}}}
\newcommand{\upchifullmodinhom}{\mathfrak{X}}

\newcommand{\upchipartialmodarg}[1]{{^{(#1)} \mkern-4mu \widetilde{\mathscr{X}}}}

\newcommand{\upchipartialmodinhomarg}[1]{{^{(#1)} \mkern-2mu \widetilde{\mathfrak{X}}}}

\newcommand{\smoothfunction}{\mathrm{f}}

\newcommand{\mydiam}{{\mkern-1mu \scaleobj{.75}{\blacklozenge}}}

\newcommand{\toprate}{M_*}
\newcommand{\Ccrit}{C'}

\newcommand{\f}{\frac}
\newcommand{\de}{\delta}
\newcommand{\ls}{\lesssim}
\newcommand{\rd}{\partial}
\newcommand{\Vr}{\Vortrenormalized}

\newcommand{\rr}{\Densrenormalized}
\newcommand{\bX}{\Rad}

\newcommand{\epd}{\mathring{\upepsilon}}

\newcommand{\rdb}{\underline{\rd}}
\newcommand{\srd}{\slashed{\rd}}
\newcommand{\chih}{\hat{\upchi}}
\newcommand{\nab}{\slashed{\nabla}}

\newcommand{\alp}{\alpha}
\newcommand{\bt}{\beta}
\newcommand{\mfp}{\mathfrak p}

 \newcommand{\pfstep}[1]{\vspace{.5em} {\it \noindent #1.}}

\newcommand{\inhom}{\mathfrak{F}}
\newcommand{\Mtu}{\mathcal{M}_{t,u}}
\newcommand{\Ntop}{N_{top}}

\newcommand{\curl}{\mbox{curl }}

\numberwithin{equation}{section}

\begin{document}
\title{The stability of simple plane-symmetric shock formation for 3D compressible Euler flow with vorticity and entropy
}
\author[JL,JS]{Jonathan Luk$^{* \dagger}$ and Jared Speck$^{** \dagger\dagger}$}

\thanks{$^\dagger$ JL is supported by a Terman fellowship and the NSF Grants \# DMS-1709458 and DMS-2005435.}

\thanks{$^{\dagger\dagger}$JS gratefully acknowledges support from NSF grant \# DMS-2054184 and NSF CAREER grant \# DMS-1914537.
}

\thanks{$^{*}$Stanford University, Palo Alto, CA, USA.
\texttt{jluk@stanford.edu}}

\thanks{$^{**}$Vanderbilt University, Nashville, TN, USA.
\texttt{jared.speck@vanderbilt.edu}}

\begin{abstract}
Consider a $1$D simple small-amplitude solution 
$(\varrho_{(bkg)}, v^1_{(bkg)})$ to the isentropic compressible Euler equations
which has smooth initial data, coincides with a constant state outside a compact set, and forms a shock in finite time. Viewing $(\varrho_{(bkg)}, v^1_{(bkg)})$ as a plane-symmetric solution to the full compressible Euler equations in $3$D, we prove that the shock-formation mechanism for the solution $(\varrho_{(bkg)}, v^1_{(bkg)})$ is stable against all sufficiently small and compactly supported perturbations. 
In particular, these perturbations are allowed to break the symmetry and have non-trivial vorticity and variable entropy.

Our approach reveals the full structure of the set of blowup-points at the first singular time:
within the constant-time hypersurface of first blowup,
the solution's first-order Cartesian coordinate partial derivatives blow up precisely on the zero level
set of a function that measures the inverse foliation density of a family of characteristic hypersurfaces.
Moreover, relative to a set of geometric coordinates constructed out of an acoustic eikonal function, 
the fluid solution and the inverse foliation density function remain smooth up to the shock; the blowup of the solution's
Cartesian coordinate partial derivatives is caused by a degeneracy between the geometric and Cartesian coordinates,
signified by the vanishing of the inverse foliation density (i.e., the intersection of the characteristics).

\bigskip

\noindent \textbf{Keywords}:
compressible Euler equations;
shock formation;
stable singularity formation;
wave breaking;
vectorfield method;
characteristics;
eikonal function;
null condition;
null hypersurface;
null structure
\bigskip

\noindent \textbf{Mathematics Subject Classification (2010)} 
Primary: 35L67 - Secondary: 35L05, 35Q31, 76N10
\end{abstract}

\maketitle

\centerline{\today}

\tableofcontents

\section{Introduction}
\label{S:INTRO}
It is classically known --- going back to the work of Riemann --- 
that the compressible Euler equations admit solutions for which singularities 
develop from smooth initial data. Indeed, such examples can already be found in the
plane symmetric isentropic case. In this case,
the compressible Euler equations reduce to a $2\times 2$
hyperbolic system in $1+1$-dimensions, 
which can be analyzed using Riemann invariants. 
In particular, it is easy to show that simple plane-symmetric solutions 
--- solutions with one vanishing Riemann invariant --- 
obey a Burgers-type equation, and that a shock can form in finite time.
By a shock, we mean that the solution remains bounded but its first-order partial derivative
with respect to the standard spatial coordinate blows up, and that the blowup is tied to the
intersection of the characteristics.

In this article, we prove that a class of simple plane-symmetric isentropic small-amplitude shock-forming solutions
to the compressible Euler equations
are \underline{stable} under small perturbations which break the symmetry and admit variable vorticity and entropy.
In particular, the perturbed solutions develop a shock singularity in finite time.
This provides the details of the argument sketched in \cites{jSgHjLwW2016,jLjS2020a} and completes the program that we have initiated (partly joint also with
Gustav Holzegel and Willie Wai-Yeung Wong) in \cites{jSgHjLwW2016,jS2019c,LS,jLjS2020a}.

We will consider the spatial domain\footnote{It is only for technical convenience that we chose the spatial topology 
$\mathbb{R} \times \mathbb T^2$. Similar results also hold, for instance, on $\mathbb{R}^3$.}
$\Sigma \doteq \mathbb R\times \mathbb T^2 = \mathbb R \times (\mathbb R/\mathbb Z)^2$ and a time interval $I$. Our unknowns
are the density $\varrho:I\times \Sigma \to \mathbb R_{>0}$, the velocity $v:I\times \Sigma \to \mathbb R^3$, 
and the entropy $s:I\times \Sigma\to \mathbb{R}$. 
Relative to the standard Cartesian coordinates 
$(t,x^1,x^2,x^3)$ on 
$I \times \mathbb{R} \times \mathbb{T}^2$, the compressible Euler equations can be expressed as follows:
\begin{align}
\label{eq:Euler.1}(\rd_t + v^a \rd_a) \varrho = &\: - \varrho \,\mathrm{div} v,\\
\label{eq:Euler.2}(\rd_t + v^a \rd_a) v^j = &\: - \f 1\varrho \de^{ja} \rd_a p,
&& (j=1,2,3),\\
\label{eq:Euler.3}(\rd_t + v^a \rd_a) s =&\:  0,
\end{align}
where (from now on) $\de^{ij}$ denotes the Kronecker delta,
$\mathrm{div} v \doteq \partial_a v^a$ is the Euclidean divergence of $v$,
repeated lowercase Latin indices are summed over $i,j= 1,2,3$, 
and the pressure $p$ relates to $\varrho$ and $s$ by a prescribed 
smooth
\emph{equation of state} $p = p(\varrho,s)$.
In other words, the right-hand side of \eqref{eq:Euler.2} can be expressed as 
$- \f 1\varrho \de^{ja} \rd_a p = - \f 1\varrho p_{;\varrho}\de^{ja} \rd_a \varrho 
- 
\f 1\varrho p_{;\Ent} \de^{ja} \rd_a s$,
where $p_{;\varrho}$ denotes\footnote{Later in the paper, we will take the partial derivative of 
various quantities with respect
to the logarithmic density $\Densrenormalized$. If $f$ is a function of the fluid unknowns, then 
$f_{;\Densrenormalized}$ will denote the partial derivative of $f$ with respect to $\Densrenormalized$
when the other fluid variables are held fixed. Similarly, 
$f_{;\Ent}$ denotes the partial derivative of $f$ with respect to $\Ent$
when the other fluid variables are held fixed.} 
the partial derivative of the equation of state
with respect to the density at fixed $\Ent$, and analogously for $p_{;\Ent}$.

For the remainder of the paper:
\begin{enumerate}
\item We fix a constant $\bar{\varrho}>0$ and a constant solution $(\varrho, v^i,s) = (\bar{\varrho}, 0,0)$ to \eqref{eq:Euler.1}--\eqref{eq:Euler.3}
\item We fix an equation of state $p = p(\varrho,s)$ such that\footnote{This normalization can always be achieved by a change of variables as long as $\f{\rd p}{\rd\varrho}(\bar{\varrho}, 0) >0$; see \cite[Footnote~19]{LS}.} $\f{\rd p}{\rd\varrho}(\bar{\varrho}, 0) = 1$.
\end{enumerate}

For notational convenience, we define the logarithmic density 
$\rr \doteq \log \left( \f{\varrho}{\bar{\varrho}} \right)$ 
and the speed of sound $\Speed(\rr,s) \doteq \sqrt{\f{\rd p}{\rd\varrho}}(\varrho,s)$. We will from now on think of $\Speed$ as a function of $(\rr,s)$.

We will study
perturbations of a shock-forming
background solution $(\varrho_{(bkg)}, v^i_{(bkg)},s_{(bkg)})$ 
arising from smooth initial data such that the following hold:
\begin{enumerate}
\item The background solution is plane-symmetric and isentropic, i.e., $v^2_{(bkg)}=v^3_{(bkg)}=s_{(bkg)}=0$, and $(\varrho_{(bkg)}, v^1_{(bkg)})$ are functions only of $t$ and $x^1$.
\item The background solution is simple, i.e., the Riemann invariant $\mathcal{R}^{(bkg)}_{(-)}$ satisfies:
$$\mathcal{R}^{(bkg)}_{(-)} \doteq v_{(bkg)}^1 - \int_{0}^{\rr_{(bkg)}} c(\rr',0)\, d\rr'= 0.$$
\item The background solution is initially compactly supported in an $x^1$-interval of length 
$\leq 2\mathring{\upsigma}$, 
i.e., outside this interval, $(\varrho_{(bkg)}, v^i_{(bkg)},s_{(bkg)})\restriction_{t=0} = (\bar{\varrho},0,0)$.
\item At time $0$ (and hence throughout the evolution), the Riemann invariant 
$\mathcal{R}_{(+)}^{(bkg)} \doteq v^1_{(bkg)} + \int_{0}^{\rr_{(bkg)}} c(\rr',0)\, d\rr'$ has small $\leq \mathring{\upalpha}$ amplitude.
\item At time $0$, the
Cartesian spatial derivatives of 
$\mathcal{R}_{(+)}^{(bkg)}$ up to the third order\footnote{In the $1$D case,
one only needs information about the data's first derivative to close a proof of blowup
for a simple plane wave. However, when studying perturbations in $3$D, we need estimates
on these derivatives up the third order in order to close the proof.
For example, the proof of the bound \eqref{E:RADRADUPMULINFTY} relies on having 
control of up to these third-order derivatives 
(as is provided by 
\eqref{E:LINFINITITYTHREETRANSVERSALOFLARGERIEMANNINVARIANT}--\eqref{E:LINFINITITYTHREETRANSVERSALOFSMALLWAVEVARIABLES}), and we use the bound \eqref{E:RADRADUPMULINFTY}
in the proof of Lemma~\ref{lem:higher.transversal.for.Holder}
as well as in the proof of the energy estimates in the appendix.} are
bounded above pointwise by 
$\leq \mathring{\updelta}^{(bkg)}$ (where $\mathring{\updelta}^{(bkg)}$ is not necessarily small).
\item The quantity\footnote{One can check that 
this rules out the Chaplygin gas, whose speed of sound (after normalization) is given by $\Speed(\rr,s) = \exp(-\rr)$. 
One can also check that for any other equation of state, 
it is possible to choose
$\bar{\varrho}$ appropriately so that 
$\mathring{\updelta}_*^{(bkg)} > 0$.} $\mathring{\updelta}_*^{(bkg)}$ 
(where $\mathring{\updelta}_*^{(bkg)}$ is not necessarily small) that controls the 
blowup-time satisfies:\footnote{Here, $[\cdot]_+$ denotes the positive part.} 
$$\mathring{\updelta}_*^{(bkg)} \doteq \f 12 \sup_{\{t=0\}} \left[ \f 1\Speed\left\lbrace \f 1\Speed \f{\rd\Speed}{ \rd\rr}(\rr_{(bkg)},0) + 1\right\rbrace (\rd_{1} \mathcal{R}^{(bkg)}_{(+)})\right]_{+}  >0,$$
and the solution forms a shock at time\footnote{In the plane-symmetric, isentropic, simple case, 
$\mathcal R^{(bkg)}_{(+)}$ 
satisfies the transport equation 
$\rd_t \mathcal R^{(bkg)}_{(+)} + (v^1_{(bkg)} + c(\rr_{(bkg)})) \rd_1 \mathcal R^{(bkg)}_{(+)} = 0$,
and the blowup-time of $\rd_1 \mathcal R^{(bkg)}_{(+)}$ can easily be computed explicitly 
by commuting this transport equation with $\partial_1$ to obtain a Riccati-type ODE in 
$\partial_1 R^{(bkg)}_{(+)}$ along the integral curves of $\rd_t + (v^1_{(bkg)} + c(\rr_{(bkg)})) \rd_1$.} 
$T_{(Sing)}^{(bkg)} = (\mathring{\updelta}_*^{(bkg)})^{-1}$.
\end{enumerate}
The analysis for plane-symmetric solutions can be carried out easily using Riemann invariants. It is then straightforward to check that there exists a large class of plane-symmetric solutions satisfying (1)--(6) above.

We now provide a rough version of our main theorem; see Section~\ref{sec:statement} for a more precise statement.

\begin{theorem}[\textbf{Main theorem -- Rough version}]
\label{thm:intro.main}
Consider a plane-symmetric, shock-forming background solution $(\varrho_{(bkg)}, v^i_{(bkg)},s_{(bkg)})$ 
satisfying (1)--(6) above, where the parameter $\mathring{\upalpha}$ from point (4) is small.
Consider a small perturbation of the initial data of this background solution
satisfying the following assumptions 
(see Section~\ref{SS:FLUIDVARIABLEDATAASSUMPTIONS} for the precise assumptions):
\begin{itemize}
\item The perturbation is compactly supported in a region of $x^1$-length $\leq 2\mathring{\upsigma}$;
\item The perturbation belongs to a high-order Sobolev space, where 
the required Sobolev regularity is independent of the background solution 
and equation of state; 
\item The perturbation is small, where the smallness 
is captured by the small parameter $0 < \epd \ll 1$,
and the required smallness depends on the order of 
the Sobolev space, the equation of state, and the
parameters of the background solution.
\end{itemize}
Then the corresponding unique perturbed solution satisfies the following:
\begin{enumerate}
\item The solution is initially smooth, but 
it becomes singular at a time 
$T_{(Sing)}$, which is a small perturbation of the background blowup-time 
$(\mathring{\updelta}_*^{(bkg)})^{-1}$.
\item Defining\footnote{In higher dimensions or in the presence of dynamic entropy, 
$\mathcal{R}_{(+)}$ is \underline{not} a Riemann invariant because 
its dynamics is not determined purely by a transport equation. Nonetheless, for comparison
purposes, we continue to use the symbol ``$\mathcal{R}_+$'' to denote this quantity.}  
$\mathcal{R}_{(+)} \doteq v^1 + \int_{0}^{\rr} c(\rr',s)\, d\rr'$, we have
the following singular behavior:
\begin{equation}\label{eq:themoststupidblowup}
\limsup_{t \to T_{(Sing)}^-}
\sup_{\{t\}\times \Sigma} | \rd_1 \mathcal R_{(+)}| = + \infty.
\end{equation}
\item Relative to a geometric coordinate system $(t,u,x^2,x^3)$,
	where $u$ is an eikonal function, the solution remains smooth, all the way up to time
	$T_{(Sing)}$. In particular, the partial derivatives of the solution with respect
	to the geometric coordinates \underline{do not blow up}.
\item The blowup at time $T_{(Sing)}$ 
is characterized by the vanishing of the
inverse foliation density $\upmu$ (see Definition~\ref{D:FIRSTUPMU}) 
of a family of acoustically null hypersurfaces defined to be the level sets of $u$.
\item In particular, the set of blowup-points at time $T_{(Sing)}$ is characterized 
by:
\begin{align*}
& \left\{(u,x^2,x^3) \in \mathbb{R} \times \mathbb{T}^2: 
	\limsup_{(\tilde{t},\tilde{u},\tilde{x}^2,\tilde{x}^3) 
	\to 
	(T_{(Sing)}^-,u,x^2,x^3)}  
	|\rd_1 \mathcal R_{(+)}|(\tilde{t},\tilde{u},\tilde{x}^2,\tilde{x}^3) = \infty \right\}
		\\
& =
\left\{ (u,x^2,x^3) \in \mathbb{R} \times \mathbb{T}^2: \upmu(T_{(Sing)},u,x^2,x^3) = 0 \right\},
\end{align*}
where $|\rd_1 \mathcal R_{(+)}|(\tilde{t},\tilde{u},\tilde{x}^2,\tilde{x}^3)$
denotes the absolute value of the Cartesian partial derivative 
$\rd_1 \mathcal R_{(+)}$
evaluated at the point with geometric coordinates
$(\tilde{t},\tilde{u},\tilde{x}^2,\tilde{x}^3)$.
\item At the same time, as $T_{(Sing)}$ is approached from below, 
the fluid variables $\varrho$, $v^i$, $s$ all remain bounded, as 
do the specific vorticity $\Vr^i \doteq \f{(\mathrm{curl} v)^i}{(\varrho/\bar{\varrho})}$ and the entropy gradient $S \doteq \nabla s$.
\end{enumerate}
\end{theorem}

The proof of Theorem~\ref{thm:intro.main} relies on two main ingredients: 
\textbf{i)} Christodoulou's geometric theory of shock formation
for irrotational and isentropic solutions, in which case the dynamics reduces
to the study of quasilinear wave equations and \textbf{ii)} a (re)formulation of the
compressible Euler equations as a quasilinear system of wave-transport equations, 
which was derived in \cite{jS2019c}, following the earlier works 
\cites{LS,jLjS2020a} in the barotropic\footnote{A barotropic fluid is such that the 
equation of state for the pressure is a function of the density alone, as opposed to being a function
of the pressure and entropy.\label{footnote:barotropic}} case.
 This formulation exhibits remarkable null 
structures and regularity properties, which in total allow us to perturbatively control the
vorticity and entropy gradient all the way up to the singular time ---
even though generically, their first-order Cartesian partial derivatives blow up at the singularity. 
See Section~\ref{sec:ideas} for further discussion of the proof.

Some remarks are in order.

\begin{remark}
\label{R:NOSUBTRACTBACKGROUND}
Note that even though the rough Theorem~\ref{thm:intro.main} is formulated in terms of plane-symmetric background solutions, we do not actually ``subtract off a background'' in the proof. See Theorem~\ref{thm:shock} for the precise formulation.
\end{remark}

\begin{remark}[\textbf{Results building up towards Theorem~\ref{thm:intro.main}}]\label{rmk:theearlier}
\hfill

\begin{itemize}
\item Concerning stability of simple plane-symmetric shock-forming solutions to the compressible Euler equations, the first result
was our joint work with G.~Holzegel and W.~Wong \cite{jSgHjLwW2016}, which proved the analog\footnote{We remark that while \cite{jSgHjLwW2016} is only explicitly stated a theorem in $2$ spatial dimensions,
the analogous result in $3$ (or indeed higher) dimensions can be proved using similar arguments; 
see \cite[Remarks~1.4,1.11]{jSgHjLwW2016}.} 
of Theorem~\ref{thm:intro.main}
in the case\footnote{The main theorem in \cite{jSgHjLwW2016} is stated for general quasilinear wave equations. Particular applications to the relativistic compressible Euler equations in the irrotational and isentropic regime can be found in  \cite[Appendix~B]{jSgHjLwW2016}. It applies equally well to the non-relativistic case.} where the perturbation is irrotational and isentropic (i.e., $\Vr \equiv 0$,  $S \equiv 0$).

\item
In \cite{LS}, we proved the first
stable shock formation result 
without symmetry assumptions for the compressible Euler equations
for open sets of initial data that can have
non-trivial specific vorticity $\Vr$.
Specifically, in \cite{LS},
we treated the $2$D 
barotropic compressible Euler equations (see Footnote~\ref{footnote:barotropic}). 
One of the key points in \cite{LS} was our 
reformulation of equations into a system of quasilinear wave-transport equations which has favorable nonlinear 
null structures. This allowed us to use the full power of the geometric
vectorfield method on the wave part of the system while
treating the vorticity perturbatively. 
\item
In \cite{jLjS2020a}, we considered $3$D barotropic compressible Euler flow
and derived a similar reformulation of the equations
that allowed for non-zero vorticity.
In contrast to the $2$D case, the transport equation satisfied by the specific
vorticity $\Vr$ featured vorticity-stretching
source terms (of the schematic form $\Vr \cdot \partial v$). In order to handle the vorticity-stretching source terms 
in the framework of \cite{LS}, we also showed in \cite{jLjS2020a} that $\Vr$ satisfies a
div-curl-transport system with source terms that
are favorable from the point of view of regularity and from the point of view of null
structure. 
We refer to Section~\ref{sec:intro.elliptic} for further discussion of this point.
\item To incorporate thermodynamic effects into compressible fluid flow,
one must look beyond the family of barotropic equations of state, e.g.,
consider equations of state in which the pressure depends on the density and entropy.\footnote{Incorporating entropy into the analysis is expected to 
be especially important for studying weak solutions after the shock
(see Section~\ref{sec:shock.development} for further discussion),
since formal calculations \cite{dC2019} 
suggest that the entropy (even if initially zero)
should jump across the shock hypersurface, which in turn should induce
a jump in vorticity.}
Fortunately, in \cite{jS2019c}, it was shown that
a similar good reformulation of the compressible Euler equations holds
under an arbitrary equation of state (in which the pressure is a function of the density and the entropy)
in the presence of vorticity
and variable entropy. In the present paper, we use this reformulation
to prove our main results; we recall it below as Theorem~\ref{T:GEOMETRICWAVETRANSPORTSYSTEM}. 
The analysis in \cite{jS2019c} 
is substantially more complicated compared to the barotropic case,
and the basic setup requires the observation of some new structures tied
to elliptic estimates for $\Vr$ and $S$, 
such as good regularity and null structures tied to 
the modified fluid variables from Definition~\ref{def:variables.HO}.
\end{itemize}

This paper completes the program described above
by giving the analytic details already sketched in \cites{jLjS2020a,jS2019c}. 
Chief among the analytic novelties in the present paper are 
the elliptic estimates for $\Vr$ and $S$ at the top-order; 
see \cite[Sections~1.3, 4.2.7]{jLjS2020a}, 
\cite[Section~4.3]{jS2019c} and Section~\ref{sec:intro.elliptic}.
We also point out that there are other related works, which we discuss in Section~\ref{sec:related}.
\end{remark}

\begin{remark}[\textbf{Blowup and boundedness of quantities involving higher derivatives}]
For generic perturbations, derivatives of fluid variables other than 
{$\mathcal{R}_+$ (whose blowup was highlighted in \eqref{eq:themoststupidblowup}) can} 
also blow up. In particular, while the $\rd_2$ and $\rd_3$ derivatives of the fluid variables are identically $0$ for the plane-symmetric background solutions,
for the perturbed solution, $\rd_2 v^i$, say, is generically unbounded at the singularity. This is because the perturbation changes
the geometry of the solution, and the regular directions no longer align with the Cartesian directions.

On the other hand, there are indeed higher derivatives of the fluid variables that remain bounded up to the singular time. These include the specific vorticity and the entropy gradient that we already mentioned explicitly in Theorem~\ref{thm:intro.main}. Moreover, any null-hypersurface-tangential geometric derivatives (see further discussions in Section~\ref{sec:ideas}) of the fluid variables are also bounded up to the singular time. This is not just a curiosity,
but rather is a fundamental aspect of the proof.

Remarkably, there are additionally quantities, denoted by $\mathcal{C}$ and $\mathcal{D}$ 
(these variables were identified in \cite{jS2019c}, see \eqref{E:RENORMALIZEDCURLOFSPECIFICVORTICITY}--\eqref{E:RENORMALIZEDDIVOFENTROPY}), which are special combinations of 
up-to-second-order
Cartesian coordinate derivatives of the fluid variables,
which remain uniformly bounded up to the singularity 
(as do their derivatives in directions tangent to a family of null hypersurfaces);
$\mathcal{C}$ and $\mathcal{D}$ are precisely the modified fluid variables mentioned in Remark~\ref{rmk:theearlier}.
The existence of such regular higher order quantities is not only an interesting fact, but is also quite helpful in controlling the solution up to the first singularity; see Section~\ref{sec:ideas}.

Finally, as a comparison with our $2$D work \cite{LS}, note that in the $2$D case,
we proved that the specific vorticity remains Lipschitz (in Cartesian coordinates) up to the first singular time. This is no longer the case in $3$D. Indeed, in the language of this paper, the improved regularity for the specific vorticity in \cite{LS} stems from the fact that in 
$2$D, the Cartesian coordinate derivatives of the specific vorticity $\Vr$ coincide with $\mathcal{C}$.
\end{remark}

\begin{remark}[\textbf{Additional information on sub-classes of solutions}]\label{rmk:subregime}
Within the solution regime we study, we are able to derive additional information 
about the solution 
by making further assumptions on the data.
For instance, 
there are open subsets of data such that 
the vorticity/entropy gradient are non-vanishing at the first singularity, and
also open subsets of data such that 
the fluid variables remain H\"older\footnote{The H\"older estimates hold only for an open subset of data 
satisfying certain non-degeneracy assumptions. They were not announced in \cites{jLjS2020a,jS2019c}. 
We were instead inspired by \cite{tBsSvV2019b, tBsSvV2020} to include such estimates.} 
$C^{1/3}$ up to the singularity. 
See Section~\ref{sec:statement} for details.
\end{remark}


\begin{remark}[\textbf{The maximal smooth development}]\label{rmk:maximal}
The approach we take here allows us to analyze the solution up to the first singular time,
and our main results yield a complete description of the set of blowup-points at that time
(see, for example, conclusions (4)--(5) of Theorem~\ref{thm:intro.main}).
However, since the compressible Euler equations
are a hyperbolic system, it is desirable to go beyond our results by deriving
a full description of the maximal smooth development of the initial data, in analogy with \cite{dC2007}. Understanding the maximal smooth development is particularly important for the shock development problem; 
see Section~\ref{sec:shock.development} below.

Our methods, at least on their own,
are not enough to construct the maximal smooth development\footnote{Notice that in our earlier result \cite{LS} for the isentropic Euler equations in two spatial dimensions, we also only solved the equations up to the first singular time. However, there is an important difference. In the $2$D case, there does not seem to be a philosophical obstruction in extending \cite{LS} to provide a complete description of 
the maximal smooth development. In contrast, in the $3$D case it seems that ideas in \cite{lAjS2020} would be needed in a fundamental way.}. This is in part because our approach here relies on
spatially global elliptic estimates on constant-$t$ hypersurfaces; the point is that a full description of the smooth maximal development would require spatially localized estimates. 
On the other hand, the recent preprint \cite{lAjS2020} discovered an integral identity that allows the elliptic estimates to be localized, and thus gives hope that Theorem~\ref{thm:intro.main} 
can be extended to derive the structure of the full maximal smooth development.
\end{remark}

\begin{remark}[\textbf{No universal blowup-profile}]
\label{R:LACKOFUNIVERSALBLOWUPPROFILE}
One of the main advantages of our geometric framework is that it works for many kinds of singular solutions,
not just those exhibiting a specific blowup-profile. 
In particular, the solutions featured in Theorem~\ref{thm:intro.main} do not exhibit a universal blowup-profile.
Although we do not rigorously study the full class of blowup-profiles
exhibited by the solutions from Theorem~\ref{thm:intro.main}, 
the full class is likely quite complicated to describe.
This can already be seen in model case of Burgers' equation, 
where there are a continuum of possible blowup-profiles and corresponding blowup-rates \cite{jEmF2008}
(recall that we work in the near plane-symmetric regime and our work includes, as special cases, 
plane-symmetric solutions, which are analogs of Burgers' equation solutions).
A related issue is that at the time of first singularity formation, 
the set of blowup-points can be complicated and/or of infinite cardinality
(as one can already see in the special case of plane-symmetric solutions, viewed as solutions in $3$D with symmetry).
\end{remark}

\begin{remark}[\textbf{The relativistic case}]
While our present work treats only the non-relativistic case, it is likely that the relativistic case can also be treated in the same way. 
This is because the relativistic compressible Euler equations also admit a similar reformulation as we consider here,
and likewise the variables in the reformulation also exhibit a very similar null structure \cite{mDjS2019}.
\end{remark}

In the remainder of the introduction, we will first discuss the proof in \textbf{Section~\ref{sec:ideas}} and then discuss some related works in \textbf{Section~\ref{sec:related}}. We will end the introduction with an outline of the remainder of the paper.

\subsection{Ideas of the proof}\label{sec:ideas}

\subsubsection{The Christodoulou theory}\label{sec:ideas.Christodoulou} The starting point of our proof is the work of Christodoulou \cite{dC2007} on shock formation
for quasilinear wave equations.\footnote{Strictly speaking, \cite{dC2007} is only concerned with the irrotational isentropic relativistic Euler equations. However, its methods apply to much more general quasilinear wave equations; see further discussions in \cites{jS2016b,gHsKjSwW2016}.}

Consider the following model quasilinear covariant wave equation for the scalar function $\Psi$: 
$\square_{g(\Psi)}\Psi = 0$, where the Cartesian component functions $g_{\alpha \beta}$ 
are given (nonlinear in general) functions of $\Psi$, i.e., $g_{\alpha \beta} = g_{\alpha \beta}(\Psi)$.
Our study of compressible Euler flow in this paper essentially amounts
to studying a system of similar equations with source terms and showing that the source terms do not 
radically distort the dynamics. This is possible only because the source terms have remarkable null structure,
described below.



A key insight for studying the formation of shocks, 
going back to \cite{dC2007}, is that it is advantageous to 
study the shock formation via a system of geometric
coordinates. The point is that when appropriately constructed, such coordinates regularize the problem 
which allows one to treat the problem of shock formation as if it were a standard local existence problem.
More precisely,
one constructs geometric coordinates, adapted to the flow, such that the solution
remains regular relative to them.\footnote{It should be emphasized that it is only at the low derivative levels that the solution is regular. The high-order geometric energies can still blow up, even though the low-order energies remain bounded. The 
possible
growth of the high-order energies is one of the central
technical difficulties in the problem, and we will discuss it below in more detail. \label{FN:HIGHDERIVATIVESBLOWUP}}  
However, the geometric coordinates \underline{degenerate} relative to the Cartesian ones,
and the blowup of the solution's first-order Cartesian coordinate partial derivatives 
can be derived as a consequence of this degeneracy.

To carry out this strategy, one must
use the Lorentzian geometry associated to the acoustical metric $g$ (see Definition~\ref{D:ACOUSTICALMETRIC}). 
The following geometric objects are
of central importance in implementing this program:
\begin{itemize}
\item A foliation by constant-$u$ characteristic hypersurfaces $\mathcal F_u$ 
	(where~$g^{-1}(du, du) = 0$; see equation \eqref{E:INTROEIKONAL}).
	The function $u$ is known as an ``acoustic eikonal function.''
\item The inverse foliation density $\upmu$ ($\doteq -\f 1{g^{-1}(dt, du)}$), where $\upmu^{-1}$ measures the density of $\mathcal F_u$ with respect to the constant-$t$ hypersurfaces.
\item A frame of vectorfields $\{L, X, Y, Z\}$, where $\{L, Y, Z\}$ are tangent to $\mathcal F_u$ (with $L$ being its null generator) and $X$ is transversal to $\mathcal F_u$; see Figure~\ref{F:FRAME},
where we have suppressed the $Z$ direction.
\item $\{L, X, Y, Z\}$ is a frame that is ``comparable'' to the
Cartesian frame $\{ \rd_t, \rd_1, \rd_2, \rd_3\}$, by which we mean the coefficients relating the frames to each other
are size $\mathcal{O}(1)$.
\item However, in the analysis, uniform boundedness estimates are generally available
	for the derivatives of quantities with respect to only the \emph{rescaled} frame elements 
	$\{L, \bX \doteq \upmu X, Y, Z\}$.
\end{itemize}

\begin{center}
\begin{overpic}[scale=.25]{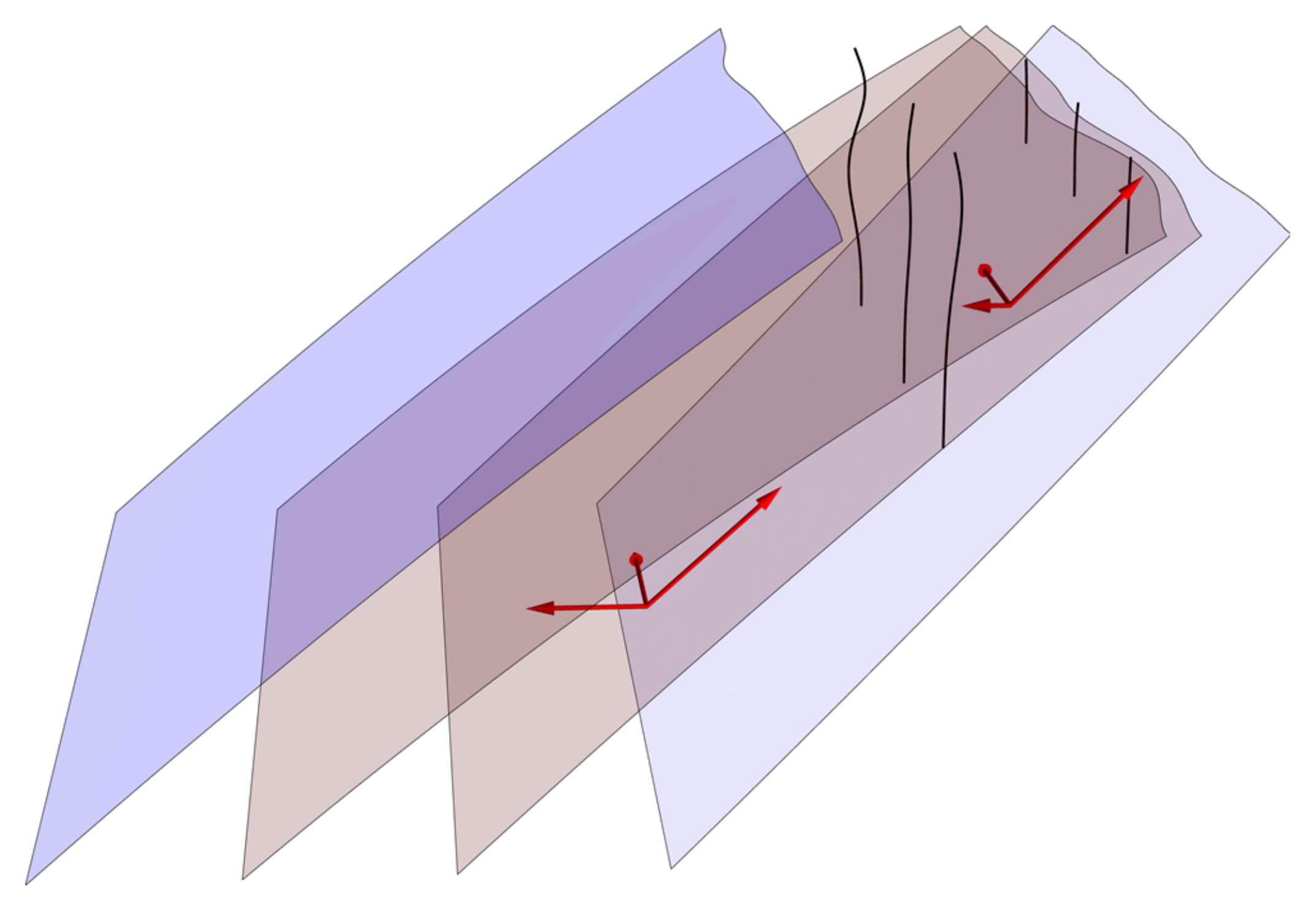} 
\put (82,47.7) {\large$\displaystyle \Lunit$}
\put (72,40.2) {\large$\displaystyle \Rad$}
\put (73.5,50) {\large$\displaystyle Y$}
\put (54,24) {\large$\displaystyle \Lunit$}
\put (38.5,17) {\large$\displaystyle \Rad$}
\put (47,28) {\large$\displaystyle Y$}
\put (49.5,10) {\large$\displaystyle \mathcal{F}_0$}
\put (35,10) {\large$\displaystyle \mathcal{F}_u$}
\put (5,10) {\large$\displaystyle \mathcal{F}_1$}
\put (22,26) {\large$\displaystyle \upmu \approx 1$}
\put (66,67) {\large$\displaystyle \upmu \ \mbox{\upshape small}$}
%
%
\end{overpic}
\captionof{figure}{The dynamic vectorfield frame at two distinct points on $\mathcal{F}_u$
with the $Z$-direction suppressed, 
and the integral curves of the transport operator $\Transport$ for the
specific vorticity and entropy}
\label{F:FRAME}
\end{center}

The analysis simultaneously yields control of the derivatives of $\Psi$ with respect to the rescaled frame and gives also quantitative estimates on the geometry. In this geometric picture, the blowup is completely captured by $\upmu\to 0$. The connection between the vanishing of $\upmu$ and the blowup of some Cartesian coordinate
partial derivative of $\Psi$ can be understood as follows: 
one proves an estimate of the form
$|\Rad \Psi| \approx 1$ (which is consistent with the uniform boundedness estimates mentioned above). 
In view of the relation $\Rad = \upmu \Radunit$, this estimate 
implies that $|\Radunit \Psi|$ blows up like $1/\upmu$ as $\upmu \to 0$.

We now give a more detailed description of the behavior of the solution, with a focus on
how it behaves at different derivative levels.
\begin{itemize}
\item  As our discussion above suggested, at
the lower derivative levels, derivatives of quantities 
with respect to the \underline{rescaled} frame are \underline{regular}, e.g.,~$L\Psi$, $\bX \Psi$, $Y\Psi$, $Z\Psi$, $\dots$, $L^3\bX Y\Psi$, etc.~are uniformly bounded.
\item As we highlighted above, the formation of the shock corresponds to
$\upmu\to 0$ in finite time, and moreover, the \underline{non-rescaled}
first-order derivative $X\Psi$ \underline{blows up} in finite time, exactly at points where 
 $\upmu$ vanishes.
\item The main difficulty in the proof is that the only known approach 
to the solution's regularity theory \underline{with respect to the rescaled frame derivatives}
that is able to avoid a loss of derivatives allows for the following possible scenario: 
the energy estimates are such that the high-order geometric energies might blow up when 
the shock forms. This leads to severe difficulties in the proof, especially considering that
one needs to show that the low-order derivatives of the solution remain bounded in order to derive the singular high-order energy estimates.\footnote{The possible high-order energy blowup has its origins
in the presence of some difficult factors of $1/\upmu$ in the top-order 
energy identities, where one must work hard to avoid a loss of derivatives. To close the energy estimates, 
one commutes the wave equation many times with the $\mathcal{F}_u$-tangent subset $\lbrace L,Y,Z \rbrace$
of the rescaled frame.
The most difficult terms in the commuted wave equation
are top-order terms in which all the derivatives fall onto the components of $\lbrace L,Y,Z \rbrace$.
It turns out that due to the way the rescaled frame is constructed, 
the corresponding difficult error terms depend on the top-order derivatives of the eikonal function $u$.
In Proposition~\ref{prop:identity.main.wave.commutator.terms}, we identify these difficult commutator terms.
To avoid the loss of derivatives, one must work with modified quantities and use elliptic estimates.
It is in this process that one creates difficult factors of $1/\upmu$.}


In \cite{dC2007}, Christodoulou showed that the maximum possible blowup-rate of the high-order energies
is of the form $\upmu_{\star}^{-2 P}(t)$, where $P$ is a universal positive constant
and $\upmu_{\star}(t) \doteq 
\min \left\lbrace 1, \min_{\Sigma_t} \upmu \right\rbrace$. To reconcile this possible high-order energy blowup
with the regular behavior at the lower derivative levels, one is forced to derive 
a hierarchy of energy estimates of the form, where
$\toprate$ is a universal\footnote{Our proof of the universality of $\toprate$ 
in the presence of vorticity and entropy
requires some new observations, described below equation \eqref{eq:intro.the.difficult.term}.
}  
positive integer:
\begin{equation}\label{eq:energy.blow.up}
\begin{split}
\mathbb{E}_{\Ntop}(t) \ls \epd^2 \upmu_{\star}^{-2 \toprate+1.8}(t), 
\quad \mathbb{E}_{\Ntop-1}(t) \ls \epd^2 \upmu_{\star}^{-2 \toprate+3.8}(t),\\
\mathbb{E}_{\Ntop-2}(t) \ls \epd^2 \upmu_{\star}^{-2 \toprate+5.8}(t),\quad \dots,
\end{split}
\end{equation}
where $\mathbb{E}_N$ denotes the energy after $N$ commutations
and all energies are by assumption initially of small size $\epd^2$. 
In other words, 
the energy estimates become less singular by
two powers of $\upmu_{\star}$ for each descent below the top derivative level.
Importantly, despite the possible blowup at higher orders, all the sufficiently 
low-order energies are bounded, which, by Sobolev embedding, 
is what allows one to show the uniform pointwise boundedness of the solution's lower-order derivatives:\footnote{
The lowest order energy $\mathbb{E}_{N}(t)$ is excluded from this estimate because it is not of small size $\epd^2$, 
owing to the largeness of $\bX \mathcal R_{(+)}$.}
\begin{equation}\label{eq:energy.not.blow.up}
\sum_{N = 1}^{\Ntop - \toprate} \mathbb{E}_{N}(t) \ls \epd^2.
\end{equation}
\end{itemize}

\subsubsection{The nearly simple plane-symmetric regime} 
Christodoulou's work \cite{dC2007} concerned compactly supported\footnote{More precisely, his work addressed compactly supported irrotational perturbations of constant, non-vacuum fluid solutions.} 
initial data in $\mathbb{R}^3$,
a regime in which dispersive effects dominate for a long time before the singularity formation processes 
eventually take over.
In a joint work with Holzegel and Wong \cite{jSgHjLwW2016}, 
we adapted the Christodoulou theory to the almost simple plane symmetric regime. 
The important point is that the commutators $\{L,Y,Z\}$, in addition to being regular derivatives near the singularity, also simultaneously capture the fact that the solution is
``almost simple plane symmetric.'' Moreover, the following analytical considerations
were fundamental to the philosophy of the proof in \cite{jSgHjLwW2016}:
\begin{itemize}
\item All energy estimates can be closed by commuting only with tangential derivatives $\{L,Y,Z\}$ (and with\underline{out} $\bX$). 
	This is a slightly different strategy than we used in our paper \cite{LS} in the $2$D case, 
	in which we closed the energy estimates by commuting the equations with strings of tangential derivatives 
	$\{L,Y,Z\}$ as well as strings that contain up to one factor of $\bX$. 
	In \cite{LS}, we also could have closed the energy estimates by commuting only with
	tangential derivatives $\{L,Y,Z\}$, but we would have had to work with the modified fluid variable
	$\CurlofVortrenormalized$ (which, though fundamental in $3$D, was not needed in \cite{LS} 
	due to the absence of the vorticity stretching term)
	or to treat the Cartesian gradients $\partial_{\alpha} \Vortrenormalized^i$ as independent unknowns.
\item After being commuted with (at least one of) $L,Y,Z$, the wave equation solutions are \underline{small}. In particular, we can capture the smallness from ``nearly simple plane-symmetric'' data without explicitly subtracting the simple plane-symmetric background solution; 
see also Remark~\ref{R:NOSUBTRACTBACKGROUND}.
\end{itemize}

\subsubsection{The reformulation of the equations}\label{sec:ideas.reformulation} In order to
extend Christodoulou's theory so that it can be applied to
the compressible Euler equations, a crucial first step
is to reformulate the compressible Euler equations as a system of quasilinear wave equations and transport equations.
Here, the  transport-part of the system refers to the vorticity and the entropy, 
and the intention is to handle them perturbatively.

As we mentioned earlier, the reformulation has been carried out in \cite{LS,jLjS2020a,jS2019c}. Here we highlight the main features and philosophy of the reformulation, and explain how we derived it.
\begin{enumerate}
\item To the extent possible, 
formulate compressible Euler flow as a perturbation of 
a system of quasilinear wave equations.
\begin{enumerate}
\item We compute $\square_g v^i$, $\square_g \rr$, and $\square_g s$,
where $\square_g$ is the covariant 
wave operator associated to the acoustical metric (see \eqref{E:ACOUSTICALMETRIC}). Then using
the compressible Euler equations \eqref{eq:Euler.1}--\eqref{eq:Euler.3}, 
we eliminate and re-express many terms.
\item We find that $v^i$, $\rr$, and $s$ 
do not exactly satisfy wave equations; instead, the 
right-hand sides contain second derivatives of the fluid variables, 
which we will show to be \emph{perturbative}, despite their appearance of being principal order in terms of the number of derivatives.
\end{enumerate}
\item The ``perturbative'' terms mentioned above are equal to 
\emph{good transport variables} that we identify, specifically 
$(\Vr,S,\mathcal{C},\mathcal{D})$.
These variables behave better 
than what one might na\"ively expect, from the points of view of their regularity and their singularity strength.
\begin{enumerate}
\item While both $\Vr^i \doteq \f{(\mathrm{curl} v)^i}{(\varrho/\bar{\varrho})}$ and $S \doteq \nabla s$ are derivatives of the fluid variables, they play a distinguished role since they verify independent transport equations, and obey better bounds than generic first derivatives of the fluid variables.
\item We have introduced the \emph{modified fluid variables} 
$\mathcal{C}^i$ and $\mathcal{D}$ 
(see Definition~\ref{def:variables.HO}),
which, up to lower-order correction terms, are equal to 
$(\mathrm{curl}\, \Vr)^i$ and $\Delta s = \mathrm{div}\,S$ respectively.
These quantities satisfy better estimates than generic first derivatives of $\Vr$ and $S$,
which is crucial for our proof.
\end{enumerate}
\end{enumerate}

\subsubsection{The remarkable null structure of the reformulation}
In the reformulation of compressible Euler flow, we consider
the unknowns to be all of 
$(v^i,\rr,s,\Vr^i,S^i,\mathcal{C}^i,\mathcal{D})$. Note that these include not only the fluid variables,
but also higher-order variables which can be derived from the fluid variables.

The equations satisfied by these variables 
take the following schematic form (see Theorem~\ref{T:GEOMETRICWAVETRANSPORTSYSTEM}
for the precise equations):\footnote{Here, our notation above the brackets 
is such that $\rd (v, \rr) \cdot\rd (v, \rr)$ may contain all of 
$\rd v^i \rd v^j$, $\rd v^i \rd\rr$ and $\rd \rr \rd \rr$. A similar convention applies for other terms.}
\begin{align}
\label{eq:intro.wave} \square_g (v,\rr,s) = &\: \underbrace{\rd (v, \rr) \cdot\rd (v, \rr) }_{\doteq I} 
+ (\Vr, S) \cdot\rd (v, \rr) + (\mathcal{C}, \mathcal{D}), \\
\label{eq:intro.V.S}  B (\Vr, S) = &\: (\Vr, S)\cdot \rd (v,\rr), \\
\label{eq:intro.C.D} B(\mathcal{C}, \mathcal{D}) = &\: \underbrace{\rd (v,\rr) \cdot\rd (\Vr, S)}_{\doteq II} + \underbrace{(\Vr, S) \cdot  \rd (v,\rr) \cdot \rd (v, \rr)}_{\doteq III} + S \cdot S\cdot \rd (v,\rr).
\end{align}
Here, $\square_g$ is the covariant wave operator associated to the acoustical metric (see \eqref{E:ACOUSTICALMETRIC}) and $B \doteq \rd_t + v^a\rd_a$ is the transport operator associated with the material derivative (cf.~\eqref{eq:Euler.1}--\eqref{eq:Euler.3}).

Although it is not apparent from the way we have written it, the system of equations 
\eqref{eq:intro.wave}--\eqref{eq:intro.C.D} has a remarkable null structure! Importantly, the terms $I$, $II$ and $III$ are $g$-null forms: when decomposed in the $\{L, X, Y, Z\}$ frame, we do not have $X (v^i,\rr) \cdot X (v^i,\rr)$ in $I$ and $III$, nor do we have $X(v,\rr) \cdot X(\Vr,S)$ in $II$.

Because $X(v^i,\rr)$ is the only derivative that blows up 
(while $\bX(v^i, \rr)$ is bounded), it follows that given a 
$g$-null form $\mathcal Q$ in the fluid variables
(see Definition~\ref{D:NULLFORMS} concerning $g$-null forms), 
such as $\mathcal Q(\rd v^i, \rd v^j)$, the quantity $\upmu \mathcal Q(\rd v^i, \rd v^j)$ 
remains bounded up to the singularity, while a generic quadratic nonlinearity
$\mathcal Q_{Bad}$ would be such that
$\upmu \mathcal Q_{Bad}(\rd v^i, \rd v^j)$
blows up when $\upmu$ vanishes. 

As is already observed in \cite{jS2016b}, a null form $I$ on the RHS of the wave equation allows all the wave estimates in Section~\ref{sec:ideas.Christodoulou} to be proved. 
As we will discuss below, the null forms $II$ and $III$ in \eqref{eq:intro.C.D} will also be important for estimating the full system.

\subsubsection{Estimates for the transported variables}\label{sec:intro.ideas.transport}
To control solutions to the system \eqref{eq:intro.wave}--\eqref{eq:intro.C.D}, 
we in particular need to estimate the transport variables $(\Vr,S,\mathcal{C},\mathcal{D})$ and understand how they interact with the wave variables 
$(v,\rr,s)$ on LHS~\eqref{eq:intro.wave}.
Here, we will discuss the estimates at the low derivative levels.
We will discuss the difficult technical issues
of a potential loss of derivatives and the blowup of the higher-order
energies
in Sections~\ref{sec:intro.elliptic} and \ref{sec:intro.transport.top.order} respectively.



We begin with two basic --- but crucial --- properties regarding the transport operator 
for the compressible Euler system, which were already observed in \cite{LS}:
\begin{itemize}
\item \underline{The transport vectorfield $B$ is \emph{transversal} to the null hypersurfaces $\mathcal F_u$}; see Figure~\ref{F:FRAME}, where some integral curves of $B$ are depicted.
As a result, one \emph{gains} a power of $\upmu$ by integrating along $B$, 
i.e., for solutions $\phi$ to
$B\phi = \inhom$, we have 
$\|\phi \|_{L^\infty} \ls \| \upmu \inhom \|_{L^\infty}$.
\item \underline{$\upmu B$ is a \emph{regular} vectorfield in the $(t,u,x^2,x^3)$ differential structure.} 
Thus, if $B \phi = \inhom$ and $\upmu \inhom$ has bounded $\{L,Y,Z \}$ derivatives, then $\phi$ also has bounded $\{L,Y,Z \}$ derivatives.
\end{itemize}

We now apply these observations to \eqref{eq:intro.V.S} and \eqref{eq:intro.C.D}:
\begin{itemize}
\item Even though $\rd(v,\rr)$ blows up as the shock forms, 
$\upmu \rd(v,\rr)$ remains regular. This is because $\upmu \rd$ can be written as
a linear combination of the rescaled frame vectorfields $\{\upmu X,L,Y,Z \}$ 
(see Section~\ref{sec:ideas.Christodoulou})
with coefficients that are $\mathcal{O}(1)$ or $\mathcal{O}(\upmu)$.
Hence, the above observations imply that
$(\Vr,S)$ and their $\{L,Y,Z \}$ derivatives are bounded.
\item The null structure and the bounds for the wave variables and $(\Vr,S)$ together imply that 
the RHS of \eqref{eq:intro.C.D} is $\mathcal{O}(\upmu^{-1})$. Thus, 
$\mathcal{C}$, $\mathcal{D}$ and their $\{L,Y,Z \}$ derivatives are also bounded.
\end{itemize}

\subsubsection{Elliptic estimates for the vorticity and the entropy gradient}\label{sec:intro.elliptic}
Despite the favorable structure of \eqref{eq:intro.wave}--\eqref{eq:intro.C.D}, there is apparently a potential loss of derivatives. To see this, consider the following simple derivative count. Suppose we bound 
$(v,\rr,s)$ with $\Ntop +1$ derivatives. \eqref{eq:intro.wave} dictates\footnote{We use here the fact that inverting the wave operator gains one derivative.} that we should control $(\mathcal{C}, \mathcal{D})$ with $\Ntop$ derivatives. If we rely only on \eqref{eq:intro.V.S}, 
then we can only bound $\Ntop$ derivatives of $(\Vr, S)$. 
However, this is insufficient:
plugging this into \eqref{eq:intro.C.D} and using only transport estimates, 
we are only able to control $\Ntop-1$ derivatives of $(\mathcal{C}, \mathcal{D})$, 
which is not enough.

The key to handling this difficulty is the observation that in fact, 
$\mathcal{C}$ and $\mathcal{D}$ 
can be used in conjunction with elliptic estimates to control one derivative of $\Omega$ and $S$. 
This is because up to lower-order terms, 
$\mathcal{C} \approx \mathrm{curl}\, \Vr$ and $\mathcal{D} \approx\mathrm{div}\, S$, while at the same time, 
by definition of $\Vr$ and $S$ --- precisely that $\Vr$ is almost a curl of a vectorfield and $S = \nabla s$ is an exact 
gradient --- $\mathrm{div} \Vr$ and $\mathrm{curl}\, S$ 
are of lower-order in terms of the number of derivatives. It follows that we can control \emph{all}
first-order spatial derivatives 
 of $\Vr$ and $S$, including $\mathcal{C}$ and $\mathcal{D}$,
using elliptic estimates.

\subsubsection{$L^2$ estimates for the transport variables and the high order blowup-rate}
\label{sec:intro.transport.top.order}
We end this section with a few comments on the 
$L^2$ energy estimates for the transport variables $(\Vr,S)$ (and $(\mathcal{C}, \mathcal{D})$){, with 
a focus on how to handle the degeneracies tied to the vanishing of $\upmu$.}

First, due to the eventual 
{vanishing of $\upmu$ and the corresponding}
blowup of the wave variables, we need to {incorporate 
$\upmu$ weights into our analysis of the} transport variables 
$(\Vr,S)$ (and $(\mathcal{C}, \mathcal{D})$){. In particular,
we need to incorporate $\upmu$ weights into the transport equations and energies}
so that the wave terms appearing as inhomogeneous terms in the energy estimates 
{for the transport variables} are regular. 
Importantly, {despite the need to rely on $\upmu$ weights in some parts of the analysis,
the ``transport energy'' that we construct} controls a non-{degenerate} energy flux 
{(i.e., an energy flux without $\upmu$ weights)}
on constant-$u$ hypersurfaces $\mathcal F_u$. That this {energy flux} is bounded can be thought of as another manifestation of the transversality of the transport operator and $\mathcal F_u$. 
{More precisely, with} $\Sigma_t$ denoting constant-$t$ hypersurfaces, we have, roughly,
{$L^2$ estimates} of the following form, where $\mathcal{P}^N$ is an order $N$
differential operator corresponding to repeated differentiation with respect to the $\mathcal F_u$-tangent 
vectorfields $\{L,Y,Z \}$:
\begin{equation}\label{eq:intro.top.order.transport}
\begin{split}
&\: \sup_{t'\in [0,t)}\|\sqrt{\upmu} \mathcal{P}^{N} (\Vr, S)\|_{L^2(\Sigma_{t'})}^2 
+ 
\sup_{u'\in [0,u)} \| \mathcal{P}^N (\Vr, S)\|_{L^2(\mathcal F_{u'})}^2 \\
\ls &\: \mbox{data terms } + \mbox{regular wave terms } + \int_{u'=0}^{u'=u} 
\| \mathcal{P}^N (\Vr, S)\|_{L^2(\mathcal F_{u'})}^2 \, du'.
\end{split}
\end{equation}
Here, the non-degenerate energy flux 
(i.e., the energy along $\mathcal{F}_{u'}$ on LHS~\eqref{eq:intro.top.order.transport}, which does 
not have a $\upmu$ weight)
allows one to absorb the last term 
on RHS~\eqref{eq:intro.top.order.transport}
using Gr\"onwall's inequality\footnote{Our analysis takes place in regions of bounded $u$ width,
so that factors of $e^{Cu}$ 
which arise in our Gr\"onwall estimates can be bounded by a constant.} 
in $u$ (as opposed to Gr\"onwall's inequality in $t$ which has a loss in $\upmu$). For the lower-order energies, the ``regular wave terms'' are indeed bounded (see \eqref{eq:energy.not.blow.up}), 
which in total allows us to prove that the transport energies on LHS~\eqref{eq:intro.top.order.transport}
are also bounded at the lower derivative levels.

Second, since the higher-order energies of the wave variables $(v,\rr,s)$ 
can blow up as $\upmu_{\star}(t) \rightarrow 0$
(even in the absence of inhomogeneous terms; see \eqref{eq:energy.blow.up}), 
\eqref{eq:intro.top.order.transport} allows for the possibility that
the higher-order energies of the transport variables $(\Vr, S)$ (and $(\mathcal{C},\mathcal{D})$) 
might also blow up.
Hence, one needs to verify that there is consistency between the blowup-rates 
(with respect to powers of $\upmu_{\star}^{-1}$) associated to the different kinds of solution variables. That is,
using \eqref{eq:intro.top.order.transport} and the wave energy blowup-rates from
\eqref{eq:energy.blow.up}, one needs to compute the expected blowup-rate of the transport
variables and then plug these back into the energy estimates for the wave variables to 
confirm that the transport terms have an expected singularity strength that is consistent with
wave energy blowup-rates. See, for example, the proof of Proposition~\ref{prop:wave.final}.

Third, due to issues mentioned in Section~\ref{sec:intro.elliptic}, the transport estimates at the top-order are necessarily coupled with elliptic estimates. 
By their nature, the elliptic estimates treat derivatives in all spatial directions on the same footing. This clashes with the philosophy of bounding the solution with respect to the rescaled frame 
(which would mean that derivatives in the $Y$ and $Z$ frame directions should be more regular
than those in the $X$ direction), and it leads to estimates that are 
singular in $\upmu_{\star}^{-1}$. 
To illustrate the difficulties and our approach to overcoming them, we first note that,
suppressing many error terms, we can derive a top-order inequality
of the following form, 
with $\rdb$ denoting Cartesian spatial derivatives and $A$ denoting a constant depending on the equation of state:
\begin{equation}\label{eq:intro.the.difficult.term}
\begin{split}
&\|\sqrt{\upmu}\rdb \mathcal{P}^{\Ntop} (\Vr, S)\|_{L^2(\Sigma_t)} 
	\\
 \leq 
&
C \epd^{\f 32} \upmu^{-2 \toprate+2.8}(t)
+  
A \int_{t'=0}^{t'=t} \upmu_{\star}^{-1}(t') \|\sqrt{\upmu} \rdb \mathcal{P}^{\Ntop} (\Vr, S)\|_{L^2(\Sigma_{t'})} \, dt' 
+ 
\cdots.
\end{split}
\end{equation}

To apply Gr\"onwall's inequality to \eqref{eq:intro.the.difficult.term},
one must quantitatively control the behavior of the crucial ``Gr\"onwall factor''
$\int_{t'=0}^{t'=t} \frac{A}{\upmu_{\star}(t')}\, dt'$.
A fundamental aspect of our analysis is that $\upmu_{\star}(t)$ tends to $0$ 
linearly\footnote{The linear vanishing rate is crucial for the proof of
Proposition~\ref{prop:mus.int} and for the Gr\"onwall-type estimates for the energies
that we carry out in Proposition~\ref{prop:wave.final} and in the appendix.
See \eqref{E:MUSTARFINALEST} for a precise description of how $\upmu_{\star}$ goes to $0$.}  
in $t$ towards the blowup-time.
It follows that one can at best prove an estimate of the form
$\int_{t'=0}^{t'=t} \upmu_{\star}^{-1}(t')\,dt' \ls \log (\upmu_{\star}^{-1})(t)$ 
(recall that $\upmu_{\star}(t) = 
\min \left\lbrace 1, \min_{\Sigma_t} \upmu \right\rbrace$, and
see Proposition~\ref{prop:mus.int} for related estimates). 
Using only this estimate and applying Gr\"onwall's inequality to \eqref{eq:intro.the.difficult.term},
we find (ignoring the error terms ``$\cdots$'') that
$\|\rdb \mathcal{P}^{\Ntop}(\Vr,S)\|_{L^2(\Sigma_t)} \ls 
\epd^{\f 32}
\upmu_{\star}^{-\max\{\mathcal{O}(A),\, 2 \toprate-2.8\}}(t)$. Notice that unless $A$ is small, 
the dominant blowup-rate in the problem would be the one corresponding to these elliptic 
estimates for $(\Vr, S)$, 
which could in principle be much larger than the blowup-rates corresponding
to the irrotational and isentropic case.\footnote{In principle, the largeness of $A$ would not be an obstruction to closing
the estimates. It would just mean that the number of derivatives needed to close the problem would increase
in the presence of vorticity and entropy. We refer readers to the technical estimates
in Section~\ref{sec:sketch.prop.wave} for clarification on the role that the sizes of various constants
play in determining the blowup-rates in the problem as well as the number of derivatives needed
to close the proof.}

However, we can prove a better result: 
we can show that the blowup-rates are not dominated by the
top-order elliptic estimates for the transport variables, 
but rather by the blowup-rates for the wave variables.\footnote{In other words, 
our approach yields the same maximum possible high-order energy
blowup-rates for the wave variables in the general case as it does for irrotational and isentropic solutions.}
The key to showing this is to replace the estimate \eqref{eq:intro.the.difficult.term}
with a related $L^2$ estimate that features \emph{weights in the eikonal function} $u$;
see Proposition~\ref{prop:elliptic}.
Thanks to the $u$ weights, the 
corresponding constant $A$ in this analogue of \eqref{eq:intro.the.difficult.term} 
can be chosen to be arbitrarily \emph{small}, and thus
the main contribution to the blowup-rate comes from the wave variables
error terms, which are present in the ``$\cdots$'' on RHS~\eqref{eq:intro.the.difficult.term}. 
That this can be done is related to the fact that we have good flux estimates for top derivatives of
$\mathcal{C}$ and $\mathcal{D}$ on $\mathcal F_u$.
We refer to Propositions~\ref{prop:C.D.transport.main},
\ref{prop:C.D.elliptic.final}, and \ref{prop:elliptic.putting.everything.together} for the details.

\subsection{Related works}\label{sec:related}

\subsubsection{Shock formation in one spatial dimension} One-dimensional shock formation has a long tradition starting from \cite{bR1860}. See the works of Lax \cite{pL1964},
John \cite{fJ1974}, Liu \cite{tpL1979}, and Christodoulou--Raoul Perez \cite{dCdRP2016},
as well as the surveys \cites{gqCdW2002,cD2010} for details.

\subsubsection{Multi-dimensional shock formation for quasilinear wave equations}\label{sec:related.waves}
Multi-dimensional shock formation for quasilinear wave equations was first proven in Alinhac's groundbreaking
papers \cites{sA1999a, sA1999b, sA2002}. Alinhac's methods allowed him to prove the formation of non-degenerate shock singularities which, roughly speaking, are shock singularities that are isolated within the constant-time hypersurface of first blowup.
The problem was revisited in Christodoulou's monumental 
book \cite{dC2007}, which concerned the quasilinear
wave equations of irrotational and isentropic relativistic fluid mechanics. In this book, Christodoulou
introduced methods that apply to a more general class of shock singularities than the non-degenerate ones treated by Alinhac and, for a large open subset of these solutions, 
are able to yield a complete description of the maximal smooth development, up to the boundary. This was the starting point of his follow-up breakthrough monograph \cite{dC2019} on the 
restricted shock development problem.

For quasilinear wave equations, there are many extensions, variations,
and simplifications of \cite{dC2007}, some of which adapted Christodoulou's geometric framework to other solution regimes. See, for instance, 
\cites{dC2007b,jS2016b,gHsKjSwW2016,jSgHjLwW2016,dCsM2014,sMpY2017,dCaL2016,sM2018}.

\subsubsection{Multi-dimensional shock formation for the compressible Euler equations} 
Multi-dimensional singularity formation for the compressible Euler equations 
without symmetry assumptions
was first discovered by Sideris \cite{tS1985} via an indirect argument. A constructive proof of stable shock formation in a symmetry-reduced regime for which multi-dimensional phenomena (such as dispersion and vorticity) are present was given by Alinhac in \cites{sA1993}. 
See also \cite{tBsSvV2020,tBsSvV2022}.

All the works in Section~\ref{sec:related.waves} on quasilinear wave equations can be used to obtain an analogous result for the compressible Euler equations in the irrotational and isentropic regime,
where the dynamics reduces to a single, scalar quasilinear wave equation for a potential function.
The regime of small, compact, irrotational perturbations of non-vacuum constant fluid
states was treated in Christodoulou's aforementioned breakthrough work \cite{dC2007} in the relativistic case,
and later in \cite{dCsM2014} in the non-relativistic case.

Shock formation beyond the irrotational and isentropic regime was first proven in \cites{LS,jLjS2020a,jS2019c}. These are already discussed above; see Remark~\ref{rmk:theearlier}.

In very interesting recent works \cites{tBsSvV2020,tBsSvV2022}, 
Buckmaster--Shkoller--Vicol provided a philosophically new proof of stable singularity formation 
without symmetry assumptions in $3$D under adiabatic equations of state
in a solution regime 
with vorticity and/or dynamic entropy
for initial data such that precisely one singular point forms
at the first singular time; these are analogs of the non-degenerate singularities that
Alinhac studied \cites{sA1999a,sA1999b,sA2002}
in the case of quasilinear wave equations.
Moreover, in their regime (compare with Remark~\ref{R:LACKOFUNIVERSALBLOWUPPROFILE}), 
they proved that the singularity is a perturbation of a self-similar Burgers shock. See also 
 the $2$D precursor work \cite{tBsSvV2019b} in symmetry, and the recent work \cite{tBsI2022},
which, in $2$D in azimuthal symmetry, constructed a set of shock-forming solutions whose cusp-like spatial
behavior at the singularity is unstable (non-generic).

\subsubsection{Shock development problem}\label{sec:shock.development}
In the $1$D case, the theory of global solutions of small bounded variation (BV) norms \cite{jG1965,aBgCbP2000}
allows one to study solutions that form shocks 
as well as the subsequent interactions of the shocks in the corresponding weak solutions. 
In higher dimensions, the compressible Euler equations are ill-posed in BV spaces \cite{jR1986}.
Nonetheless, in $2$D or $3$D, 
one still hopes to develop a theory that allows one to uniquely extend the solution as a piecewise smooth
weak solution beyond the 
first shock singularity and to prove that the resulting solution has a propagating shock hypersurface. 
This is known as \emph{the shock development problem.}

Even though the shock development problem for the compressible Euler equations in its full generality is open in higher dimensions, 
it has been solved under spherical symmetry in 3D, or in azimuthal symmetry in 2D. See \cite{hcY2004, dCaL2016} and, most recently, \cite{tBtdDsSvV2021}.

In the irrotational and isentropic regime, 
the restricted shock development problem was 
solved in the recent monumental work 
\cite{dC2019} of Christodoulou \emph{without any symmetry assumptions}. 
Here, the word ``restricted'' means that 
the approach of \cite{dC2019} does not exactly construct a weak solution to the compressible Euler equations, but instead yields a weak solution to a closely related hyperbolic PDE system such that the solution was ``forced'' to remain irrotational and isentropic. Nonetheless, this
gives hope that under an arbitrary equation of state for the compressible Euler equations in $3$D,
one could construct a unique weak solution with a propagating shock hypersurface,
starting from the first singular time exhibited in Theorem~\ref{thm:intro.main}.
To solve this problem would in particular require extending the ideas in \cite{dC2019} beyond the irrotational and isentropic regime. 
This is an outstanding open problem.

\subsubsection{Other singularities for the compressible Euler equations}
It has been known since \cites{gG1942,liS1982} that the compressible Euler equations admit self-similar solutions. Recently, this has been revisited by Merle--Raph\"ael--Rodnianski--Szeftel in \cite{fMpRiRjS2019a} to show that singularities more severe than shocks can arise in 
$3$D starting from smooth initial data. See also \cites{fMpRiRjS2019b,fMpRiRjS2022c} for some spectacular applications.

\subsubsection{Singularity formation in related models}
For shock formation results concerning some other multi-speed hyperbolic problems, see \cites{jS2020, jS2018b} by the second author. 

Interestingly, there are also \emph{non-hyperbolic} models with stable self-similar blowup-profiles modeled on a self-similar Burgers shock. Examples include the Burgers equation with transverse viscosity \cite{cCteGnM2018}, the Burgers--Hilbert equations \cite{rY2021}, the fractal Burgers equation \cite{krCrcMVgP2021}, as well as general dispersive or dissipative perturbations of the Burgers equation \cite{sjOfP2021}. See also \cites{cCteGsInM2018, cCteGnM2021}.

\subsubsection{Other works} The framework we introduced in \cites{LS,jLjS2020a,jS2019c} is useful in other low-regularity settings. See for example results on improved regularity for vorticity/entropy in \cite{mDjS2019}, and results on local existence with rough data in 
\cites{mDcLgMjS2022,qW2022,sY2022}.

\subsection{Structure of the paper}
\label{SS:STRUCTUREOFPAPER}
The remainder of the paper is structured as follows. 

Sections~\ref{S:GEOMETRICSETUP}--\ref{S:DATASIZEANDBOOTSTRAP} are introductory sections. We introduce the basic setup in \textbf{Section~\ref{S:GEOMETRICSETUP}}, 
and we define
the norms and energies in \textbf{Section~\ref{S:FORMSANDENERGY}}. 
The setup is similar to the setups in \cites{jSgHjLwW2016,LS}.
Then 
in \textbf{Section~\ref{S:DATASIZEANDBOOTSTRAP}}, we 
state our precise assumptions on the initial data and
give a precise statement 
of our main results, which we split into several theorems and corollaries.

In \textbf{Section~\ref{sec:speck.reformulation}}, we recall the results of \cite{jS2019c} on the reformulation of the equations{,} which is important for the remainder of the paper.

The bulk of paper is devoted to proving the main a priori estimates{,
which we state in \textbf{Section~\ref{sec:bootstrap}} as Theorem~\ref{thm:bootstrap}. The proof
of Theorem~\ref{thm:bootstrap}, which we provide in Section~\ref{sec:everything},
relies on a set of bootstrap assumption that we also state in Section~\ref{sec:bootstrap}.
Next,} 
after an easy (but crucial) finite speed of propagation argument in \textbf{Section~\ref{sec:trivial}}, 
in \textbf{Section~\ref{sec:geometry}},
we cite {various straightforward pointwise and $L^{\infty}$
estimates for geometric quantities found in \cite{jSgHjLwW2016}, 
and we complement these results with a few related ones that allow 
us to handle the transport variables.}

We then turn to the main estimates in this paper. 
In \textbf{Section~\ref{sec:transport.easy.main}}, we carry out the transport estimates{, specifically $L^{\infty}$ estimates and energy estimates,} for $\Vr$, $S$ and their derivatives. 
In \textbf{Section~\ref{sec:transport.harder}}, we prove {analogous} transport estimates for $\mathcal{C}$, $\mathcal{D}$, 
and their derivatives, except {we delay the proof of the top-order estimates until the next section.} 
In \textbf{Section~\ref{sec:transport.hard}}, 
we derive the top-order estimates for $\mathcal{C}$ and $\mathcal{D}$, 
which, as we described in Section~\ref{sec:intro.transport.top.order}, 
requires elliptic estimates in addition to transport estimates.
In total, these estimates for the transport variables can be viewed as the main
new contribution of the paper.

Next, in \textbf{Section~\ref{sec:wave.main.est}}, we derive energy estimates
for the fluid wave variables. For convenience, 
we have organized the wave equation estimates 
so that they rely on an auxiliary proposition, namely Proposition~\ref{prop:wave}, that 
provides estimates for solutions to the fluid wave equations in terms of 
various norms of their inhomogeneous terms, which for purposes of the 
proposition, we simply denote by ``$\mathfrak{G}$.''
To prove the final a priori energy estimates for the wave equations,
which are located in Proposition~\ref{prop:wave.final},
we must use the bounds for $\mathfrak{G}$ that
we obtained in the previous sections, including the bounds
for the transport variables.
Since the auxiliary Proposition~\ref{prop:wave} does not rely on
the precise structure of $\mathfrak{G}$,
it can be proved using essentially same arguments that have been used in previous works on shock formation
for wave equations.
For this reason, and to aid the flow of the paper, we delay the proof of
Proposition~\ref{prop:wave} until the appendix.

Next, in \textbf{Section~\ref{sec:Linfty}}, 
we {use the energy estimates to derive} 
$L^\infty$ estimates {for the wave variables}.
In particular, these estimates
yield improvements of the $L^{\infty}$ bootstrap assumptions that
we made in Section~\ref{sec:bootstrap}.

In \textbf{Section~\ref{sec:everything}}, we combine the results of the previous
sections to provide the proof of the main a priori estimates as well as the main theorems
and their corollaries.

Finally, in \textbf{Appendix~\ref{app:elliptic}}, 
we provide the details behind the proof of the auxiliary Proposition~\ref{prop:wave}.
The proof relies on small modifications to the proofs of \cites{jSgHjLwW2016,LS}
that account for the third spatial dimension (note that $3$D wave equations were also handled in \cites{dC2007,jS2016b})
as well as the presence of the inhomogeneous terms $\mathfrak{G}$ in the wave equations.

\section{Geometric setup}
\label{S:GEOMETRICSETUP}
In this section, we construct most of the geometric objects
that we use to study shock formation and exhibit their basic
properties. 

\subsection{Notational conventions and remarks on constants}
\label{SS:NOTATION}
The precise definitions of some of the concepts referred to
here are provided later in the article.

\begin{itemize}
	\item Lowercase Greek spacetime indices 
	$\alpha$, $\beta$, etc.\ correspond to the Cartesian spacetime coordinates 
	(see Section~\ref{sec:ambient})
	and vary over $0,1,2,3$.
	Lowercase Latin spatial indices
	$a$,$b$, etc.\ correspond to the Cartesian spatial coordinates and vary over $1,2,3$.
	Uppercase Latin spatial indices $A$,$B$, etc.~correspond to the coordinates on $\ell_{t,u}$ and vary over $2,3$.
	All lowercase Greek indices are lowered and raised with the acoustical metric
	$g$ and its inverse $g^{-1}$, and \emph{not with the Minkowski metric}.
	We use Einstein's summation convention in that repeated indices are summed.
\item $\cdot$ denotes the natural contraction between two tensors. 
		For example, if $\xi$ is a spacetime one-form and $V$ is a 
		spacetime vectorfield,
		then $\xi \cdot V \doteq \xi_{\alpha} V^{\alpha}$.
\item If $\xi$ is an $\ell_{t,u}$-tangent one-form
	(as defined in Section~\ref{SS:PROJECTIONTENSORFIELDANDPROJECTEDLIEDERIVATIVES}),
	then $\xi^{\#}$ denotes its $\gsphere$-dual vectorfield,
	where $\gsphere$ is the Riemannian metric induced on $\ell_{t,u}$ by $g$.
	Similarly, if $\xi$ is a symmetric type $\binom{0}{2}$ $\ell_{t,u}$-tangent tensor, 
	then $\xi^{\#}$ denotes the type $\binom{1}{1}$ $\ell_{t,u}$-tangent tensor formed by raising one index with $\ginversesphere$
	and $\xi^{\# \#}$ denotes the type $\binom{2}{0}$ $\ell_{t,u}$-tangent tensor formed by raising both indices with $\ginversesphere$.
\item If $V$ is an $\ell_{t,u}$-tangent vectorfield,
	then $V_{\flat}$ denotes its $\gsphere$-dual one-form.
\item If $V$ and $W$ are vectorfields, then $V_W \doteq V^{\alpha} W_{\alpha} = g_{\alpha \beta} V^{\alpha} W^{\beta}$.
\item If $\xi$ is a one-form and $V$ is a vectorfield, then $\xi_V \doteq \xi_{\alpha} V^{\alpha}$.
	We use similar notation when contracting higher-order tensorfields against vectorfields.
	For example, if $\xi$ is a type $\binom{0}{2}$ tensorfield and
	$V$ and $W$ are vectorfields, then $\xi_{VW} \doteq \xi_{\alpha \beta} V^{\alpha} W^{\beta}$.
\item Unless otherwise indicated, 
	all quantities in our estimates that are not explicitly under
	an integral are viewed as functions of 
	the geometric coordinates $(t,u,x^2,x^3)$.
	Unless otherwise indicated, integrands have the functional dependence established below in
	Definition~\ref{D:NONDEGENERATEVOLUMEFORMS}.
\item
	$[Q_1,Q_2] = Q_1 Q_2 - Q_2 Q_1$ denotes the commutator of the operators
	$Q_1$ and $Q_2$.
\item $A \lesssim B$ means that there exists $C > 0$ such that $A \leq C B$.
		$A \approx B$ means that $A \lesssim B$ and $B \lesssim A$.
		$A = \mathcal{O}(B)$ means that $|A| \lesssim |B|$.
\item The constants $C$ are free to vary from line to line.
	\textbf{These constants, and implicit constants as well, 
	are allowed 
	to depend on the equation of state, the background $\bar{\varrho}$,
	the maximum number of times $\Ntop$ that we commute the equations,
	and the parameters 
	$\mathring{\upsigma}$, 
	$\Trandatasize$
	and 
	$\TranminusdatasizeWithFactor^{-1}$}
	from
	Section~\ref{SS:FLUIDVARIABLEDATAASSUMPTIONS}.
\item Constants $C_{\mydiam}$ are also allowed to vary from line to line, but
		unlike $C$, the $C_{\mydiam}$ are 
		\textbf{only allowed to depend on the equation of state and the background $\bar{\varrho}$}.
\item In the appendix, 
there appear \textbf{absolute constants} $M_{\mathrm{abs}}$,
which can be chosen to be independent of the equation of state and all other parameters in the problem.
\item For our proof to close, the high-order energy blowup-rate parameter $\toprate$ needs to be chosen to be 
large in a manner that depends only on $M_{\mathrm{abs}}$; hence, $\toprate$ can also be chosen to be
an absolute constant. 
\item The integer $\Ntop$ denotes the maximum number of times we need to commute the equations
	to close the estimates. For our proof to close, $\Ntop$ 
	needs to be chosen to be 
	large in a manner that depends only on $\toprate$. $\Ntop$ could be chosen to be an absolute
	constant, but we choose to think of it as a parameter that we are free to adjust so that
	we can study solutions with arbitrary sufficiently large regularity.
\item For our proof to close, the data-size parameters $\mathring{\upalpha}$ and $\epd$ must be chosen
	to be sufficiently small, where the required smallness is clarified in Theorem~\ref{thm:bootstrap}.
	We always assume that $\epd^{\frac{1}{2}} \leq \mathring{\upalpha}$.
\item $A \ls_{\mydiam} B$ means that $A \leq C_{\mydiam} B$, with $C_{\mydiam}$ as above.
	Similarly, $A = \mathcal{O}_{\mydiam}(B)$ means that $|A| \leq C_{\mydiam} |B|$.
\item For example, $\TranminusdatasizeWithFactor^{-2} = \mathcal{O}(1)$,
		$2 + \Psiep + \Psiep^2 = \mathcal{O}_{\mydiam}(1)$,
		$\Psiep \epd = \mathcal{O}(\epd)$,
		$C_{\mydiam} \Psiep^2 =  \mathcal{O}_{\mydiam}(\Psiep)$,
		$N! \epd = \mathcal{O}(\epd)$,
		and $C \Psiep = \mathcal{O}(1)$. Some of these examples are non-optimal;
		e.g., we actually have $\Psiep \epd = \mathcal{O}_{\mydiam}(\epd)$.
\item $\lfloor \cdot \rfloor$
	and $\lceil \cdot \rceil$
	respectively denote the standard floor and ceiling functions. 
\end{itemize}

\subsection{Caveats on citations}
\label{SS:REMARKSONCITATIONSOFEARLIERWORK}
Before we introduce our geometric setup, we should say that our setup is essentially the same as that in 
\cites{jSgHjLwW2016,LS}, 
except for some small differences. We will therefore cite whenever possible the computations in 
\cites{jSgHjLwW2016,LS}, except 
we will need to take into account the following differences:
\begin{itemize}
	\item The work \cite{jSgHjLwW2016} allows for very general metrics, 
	while in the present paper,
	we are only concerned with the acoustical metric for the compressible Euler equations.
	In citing \cite{jSgHjLwW2016}, we sometimes adjust formulas
	to take into account the explicit
	form of the Cartesian metric components
	$g_{\alpha \beta}$  stated in Definition~\ref{D:ACOUSTICALMETRIC}.
	\item The papers \cites{jSgHjLwW2016,LS} concern two spatial dimensions (with ambient manifold $\Sigma = \mathbb R\times \mathbb T$), while in the present paper, we are concerned with three spatial dimensions (with $\Sigma = \mathbb R\times \mathbb T^2$).
	\item 
	In \cite{jSgHjLwW2016},
	the metric components $g_{\alpha \beta}$ were functions of a scalar function $\Psi$,
	as opposed to the array $\threePsi$ (defined in \eqref{E:THREEVECPSI}).
	For this reason, we must make minor adjustments 
	to many of the formulas
	from \cite{jSgHjLwW2016} 
	to account for the fact that in the present article, $\threePsi$ is an array.
\end{itemize}
In all cases, our minor adjustments can easily be verified by examining the proof in \cite{jSgHjLwW2016}.

\subsection{Basic setup and ambient manifold}\label{sec:ambient}
We recall again the setup from the introduction. 
We will work on the spacetime manifold $I\times \Sigma$ (with $I \subseteq \mathbb R$ 
a time 
interval and $\Sigma \doteq \mathbb R\times \mathbb T^2$ the spatial domain). We fix a standard Cartesian coordinate system 
$\lbrace x^{\alpha} \rbrace_{\alpha = 0,1,2,3}$
on $I \times \Sigma$, where $t \doteq x^0 \in I$
is the time coordinate and $x \doteq (x^1,x^2,x^3) \in \mathbb{R} \times \mathbb{T}^2$
are the spatial coordinates\footnote{While the coordinates $x^2,x^3$ on $\mathbb{T}^2$ are only locally defined, the corresponding partial derivative vectorfields $\partial_2,\partial_3$ can be extended
so as to form a global smooth frame on $\mathbb{T}^2$. Similar remarks apply to the one-forms $dx^2,dx^3$
These simple observations are relevant for this paper because when we derive estimates, the coordinate functions $x^2,x^3$
themselves are never directly relevant; what matters are estimates for the components of various tensorfields
with respect to the frame $\lbrace \partial_t, \partial_1, \partial_2, \partial_3 \rbrace$
and the basis dual co-frame $\lbrace dt, dx^1, dx^2, dx^3 \rbrace$, which are everywhere smooth.} \label{FN:COORDINATESARENOTGLOBAL}. 
We use the notation $\lbrace \rd_\alpha \rbrace_{\alpha = 0,1,2,3}$ 
(or $\rd_t \doteq \rd_0$) to denote the Cartesian coordinate partial derivative
vectorfields.

In this coordinate system, the \emph{plane-symmetric} solutions are exactly those whose fluid variables are independent of $(x^2,x^3)$.

\subsection{Fluid variables and new variables useful for the reformulation}
As we already discussed in Section~\ref{sec:ideas.reformulation}, 
at the heart of our approach is a reformulation of the compressible Euler equations
in terms of new variables. We introduce these new variables in this subsection; see
Definitions~\ref{def:variables.fluid} and \ref{def:variables.HO}.

The basic fluid variables are $(\varrho, v^i,s)$ (see the introduction).
We fix an equation of state $p = p(\varrho,s)$ and a constant $\bar{\varrho}>0$ 
 such 
that
$p_{;\varrho}(\bar{\varrho},0) = 1$.

\begin{definition}\label{def:rr.c}
Define the \textbf{logarithmic density} $\rr$ and the \textbf{speed of sound} $\Speed(\rr,s)$ by:
$$\rr = \log \left( \f{\varrho}{\bar{\varrho}}\right),\quad \Speed(\rr,s) = \sqrt{\f{\rd p}{\rd\varrho}}(\varrho,s).$$
\end{definition}

\begin{remark}
As is suggested by our notation, we will consider $\Speed(\rr,s)$ as a function of $(\rr,s)$. 
The normalization of $p_{;\varrho}$ that we stated above is equivalent to:
\begin{equation}\label{eq:c.normalization}
c(0,0) = 1.
\end{equation}
\end{remark}

\begin{definition}[\textbf{The fluid variables arrays}]\label{def:variables.fluid}
\

\begin{enumerate}
\item Define the \textbf{almost Riemann invariants}\footnote{$\mathcal R_{(\pm)}$ coincide with the well-known \emph{Riemann invariants} in the plane-symmetric isentropic case.
Even though they are no longer ``invariant'' in our case, they are useful in capturing smallness.} $\mathcal{R}_{(\pm)}$ as follows (recall Definition~\ref{def:rr.c}):
\begin{align} \label{E:RIEMANNINVARIANTS}
	\mathcal{R}_{(\pm)}
	& \doteq v^1 \pm F(\Densrenormalized,\Ent),
	\quad F(\Densrenormalized,\Ent) \doteq \int_0^{\rr} c(\rr',s) \, d\rr'.
\end{align}
\item Define the \textbf{array of wave variables}:\footnote{Throughout, we consider $\threePsi$ as an array of
	scalar functions; we will not attribute any 
	tensorial structure to the labeling index $\imath$ of $\Psi_{\imath}$
	besides simple contractions, denoted by $\contr$,
	corresponding to the chain rule; see Definition~\ref{D:ARRAYNOTATION}.}
		\begin{align} \label{E:THREEVECPSI}
			\threePsi
			\doteq 
			(\Psi_1,\Psi_2,\Psi_3,\Psi_4,\Psi_5)
			\doteq
			(\mathcal{R}_{(+)},
			 \mathcal{R}_{(-)},
				v^2,v^3,s).
			\end{align}	
\end{enumerate}
\end{definition}

\begin{remark}
	We sometimes use the simpler notation ``$\Psi$'' in place of ``$\threePsi$'' when there
	is no danger of confusion. At other times, we use the notation ``$\Psi$'' to denote a generic
	element of $\threePsi$. The precise meaning of the symbol ``$\Psi$'' will be clear from context.
\end{remark}

\begin{remark}[\textbf{Clarification on our approach to estimating $\Densrenormalized$ and $v^1$}]
	\label{R:HOWWEESTIMATEDENSITYANDV1}
	Recall that we have introduced $\mathcal{R}_{(\pm)}$ to allow us to
	capture the fact that our solutions are perturbations of simple plane 
	waves (for which only $\mathcal{R}_{(+)}$ is non-vanishing). 
	In the $1$D isentropic case, $\lbrace \mathcal{R}_{(+)}, \mathcal{R}_{(-)} \rbrace$ can be taken
	to be the unknowns in place of $\lbrace \Densrenormalized, v^1 \rbrace$.
	A similar remark holds in the present $3$D case as well, 
	provided we take into account the entropy.
	Specifically, from \eqref{eq:c.normalization} and Definition~\ref{def:variables.fluid},
	it follows that $v^1 = \frac{1}{2}(\mathcal{R}_{(+)} + \mathcal{R}_{(-)})$,
	and that when $\Densrenormalized$, $v^1$, and $s$ are sufficiently small
	(as is captured by the smallness parameters $\mathring{\upalpha}$ and $\epd$
	described at the beginning of Section~\ref{SS:FLUIDVARIABLEDATAASSUMPTIONS}),
	we have (via the implicit function theorem)
	$\Densrenormalized = 
	(\mathcal{R}_{(+)} - \mathcal{R}_{(-)})
	\cdot
	\widetilde{F}(\mathcal{R}_{(+)} - \mathcal{R}_{(-)},s)$,
	where $\widetilde{F}$ is a smooth function.
	This allows us to control $\Densrenormalized$ and $v^1$ in terms of 
	$\mathcal{R}_{(+)}$, $\mathcal{R}_{(-)}$, and $s$. Throughout the article,
	we use this observation without explicitly pointing it out.
	In particular, even though many of the equations we cite explicitly involve 
	$\Densrenormalized$ and $v^1$, it should be understood that we always
	estimate these quantities in terms of the wave variables
	$\mathcal{R}_{(+)}$, $\mathcal{R}_{(-)}$, and $s$,
	which are featured in the array \eqref{E:THREEVECPSI}.
\end{remark}

\begin{definition}[\textbf{Euclidean divergence and curl}]
Denote by\footnote{This is in contrast to $\angdiv$; see Definition~\ref{D:CONNECTIONS}.} $\mathrm{div}$ and $\mathrm{curl}$ the Euclidean spatial divergence and curl operator. That is,
given a $\Sigma_t$-tangent 
vectorfield $V = V^a \rd_a$, define:
\begin{align} \label{E:FLATDIVANDCURL}
		\Flatdiv V
		& \doteq \partial_a V^a,
			\qquad
		(\Flatcurl V)^i
		\doteq \epsilon_{iab} \partial_a V^b,
	\end{align}
	where $\epsilon_{iab}$ is the fully antisymmetric symbol normalized by $\epsilon_{123} = 1$.
\end{definition}

\begin{definition}[\textbf{The higher-order variables}]\label{def:variables.HO}
\

\begin{enumerate}
\item Define the \textbf{specific vorticity} 
	to be the $\Sigma_t$-tangent vectorfield with the following Cartesian spatial components:
$$\Vr^i \doteq \f{(\mathrm{curl}\,v)^i}{\varrho/\bar{\varrho}} = \f{(\mathrm{curl}\,v)^i}{\exp(\rr)}.$$
\item Define the \textbf{entropy gradient} 
	to be the $\Sigma_t$-tangent vectorfield with the following Cartesian spatial components:
$$\GradEnt^i \doteq \rd_i s.$$
\item Define the \textbf{modified fluid variables} by:
\begin{subequations}
	\begin{align} \label{E:RENORMALIZEDCURLOFSPECIFICVORTICITY}
		\CurlofVortrenormalized^i
		& \doteq
			\exp(-\Densrenormalized) (\Flatcurl \Vortrenormalized)^i
			+
			\exp(-3\Densrenormalized) \Speed^{-2} \frac{p_{;\Ent}}{\bar{\varrho}} \GradEnt^a \partial_a v^i
			-
			\exp(-3\Densrenormalized) \Speed^{-2} \frac{p_{;\Ent}}{\bar{\varrho}} (\partial_a v^a) \GradEnt^i,
				\\
		\DivofEntrenormalized
		& \doteq
			\exp(-2 \Densrenormalized) \Flatdiv \GradEnt 
			-
			\exp(-2 \Densrenormalized) \GradEnt^a \partial_a \Densrenormalized.
			\label{E:RENORMALIZEDDIVOFENTROPY}
	\end{align}
	\end{subequations}
We think of $\CurlofVortrenormalized$ as a
$\Sigma_t$-tangent vectorfield with Cartesian spatial components given by 
\eqref{E:RENORMALIZEDCURLOFSPECIFICVORTICITY}.
\end{enumerate}

\end{definition}

\subsection{The acoustical metric and related objects in Cartesian coordinates}
\label{SS:GEOMETRICTENSORFIELDSASSOCITEDTOTHEFLOW}
Hidden within compressible Euler flow lies
a geometric structure captured by the acoustical metric, which governs the dynamics of the sound waves.
We introduce in this subsection the acoustical metric $g$ in Cartesian coordinates.

\begin{definition}[\textbf{Material derivative vectorfield}]
\label{D:MATERIALVECTORVIELDRELATIVETORECTANGULAR}
 We define the \emph{material derivative vectorfield}
 as follows relative to the Cartesian coordinates:
\begin{align} \label{E:MATERIALVECTORVIELDRELATIVETORECTANGULAR}
	\Transport 
	\doteq \partial_t + v^a \partial_a.
\end{align}
\end{definition}

\begin{definition}[\textbf{The acoustical metric}] 
\label{D:ACOUSTICALMETRIC}
Define the \emph{acoustical metric} $g$ 
(in Cartesian coordinates) by:
	\begin{align}
		g 
		\doteq
		-  dt \otimes dt
			+ 
			\Speed^{-2} \sum_{a=1}^3(dx^a - v^a dt) \otimes (dx^a - v^a dt).
				\label{E:ACOUSTICALMETRIC} 
	\end{align}
\end{definition}

The following lemma follows from straightforward computations.
\begin{lemma}[\textbf{The inverse acoustical metric}] 
The inverse of the acoustical metric $g$ from \eqref{E:ACOUSTICALMETRIC}
can be expressed as follows:
\begin{align}
g^{-1} 
		= 
			- \Transport \otimes \Transport
			+ \Speed^2 \sum_{a=1}^3 \partial_a \otimes \partial_a.
			\label{E:INVERSEACOUSTICALMETRIC}
	\end{align}

\end{lemma}


\begin{remark}[\textbf{Closeness to the Minkowski metric}]
In our analysis, $v$ and $c-1$ will be small,
where the smallness is
captured by the parameters $\mathring{\upalpha}$ and $\epd$
described at the beginning of Section~\ref{SS:FLUIDVARIABLEDATAASSUMPTIONS}.
Recalling \eqref{E:ACOUSTICALMETRIC}, we see that $g$ will be 
$L^\infty$-close to the Minkowski metric.
It is therefore convenient to introduce the following decomposition:
\begin{align} \label{E:LITTLEGDECOMPOSED}
	g_{\alpha \beta}(\threePsi)
	& = m_{\alpha \beta} 
		+ g_{\alpha \beta}^{(Small)}(\threePsi),
	&& m_{\alpha \beta} \doteq \mbox{\upshape diag}(-1,1,1,1),
\end{align}
where $m$
is the Minkowski metric and
$g_{\alpha \beta}^{(Small)}(\threePsi)$ is a smooth function of $\threePsi$
such that:
\begin{align} \label{E:METRICPERTURBATIONFUNCTION}
	g_{\alpha \beta}^{(Small)}(\threePsi = 0)
	& = 0.
\end{align}
\end{remark}


\begin{definition}[\textbf{$\threePsi$-derivatives of $g_{\alpha \beta}$}]
	\label{D:BIGGANDBIGH}
	For $\alpha,\beta = 0,\dots,3$ and $\imath = 1,\dots,5$, we define:
	\begin{align} \label{E:BIGGDEF} 
		G_{\alpha \beta}^{\imath}(\threePsi)
		& \doteq \frac{\partial}{\partial \Psi_{\imath}} 
					g_{\alpha \beta}(\threePsi),
		&
		\vec{G}_{\alpha \beta}
		=
		\vec{G}_{\alpha \beta}(\threePsi)
		&\doteq 
		\left(
			G_{\alpha \beta}^1(\threePsi),
			G_{\alpha \beta}^2(\threePsi),
			G_{\alpha \beta}^3(\threePsi),
			G_{\alpha \beta}^4(\threePsi),
			G_{\alpha \beta}^5(\threePsi)
		\right).
	\end{align}
\end{definition}
For each fixed $\imath \in \lbrace 1,\dots,5 \rbrace$, we think of
$\lbrace G_{\alpha \beta}^{\imath} \rbrace_{\alpha,\beta = 0,\dots,3}$,
as the Cartesian components of a spacetime tensorfield. 
Similarly, we think of
$\lbrace \vec{G}_{\alpha \beta} \rbrace_{\alpha,\beta = 0,\dots,3}$
as the Cartesian components of an array-valued spacetime tensorfield.

\begin{definition}[\textbf{Operators involving $\threePsi$}]
	\label{D:ARRAYNOTATION}
	Let $U_1,U_2,V$ be vectorfields. We define:
	\begin{align} 
		V \threePsi
		& \doteq 
		(V \Psi_1, V \Psi_2, V \Psi_3, V \Psi_4, V \Psi_5),
	&
	\vec{G}_{U_1 U_2}\contr V \threePsi
	& \doteq \sum_{\imath=1}^5 G_{\alpha \beta}^{\imath}U_1^{\alpha} U_2^{\beta} V \Psi_{\imath}.
	\end{align}

	We use similar notation with other differential operators in place of vectorfield differentiation.
	For example, 
	$\vec{G}_{U_1 U_2} \contr\angLap \threePsi 
	\doteq \sum_{\imath=1}^5 G_{\alpha \beta}^{\imath} U_1^{\alpha} U_2^{\beta} \angLap \Psi_{\imath}
	$
	(where $\angLap$ is defined in Definition~\ref{D:CONNECTIONS}).

\end{definition}

\subsection{The acoustic eikonal function and related constructions}
\label{SS:EIKONALFUNCTIONANDRELATED}
To control the solution up to the shock, 
we will crucially rely on an eikonal function
for the acoustical metric.

\begin{definition}[\textbf{Acoustic eikonal function}]
\label{D:INTROEIKONAL}
The acoustic eikonal function (eikonal function for short) $u$ 
solves the following eikonal equation initial value problem:
\begin{align} \label{E:INTROEIKONAL}
	(g^{-1})^{\alpha \beta}
	\partial_{\alpha} u \partial_{\beta} u
	& = 0, 
	\qquad \partial_t u > 0,
	\qquad
	u\restriction_{t=0}
	= \mathring{\upsigma} - x^1,
\end{align}
where $\mathring{\upsigma}>0$ is the constant controlling the initial support (recall Theorem~\ref{thm:intro.main}).
\end{definition}

\begin{definition}[\textbf{Inverse foliation density}]
\label{D:FIRSTUPMU}
Define the inverse foliation density $\upmu$ by:
\begin{align} \label{E:FIRSTUPMU}
	\upmu 
	& \doteq \frac{-1}{(g^{-1})^{\alpha \beta}(\threePsi) \partial_{\alpha} t \partial_{\beta} u} > 0.
\end{align}
\end{definition}

$1/\upmu$ measures the density of the level sets of $u$
relative to the constant-time hypersurfaces $\Sigma_t$. 
For the data that we will consider,
we have $\upmu\restriction_{\Sigma_0} \approx 1$.
When $\upmu$ vanishes, the level sets of $u$ intersect
and, as it turns out,
$\max_{\alpha = 0,1,2,3}|\partial_{\alpha} u|$
and
$\max_{\alpha = 0,1,2,3}|\partial_{\alpha} \mathcal{R}_{(+)}|$, 
blow up.

The following quantities, tied to $\upmu$, play an important role in our description of the singular behavior of our
high-order energies.
\begin{definition}\label{def:mustar}
Define $\upmu_{\star}(t,u)$ and $\upmu_{\star}(t)$ by:\footnote{By definition, $\upmu_{\star}(t,u)\geq \upmu_{\star}(t)$, $\forall u\in \mathbb R$. Note that by the localization lemma (Lemma~\ref{lem:localization}) we prove below, we have $\upmu_{\star}(t) = \upmu_{\star}(t,U_0)$.  In most of the proof, it suffices to consider the function $\upmu_{\star}(t)$ without considering $\upmu_{\star}(t,u)$. The more refined definition for $\upmu_{\star}(t,u)$ will only be referred to in the appendix, so that the formulas take the same forms as their counterparts in \cites{jSgHjLwW2016,LS}.}
$$\upmu_{\star}(t,u) \doteq \min \left\lbrace 1, \min_{u' \leq u} \upmu(t,u') \right\rbrace,
\quad \upmu_{\star}(t) \doteq \min \left\lbrace 1, \min_{\Sigma_t} \upmu \right\rbrace.$$
\end{definition}

\subsection{Subsets of spacetimes}

\begin{definition} [\textbf{Subsets of spacetime}]
\label{D:HYPERSURFACESANDCONICALREGIONS}
For $0 \leq t'$ and $0 \leq u'$,
define: 
\begin{subequations}
\begin{align}
	\Sigma_{t'} & \doteq \lbrace (t,x) \in \mathbb{R} \times (\mathbb{R} \times \mathbb{T}^2 )  
		\ | \ t = t' \rbrace, 
		\label{E:SIGMAT} \\
	\Sigma_{t'}^{u'} & \doteq  \lbrace (t,x) \in \mathbb{R} \times (\mathbb{R} \times \mathbb{T}^2 ) 
		 \ | \ t = t', \ 0 \leq u(t,x) \leq u' \rbrace, 
		\label{E:SIGMATU} 
		\\
	\mathcal{F}_{u'}
	& \doteq 
		\lbrace (t,x) \in \mathbb{R} \times (\mathbb{R} \times \mathbb{T}^2)
			\ | \ u(t,x) = u' 
		\rbrace, 
		\label{E:PU} \\
	\mathcal{F}_{u'}^{t'} 
	& \doteq 
		\lbrace (t,x) \in \mathbb{R} \times (\mathbb{R} \times \mathbb{T}^2 )
			\ | \ 0 \leq t \leq t', \ u(t,x) = u' 
		\rbrace, 
		\label{E:PUT} \\
	\ell_{t',u'} 
		&\doteq  \mathcal{F}_{u'}^{t'} \cap \Sigma_{t'}^{u'}
		= \lbrace (t,x) \in \mathbb{R} \times (\mathbb{R} \times \mathbb{T}^2 )
			\ | \ t = t', \ u(t,x) = u' \rbrace, 
			\label{E:LTU} \\
	\mathcal{M}_{t',u'} & \doteq  \cup_{u \in [0,u']} \mathcal{F}_u^{t'} \cap \lbrace (t,x) 
		\in \mathbb{R} \times (\mathbb{R} \times \mathbb{T}^2 ) \ | \ 0 \leq t < t' \rbrace.
		\label{E:MTUDEF}
\end{align}
\end{subequations}
\end{definition}
\begin{center}
\begin{overpic}[scale=.15]{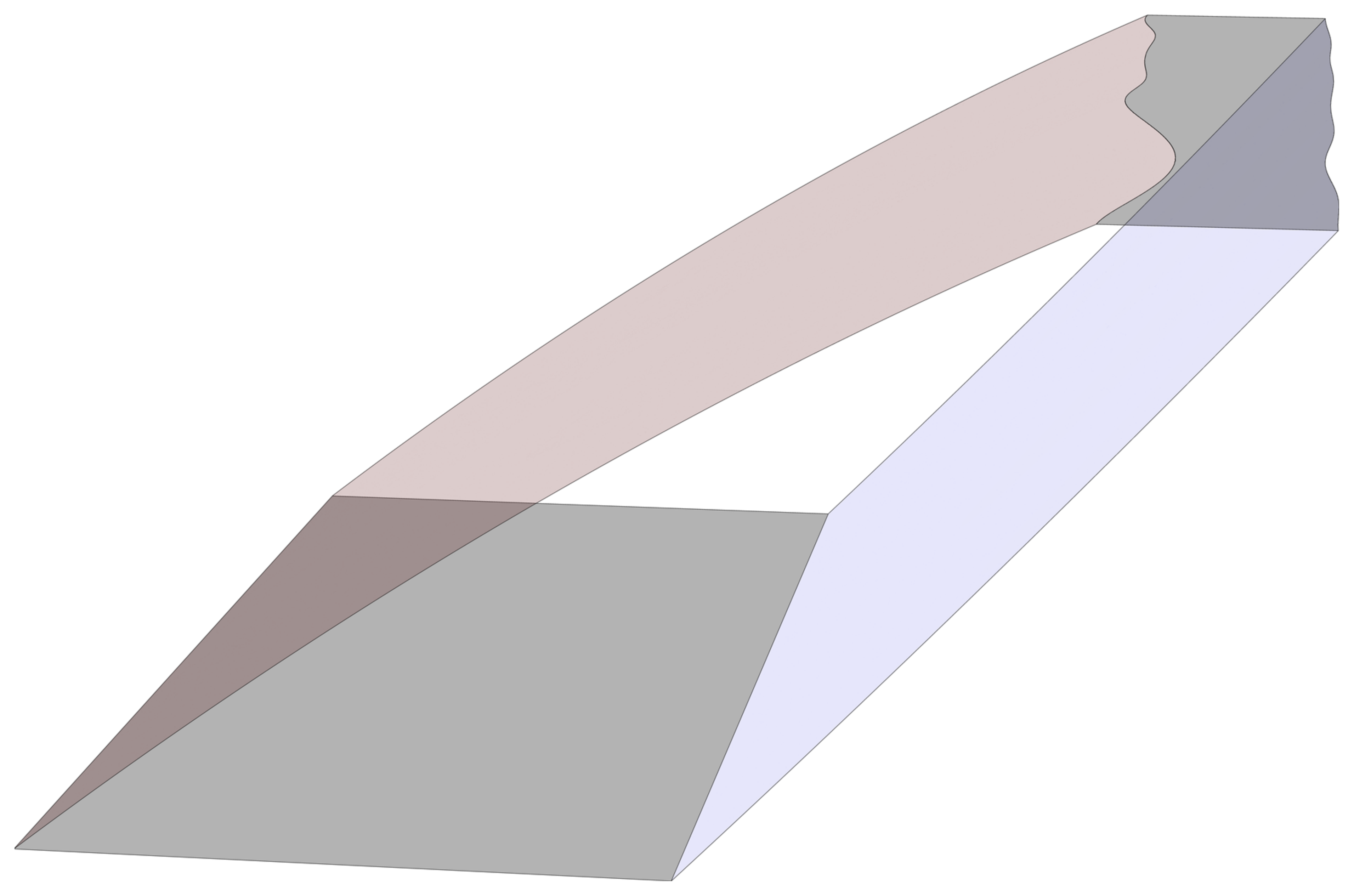} 
\put (54,34) {$\displaystyle \mathcal{M}_{t,u}$}
\put (38,33) {$\displaystyle \mathcal{F}_u^t$}
\put (74,37) {$\displaystyle \mathcal{F}_0^t$}
\put (32,17) {$\displaystyle \Sigma_0^u$}
\put (48.5,13) {$\displaystyle \ell_{0,0}$}
\put (12,13) {$\displaystyle \ell_{0,u}$}
\put (89.5,61) {$\displaystyle \Sigma_t^u$}
\put (92.8,56) {$\displaystyle \ell_{t,0}$}
\put (86,56) {$\displaystyle \ell_{t,u}$}
\put (-11,16) {$\displaystyle (x^2,x^3) \in \mathbb{T}^2$}
\put (22,-4) {$\displaystyle x^1 \in \mathbb{R}$}
\thicklines
\put (-0.9,3){\vector(0.9,1){22}}
\put (.7,1.8){\vector(100,-4.5){48}}
\end{overpic}
\captionof{figure}{The spacetime region and various subsets\protect\footnote{The (unlabeled and uncolored)
flat front and back surfaces should be identified.}.}
\label{F:SOLIDREGION}
\end{center}

We refer to the $\Sigma_t$ and $\Sigma_t^u$ as ``constant time slices,'' 
the 
$\mathcal{F}_u$
and
$\mathcal{F}_u^t$ as ``null hyperplanes,'' 
``null hypersurfaces,''
``characteristics,'' or ``acoustic characteristics,''
and the $\ell_{t,u}$ as ``tori.'' 
Note that $\mathcal{M}_{t,u}$ is ``open-at-the-top'' by construction.

\subsection{Important vectorfields, 
the rescaled frame, 
and the non-rescaled frame}
\label{SS:FRAMEANDRELATEDVECTORFIELDS}

\begin{definition}[\textbf{Important vectorfields}]
\

\begin{enumerate}
\item Define the \textbf{geodesic null vectorfield} by:
\begin{align} \label{E:LGEOEQUATION}
	\Lgeo^{\nu} 
	& \doteq - (g^{-1})^{\nu \alpha} \partial_{\alpha} u.
\end{align}
\item Define the \textbf{rescale null vectorfield} (recall the
	definition of $\upmu$ in \eqref{E:FIRSTUPMU}) by:
\begin{align} \label{E:LUNITDEF}
		\Lunit
		& \doteq \upmu \Lgeo.
	\end{align}
\item Define $X$ to be the unique vectorfield that is $\Sigma_t$-tangent, 
$g$-orthogonal
	to the $\ell_{t,u}$, and normalized by:
	\begin{align} \label{E:GLUNITRADUNITISMINUSONE}
		g(\Lunit,\Radunit) = -1.
	\end{align}
Define the ``rescaled'' vectorfield $\bX$ by:
\begin{align} \label{E:RADDEF}
		\Rad \doteq \upmu \Radunit.
	\end{align}
\item Define $Y$ and $Z$ respectively to be the $g$-orthogonal projection\footnote{
To see that $Y$ and $Z$ are tangent to $\ell_{t,u}$, 
one can use \eqref{E:GLUNITRADUNITISMINUSONE}, \eqref{E:TRANSPORTVECTORFIELDINTERMSOFLUNITANDRADUNIT},
the fact that $\Transport$ is $g$-orthogonal to $\Sigma_t$,
and the fact that $\rd_i$ is tangent to $\Sigma_t$. Alternatively, see \eqref{E:LINEPROJECTION}.} 
of the Cartesian partial derivative vectorfields
$\rd_2$ and $\rd_3$ to the tangent space of $\ell_{t,u}$, i.e.,
\begin{equation}\label{eq:Y.Z.def}
Y \doteq \rd_2 - g(\rd_2, X) X, \quad Z \doteq \rd_3 - g(\rd_3,X) X. 
\end{equation}
\item We will use vectorfields in $\mathscr{P} \doteq \{L,Y,Z\}$ for commutation, and
we therefore refer to them as \textbf{commutation vectorfields}. An element of $\mathscr{P}$ will often be denoted schematically by $\mathcal{P}$ (see also Definition~\ref{def:array.convention}).
\end{enumerate}
\end{definition}

We collect some basic properties of these vectorfields; see \cite[(2.12), (2.13) and Lemma~2.1]{jSgHjLwW2016} for proofs.

\begin{lemma}[\textbf{Basic properties of the vectorfields}] 
\label{L:VFBASIC}
\

\begin{enumerate}
\item $\Lgeo$ is geodesic and null, i.e., 
$$g(\Lgeo,\Lgeo) = 0,\quad \D_{\Lgeo} \Lgeo = 0,$$
where $\D$ is the Levi-Civita connection associated to $g$.
\item The following identities hold:
\begin{align} \label{E:LUNITANDRADOFUANDT}
	\Lunit u & = 0, 
	\qquad \Lunit t = \Lunit^0 = 1,
	\qquad
	\Rad u = 1,
	\qquad
	\Rad t = \Rad^0 = 0,
		\\
	g(\Radunit,\Radunit)
	& = 1,
	\qquad
	g(\Rad,\Rad)
	= \upmu^2,
		\qquad
	g(\Lunit,\Radunit)
	= -1,
	\qquad
	g(\Lunit,\Rad) = -\upmu.
		\label{E:RADIALVECTORFIELDSLENGTHSANDLRADIALVECTORFIELDSNORMALIZATIONS}
\end{align}
\item The vectorfield 
$\Transport$ (see \eqref{E:MATERIALVECTORVIELDRELATIVETORECTANGULAR})
is future-directed, $g$-orthogonal
to $\Sigma_t$, and is normalized by $
	g(\Transport,\Transport) 
	 = - 1
$. 
Moreover,
\begin{align} \label{E:TRANSPORTVECTORFIELDINTERMSOFLUNITANDRADUNIT}
	\Transport
	& = \rd_t + v^a \rd_a = \Lunit + \Radunit,
		\\
	\Transport_{\alpha}
	& = - \delta_{\alpha}^0,
	\label{E:MATERIALDERIVATIVEINDICESDOWN}
\end{align}
where $\delta_{\alpha}^{\beta}$ is the Kronecker delta.
\end{enumerate}
\end{lemma}

\subsection{Transformations}
Having introduced various vectorfields in Section~\ref{SS:FRAMEANDRELATEDVECTORFIELDS}, 
we now derive some related transformation formulas that we will use later on.

\begin{definition}[\textbf{Coordinate vectorfields in geometric $(t,u,x^2,x^3)$ coordinates}]
\label{D:GEOMETRICCOORDINATEVECTORFIELDS}
Define $(\srd_t, \srd_u, \srd_2, \srd_3)$ to be the coordinate partial derivative 
vectorfields in the geometric $(t,u,x^2,x^3)$ coordinate system.
\end{definition}

\begin{definition}[\textbf{Cartesian components of geometric vectorfields}] \label{D:CARTESIANCOMPONENTSOFVECTORFIELDS}
\

\begin{enumerate}
\item Define $L^i$ and $X^i$ to be the Cartesian $i$-th components of $L$ and $X$ respectively. 
(Note that $L^i +X^i -v^i = 0$; see \eqref{E:TRANSPORTVECTORFIELDINTERMSOFLUNITANDRADUNIT}.)
\item Define\footnote{The notation is suggestive of the fact that these quantities are of size $\mathcal{O}(\mathring{\upalpha})$ (and hence small).} $L_{(Small)}$ and $X_{(Small)}$ by:
\begin{subequations}
\begin{align}
L_{(Small)}^1 & \doteq L^1 - 1,\quad  L_{(Small)}^2 \doteq L^2,\quad L_{(Small)}^3 \doteq L^3,
	\label{E:LSMALLDEF} \\
X_{(Small)}^1 & \doteq X^1 + 1,\quad  X_{(Small)}^2 \doteq X^2,\quad X_{(Small)}^3 \doteq X^3.
\label{E:XSMALLDEF}
\end{align}
\end{subequations}
\end{enumerate}
\end{definition}

\begin{lemma}[\textbf{Relations between $\{\rd_\alp\}_{\alpha=0,1,2,3}$ 
and $\{L,X,Y,Z\}$}]\label{lem:Cart.to.geo}
The following identities hold:
\begin{subequations}
\begin{align}
\label{eq:Cart.to.geo.0} \rd_t \doteq & \ \rd_0 = \: L + X - v^a \rd_a,\\
\label{eq:Cart.to.geo.1} \rd_1 = &\: \Speed^{-2} X^1 X - \f {X^2}{X^1} Y - \f {X^3}{X^1} Z, \\
\label{eq:Cart.to.geo.2} \rd_2 = &\: Y+ (\Speed^{-2} X^2) X,\quad \rd_3 = Z + (\Speed^{-2} X^3) X.
\end{align}
\end{subequations}
\end{lemma}
\begin{proof}
\eqref{eq:Cart.to.geo.0} is simply a restatement of \eqref{E:TRANSPORTVECTORFIELDINTERMSOFLUNITANDRADUNIT}. \eqref{eq:Cart.to.geo.2} follows from \eqref{eq:Y.Z.def} and $g(\rd_A, X) = \Speed^{-2} X^A$ for $A=2,3$
(see \eqref{E:ACOUSTICALMETRIC}). Finally, to obtain \eqref{eq:Cart.to.geo.1}, we write 
$X = X^a \rd_a$ and use \eqref{eq:Cart.to.geo.2} to obtain $\rd_1 = \f{1}{X^1} [1 - c^{-2} ((X^2)^2 + (X^3)^2)] X - \f {X^2}{X^1} Y - \f {X^3}{X^1} Z$. This then implies \eqref{eq:Cart.to.geo.1} since 
 $\sum_{a=1}^3 (X^a)^2 = c^2$ by $g(X,X)=1$ 
(see \eqref{E:RADIALVECTORFIELDSLENGTHSANDLRADIALVECTORFIELDSNORMALIZATIONS}) and 
\eqref{E:ACOUSTICALMETRIC}. \qedhere
\end{proof}

\begin{lemma}[\textbf{Relation between $\{\srd_\alp\}$ and $\{L,X,Y,Z\}$}]\label{lem:slashed}
The following identities hold, where repeated capital Latin indices are summed over $A=2,3$:
\begin{subequations}
\begin{equation}\label{eq:slashes.1}
L = \srd_t + L^A \srd_A,\quad \bX = \srd_u + \upmu X^A \srd_A, 
\end{equation}
\begin{equation}\label{eq:slashes.2}
 Y = (1-\Speed^{-2} (X^2)^2) \srd_2 - \Speed^{-2} X^2 X^3 \srd_3, \quad  
Z = (1-\Speed^{-2} (X^3)^2) \srd_3 - \Speed^{-2} X^2 X^3 \srd_2.
\end{equation}

\end{subequations}
\end{lemma}
\begin{proof}
\eqref{eq:slashes.1} is an immediate consequence of \eqref{E:LUNITANDRADOFUANDT}
(and \eqref{E:RADDEF}).

To derive the first equation in \eqref{eq:slashes.2}, simply note that $Yx^2 = 1-\Speed^{-2}(X^2)^2$ and $Yx^3 = -\Speed X^2 X^3$ by \eqref{eq:Cart.to.geo.2}, and that $Yt = Y u = 0$ since
$Y$ is $\ell_{t,u}$-tangent. 
The second equation in \eqref{eq:slashes.2} 
follows from similar reasoning. \qedhere 
\end{proof}

\begin{lemma}[\textbf{Relation\footnote{We could also obtain $\srd_t = \rd_t + (L^1 +\f{X^2L^2 + X^3L^3}{X_1})\rd_1$. Since this will not be explicitly needed, we will not prove it.} between 
$\{\rd_a\}_{a=1,2,3}$, $\{\srd_u, \srd_2, \srd_3\}$, and
$\{\bX,Y,Z\}$}]
\label{L:GEOMETRICCOORDINATEVECTORFIELDSINTERMSOFCARTESIANVECTORFIELDS}
The following identities hold:
\begin{subequations}
\begin{align}
	\srd_u 
		& 
		= \f{\upmu \Speed^2}{X^1} \rd_1
		= \Rad - \upmu \Speed^2 \f {X^2}{(X^1)^2} Y - \upmu \Speed^2 \f {X^3}{(X^1)^2} Z,
			\label{E:GEOMETRICUCOORDINATEPARTIALDERIVATIVESINTERMSOFOTHERVECTORFIELDS} \\
	\srd_2 & = \rd_2 -\f{X^2}{X^1} \rd_1	
		= \left\lbrace 1 + \left(\frac{X^2}{X^1}\right)^2 \right\rbrace Y
			+
			\frac{X^2 X^3}{(X^1)^2}
			Z,
		\label{E:GEOMETRIC2COORDINATEPARTIALDERIVATIVESINTERMSOFOTHERVECTORFIELDS} \\
	\srd_3 & = \rd_3 -\f{X^3}{X^1}\rd_1
	= \frac{X^2 X^3}{(X^1)^2} Y
		+
		\left\lbrace 1 + \left(\frac{X^3}{X^1}\right)^2 \right\rbrace
		Z.
	\label{E:GEOMETRIC3COORDINATEPARTIALDERIVATIVESINTERMSOFOTHERVECTORFIELDS} 
\end{align}
\end{subequations}
\end{lemma}
\begin{proof}
It suffices to derive the following identities:
\begin{equation}\label{eq:srd.x1}
\srd_u x^1=  \f{\upmu \Speed^2}{X^1},\quad \slashed \rd_2 x^1 = -\f{X^2}{X^1},\quad \slashed \rd_3 x^1 = -\f{X^3}{X^1};
\end{equation}
it is straightforward to see that the first identities in each of 
\eqref{E:GEOMETRICUCOORDINATEPARTIALDERIVATIVESINTERMSOFOTHERVECTORFIELDS}--\eqref{E:GEOMETRIC3COORDINATEPARTIALDERIVATIVESINTERMSOFOTHERVECTORFIELDS} follow from \eqref{eq:srd.x1}; 
the second identities in
\eqref{E:GEOMETRICUCOORDINATEPARTIALDERIVATIVESINTERMSOFOTHERVECTORFIELDS}--\eqref{E:GEOMETRIC3COORDINATEPARTIALDERIVATIVESINTERMSOFOTHERVECTORFIELDS}
then follow from the first ones and Lemma~\ref{lem:Cart.to.geo}. To prove \eqref{eq:srd.x1},
we invert \eqref{eq:slashes.2} to obtain
(with the help of the identity $\sum_{a=1}^3 (X^a)^2 = c^2$,
which follows from \eqref{E:RADIALVECTORFIELDSLENGTHSANDLRADIALVECTORFIELDSNORMALIZATIONS} and \eqref{E:ACOUSTICALMETRIC}):
$$\slashed \rd_2 = \left\lbrace \f{\Speed^2}{(X^1)^2} - (\f{X^3}{X^1})^2 \right\rbrace Y 
	+ 
	\f{X^2 X^3}{(X^1)^2} Z,\quad \slashed \rd_3 = \f{X^2 X^3}{(X^1)^2} Y 
	+ 
	\left\lbrace(\f{\Speed^2}{(X^1)^2} - (\f{X^2}{X^1})^2 \right\rbrace)Z.$$ 

On the other hand, by \eqref{eq:Cart.to.geo.2}, $Yx^1 = -\Speed^{-2} X^2 X^1$, $Zx^1 = -\Speed^{-2} X^3 X^1$. 
Hence,
$$\slashed \rd_2 x^1 = -\f{X^2}{X^1},\quad \slashed \rd_3 x^1 = -\f{X^3}{X^1}.$$
Plugging back into the second identity in \eqref{eq:slashes.1}, we obtain:
$$\srd_u x^1= \upmu X^1 - \sum_{A=2}^3 \upmu X^A \srd_A x^1= \upmu X^1 + \sum_{A=2}^3 \upmu \f{(X^A)^2}{X^1} = \f{\upmu \Speed^2}{X^1},$$
where we again used $\sum_{i=1}^3 (X^i)^2 = c^2$. \qedhere
\end{proof}

\subsection{Projection tensorfields, \texorpdfstring{$\vec{G}_{(Frame)}$,}{frame components,} 
and projected Lie derivatives}
\label{SS:PROJECTIONTENSORFIELDANDPROJECTEDLIEDERIVATIVES}

\begin{definition}[\textbf{Projection tensorfields}]
We define the $\Sigma_t$ projection tensorfield\footnote{In \eqref{E:SIGMATPROJECTION},
we have corrected a sign error that occurred in
\cite{jSgHjLwW2016}*{Definition 2.8}.} 
$\Sigmatproject$
and the $\ell_{t,u}$ projection tensorfield
$\Lineproject$ relative to Cartesian coordinates as follows:
\begin{subequations}
\begin{align} 
	\Sigmatproject_{\nu}^{\ \mu} 
	&\doteq	\delta_{\nu}^{\ \mu}
			+ 
			\Transport_{\nu} \Transport^{\mu} 
		= \delta_{\nu}^{\ \mu}
			- \delta_{\nu}^{\ 0} \Lunit^{\mu}
			- \delta_{\nu}^{\ 0} \Radunit^{\mu},
			\label{E:SIGMATPROJECTION} \\
	\Lineproject_{\nu}^{\ \mu} 
	&\doteq	\delta_{\nu}^{\ \mu}
			+ \Radunit_{\nu} \Lunit^{\mu} 
			+ \Lunit_{\nu} (\Lunit^{\mu} + \Radunit^{\mu}) 
		= \delta_{\nu}^{\ \mu}
			- \delta_{\nu}^{\ 0} \Lunit^{\mu} 
			+  \Lunit_{\nu} \Radunit^{\mu}.
			\label{E:LINEPROJECTION}
	\end{align}
\end{subequations}
In \eqref{E:SIGMATPROJECTION}--\eqref{E:LINEPROJECTION},
$\delta_{\nu}^{\ \mu}$ is the standard Kronecker delta.
The last equalities in \eqref{E:SIGMATPROJECTION} and \eqref{E:LINEPROJECTION} follow from 
\eqref{E:TRANSPORTVECTORFIELDINTERMSOFLUNITANDRADUNIT}--\eqref{E:MATERIALDERIVATIVEINDICESDOWN}.
\end{definition}

\begin{definition}[\textbf{Projections of tensorfields}]
Given any type $\binom{m}{n}$ spacetime tensorfield $\xi$,
we define its $\Sigma_t$ projection $\Sigmatproject \xi$
and its $\ell_{t,u}$ projection $\Lineproject \xi$
as follows:
\begin{subequations}
\begin{align} 
(\Sigmatproject \xi)_{\nu_1 \cdots \nu_n}^{\mu_1 \cdots \mu_m}
& \doteq
	\Sigmatproject_{\widetilde{\mu}_1}^{\ \mu_1} \cdots \Sigmatproject_{\widetilde{\mu}_m}^{\ \mu_m}
	\Sigmatproject_{\nu_1}^{\ \widetilde{\nu}_1} \cdots \Sigmatproject_{\nu_n}^{\ \widetilde{\nu}_n} 
	\xi_{\widetilde{\nu}_1 \cdots \widetilde{\nu}_n}^{\widetilde{\mu}_1 \cdots \widetilde{\mu}_m},
		\\
(\Lineproject \xi)_{\nu_1 \cdots \nu_n}^{\mu_1 \cdots \mu_m}
& \doteq 
	\Lineproject_{\widetilde{\mu}_1}^{\ \mu_1} \cdots \Lineproject_{\widetilde{\mu}_m}^{\ \mu_m}
	\Lineproject_{\nu_1}^{\ \widetilde{\nu}_1} \cdots \Lineproject_{\nu_n}^{\ \widetilde{\nu}_n} 
	\xi_{\widetilde{\nu}_1 \cdots \widetilde{\nu}_n}^{\widetilde{\mu}_1 \cdots \widetilde{\mu}_m}.
	\label{E:STUPROJECTIONOFATENSOR}
\end{align}
\end{subequations}
\end{definition}
We say that a spacetime tensorfield $\xi$ is $\Sigma_t$-tangent 
(respectively $\ell_{t,u}$-tangent)
if $\Sigmatproject \xi = \xi$
(respectively if $\Lineproject \xi = \xi$).
Alternatively, we say that $\xi$ is a
$\Sigma_t$ tensor (respectively $\ell_{t,u}$ tensor).


\begin{definition}[\textbf{$\ell_{t,u}$ projection notation}]
	\label{D:STUSLASHPROJECTIONNOTATION}
	If $\xi$ is a spacetime tensor, then 
	$\angxi \doteq \Lineproject \xi$.

	If $\xi$ is a symmetric type $\binom{0}{2}$ spacetime tensor and $V$ is a spacetime
	vectorfield, then $\angxiarg{V} \doteq \Lineproject (\xi_V)$,
	where $\xi_V$ is the spacetime one-form with 
	Cartesian components $\xi_{\alpha \nu} V^{\alpha}$, $(\nu = 0,1,2,3)$.
\end{definition}

\begin{remark}[Clarification of the symbols $(\srd_t, \srd_u, \srd_2, \srd_3)$]
	\label{R:MEANINGOFSRDTANDSRDU}
	We caution that the coordinate partial derivative vectorfields $(\srd_t, \srd_u, \srd_2, \srd_3)$
	from Definition~\ref{D:GEOMETRICCOORDINATEVECTORFIELDS}
	are not $\ell_{t,u}$ projections of other vectorfields, i.e.,
	for $(\srd_t, \srd_u, \srd_2, \srd_3)$, we are not using the ``slash conventions''
	of Definition~\ref{D:STUSLASHPROJECTIONNOTATION}.
\end{remark}

Throughout, $\Lie_V \xi$ denotes the Lie derivative of the tensorfield
$\xi$ with respect to the vectorfield $V$.
We often use the Lie bracket notation
$[V,W] \doteq \Lie_V W$ when $V$ and $W$ are vectorfields.

\begin{definition}[$\Sigma_t$- \textbf{and} $\ell_{t,u}$-\textbf{projected Lie derivatives}]
\label{D:PROJECTEDLIE}
If $\xi$ is a tensorfield 
and $V$ is a vectorfield,
we define the $\Sigma_t$-projected Lie derivative
$\SigmatLie_V \xi$
and the $\ell_{t,u}$-projected Lie derivative
$\angLie_V \xi$ as follows:
\begin{align} 
	\SigmatLie_V \xi
	& \doteq \Sigmatproject \Lie_V \xi,
	&&
	\angLie_V \xi
	\doteq \Lineproject \Lie_V \xi.
	\label{E:PROJECTIONS}
\end{align}
\end{definition}

\begin{definition}[\textbf{Components of $\vec{G}$ relative to the non-rescaled frame}]
	\label{D:GFRAMEANDHFRAMEARRAYS}
	We define:
	\begin{align}
		\vec{G}_{(Frame)} 
		& \doteq 
			\left\lbrace 
				\vec{G}_{\Lunit \Lunit}, \vec{G}_{\Lunit \Radunit}, 
				\vec{G}_{\Radunit \Radunit}, \angGarg{\Lunit}, 
				\angGarg{\Radunit}, 
				\angG 
			\right\rbrace,
				\label{E:GFRAME}
	\end{align}
	where $\vec{G}_{\alpha \beta}$
	is defined in \eqref{E:BIGGDEF}.
\end{definition}

Our convention is that derivatives of $\vec{G}_{(Frame)}$
form a new array consisting of the differentiated components.
For example,
$\angLie_{\Lunit} \vec{G}_{(Frame)} 
		\doteq 
			\left\lbrace 
				\Lunit(\vec{G}_{\Lunit \Lunit}), 
				\Lunit (\vec{G}_{\Lunit \Radunit}), 
				\cdots,
				\angLie_{\Lunit} \angG
			\right\rbrace
$,
where 
$
\Lunit(\vec{G}_{\Lunit \Lunit})
\doteq \left\lbrace
					\Lunit(G_{\Lunit \Lunit}^1),
					\Lunit(G_{\Lunit \Lunit}^2),
					\cdots,
					\Lunit(G_{\Lunit \Lunit}^5)
			 \right\rbrace
$,
$
\angLie_{\Lunit} (\angGarg{\Radunit})
	\doteq \left\lbrace
					\angLie_{\Lunit} (\NovecangGarg{\Radunit}{1}),
					\angLie_{\Lunit} (\NovecangGarg{\Radunit}{2}),
					\cdots,
					\angLie_{\Lunit} (\NovecangGarg{\Radunit}{5})
			 \right\rbrace,
$
etc.

\subsection{First and second fundamental forms and covariant differential operators}
\label{SS:FUNDFORMSANDCOVDERIVOPS}

\begin{definition}[\textbf{First fundamental forms}] \label{D:FIRSTFUND}
	Let $\Sigmatproject$ and $\Lineproject$ be as in Definition~\ref{D:STUSLASHPROJECTIONNOTATION}.
	We define the first fundamental form $\gt$ of $\Sigma_t$ and the 
	first fundamental form $\gsphere$ of $\ell_{t,u}$ as follows:
	\begin{align}
		\gt
		\doteq \Sigmatproject g,
		\qquad
		\gsphere
		\doteq \Lineproject g.
		\label{E:GTANDGSPHERESPHEREDEF}
	\end{align}

	We define the inverse first fundamental forms by raising the indices with $g^{-1}$:
	\begin{align} \label{E:GGTINVERSEANDGSPHEREINVERSEDEF}
		(\gt^{-1})^{\mu \nu}
		& \doteq (g^{-1})^{\mu \alpha} (g^{-1})^{\nu \beta} \gt_{\alpha \beta},
		&
		(\gsphere^{-1})^{\mu \nu}
		& \doteq (g^{-1})^{\mu \alpha} (g^{-1})^{\nu \beta} \gsphere_{\alpha \beta}.
	\end{align}
\end{definition}
$\gt$ is the Riemannian metric on $\Sigma_t$ induced by $g$ while
$\gsphere$ is the Riemannian metric on $\ell_{t,u}$ induced by $g$.
Simple calculations imply that
$(\gt^{-1})^{\mu \alpha} \gt_{\alpha \nu} = \Sigmatproject_{\nu}^{\ \mu}$
and $(\gsphere^{-1})^{\mu \alpha} \gsphere_{\alpha \nu} = \Lineproject_{\nu}^{\ \mu}$.

\begin{lemma}[\textbf{Identities for induced metrics}]
\label{lem:induced.metric}
In the $(t,u,x^2,x^3)$ coordinate system, we have:
$$\underline{g} = \f{\upmu^2 \Speed^{2}}{(X^1)^2}  du \otimes du  
- \upmu \sum_{A=2}^3 \f{X^A}{(X^1)^2} (d x^A \otimes du + du \otimes d x^A) 
+ 
\slashed{g},\quad \slashed{g} = \sum_{A,B=2}^3 \Speed^{-2}(\de_{AB} + \frac{X^A X^B}{(X^1)^2}) dx^A \otimes dx^B. $$
Moreover, 
$$\slashed g^{-1} = \sum_{A,B=2}^3 (\Speed^2 \de^{AB} - X^A X^B) \srd_A \otimes \srd_B.$$
\end{lemma}
\begin{proof}
	The identities for $\underline{g}$ and $\slashed g$
	follow easily from Lemma~\ref{L:GEOMETRICCOORDINATEVECTORFIELDSINTERMSOFCARTESIANVECTORFIELDS}
	and the fact that $\underline{g}_{ij} = \Speed^{-2} \updelta_{ij}$ in Cartesian coordinates
	(see \eqref{E:ACOUSTICALMETRIC}).
	The identity for $\slashed g^{-1}$ follows from inverting the matrix $(\slashed g_{AB})_{A,B=2,3}$
	and using the identity
	$\sum_{i=1}^3 (X^i)^2 = c^2$,
	which follows from the first identity in \eqref{E:RADIALVECTORFIELDSLENGTHSANDLRADIALVECTORFIELDSNORMALIZATIONS} 
	and \eqref{E:ACOUSTICALMETRIC}.
\end{proof}



\begin{definition}[\textbf{Differential operators associated to the metrics}] 
\label{D:CONNECTIONS}
	\ \\
	\begin{itemize}
		\item $\D$ denotes the Levi-Civita connection of the acoustical metric $g$.
		\item $\angD$ denotes the Levi-Civita connection of $\gsphere$.
		\item If $f$ is a scalar function on $\ell_{t,u}$, 
	then $\angdiff f \doteq \angD f = \Lineproject \D f$,
	where $\D f$ is the gradient one-form associated to $f$.
		\item If $\xi$ is an $\ell_{t,u}$-tangent one-form,
			then $\angdiv \xi$ is the scalar function
			$\angdiv \xi \doteq \ginversesphere \cdot \angD \xi$.
		\item Similarly, if $V$ is an $\ell_{t,u}$-tangent vectorfield,
			then $\angdiv V \doteq \ginversesphere \cdot \angD V_{\flat}$,
			where $V_{\flat}$ is the one-form $\gsphere$-dual to $V$.
		\item If $\xi$ is a symmetric type $\binom{0}{2}$ 
		 $\ell_{t,u}$-tangent tensorfield, then
		 $\angdiv \xi$ is the $\ell_{t,u}$-tangent 
		 one-form $\angdiv \xi \doteq \ginversesphere \cdot \angD \xi$,
		 where the two contraction indices in $\angD \xi$
		 correspond to the operator $\angD$ and the first index of $\xi$.
		\item  $\angLap \doteq \ginversesphere \cdot \angDsquared$ denotes the covariant Laplacian 
			corresponding to $\gsphere$.
		\end{itemize}
\end{definition}

\subsection{Ricci coefficients}

\begin{definition}[\textbf{Ricci coefficients}]
\label{D:SECONDFUND}
\

\begin{enumerate}
\item Define the \textbf{second fundamental form} $k$ of $\Sigma_t$
and the \textbf{null second fundamental form} $\upchi$ of $\ell_{t,u}$ 
as follows:
\begin{align} \label{E:CHIANDSECONDFUNDDEF}
	k 
	&\doteq \frac{1}{2} \SigmatLie_{\Transport} \gt,
	&
	\upchi
	& \doteq \frac{1}{2} \angLie_{\Lunit} \gsphere.
\end{align}
\item Define $\upzeta$ to be the $\ell_{t,u}$-tangent one-form whose components are given by:
\begin{align} \label{E:ZETADEF}
		\upzeta(\srd_A)
		& \doteq g(\D_{\srd_A} \Lunit, \Radunit) 
		= \upmu^{-1} g(\D_{\srd_A} \Lunit, \Rad),\quad A=2,3.
	\end{align}
	\item Given any symmetric type $\binom{0}{2}$ 
		 $\ell_{t,u}$-tangent tensorfield $\xi$, define its \textbf{trace} by:
		 $$\mytr \xi\doteq (\slashed g^{-1})^{AB} \xi_{AB}.$$
\end{enumerate}

\end{definition}

\begin{lemma}[\textbf{Useful identities for the Ricci coefficients}] 
	\label{L:IDENTITIESINVOLVING}
	The following identities hold:\footnote{Here, $
\angGarg{\Lunit} \overset{\contr}{\otimes} \angdiff \threePsi
\doteq
\sum_{\imath=1}^5 
\NovecangGarg{\Lunit}{\imath} \otimes \angdiff \Psi_{\imath}
$,
and similarly for the other terms involving $\overset{\contr}{\otimes}$.}
	\begin{subequations}
	\begin{align} 
	\upchi
	& =  g_{ab} (\angdiff L^a)\otimes (\angdiff x^b) 
		+ 
		\frac{1}{2} \angG \contr \Lunit \threePsi
		+ 
		\frac{1}{2} \angdiff \threePsi \overset{\contr}{\otimes}  \angGarg{\Lunit} 
		-
		\frac{1}{2} \angGarg{\Lunit} \overset{\contr}{\otimes} \angdiff \threePsi, \label{E:CHIINTERMSOFOTHERVARIABLES}
		\\
	\mytr \upchi 
	& = g_{ab} \ginversesphere \cdot \left\lbrace (\angdiff \Lunit^a) \otimes (\angdiff x)^b \right\rbrace
		+ \frac{1}{2} \ginversesphere \cdot \angG\contr\Lunit\threePsi,
			 \label{E:TRCHIINTERMSOFOTHERVARIABLES}\\
	\angk & = 
			\frac{1}{2} \upmu^{-1} \angG\contr \Rad \threePsi 
				+ 
				\frac{1}{2} \angG \contr \Lunit \threePsi
			- 
			\frac{1}{2} \angGarg{\Lunit} \overset{\contr}{\otimes} \angdiff \threePsi
			- 
			\frac{1}{2} \angdiff \threePsi \overset{\contr}{\otimes} \angGarg{\Lunit} 
			- 
			\frac{1}{2} \angGarg{\Radunit} \overset{\contr}{\otimes} \angdiff \threePsi
			- 
			\frac{1}{2} \angdiff \threePsi \overset{\contr}{\otimes} \angGarg{\Radunit}, \label{eq:angk.def} \\
	\upzeta & = - \frac{1}{2} \upmu^{-1}\angGarg{\Lunit} \contr \Rad \threePsi + \frac{1}{2} \angGarg{\Radunit}\contr\Lunit \threePsi
			- \frac{1}{2} \vec{G}_{\Lunit \Radunit}\contr\angdiff \threePsi
			- \frac{1}{2} \vec{G}_{\Radunit \Radunit}\contr\angdiff \threePsi. \label{eq:upzeta.def}
\end{align}
\end{subequations}
\end{lemma}
\begin{proof}
		This is the same as \cite[Lemmas~2.13, 2.15]{jSgHjLwW2016}
		except for small modifications incorporating the third dimension. \qedhere
\end{proof}


\subsection{Pointwise norms}
\label{SS:POINTWISENORMS}
We always measure the magnitude of $\ell_{t,u}$ tensors\footnote{Note that in contrast, for $\Sigma_t$ tensors, we measure 
 their magnitude using the Euclidean metric or an equivalent norm; see, for example, 
Definition~\ref{def:spatial.derivatives}.} using $\gsphere$.

\begin{definition}[\textbf{Pointwise norms}]
	\label{D:POINTWISENORM}
	For any type $\binom{m}{n}$ $\ell_{t,u}$ tensor $\xi_{\nu_1 \cdots \nu_n}^{\mu_1 \cdots \mu_m}$,
	we define:
	\begin{align} \label{E:POINTWISENORM}
		|\xi|
		\doteq 
		\sqrt{
		\gsphere_{\mu_1 \widetilde{\mu}_1} \cdots \gsphere_{\mu_m \widetilde{\mu}_m}
		(\ginversesphere)^{\nu_1 \widetilde{\nu}_1} \cdots (\ginversesphere)^{\nu_n \widetilde{\nu}_n}
		\xi_{\nu_1 \cdots \nu_n}^{\mu_1 \cdots \mu_m}
		\xi_{\widetilde{\nu}_1 \cdots \widetilde{\nu}_n}^{\widetilde{\mu}_1 \cdots \widetilde{\mu}_m}
		}.
	\end{align}
\end{definition}

\subsection{Transport equations for the eikonal function quantities}
\label{SS:TRANSPORT}
The next lemma provides the transport equations that, 
in conjunction with \eqref{E:TRCHIINTERMSOFOTHERVARIABLES}, 
we use to estimate
the eikonal function quantities 
$\upmu$,
$\Lunit_{(Small)}^i$,
and $\mytr \upchi$
below top-order.

\begin{lemma}\cite{jSgHjLwW2016}*{Lemma 2.12; \textbf{The transport equations verified by} $\upmu$ \textbf{and} $\Lunit_{(Small)}^i$}
 \label{L:UPMUANDLUNITIFIRSTTRANSPORT}
The following transport equations hold:
\begin{align} \label{E:UPMUFIRSTTRANSPORT}
	\Lunit \upmu 
	& =
		\frac{1}{2} \vec{G}_{\Lunit \Lunit}\contr \Rad \threePsi
		- \frac{1}{2} \upmu \vec{G}_{\Lunit \Lunit}\contr\Lunit \threePsi
		- \upmu \vec{G}_{\Lunit \Radunit}\contr\Lunit \threePsi,
			\\
	\Lunit \Lunit^i
	&  = \frac{1}{2} \vec{G}_{\Lunit \Lunit}\contr(\Lunit \threePsi) \Radunit^i
			- 
			\angGmixedarg{\Lunit}{\#}\contr(\Lunit \threePsi) \cdot \angdiff x^i
			+ 
			\frac{1}{2} \vec{G}_{\Lunit \Lunit}\contr(\angdiffuparg{\#} \threePsi) \cdot \angdiff x^i.
			\label{E:LLUNITI}
\end{align}
\end{lemma}

\subsection{Calculations connected to the failure of the null condition}
\label{SS:CALCSFORFAILUREOFNULLCONDITION}
Many important estimates are tied to
the coefficients $\vec{G}_{\Lunit \Lunit}$.
In the next two lemmas,
we derive expressions for
$\vec{G}_{\Lunit \Lunit}$
and
$
\frac{1}{2} \vec{G}_{\Lunit \Lunit} \contr \Rad \threePsi
$.
This presence of the latter term on RHS~\eqref{E:UPMUFIRSTTRANSPORT}
is tied to the failure of
Klainerman's null condition \cite{sK1984} and thus one expects
that the product must be non-zero for shocks to form;
this is explained in more detail in
the survey article \cite{gHsKjSwW2016} in 
a slightly different context.

\begin{lemma}[\textbf{Formula for} $\frac{1}{2} \vec{G}_{\Lunit \Lunit}\contr\Rad \threePsi$]
	\label{L:FORMULAFORGLLRADPSI}
	Let $F$ be the smooth function of $(\Densrenormalized,\Ent)$ from \eqref{E:RIEMANNINVARIANTS},
	and let $F_{;\Ent}$ denote its partial derivative with respect to $\Ent$ at fixed $\Densrenormalized$.
	For solutions to 
	\eqref{eq:Euler.1}--\eqref{eq:Euler.3},
	we have:
\begin{align} 
\begin{split} \label{E:KEYLARGETERMEXPANDED}
	\frac{1}{2} \vec{G}_{\Lunit \Lunit}\contr\Rad \threePsi
	& =
		-
			\frac{1}{2}
			\Speed^{-1}
			(\Speed^{-1} \Speed_{;\rr} + 1)
			\left\lbrace
				\Rad \mathcal{R}_{(+)} - \Rad \mathcal{R}_{(-)}
			\right\rbrace
				\\
	& \ \
		-
		\frac{1}{2}
		\upmu \Speed^{-2} \Radunit^1 
		\left\lbrace
			\Lunit \mathcal{R}_{(+)} + \Lunit \mathcal{R}_{(-)}
		\right\rbrace
		-
		\upmu 
		\Speed^{-2} 
		(\Radunit^2 \Lunit v^2 + \Radunit^3 \Lunit v^3) 	
			 \\
	& \ \
		-
		\upmu
		\Speed^{-1} \Speed_{;s} \Radunit^a \GradEnt^a
		+
		\upmu
		\Speed^{-1}
		(\Speed^{-1} \Speed_{;\rr} + 1)
		F_{;\Ent} \Radunit^a \GradEnt^a.
\end{split}
\end{align}
\end{lemma}

\begin{proof}
		This is the same as \cite[Lemmas~2.45, 2.46]{LS}, except for minor
		modifications incorporating the third dimension and the entropy
		(via the $\Speed_{;s}$-dependent and $F_{;\Ent}$-dependent products).
		\qedhere
\end{proof}



\section{Volume forms and energies}
\label{S:FORMSANDENERGY}
In this section, we first define geometric integration forms
and corresponding integrals. 
We then define the energies and null fluxes
which we will use in the remainder of the paper to derive a priori $L^2$-type estimates.

\subsection{Geometric forms and related integrals}
\label{SS:FORMS}
We define our geometric integrals in terms of area and volume
forms that remain non-degenerate relative to the geometric coordinates
throughout the evolution (i.e., all the way up to the shock).

\begin{definition}[\textbf{Geometric forms and related integrals}]
	\label{D:NONDEGENERATEVOLUMEFORMS}
	Define the area form
	$d \spherevol$ on $\ell_{t,u}$,
	the area form $d \tvol$ on $\Sigma_t^u$,
	the area form $d \conevol$ on $\mathcal{F}_u^t$,
	and the volume form $d \vol$ on $\mathcal{M}_{t,u}$
	as follows (relative to the $(t,u,x^2,x^3)$ coordinates):
	\begin{align*} 
			d \spherevol
			& = d \spherevol{(t,u,x^2,x^3)}
			\doteq \f{dx^2\, dx^3}{\Speed |X^1|},
				&&
			d \tvol
			=
			d \tvol(t,u',x^2,x^3)
			\doteq d \spherevol(t,u',x^2,x^3) du',
				\\
			d \conevol 
			& = d \conevol(t',u,x^2,x^3)
			\doteq d \spherevol(t',u,x^2,x^3) dt',
				&&
			d \vol 
			= d \vol(t',u',x^2,x^3)
			\doteq d \spherevol(t',u',x^2,x^3) du' dt'.
	\end{align*}
\end{definition}

	It is understood that unless we explicitly indicate otherwise, all integrals
	are defined with respect to the forms of Definition~\ref{D:NONDEGENERATEVOLUMEFORMS}.
	Moreover, in our notation, we often suppress the variables with respect to which we integrate,
	i.e., we write
	$
	\int_{\ell_{t,u}}
		f
		\, d \spherevol
		\doteq
		\int_{(x^2,x^3) \in \mathbb{T}^2}
			f(t,u,x^2,x^3)
		\, d \spherevol(t,u,x^2,x^3)
	 $, etc.
	
The following lemma clarifies the geometric and analytic significance
of the forms from Definition~\ref{D:NONDEGENERATEVOLUMEFORMS}.
\begin{lemma}[\textbf{Identities concerning the forms}]
\

\begin{enumerate}
\item $d \spherevol$ is the volume measure induced by $\slashed g$ on $\ell_{t,u}$.
\item $\upmu \, d\tvol$ is the volume measure induced by $\underline{g}$ on $\Sigma_t^u$.
\item Let $dx$ be the standard Euclidean volume measure on $\Sigma_t^u$,
	i.e., $dx = dx^1 \, dx^2\, dx^3$ relative to the Cartesian spatial coordinates. Then:
\begin{equation}\label{eq:volume.form.Sigma}
dx = \upmu\Speed^3 \, d\tvol.
\end{equation}
\end{enumerate}
\end{lemma}	
\begin{proof}
A computation based on
Lemma~\ref{lem:induced.metric} 
and the identity $\sum_{a=1}^3 (X^a)^2 = c^2$
(which follows from \eqref{E:RADIALVECTORFIELDSLENGTHSANDLRADIALVECTORFIELDSNORMALIZATIONS} 
and \eqref{E:ACOUSTICALMETRIC}) yields
that $\det \slashed g = \f{1}{\Speed^2 (X^1)^2}$. 
Since
$d\spherevol = \sqrt{\det\slashed g} \, dx^2\, dx^3$, we thus obtain (1). 

Next, we again use
Lemma~\ref{lem:induced.metric} 
and the identity $\sum_{a=1}^3 (X^a)^2 = c^2$
to compute that relative to the $(u,x^2,x^3)$ coordinates, we have
$\det \underline{g} = \f{\upmu^2}{\Speed^2 (X^1)^2}$. 
Taking the square root, we see that
the volume measure induced by $\underline{g}$ on $\Sigma_t^u$ is given in the $(u,x^2,x^3)$ coordinates 
by $\f{\upmu}{\Speed |X^1|} \, du\, dx^2\, dx^3$, which gives (2). 

Finally, we obtain (3) from (2) via \eqref{E:ACOUSTICALMETRIC}, which
implies that relative to the Cartesian spatial coordinates, the canonical volume form 
induced by $\underline{g}$ on $\Sigma_t$ is
$\Speed^{-3} dx^1\, dx^2\, dx^3$.
\qedhere
\end{proof}
	
\subsection{The definitions of the energies and null fluxes}
\label{SS:DEFSOFENANDFLUX}

\subsubsection{Forms and conventions}

\begin{definition}[\textbf{Volume forms for $L^p$ norms}]
For $p \in \lbrace 1,2 \rbrace$,
we define $L^p$ norms with respect to the volume forms introduced in Definition~\ref{D:NONDEGENERATEVOLUMEFORMS}.
That is, for scalar functions or $\ell_{t,u}$-tangent tensorfields $\upxi$, we define:
$$\|  \upxi \|_{L^p(\ell_{t,u})}\doteq 
\left(\int_{\ell_{t,u}}  |\upxi|^p \, d \spherevol \right)^{\f 1p}, \quad \| \upxi \|_{L^p(\mathcal F_{u}^t)}\doteq \left(\int_{\mathcal F_{u}^t}  |\upxi|^p \, d \conevol \right)^{\f 1p},$$
$$\|  \upxi \|_{L^p(\Sigma_{t}^u)}
\doteq \left(\int_{\Sigma_{t}^u} | \upxi|^p \, d \tvol \right)^{\f 1p}, \quad \|  \upxi \|_{L^p(\Sigma_{t})}\doteq \left(\int_{\Sigma_{t}}  |\upxi|^p \, d \tvol \right)^{\f 1p},$$
$$\|  \upxi \|_{L^p(\Mtu)} \doteq
\left(\int_{\Mtu} | \upxi|^p \, d \vol \right)^{\f 1p}.$$  
\end{definition}

\begin{definition}[\textbf{Conventions with variable arrays and differentiated quantities}]\label{def:array.convention}
	\ 

\begin{enumerate}
\item Given the fluid variable array $\threePsi$ in Definition~\ref{def:variables.fluid}, define:
$$|\threePsi| = |\Psi| \doteq \max_{\iota=1,\cdots,5} |\Psi_{\imath}|.$$ We also set: $$|\Vr| \doteq \max_{a=1,2,3} |\Vr^a|,$$ 
and similarly for the other $\Sigma_t$-tangent
tensorfields such as $S$ and $\mathcal{C}$ that correspond to the transport variables.
For $p = 2$ or $p=\infty$, define also:
$$\|\Psi \|_{L^p(\ell_{t,u})} \doteq \max_{\iota=1,\cdots,5} \|\Psi_{\imath}\|_{L^p(\ell_{t,u})},$$
and similarly for
$L^p(\Sigma_t^u)$, $L^p(\Sigma_t)$, $L^p(\mathcal F_u^t)$, and $L^p(\Mtu)$.
Similarly, we set: 
$$\| \Vr \|_{L^p(\ell_{t,u})} \doteq 
\max_{a=1,2,3} \| \Vr^a \|_{L^p(\ell_{t,u})},$$ 
and we analogously define $L^p$ norms of other $\Sigma_t$-tangent tensorfields 
that correspond to the transport variables,
such as $S$ and $\mathcal{C}$.
\item When estimating multiple solution variables
simultaneously, we use the following convention (for $p=2$ or $p=\infty$):
$$\|(\Vr,S) \|_{L^p(\ell_{t,u})} \doteq \max \{\| \Vr \|_{L^p(\ell_{t,u})},\, \| S \|_{L^p(\ell_{t,u})}\},$$
and similarly for $L^p(\Sigma_t^u)$, $L^p(\Sigma_t)$, $L^p(\mathcal F_u^t)$, and $L^p(\Mtu)$.
\item Let $\mathscr{P} \doteq \{L,Y,Z\}$ be the set of commutation vectorfields and $\mathscr{P}^{(N)} \doteq \{\mathcal{P}_1 \mathcal{P}_2\cdots \mathcal{P}_N | \mathcal{P}_i \in \mathscr{P} \mbox{ for } 1 \leq i \leq N \}$. For any smooth scalar function $\phi$, define:
$$|\mathcal{P}^N \phi| \doteq 
\max_{\mathcal{P}_1,\cdots,\mathcal{P}_N \in \mathscr{P}^{(N)}} |\mathcal{P}_1 \cdots \mathcal{P}_N \phi|.$$
For $p=2$ or $p=\infty$, the $L^p$ norms are defined similarly, with:
$$\|\mathcal{P}^N \phi\|_{L^p(\ell_{t,u})} \doteq \left\| |\mathcal{P}^N \phi| \right\|_{L^p(\ell_{t,u})}, \mbox{etc.}$$
Moreover, we let $\mathcal{P}^N \Vr$ denote the $\Sigma_t$-tangent vectorfield
with Cartesian spatial components $\mathcal{P}^N \Vr^i$,
and we define:
$$|\mathcal{P}^N \Vr| 
\doteq \max_{\mathcal{P}_1 \cdots \mathcal{P}_N \in \mathscr{P}^{(N)}} 
\max_{a=1,2,3} |\mathcal{P}_1 \cdots \mathcal{P}_N \Vr^a|,$$
$$\|\mathcal{P}^N \Vr \|_{L^p(\ell_{t,u})} 
\doteq 
\left\| |\mathcal{P}^N \Vr| \right\|_{L^p(\ell_{t,u})}, \mbox{etc.,}$$
and similarly for other $\Sigma_t$-tangent tensorfields 
that correspond to the transport variables,
such as $S$ and $\mathcal{C}$.
\item Similarly, we let
$\slashed{\mathscr{P}} \doteq \{Y,Z\}$ be the set of $\ell_{t,u}$-tangent commutation vectorfields,
$\slashed{\mathscr{P}}^{(N)} \doteq \{\slashed{\mathcal{P}}_1 \slashed{\mathcal{P}}_2\cdots \slashed{\mathcal{P}}_N 
|\slashed{\mathcal{P}}_i \in \slashed{\mathscr{P}} \mbox{ for } 1 \leq i \leq N \}$, 
$|\slashed{\mathcal{P}}^N \phi| 
\doteq \max_{\slashed{\mathcal{P}}_1,\cdots,\slashed{\mathcal{P}}_N \in \slashed{\mathscr{P}}^{(N)}} 
|\slashed{\mathcal{P}}_1 \cdots \slashed{\mathcal{P}}_N \phi|$,
etc. 
The importance of distinguishing the subset of $\ell_{t,u}$-tangent commutation vectorfields 
from the full set $\mathscr{P}$ will be made clear in the appendix.\footnote{As in the 
$2$D-case, the most difficult error terms
in the wave equation energy estimates are commutator terms 
involving the pure $\ell_{t,u}$-tangent derivatives of the null mean curvature
of the $\ell_{t,u}$.}
\item We use the following conventions for sums:
$$|\mathcal{P}^{[N_1,N_2]} \phi| \doteq \sum_{N'=N_1}^{N_2} |\mathcal{P}^{N'}\phi|,\quad |\mathcal{P}^{\leq N} \phi|\doteq |\mathcal{P}^{[0,N]}\phi|,$$
$$|\slashed{\mathcal{P}}^{[N_1,N_2]} \phi| \doteq \sum_{N'=N_1}^{N_2} |\slashed{\mathcal{P}}^{N'}\phi|,\quad 
|\slashed{\mathcal{P}}^{\leq N} \phi |\doteq |\slashed{\mathcal{P}}^{[0,N]}\phi|.$$	
\item We will combine the above conventions. For instance,
$$|\mathcal{P}^N(\Vr, S)| \doteq  \max_{\mathcal{P}_1 \cdots \mathcal{P}_N \in \mathscr{P}^{(N)}} \max\{ |\mathcal{P}_1 \cdots \mathcal{P}_N \Vr|,\, |\mathcal{P}_1 \cdots \mathcal{P}_N S|\},$$
$$|\mathcal{P}^N \Psi| \doteq \max_{\iota=1,\cdots,5} \max_{\mathcal{P}_1 \cdots \mathcal{P}_N \in \mathscr{P}^{(N)}} |\mathcal{P}_1 \cdots \mathcal{P}_N \Psi_{\imath}|.$$
\end{enumerate}
\end{definition}

\subsubsection{Definitions of the energies}
\label{SSS:DEFSOFENERGIES}
We are now ready to introduce the main energies we use to control the solution.

\medskip 

\noindent{\underline{\textbf{Wave energies}}.

\begin{subequations}
\begin{align}
\mathbb{E}_N(t,u) & \doteq  \sup_{t' \in [0,t]} \left(\|\bX \mathcal{P}^{N} \Psi\|_{L^2(\Sigma_{t'}^u)}^2+ \|\sqrt{\upmu} \mathcal{P}^{N+1} \Psi\|_{L^2(\Sigma_{t'}^u)}^2 \right), \label{eq:wave.energy.def.1}\\
\mathbb{F}_N(t,u) & \doteq \sup_{u'\in [0,u]} \left( \| L\mathcal{P}^N \Psi \|_{L^2(\mathcal F_{u'}^t)}^2 + \| \sqrt{\upmu} \slashed{d} \mathcal{P}^N \Psi \|_{L^2(\mathcal F_{u'}^t)}^2 \right), \\
\mathbb K_N(t,u) & \doteq \|\slashed{d} \mathcal{P}^N \Psi \|_{L^2(\mathcal M_{t,u}\cap\{\upmu \leq \f 14\})}^2, 
\label{eq:wave.energy.def.3}\\
\mathbb{Q}_N(t,u) & \doteq \mathbb{E}_N(t,u) + \mathbb{F}_N(t,u),\\
\mathbb{W}_N(t,u) & \doteq \mathbb{E}_N(t,u) + \mathbb{F}_N(t,u) + \mathbb K_N(t,u).
	\label{eq:wave.energy.def.4}
\end{align}
\end{subequations}

\noindent{\underline{\textbf{Specific vorticity energies}}.

\begin{subequations}
\begin{align}
\mathbb V_N(t,u) \doteq \sup_{t'\in [0,t]} \|\sqrt{\upmu} \mathcal{P}^{N} \Vr\|_{L^2(\Sigma_{t'}^u)}^2 + \sup_{u'\in [0,u]}\| \mathcal{P}^N \Vr \|_{L^2(\mathcal F_{u'}^t)}^2, \\
\mathbb C_N(t,u) \doteq \sup_{t'\in [0,t]} \|\sqrt{\upmu} \mathcal{P}^{N} \mathcal{C}\|_{L^2(\Sigma_{t'}^u)}^2 + \sup_{u'\in [0,u]}\| \mathcal{P}^N \mathcal{C}\|_{L^2(\mathcal F_{u'}^t)}^2.
\label{eq:C.energy.def}
\end{align}
\end{subequations}

\noindent{\underline{\textbf{Entropy gradient energies}}.

\begin{subequations}
\begin{align}
\mathbb S_N(t,u) & \doteq \sup_{t'\in [0,t]} \|\sqrt{\upmu} \mathcal{P}^{N} S\|_{L^2(\Sigma_{t'}^u)}^2 + \sup_{u'\in [0,u]}\| \mathcal{P}^N S \|_{L^2(\mathcal F_{u'}^t)}^2, \\
\mathbb D_N(t,u) & \doteq \sup_{t'\in [0,t]} \|\sqrt{\upmu} \mathcal{P}^{N} \mathcal{D}\|_{L^2(\Sigma_{t'}^u)}^2 + \sup_{u'\in [0,u]}\| \mathcal{P}^N \mathcal{D} \|_{L^2(\mathcal F_{u'}^t)}^2.\label{eq:D.energy.def}
\end{align}
\end{subequations}

\begin{definition}[\textbf{Important conventions for energies}]\label{def:energy.convention}
\

\begin{enumerate}
\item We define the following convention for sums (cf.~Definition~\ref{def:array.convention}(3)):
$$\mathbb{E}_{\leq N}(t,u)\doteq \sum_{N'=0}^N \mathbb{E}_{N'}(t,u),\quad \mathbb{E}_{[1,N]}(t,u)\doteq \sum_{N'=1}^N \mathbb{E}_{N'}(t,u),$$
and similarly for other energies.
\item Abusing notation slightly, if we write an energy as a function of only
$t$ (instead of a function of $(t,u)$), then it is understood that we take supremum in $u$, e.g.,
$$\mathbb{E}_N(t) \doteq \sup_{u\in \mathbb R} \mathbb{E}_N(t,u).$$
\end{enumerate}
\end{definition}

\section{Assumptions on the data and statement of the main theorems}
\label{S:DATASIZEANDBOOTSTRAP}

\subsection{Assumptions on the data of the fluid variables}
\label{SS:FLUIDVARIABLEDATAASSUMPTIONS}
We now introduce the assumptions on the data for our main theorem. 
We have five parameters (cf.~Theorem~\ref{thm:intro.main}), denoted by
$\mathring{\upsigma}$, $\mathring{\updelta}_*$, $\mathring{\updelta}$, $\mathring{\upalpha}$ and $\epd$:
\begin{itemize}
\item $\mathring{\upsigma}$ measures the size of the initial support in $x^1$.
\item $\mathring{\updelta}_*$ gives a lower bound on the quantity that controls the blowup, and in particular determines the time interval for which we need to control our solution before a singularity forms. 
\item $\mathring{\updelta}$, $\mathring{\upalpha}$ and $\epd$ are parameters that control
the sizes of various norms of the solution. 
$\mathring{\updelta}$ measures the $L^{\infty}$ size of the transversal derivatives of $\mathcal{R}_{(+)}$,
and it \emph{can be large}. 
$\mathring{\upalpha}$ limits the size of the amplitude of the fluid variables, is \emph{small} depending on the equation of state and the background density 
$\bar{\varrho} > 0$, and is used to control basic features of the Lorentzian geometry. $\epd$ is a \emph{small} parameter depending on the equation of state and all the other parameters. In particular,
$\epd$ controls the size of solution in ``directions that break the simple plane symmetry,''
and it provides the most crucial smallness that we exploit in the analysis.
\item We assume that $\epd^{\frac{1}{2}} \leq \mathring{\upalpha}$.
\end{itemize}

\medskip
Here are the assumptions on the initial data.\footnote{Of course, we are only allowed to prescribe 
$(\varrho,v^i,s)$ without explicitly specifying their derivatives transversal to $\Sigma_0$. Nevertheless, 
using \eqref{eq:Euler.1}--\eqref{eq:Euler.3}, we can compute the traces of all derivatives on $\Sigma_0$. The derivative assumptions that we specify here are to be understood in this sense. 
Notice that all the assumptions are satisfied by the data of exactly simple plane-symmetric solutions with
$\epd = 0$.
Thus, they are also satisfied by small perturbations of them.}

In what follows, $\Ntop$ and $\toprate$
denote large positive integers that are constrained in particular by
$\Ntop \geq 2 \toprate + 10$. 
In our proof of Proposition~\ref{prop:wave},
we will show that our estimates close with $\toprate$ chosen
to be a universal positive integer.
The restriction $\Ntop \geq 2 \toprate + 10$
is further explained in Remark~\ref{R:NUMBEROFDERIVATIVES}.
See also the discussion in Section~\ref{SS:NOTATION}.

\medskip


\noindent \underline{\textbf{Compact support assumptions}}.
\begin{align}\label{assumption:support}
\mbox{If $|x^1|\geq \mathring{\upsigma}$, then $(\rr,v,s) = (0,0,0)$.}
\end{align}

\medskip

\noindent \underline{\textbf{Lower bound for the quantity that controls the blowup-time}}.\footnote{Here, $z_+\doteq \max\{z,0\}$.}
\begin{align} 
\mathring{\updelta}_* \doteq 
\sup_{\Sigma_0}
\f 12 \left[ \Speed^{-1}
			(\Speed^{-1} \Speed_{;\rr} + 1) (\bX \mathcal R_{(+)})\right]_+ >0. \label{assumption:lower.bound}
\end{align}

\begin{remark}[\textbf{Non-degeneracy assumption on the factor 
$\Speed^{-1} (\Speed^{-1} \Speed_{;\rr} + 1)$}]
	\label{R:NONVANISHINGFACTOR}
	Recall that the factor $\Speed^{-1} (\Speed^{-1} \Speed_{;\rr} + 1)$
	in \eqref{assumption:lower.bound} can be viewed as a function of $(\Densrenormalized,s)$.
	For the solutions under study, we are assuming that 
	$\Speed^{-1} (\Speed^{-1} \Speed_{;\rr} + 1)$
	is non-vanishing when evaluated at the trivial background solution $(\Densrenormalized,s) \equiv (0,0)$
	(recall that this background corresponds to a state with constant density $\varrho \equiv \bar{\varrho} > 0$).
	One can check that for any smooth equation of state except that of a Chaplygin gas,
	there are always open sets of $\bar{\varrho} > 0$ such that the non-vanishing condition holds;
	see the end of \cite[Section~2.16]{LS} for further discussion.
	We also point out that for the Chaplygin gas, it is not expected that shocks will form.
\end{remark}

\medskip

\noindent \underline{\textbf{Assumptions on the amplitude and transversal derivatives of the wave variables}}.
\begin{subequations}
\begin{align} 
	\left\|
		\mathcal{R}_{(+)}
	\right\|_{L^{\infty}(\Sigma_0)}
	& \leq \Psiep,
		\label{assumption:R+} \\
	\left\|
		\Rad^{[1,3]} \mathcal{R}_{(+)}
	\right\|_{L^{\infty}(\Sigma_0)}
	& \leq \mathring{\updelta},
	\label{assumption:R+.trans} \\
	\left\|
		\Rad^{\leq 3} (\mathcal{R}_{(-)},v^2,v^3,s)
	\right\|_{L^{\infty}(\Sigma_0)}
	& \leq \epd, \label{assumption:small.trans}
		\\
	\left\|
		\Lunit \Rad \Rad \Rad \Psi
	\right\|_{L^{\infty}(\Sigma_0)} & \leq \epd.
\end{align}
\end{subequations}

\medskip

\noindent \underline{\textbf{Smallness assumptions for good derivatives of the wave variables}}.
\begin{align} 
\begin{split} \label{assumption:small}
	&
	\left\|
		\mathcal{P}^{[1,\Ntop- \toprate-2]} \Psi
	\right\|_{L^{\infty}(\Sigma_0)},
		\,
	\left\|
		\mathcal{P}^{[1,\Ntop-\toprate-4]} \bX \Psi
	\right\|_{L^{\infty}(\Sigma_0)}, 
		\,
	\left\|
		\mathcal{P}^{[1,2]} \bX \bX \Psi
	\right\|_{L^{\infty}(\Sigma_0)},	
		\\
	&
	\sup_{u \in [0,U_0]}
	\left\|
		\mathcal{P}^{[1,\Ntop-\toprate]} \Psi
	\right\|_{L^2(\ell_{0,u})},
		\\
	&
	\left\|
		\mathcal{P}^{[1,\Ntop+1]} \Psi
	\right\|_{L^2(\Sigma_0)},
		\,
	\left\|
		\mathcal{P}^{[1,\Ntop]} \bX \Psi
	\right\|_{L^2(\Sigma_0)}
		\\
	& \leq 
	\epd.
\end{split}
\end{align}

\medskip

\noindent \underline{\textbf{Smallness assumptions for the specific vorticity and entropy gradient}}.
\begin{align}
\begin{split} \label{assumption:very.small}
&
	\left\|
		\mathcal{P}^{\leq \Ntop - \toprate-2} (\Vortrenormalized,S)
	\right\|_{L^\infty(\Sigma_0)},
		\,
	\sup_{u \in [0,U_0]}	
	\left\|
		\mathcal{P}^{\leq \Ntop-\toprate} (\Vortrenormalized,S)
	\right\|_{L^2(\ell_{0,u})},
		 \\
	&
	\left\|
		\mathcal{P}^{\leq \Ntop} (\Vortrenormalized,S)
	\right\|_{L^2(\Sigma_0)}
		 \\
	& \leq \epd^{\f 32}. 
\end{split}
\end{align}

\medskip

\noindent \underline{\textbf{Smallness assumptions for the modified fluid variables}}.
\begin{align}
\begin{split} \label{assumption:very.small.modified} 
&\left\|
		\mathcal{P}^{\leq \Ntop - \toprate-3} (\mathcal{C},\mathcal{D})
	\right\|_{L^\infty(\Sigma_0)},
		\, 
	\sup_{u \in [0,U_0]}	
	\left\|
		\mathcal{P}^{\leq \Ntop-\toprate-1} (\mathcal{C},\mathcal{D})
	\right\|_{L^2(\ell_{0,u})},
			 \\
	&
	\left\|
		\mathcal{P}^{\leq \Ntop} (\mathcal{C},\mathcal{D})
	\right\|_{L^2(\Sigma_0)}
		 \\
	& \leq \epd^{\f 32}. 
\end{split}
\end{align}




\subsection{Statement of the main theorem}\label{sec:statement}
We are now ready to give a precise statement of Theorem~\ref{thm:intro.main} (see already Theorems~\ref{thm:main} and \ref{thm:shock} below), as well as the corollaries in interesting sub-regimes of solutions discussed in Remark~\ref{rmk:subregime} (see already Corollaries~\ref{cor:stupid.nonvanishing} and \ref{cor:stupid.Holder}).

We first discuss Theorem~\ref{thm:intro.main}. It will be convenient to think of Theorem~\ref{thm:intro.main} as two theorems. The first, which is the much harder theorem, is a \emph{regularity} statement, stating --- with precise estimates --- that in the region under study, the only possible singularity is that of a shock, i.e., one that is associated with the vanishing of $\upmu$. This is the content of Theorem~\ref{thm:main}. 
Once Theorem~\ref{thm:main} has been proved, 
the proof that a shock indeed occurs is much easier.
This is the content of Theorem~\ref{thm:shock}.

\begin{theorem}[\textbf{Regularity unless shock occurs}]\label{thm:main}
Let $\mathring{\upsigma},\,\mathring{\updelta},\,\mathring{\updelta}_* >0$.
There exists a large integer $\toprate$ that is 
\textbf{absolute} in the sense that it is \underline{independent} 
of the equation of state, $\bar{\varrho}$,   
$\mathring{\upsigma}$, $\mathring{\updelta}$, and $\mathring{\updelta}_*^{-1}$,
such that the following hold.
Assume that:
\begin{itemize}
\item The integer $\Ntop$ satisfies $\Ntop \geq 2 \toprate + 10$
(see Remark~\ref{R:NUMBEROFDERIVATIVES} regarding the size of $\Ntop$);
\item $\mathring{\upalpha} > 0$ is sufficiently small 
in a manner that depends only on the equation of state and $\bar{\varrho}$;
\item $\epd > 0$ satisfies\footnote{The assumption $\epd^{\frac{1}{2}} \leq \mathring{\upalpha}$ 
	allows us to simplify the presentation of various estimates, for example, by allowing us to write 
	$\mathcal{O}(\mathring{\upalpha})$ instead of
	$\mathcal{O}(\epd^{\frac{1}{2}}) + \mathcal{O}(\mathring{\upalpha})$.
	\label{FN:EPSILONSMALLERTHANALPHA}}  
$\epd^{\frac{1}{2}} \leq \mathring{\upalpha}$
and is sufficiently small
in
a manner that depends only on 
the equation of state, 
$\Ntop$,
$\bar{\varrho}$, 
 $\mathring{\upsigma}$, 
$\mathring{\updelta}$, 
and $\mathring{\updelta}_*^{-1}$;
\item The initial data satisfy the support-size and norm-size 
assumptions\footnote{Note that our plane-symmetric background solutions
satisfy these assumptions with $\epd = 0$.}  
\eqref{assumption:support}--\eqref{assumption:very.small.modified}.
\end{itemize}
Then the corresponding solution $(\varrho,v^1,v^2,v^3,s)$ to the compressible Euler equations
\eqref{eq:Euler.1}--\eqref{eq:Euler.3} exhibits the following properties.

Suppose $T \in (0,2\mathring{\updelta}_*^{-1}]$, and \underline{\textbf{assume}} 
that there is a smooth solution 
such that the following two conditions hold:
\begin{itemize}
	\item The change of variables map $(t,u,x^2,x^3) \rightarrow (t,x^1,x^2,x^3)$
	from geometric to Cartesian coordinates
	is a diffeomorphism from $[0,T) \times \mathbb{R} \times \mathbb{T}^2$ onto
	$[0,T) \times \Sigma$;
		and
	\item $\upmu>0$ in $[0,T)\times \Sigma$.
\end{itemize}
Then the following estimates hold for every $t\in [0,T)$, where
the implicit constants in 
$\ls_{\mydiam}$ depend only on the equation of state and $\bar{\varrho}$,
while the implicit constants in $\ls$ depend only on 
the equation of state,
$\Ntop$,
$\bar{\varrho}$
$\mathring{\upsigma}$, 
$\mathring{\updelta}$,
and
$\mathring{\updelta}_*^{-1}$
(in particular, all implicit constants are independent of $t$ and $T$).
\begin{enumerate}
\item The following energy estimates hold
(where the energies are defined in \eqref{eq:wave.energy.def.1}--\eqref{eq:D.energy.def} and
$\upmu_{\star}(t)$ is as in Definition~\ref{def:mustar}):
\begin{subequations}
\begin{align}
\label{conclusion:W.energy}\mathbb{W}_{N}(t) & \ls \epd^2 \max\{1, \upmu_{\star}^{-2 \toprate+2\Ntop-2N+1.8}(t)\} \quad \mbox{for $1\leq N\leq \Ntop$}, \\
\label{conclusion:V.energy}\mathbb V_N(t),\,\mathbb S_N(t) & \ls \epd^3\max\{1, \upmu_{\star}^{-2 \toprate+2\Ntop-2N+2.8}(t)\}\quad \mbox{for $0\leq N\leq \Ntop$}, \\
\label{conclusion:C.energy}\mathbb C_N(t),\, \mathbb D_N(t) & \ls \epd^3\max\{1, \upmu_{\star}^{-2 \toprate+2\Ntop-2N+0.8}(t)\}\quad \mbox{for $0\leq N\leq \Ntop$}.
\end{align}
\end{subequations}
\item The following $L^\infty$ estimates hold:
\begin{subequations}
\begin{align}
\label{conclusion:W.P.Li}
\|\mathcal{P}^{[1,\Ntop- \toprate-2]} \Psi\|_{L^\infty(\Sigma_t)} ,\, 
\|\mathcal{P}^{[1,\Ntop- \toprate-4]} \bX \Psi\|_{L^\infty(\Sigma_t)} &\ls \epd, \\
\label{conclusion:Psi.itself.Li}
 \|\mathcal R_{(+)}\|_{L^\infty(\Sigma_t)} \ls_{\mydiam} \mathring{\upalpha},
	\quad 
\|(\mathcal{R}_{(-)},v^2,v^3,s) \|_{L^\infty(\Sigma_t)}
	&\ls \epd,\\
\label{conclusion:XPsi.Li}
\| \bX \mathcal R_{(+)} \|_{L^\infty(\Sigma_t)} \leq 2\mathring{\updelta}, 
	\quad 
\|\bX (\mathcal{R}_{(-)},v^2,v^3,s) \|_{L^\infty(\Sigma_t)} 
& \ls \epd, 
	\\
\label{conclusion:V.S.C.D}
\|\mathcal{P}^{[1,\Ntop- \toprate-2]} (\Vr,S)\|_{L^\infty(\Sigma_t)},
	\, 
\|\mathcal{P}^{[1,\Ntop- \toprate-3]} (\mathcal{C},\mathcal{D})\|_{L^\infty(\Sigma_t)},
	\,
\|\mathcal{P}^{[1,\Ntop- \toprate-4]} \bX (\Vr,S)\|_{L^\infty(\Sigma_t)}
&
\ls \epd^{\f 32}.
\end{align}
\end{subequations}
\end{enumerate}

In addition, the solution can be smoothly extended to 
$[0,T] \times \mathbb{R} \times \mathbb{T}^2$
as a function of the geometric coordinates
$(t,u,x^2,x^3)$.

Finally, if $\inf_{t\in [0,T)} \upmu_{\star}(t) >0$, then the solution can be smoothly extended 
to a Cartesian slab $[0,T + \epsilon] \times \Sigma$ for some $\epsilon > 0$
such that the map $(t,u,x^2,x^3) \rightarrow (t,x^1,x^2,x^3)$
is a diffeomorphism from $[0,T + \epsilon] \times \mathbb{R} \times \mathbb{T}^2$
onto $[0,T + \epsilon] \times \Sigma$. In particular, on the extended region,
the solution is a smooth function of the geometric coordinates and the Cartesian coordinates.
\end{theorem}

\begin{theorem}[\textbf{Complete description of the shock formation at the first singular time}]\label{thm:shock}
Under the assumptions of Theorem~\ref{thm:main} 
-- perhaps taking 
$\mathring{\upalpha}$ and $\epd$ smaller in a manner that depends on the same quantities stated in the theorem -- 
there exists $T_{(Sing)} \in [0,2\mathring{\updelta}_*^{-1}]$ satisfying 
the estimate:\footnote{See Section~\ref{SS:NOTATION} regarding
our use of the notations ``$\mathcal{O}_{\mydiam}(\cdot)$,'' ``$\mathcal{O}(\cdot)$,''
etc.}
\begin{equation}\label{eq:Tsing.est}
T_{(Sing)}
=
\left 
\lbrace 1 
+ 
\mathcal{O}_{\mydiam}(\mathring{\upalpha}) 
+ 
\mathcal{O}(\epd) 
\right\rbrace
\mathring{\updelta}_*^{-1}
\end{equation}
such that the following holds:
\begin{enumerate}
\item The solution variables are smooth functions of the Cartesian coordinates 
$(t,x^1,x^2,x^3)$ in $[0,T_{(Sing)})\times \Sigma$.
\item The solution variables extend as 
	smooth functions of the geometric coordinates $(t,u,x^2,x^3)$ to $[0,T_{(Sing)}] \times \mathbb{R} \times \mathbb{T}^2$.
\item The inverse foliation density tends to zero at $T_{(Sing)}$, i.e., $\liminf_{t \uparrow T_{(Sing)}^-} \upmu_{\star}(t) = 0$.
\item $\rd_1\mathcal R_{(+)}$ blows up as $t \uparrow T_{(Sing)}^-$, i.e., 
	$\limsup_{t \uparrow T_{(Sing)}^-} \sup_{\Sigma_t} |\rd_1 \mathcal R_{(+)}| = \infty$.
\item Moreover, let:\footnote{For definiteness, in the definition of the subset $\mathscr{S}_{Regular}$, 
we have made statements only about the boundedness of the solution's $C^1$ norm. However, 
our proof shows that on $\mathscr{S}_{Regular}$, the solution inherits the full regularity enjoyed by the initial data.}
 \begin{align*}
\mathscr{S}_{Blowup} & \doteq \{(u,x^2,x^3) \in \mathbb{R} \times \mathbb{T}^2: 
	\limsup_{(\tilde{t},\tilde{u},\tilde{x}^2,\tilde{x}^3) 
	\to 
	(T_{(Sing)}^-,u,x^2,x^3)}  
	|\rd_1 \mathcal R_{(+)}|(\tilde{t},\tilde{u},\tilde{x}^2,\tilde{x}^3) = \infty \} , \\
\mathscr{S}_{Vanish} & \doteq  \{ (u,x^2,x^3) \in \mathbb{R} \times \mathbb{T}^2: \upmu(T_{(Sing)},u,x^2,x^3) = 0\}, \\
\mathscr{S}_{Regular} &\doteq \Big\{ (u,x^2,x^3) \in \mathbb{R} \times \mathbb{T}^2: 
	\mbox{all solution variables extend to be $C^1$ functions}
		\\
	& \ \ \ \ \ \ \ \ \ \ \ 
		\mbox{of the geometric \underline{and} Cartesian coordinates 
		in a neighborhood of}
			\\
		& \ \ \ \ \ \ \ \ \ \ \ \mbox{the point with geometric coordinates $(T_{(Sing)},u,x^2,x^3)$,}
				\\
		& \ \ \ \ \ \ \ \ \ \ \ \mbox{intersected with the half-space $\lbrace t \leq T_{(Sing)} \rbrace$ \Big\}}.
 \end{align*}
 Then $\mathscr{S}_{Blowup}$ and $\mathscr{S}_{Vanish}$ are non-empty, and:
 $$\mathscr{S}_{Blowup} = \mathscr{S}_{Vanish} = \mathbb{R} \times \mathbb{T}^2 \setminus \mathscr{S}_{Regular}.$$
\end{enumerate}
\end{theorem}

The proofs of both Theorem~\ref{thm:main} and Theorem~\ref{thm:shock} 
are located in Section~\ref{sec:proof.of.main.theorem}.
\medskip

The next two corollaries concern some refined conclusions one can make with \emph{additional} assumptions on the initial data.

\begin{corollary}[\textbf{Non-vanishing of the vorticity and entropy at the blowup-points}]
\label{cor:stupid.nonvanishing}
Assume the hypotheses and conclusions of Theorem~\ref{thm:main}, 
but perhaps taking $\mathring{\upalpha}$ and $\epd$ smaller in a manner that depends on the same quantities stated in the theorem. Assume in addition that:\footnote{Recall the initial condition \eqref{E:INTROEIKONAL} for $u$,
which shows that $u\restriction_{\Sigma_0} = \mathring{\upsigma} - x^1$.}
\begin{equation}\label{eq:problem.concentrated.somewhere}
\f 12 [\Speed^{-1}(\Speed^{-1} \Speed_{;\rr} +1)(\bX\mathcal R_{(+)})]_+(t=0, u, x^2, x^3) \leq \f 12 \mathring{\updelta}_*^{-1} \quad\mbox{when 
$| u - \mathring{\upsigma} + \mathring{\updelta}_*^{-1}| \geq 3\mathring{\upalpha} \mathring{\updelta}_*^{-1}$},
\end{equation}
and:
\begin{equation}\label{eq:V.S.nonvanishing.assumption}
\f 12\epd^2 \leq |\Vr(t=0,u,x^2,x^3)|\leq \epd^2,
\quad 
\f 12 \epd^3\leq |S(t=0,u,x^2,x^3)|\leq \epd^3 \quad 
\mbox{when $| u - \mathring{\upsigma}|\leq \mathring{\upalpha}^{\f 12}$}.
\end{equation}

Then $\Vr$ and $S$ are non-vanishing near the singular set, i.e., for any 
$(u,x^2,x^3) \in \mathscr{S}_{Blowup}$ (as in Theorem~\ref{thm:shock}), 
we have $\Vr(T_{(Sing)},u,x^2,x^3) \neq 0$ and $S(T_{(Sing)},u,x^2,x^3) \neq 0$.
\end{corollary}

The proof of Corollary~\ref{cor:stupid.nonvanishing} is located in Section~\ref{sec:stupid.nonvanishing}.

\begin{corollary}[\textbf{The spatial H\"{o}lder regularity of the solution relative to the Cartesian coordinates}]
\label{cor:stupid.Holder}
Let $\mathring{\upbeta}>0$ be a constant,
and assume that the following hold:
\begin{enumerate}
\item For all $u$ such that $|u - \mathring{\upsigma}| \geq \f{\mathring{\upsigma}}4$, we have:
$$\f 12 [\Speed^{-1}(\Speed^{-1} \Speed_{;\rr} +1)(\bX\mathcal R_{(+)})]_+(t=0,u,x^2, x^3) 
\leq \f 14 \mathring{\updelta}_*.$$
\item For all\footnote{This is a non-degeneracy condition 
in the sense that it guarantees that
 for every 
$(x^2,x^3) \in \mathbb{T}^2$, the quantity $(\Speed^{-1} \Speed_{;\rr} +1)(\bX\mathcal R_{(+)})\restriction_{\Sigma_0}$, when viewed as a one-variable function 
of $u$, has a non-degenerate maximum. (Note also that $(\Speed^{-1} \Speed_{;\rr} +1)(\bX\mathcal R_{(+)})\restriction_{\Sigma_0}$ is related to the quantity 
in \eqref{assumption:lower.bound}, whose reciprocal 
controls the blowup-time.)} $u \in [\f{\mathring{\upsigma}}2, \frac{3\mathring{\upsigma}}{2}]$,
\begin{equation}\label{eq:nondegeneracy}
\f 12 \bX \bX  \{ (\Speed^{-1} \Speed_{;\rr} +1)(\bX\mathcal R_{(+)})\}(t=0,u, x^2, x^3)   \leq - 3 \mathring{\updelta}_* \mathring{\upbeta} < 0.
\end{equation}
\end{enumerate}

Also assume the hypotheses and conclusions of Theorems~\ref{thm:main} and \ref{thm:shock}, 
but perhaps taking $\mathring{\upalpha}$ and $\epd$ smaller in a manner that depends on $\mathring{\upbeta}$ and
the same quantities stated in Theorem~\ref{thm:main}. Then 
the spatial $C^{1/3}$ norms 
(i.e., the standard $C^{1/3}$ H\"{o}lder norms with respect to the Cartesian spatial coordinates)
of all of the fluid variables and higher-order variables
$\rr$, $v^i$, $\Vr^i$, $S^i$, $\mathcal{C}^i$ and $\mathcal{D}$ 
are uniformly bounded up to the first singular time.
\end{corollary}

The proof of Corollary~\ref{cor:stupid.Holder} is located in Section~\ref{sec:stupid.Holder}.


\section{Reformulation of the equations and the remarkable null structure}\label{sec:speck.reformulation}
We recall in this section the main result in \cite{jS2019c}, which is of crucial importance for our analysis.

\begin{theorem}[\textbf{The geometric wave-transport-divergence-curl formulation of the compressible Euler equations}]
	\label{T:GEOMETRICWAVETRANSPORTSYSTEM}
	Consider a smooth solution to the compressible Euler equations
	\eqref{eq:Euler.1}--\eqref{eq:Euler.3}
	under an equation of state $p = p(\varrho,s)$ and constant $\bar{\varrho} > 0$ 
	such that the normalization condition \eqref{eq:c.normalization} holds.
	Then the scalar-valued functions
	$v^i$,	
	$\mathcal R_{(\pm)}$,
	$\Vortrenormalized^i$,
$\Ent$,
$\GradEnt^i$,
$\Flatdiv \Vortrenormalized$,
$\CurlofVortrenormalized^i$,
$\DivofEntrenormalized$,
and
$(\Flatcurl \GradEnt)^i$,
	$(i=1,2,3)$, (see Definitions~\ref{def:variables.fluid} and \ref{def:variables.HO})
	obey the following system of equations
	(where the Cartesian component functions $v^i$ are 
	\textbf{treated as scalar-valued functions
	under covariant differentiation} on LHS~\eqref{E:VELOCITYWAVEEQUATION}):

\medskip

\noindent \underline{\textbf{\upshape Covariant wave equations}}
	\begin{subequations}
	\begin{align}
		\square_g v^i
		& = 
			- 
			\Speed^2 \exp(2 \Densrenormalized) \CurlofVortrenormalized^i
			+ 
			\mathfrak{Q}_{(v)}^i
			+ 
			\mathfrak{L}_{(v)}^i,
			\label{E:VELOCITYWAVEEQUATION}	\\
			\square_g \mathcal R_{(\pm)} & = 
			- \Speed^2 \exp(2 \Densrenormalized) \CurlofVortrenormalized^1
			\pm
			\left\lbrace
				F_{;\Ent }\Speed^2 \exp(2\rr) 
				-
				c \exp(\Densrenormalized) \frac{p_{;\Ent}}{\bar{\varrho}} 
			\right\rbrace
			\DivofEntrenormalized
			+ 
			\mathfrak{Q}_{(\pm)}
			+ 
			\mathfrak{L}_{(\pm)}, 
			\label{E:PMWAVEEQUATION}\\
	\square_g s 
		& 
		=  
		\Speed^2 \exp(2\rr) \DivofEntrenormalized 
		+
		\mathfrak{L}_{(s)}. \label{E:ENTROPYWAVEEQUATION}
\end{align}
\end{subequations}

\noindent \underline{\textbf{\upshape Transport equations}}
\begin{subequations}
\begin{align}	\Transport \Vortrenormalized^i
	& = \mathfrak{L}_{(\Vortrenormalized)}^i,
		\label{E:RENORMALIZEDVORTICTITYTRANSPORTEQUATION}
		\\
	\Transport \Ent
	& = 0,	
		\label{E:ENTROPYTRANSPORTMAINSYSTEM}
			\\
	\Transport \GradEnt^i
	& = \mathfrak{L}_{(\GradEnt)}^i.
		\label{E:GRADENTROPYTRANSPORT}
	\end{align}
	\end{subequations}

\noindent \underline{\textbf{\upshape Transport-divergence-curl system for the specific vorticity}}
\begin{subequations}
\begin{align} \label{E:FLATDIVOFRENORMALIZEDVORTICITY}
	\Flatdiv \Vortrenormalized
	& = 
		\mathfrak{L}_{(\Flatdiv \Vortrenormalized)},
		\\
\Transport 
\CurlofVortrenormalized^i \label{E:EVOLUTIONEQUATIONFLATCURLRENORMALIZEDVORTICITY} 
& =  \mathfrak M_{(\CurlofVortrenormalized)}^i 
	+
	\mathfrak{Q}_{(\CurlofVortrenormalized)}^i
	+	
	\mathfrak{L}_{(\CurlofVortrenormalized)}^i.
\end{align}	
\end{subequations}

\noindent \underline{\textbf{\upshape Transport-divergence-curl system for the entropy gradient}}
\begin{subequations}
\begin{align} 	
\Transport \DivofEntrenormalized
	& =  		\mathfrak{M}_{(\DivofEntrenormalized)} 
			+
			\mathfrak{Q}_{(\DivofEntrenormalized)}, 	\label{E:TRANSPORTFLATDIVGRADENT}
			\\
	(\Flatcurl \GradEnt)^i & = 0.
	\label{E:CURLGRADENT}
\end{align}
\end{subequations}

\medskip

	Above, the main terms in the transport equations for the modified fluid variables take the form:
	\begin{subequations}
	\begin{align}
	\begin{split} \label{eq:def.M.C}
	\mathfrak M_{(\mathcal{C})}^i & \doteq - 
		2 \delta_{jk} \epsilon_{iab} \exp(-\Densrenormalized) (\partial_a v^j) \partial_b \Vortrenormalized^k
		+
		\epsilon_{ajk}
		\exp(-\Densrenormalized)
		(\partial_a v^i) 
		\partial_j \Vortrenormalized^k
			 \\
& \ \
		+ 
		\exp(-3 \Densrenormalized) \Speed^{-2} \frac{p_{;\Ent}}{\bar{\varrho}} 
		\left\lbrace
			(\Transport \GradEnt^a) \partial_a v^i
			-
			(\Transport v^i) \partial_a \GradEnt^a
		\right\rbrace
		 \\
	& \ \
		+
		\exp(-3 \Densrenormalized) \Speed^{-2} \frac{p_{;\Ent}}{\bar{\varrho}}  
		\left\lbrace
			(\Transport v^a) \partial_a \GradEnt^i
			- 
			(\partial_a v^a) \Transport \GradEnt^i
		\right\rbrace,
	\end{split}
	\\
	\mathfrak M_{(\DivofEntrenormalized)}
	& \doteq 2 \exp(-2 \Densrenormalized) 
			\left\lbrace
				(\partial_a v^a) \partial_b \GradEnt^b
				-
				(\partial_a v^b) \partial_b \GradEnt^a
			\right\rbrace
			+
			\exp(-\Densrenormalized) \delta_{ab} (\Flatcurl \Vortrenormalized)^a \GradEnt^b. \label{eq:def.M.D}
	\end{align} 
	\end{subequations}
	The terms
	$\mathfrak{Q}_{(v)}^i$,
	$\mathfrak{Q}_{(\pm)}$, 
	$\mathfrak{Q}_{(\CurlofVortrenormalized)}^i$,
	and
	$\mathfrak{Q}_{(\DivofEntrenormalized)}$
	are the \textbf{null forms relative to} $g$ defined by:
	\begin{subequations}
		\begin{align}
		\mathfrak{Q}_{(v)}^i	
		& \doteq 	-
					\left\lbrace
						1
						+
						\Speed^{-1} \Speed_{; \rr}
					\right\rbrace
					(g^{-1})^{\alpha \beta} (\partial_{\alpha} \Densrenormalized) \partial_{\beta} v^i,
			\label{E:VELOCITYNULLFORM} \\
		\mathfrak{Q}_{(\pm)}
		& \doteq 
		\mathfrak{Q}_{(v)}^i 
		\mp 
		2 \Speed_{; \rr} 
		(g^{-1})^{\alpha \beta} \partial_{\alpha} \Densrenormalized \partial_{\beta} \Densrenormalized
		\pm
		\Speed \left\lbrace
			(\partial_a v^a) (\partial_b v^b)
			-
			(\partial_a v^b) \partial_b v^a
		\right\rbrace,
			\label{E:DENSITYNULLFORM}
				\\
	\begin{split} \label{E:RENORMALIZEDVORTICITYCURLNULLFORM}
	\mathfrak{Q}_{(\CurlofVortrenormalized)}^i
	& \doteq
		\exp(-3 \Densrenormalized) \Speed^{-2} \frac{p_{;\Ent}}{\bar{\varrho}}  \GradEnt^i
		\left\lbrace
			(\partial_a v^b) \partial_b v^a
			-
			(\partial_a v^a) \partial_b v^b
		\right\rbrace
			 \\
& \ \
		+ 
		\exp(-3 \Densrenormalized) \Speed^{-2} \frac{p_{;\Ent}}{\bar{\varrho}}
		\left\lbrace
			(\partial_a v^a) \GradEnt^b \partial_b v^i 
			- 
			(\GradEnt^a \partial_a v^b) \partial_b v^i 
		\right\rbrace
			\\
& \ \
		+ 
		2 \exp(-3 \Densrenormalized) \Speed^{-2} \frac{p_{;\Ent}}{\bar{\varrho}}
		\left\lbrace
			(\GradEnt^a \partial_a \Densrenormalized)  \Transport v^i
			 - 
		  (\Transport \Densrenormalized) \GradEnt^a \partial_a v^i
		\right\rbrace
			\\
	&  \ \
			+ 
			2 \exp(-3 \Densrenormalized) \Speed^{-3} \Speed_{;\Densrenormalized} \frac{p_{;\Ent}}{\bar{\varrho}}
			\left\lbrace
				(\GradEnt^a \partial_a \Densrenormalized)  \Transport v^i
				- 
				(\Transport \Densrenormalized) \GradEnt^a \partial_a v^i
			\right\rbrace
		 \\
	& \ \
		+ 
		\exp(-3 \Densrenormalized) \Speed^{-2} \frac{p_{;\Ent;\Densrenormalized}}{\bar{\varrho}}
		\left\lbrace
			(\Transport \Densrenormalized) \GradEnt^a \partial_a v^i
			- 
			(\GradEnt^a \partial_a \Densrenormalized) \Transport v^i
		\right\rbrace
		 \\
	& \ \
		+
		\exp(-3 \Densrenormalized) \Speed^{-2} \frac{p_{;\Ent;\Densrenormalized}}{\bar{\varrho}} \GradEnt^i
		\left\lbrace
			(\Transport v^a) \partial_a \Densrenormalized
			-
			 (\partial_a v^a) \Transport \Densrenormalized
		\right\rbrace
		 \\
	& \ \
		+
		2 \exp(-3 \Densrenormalized) \Speed^{-2} \frac{p_{;\Ent}}{\bar{\varrho}} \GradEnt^i
		\left\lbrace
			(\partial_a v^a) \Transport \Densrenormalized
			- 
			(\Transport v^a) \partial_a \Densrenormalized
		\right\rbrace
			\\
	& \ \
		 	+
			2 \exp(-3 \Densrenormalized) \Speed^{-3} \Speed_{;\Densrenormalized} \frac{p_{;\Ent}}{\bar{\varrho}} \GradEnt^i
			\left\lbrace
				(\partial_a v^a) \Transport \Densrenormalized
		 		- 
		 		(\Transport v^a) \partial_a \Densrenormalized
		 	\right\rbrace,
\end{split}
		 		\\
\label{E:DIVENTROPYGRADIENTNULLFORM}
	\mathfrak{Q}_{(\DivofEntrenormalized)}
	& \doteq
		2 \exp(-2 \Densrenormalized) 
		\left\lbrace
			(\GradEnt^a \partial_a v^b) \partial_b \Densrenormalized
			-
			(\partial_a v^a) \GradEnt^b \partial_b \Densrenormalized
		\right\rbrace.
\end{align}
\end{subequations}
	In addition, the terms
	$\mathfrak{L}_{(v)}^i$,
	$\mathfrak{L}_{(\pm)}$,
	$\mathfrak{L}_{(s)}$,
	$\mathfrak{L}_{(\Vortrenormalized)}^i$,
	$\mathfrak{L}_{(\GradEnt)}^i$,
		$\mathfrak{L}_{(\Flatdiv \Vortrenormalized)}$,
		and
	$\mathfrak{L}_{(\CurlofVortrenormalized)}^i$,
	which are at most linear in the derivatives of the unknowns, are defined as follows:
	\begin{subequations}
	\begin{align} 
	\begin{split} \label{E:VELOCITYILINEARORBETTER} 
		\mathfrak{L}_{(v)}^i
		& \doteq 
		2 \exp(\Densrenormalized) \epsilon_{iab} (\Transport v^a) \Vortrenormalized^b
		-
		\frac{p_{;\Ent}}{\bar{\varrho}} \epsilon_{iab} \Vortrenormalized^a \GradEnt^b
			\\	
	& \ \
		- 
		\frac{1}{2} \exp(-\Densrenormalized) \frac{p_{;\Densrenormalized;\Ent}}{\bar{\varrho}} \GradEnt^a \partial_a v^i
			\\
		& \ \
		- 
		2 \exp(-\Densrenormalized) \Speed^{-1} \Speed_{;{\Densrenormalized}} \frac{p_{;\Ent}}{\bar{\varrho}} 
		(\Transport \Densrenormalized) \GradEnt^i
		+
		\exp(-\Densrenormalized) \frac{p_{;\Ent;\Densrenormalized}}{\bar{\varrho}} (\Transport \Densrenormalized) \GradEnt^i,
		\end{split}	
			\\
		\begin{split}\label{E:PMLINEARORBETTER}   
		\mathfrak {L}_{(\pm)} & 
		\doteq \mathfrak{L}_{(v)}^i 
		\pm
		F_{;\Ent}\mathfrak{L}_{(s)}
		\mp
		\f 52 \Speed \exp(-\Densrenormalized) \frac{p_{;\Ent;\Densrenormalized}}{\bar{\varrho}} 
		\GradEnt^a \partial_a \Densrenormalized 
		\pm 
		2 c^2 c_{;s} \GradEnt^a \rd_a \rr \\
		& \ \
		\mp
		\Speed \exp(-\Densrenormalized) \frac{p_{;\Ent;\Ent}}{\bar{\varrho}} \delta_{ab} \GradEnt^a \GradEnt^b 
		\pm
		F_{;\Ent;\Ent} \Speed^2 \delta_{ab} S^a S^b, 
		\end{split}
			\\	
		\mathfrak{L}_{(s)}
		& \doteq
		\Speed^2 \GradEnt^a \rd_a \rr - c c_{;\rr} \GradEnt^a \rd_a \rr - c c_{;s} \de_{ab}\GradEnt^a \GradEnt^b, \label{E:ENTROPYLINEARORBETTER} \\
		\mathfrak{L}_{(\Vortrenormalized)}^i
		& \doteq 
		\Vortrenormalized^a \partial_a v^i
		-
		\exp(-2 \Densrenormalized) \Speed^{-2} \frac{p_{;\Ent}}{\bar{\varrho}} \epsilon_{iab} (\Transport v^a) \GradEnt^b,
		\label{E:SPECIFICVORTICITYLINEARORBETTER}
			\\
		\mathfrak{L}_{(\GradEnt)}^i
		& \doteq
			- \GradEnt^a \partial_a v^i
			+ 
			\epsilon_{iab} \exp(\Densrenormalized) \Vortrenormalized^a \GradEnt^b,
			\label{E:ENTROPYGRADIENTLINEARORBETTER}
				\\
	\mathfrak{L}_{(\Flatdiv \Vortrenormalized)}
		& \doteq - \Vortrenormalized^a \partial_a \Densrenormalized,
		\label{E:RENORMALIZEDVORTICITYDIVLINEARORBETTER} \\
	\begin{split} \label{E:RENORMALIZEDVORTICITYCURLLINEARORBETTER} 
	\mathfrak{L}_{(\CurlofVortrenormalized)}^i	 		
	& \doteq
		 		2 \exp(-3 \Densrenormalized) \Speed^{-3} \Speed_{;\Ent} \frac{p_{;\Ent}}{\bar{\varrho}} 
				(\Transport v^i) \delta_{ab} \GradEnt^a \GradEnt^b
				-
				2 \exp(-3 \Densrenormalized) \Speed^{-3} \Speed_{;\Ent} \frac{p_{;\Ent}}{\bar{\varrho}} 
				\delta_{ab} \GradEnt^a (\Transport v^b) \GradEnt^i
			\\
		& \ \
			+ 
			\exp(-3 \Densrenormalized) \Speed^{-2} \frac{p_{;\Ent;\Ent}}{\bar{\varrho}} (\Transport v^i) \delta_{ab} \GradEnt^a \GradEnt^b
			- 
			\exp(-3 \Densrenormalized) \Speed^{-2} \frac{p_{;\Ent;\Ent}}{\bar{\varrho}} \delta_{ab} (\Transport v^a) \GradEnt^b \GradEnt^i.
\end{split} 
	\end{align}
	\end{subequations}
\end{theorem}
\begin{proof}
The equations are copied from \cite[Theorem~1]{jS2019c}, except 
we have replaced the wave equations for $\rr, v^1$
from \cite[Theorem~1]{jS2019c} with equivalent wave equations
for $\mathcal R_{(\pm)}$ with the help of the identity
$$
\square_g \mathcal R_{(\pm)} = \square_g v^1 \pm 
\left\lbrace
	\Speed \,\square_g\rr 
	+ 
	\Speed_{;\rr} (g^{-1})^{\alp\bt}\rd_\alp \rr\rd_\bt\rr 
	+ 
	2 c^2 c_{;s} \GradEnt^a \rd_a \rr
	+
	F_{;\Ent;\Ent} \Speed^2 \delta_{ab} S^a S^b
	+
	F_{;\Ent} \square_g \Ent
	\right\rbrace,$$
which follows from 
\eqref{E:RIEMANNINVARIANTS},
the chain rule,
the expression \eqref{E:INVERSEACOUSTICALMETRIC} for $g^{-1}$, 
and the transport equation $\Transport \Ent = 0$, i.e., \eqref{eq:Euler.3}. \qedhere
\end{proof}

\section{The bootstrap assumptions and statement of the main a priori estimates}
\label{sec:bootstrap}
We prove our theorem with a bootstrap argument.
In this section, we state the precise bootstrap assumptions as well as a theorem that features our main a priori estimates. The proof of the theorem occupies Sections~\ref{sec:trivial}--\ref{sec:end.of.bootstrap}.

\subsection{Bootstrap assumptions}
\label{SS:BOOTSTRAPASSUMPTIONS}
We now introduce our bootstrap assumptions. 
In the context of Theorem~\ref{thm:bootstrap} below,
we assume that the bootstrap assumptions in the next two subsubsections hold
hold for $t \in [0,\Tboot)$, where $\Tboot \in [0, 2 \mathring{\updelta}_*^{-1}]$
is a ``bootstrap time.''

\subsubsection{Soft bootstrap assumptions}
\label{SSS:SOFTBOOTSTRAP}
\begin{enumerate}
\item We assume that the change of variables map $(t,u,x^2,x^3) \rightarrow (t,x^1,x^2,x^3)$
from geometric to Cartesian coordinates 
is a $C^1$ diffeomorphism from $[0,\Tboot) \times \mathbb{R} \times \mathbb{T}^2$ onto
$[0,\Tboot) \times \Sigma$.
\item We assume that $\upmu > 0$ on $[0,\Tboot) \times \mathbb{R} \times \mathbb{T}^2$.
\end{enumerate}
The first of these ``soft bootstrap assumptions'' 
allows us, in particular, to switch back and forth between viewing tensorfields
as a function of the geometric coordinates (which is the dominant view we take throughout the analysis)
and the Cartesian coordinates. The second soft bootstrap assumption guarantees
that there are no shocks present in the bootstrap region
(though it allows for the possibility that a shock will form precisely at time $\Tboot$).

\subsubsection{Quantitative bootstrap assumptions}
\label{SSS:QUANTITATIVEBOOTSTRAP}
Let $\toprate\in \mathbb N$ be the absolute constant appearing in
the statements of Theorem~\ref{thm:main} above and
Proposition~\ref{prop:wave} below. 
{ Moreover, as we stated already in Section~\ref{SS:FLUIDVARIABLEDATAASSUMPTIONS}, 
$\Ntop$ denotes any fixed positive integer satisfying}
$\Ntop \geq  2 \toprate + 10$. 

\begin{remark}[\textbf{Rationale behind our choice $\Ntop \geq 2 \toprate + 10$}]
	\label{R:NUMBEROFDERIVATIVES}
	Later on, our assumption $\Ntop \geq  2 \toprate + 10$ and the bootstrap assumptions
	will allow us to control $\leq \Ntop$ derivatives of nonlinear products 
	by bounding all terms in $L^{\infty}$ except perhaps the one factor
	hit by the most derivatives. Roughly,\footnote{In reality, the different solution
	variables that we have to track, such as $\Psi$, $\Vr^i$, $\Lunit^i$, $\upmu$, etc., 
	exhibit slightly different
	amounts of $L^{\infty}$ regularity.} 
	the reason is that our derivative
	count will be such that any factor 
	that is hit by $\leq \Ntop - \toprate - 4$ or fewer
	derivatives is bounded in $L^{\infty}$. We will often avoid
	explicitly pointing out this aspect of our derivative count.
\end{remark}

\medskip

\noindent\underline{\textbf{$L^2$ bootstrap assumptions for the wave variables}}.

For $\Ntop- \toprate+1 \leq N\leq \Ntop$,\footnote{Equivalently, for $0\leq K \leq  \toprate-1$,
$\mathbb{W}_{\Ntop-K}(t) \leq \epd \upmu_{\star}^{-2 \toprate+2K+1.8}(t).$}
we assume the following bounds, where 
the energies $\mathbb{W}_N$ are defined in Section~\ref{SSS:DEFSOFENERGIES} and
$\upmu_{\star}(t)$ is defined in Definition~\ref{def:mustar}:
\begin{equation}\label{BA:W1}
\mathbb{W}_{N}(t) \leq \epd \upmu_{\star}^{-2 \toprate+2N_{top}-2N+1.8}(t).
\end{equation}
For $1\leq N \leq \Ntop- \toprate$,
\begin{equation}\label{BA:W2}
\mathbb{W}_N(t) \leq \epd.
\end{equation}



\medskip

\noindent\underline{\textbf{$L^\infty$ bootstrap assumptions for the wave variables}}.

\begin{equation}\label{BA:LARGERIEMANNINVARIANTLARGE}
\| \mathcal R_{(+)} \|_{L^\infty(\Sigma_t)} \leq \mathring{\upalpha}^{\f 12}, 
\quad \|\bX \mathcal R_{(+)} \|_{L^\infty(\Sigma_t)}\leq 3\mathring{\updelta},
\end{equation}

\begin{align} \label{BA:SMALLWAVEVARIABLESUPTOONETRANSVERSALDERIVATIVE}
	\|(\mathcal{R}_{(-)}, v^2, v^3, s)\|_{L^\infty(\Sigma_t)},
		\quad
	\|\Rad (\mathcal{R}_{(-)}, v^2, v^3, s)\|_{L^\infty(\Sigma_t)}
	&
	\leq \epd^{\f 12},
\end{align}

\begin{equation}\label{BA:W.Li.small}
\|\mathcal{P}^{[1,\Ntop- \toprate-2]} \Psi\|_{L^\infty(\Sigma_t)} \leq \epd^{\f 12},
	\quad 
\|\mathcal{P}^{{[1,\Ntop- \toprate-4]}} \bX \Psi\|_{L^\infty(\Sigma_t)} \leq \epd^{\f 12}.
\end{equation}

\medskip

\noindent\underline{\textbf{$L^\infty$ bootstrap assumptions for the specific vorticity}}.

\begin{equation}\label{BA:V}
\|\mathcal{P}^{{\leq \Ntop- \toprate-2}} \Vr \|_{L^\infty(\Sigma_t)} 
+ 
\|\mathcal{P}^{{\leq \Ntop- \toprate-4}} \bX \Vr \|_{L^\infty(\Sigma_t)} \leq \epd.
\end{equation}

\noindent\underline{\textbf{$L^\infty$ bootstrap assumptions for the entropy gradient}}.

\begin{equation}\label{BA:S}
 \|\mathcal{P}^{{\leq \Ntop- \toprate-2}} S \|_{L^\infty(\Sigma_t)} 
+ 
\|\mathcal{P}^{{\leq \Ntop- \toprate-4}} \bX S \|_{L^\infty(\Sigma_t)}\leq \epd.
\end{equation}

\noindent\underline{\textbf{$L^\infty$ bootstrap assumptions for the modified fluid variables}}.
\begin{equation}\label{BA:C.D}
\|\mathcal{P}^{{\leq \Ntop- \toprate-3}} (\mathcal{C}, \mathcal{D}) \|_{L^\infty(\Sigma_t)} 
\leq \epd.
\end{equation}

\begin{remark}[\textbf{The main large quantity in the problem}]
	\label{R:ONELARGEQUANTITY}
	From the discussion of the parameters at the beginning of Section~\ref{SS:FLUIDVARIABLEDATAASSUMPTIONS}
	and
	\eqref{BA:LARGERIEMANNINVARIANTLARGE}--\eqref{BA:C.D}
	we see that the main large quantity in the problem is
	$\bX \mathcal R_{(+)}$; all other terms exhibit smallness that is controlled
	by $\mathring{\upalpha}$ and $\epd$. This, of course, is tied to the
	kind of initial data we treat here.
\end{remark}

\subsection{Statement of the main a priori estimates}
We now state the theorem that yields our main a priori estimates.
Its proof will be the content of Sections~\ref{sec:trivial}--\ref{sec:end.of.bootstrap}.

\begin{theorem}[\textbf{The main a priori estimates}]\label{thm:bootstrap}
Let $\Tboot \in [0, 2\mathring{\updelta}_*^{-1}]$. Suppose that:
\begin{enumerate}
\item The bootstrap assumptions \eqref{BA:W1}--\eqref{BA:C.D} all hold for all $t\in [0,\Tboot)$
(where we recall that in the bootstrap assumptions, $\Ntop$ is any integer
 satisfying $\Ntop \geq 2 \toprate + 10$,
 where $\toprate\in \mathbb N$ is the absolute constant appearing in
the statements of Theorem~\ref{thm:main} and Proposition~\ref{prop:wave});
\item In \eqref{BA:LARGERIEMANNINVARIANTLARGE}, the parameter 
$\mathring{\upalpha}$ is sufficiently small in a manner that depends
only on the equation of state and $\bar{\varrho}$; 
\item The parameter $\epd > 0$ in \eqref{BA:W1}--\eqref{BA:C.D}
satisfies $\epd^{\frac{1}{2}} \leq \mathring{\upalpha}$ and
is sufficiently small in
a manner that depends only on 
the equation of state, 
$\Ntop$,
$\bar{\varrho}$, 
$\mathring{\upsigma}$, 
$\mathring{\updelta}$, 
and $\mathring{\updelta}_*^{-1}$;
	\\
and
\item The soft bootstrap assumptions stated in Section~\ref{SSS:SOFTBOOTSTRAP} hold 
	(including $\upmu>0$ in $[0,\Tboot) \times \mathbb{R} \times \mathbb{T}^2$).
\end{enumerate}

Then there exists a constant $C_{\mydiam}>0$ depending only on the equation of state and $\bar{\varrho}$, 
and a constant $C>0$ depending on 
 the equation of state,
$\Ntop$,
$\bar{\varrho}$,
$\mathring{\upsigma}$, $\mathring{\updelta}$, and $\mathring{\updelta}_*^{-1}$,
such that the following holds for all $t\in [0,\Tboot)$:
\begin{enumerate}
\item \eqref{BA:W1} and \eqref{BA:W2} hold with $\epd$ replaced by $C \epd^2$;
\item The two inequalities in \eqref{BA:LARGERIEMANNINVARIANTLARGE} hold with $\mathring{\upalpha}^{\f 12}$ replaced by $C_{\mydiam} \mathring{\upalpha}$ and 
$3 \mathring{\updelta}$ replaced by $2\mathring{\updelta}$ respectively;
\item The inequalities in 
\eqref{BA:SMALLWAVEVARIABLESUPTOONETRANSVERSALDERIVATIVE} and \eqref{BA:W.Li.small} 
hold with $\epd^{\f 12}$ replaced by $C\epd$; \\
	and
\item The inequalities \eqref{BA:V}--\eqref{BA:C.D} all hold with $\epd$ replaced by $C\epd^{\f 32}$.
\end{enumerate}
\end{theorem}

Sections~\ref{sec:trivial}--\ref{sec:Linfty} will be devoted to the proof of Theorem~\ref{thm:bootstrap}. See Section~\ref{sec:end.of.bootstrap} for the conclusion of the proof.

\textbf{From now on, we will use the conventions for constants stated in 
Section~\ref{SS:NOTATION} and Theorem~\ref{thm:bootstrap}.}

\section{A localization lemma via finite speed of propagation}\label{sec:trivial}
We work under the assumptions of Theorem~\ref{thm:bootstrap}.

\begin{lemma}[\textbf{A localization lemma}]
\label{lem:localization}
Let $U_0 \doteq  2 \mathring{\upsigma} + 4\mathring{\updelta}_*^{-1}$. Then for all $t\in [0,\Tboot)$,
$$(\rr,v,s) = (0,0,0),\quad \mbox{whenever $u\notin (0,U_0)$}.$$
\end{lemma}
\begin{proof}
Recall that we have normalized (see \eqref{eq:c.normalization}) $\Speed(0,0) = 1$, and (by \eqref{assumption:support}) the data are compactly supported in the region where $|x^1|\leq \mathring{\upsigma}$. Hence, by a standard finite speed of propagation argument, we see that $(\rr, v, s) = (0,0,0)$ whenever $|x^1| \geq \mathring{\upsigma} + t$. 
More precisely, 
this can be proved by applying standard energy methods 
to the first-order formulation 
of the compressible Euler equations provided by \cite[Equation~(1.201)]{dC2019},
where the relevant energy identities can be obtained with the help of
the ``energy current'' vectorfields defined by
\cite[Equations~(1.204),(1.205)]{dC2019}.
Since $t < \Tboot \leq 2\mathring{\updelta}_*^{-1}$, 
$$\{(t,x)\in [0,\Tboot)\times \Sigma: t-x^1 \geq 
\mathring{\upsigma} + 4\mathring{\updelta}_*^{-1}\} \subseteq 
\overbrace{\{(t,x)\in [0,\Tboot)\times \Sigma: x^1\leq -\mathring{\upsigma} - t\}}^{\mbox{\upshape solution is trivial here}}.$$
In particular, this implies: 
\begin{equation}\label{eq:support.in.x}
(\rr,v,s) = (0,0,0) \quad\mbox{unless $-\mathring{\upsigma}< t-x^1 < \mathring{\upsigma} + 4\mathring{\updelta}_*^{-1}$}.
\end{equation}

Observe now that since $u\restriction_{\{t=0\}} = \mathring{\upsigma}-x^1$, in the set $\{(t,x)\in [0,\Tboot)\times \Sigma: |x^1|\geq \mathring{\upsigma} + t\}$ (where the solution is trivial), we have $u=t+\mathring{\upsigma}-x^1$. In particular, $\{ u =0\} = \{t-x^1 = -\mathring{\upsigma}\}$ and 
$\{u = U_0\} = \{ t-x^1 = \mathring{\upsigma} + 4\mathring{\updelta}_*^{-1}\}$. The conclusion thus follows from \eqref{eq:support.in.x}. \qedhere
\end{proof}

\textbf{For the rest of the paper, $U_0 > 0$ denotes the constant appearing in the statement of
Lemma~\ref{lem:localization}.}

\section{Estimates for the geometric quantities associated to the acoustical metric}\label{sec:geometry}
We continue to work under the assumptions of Theorem~\ref{thm:bootstrap}.

In this section, we collect some estimates of the geometric quantities $\upmu$, 
$L^i_{(Small)}$ (see Definition~\ref{D:CARTESIANCOMPONENTSOFVECTORFIELDS}), 
under the bootstrap assumptions on the fluid variables. These estimates are the same as those appearing in 
\cites{jSgHjLwW2016,LS}.
Our analysis will therefore be somewhat brief in some spots, and
we will refer the reader to \cites{jSgHjLwW2016,LS} for details.

We highlight the following point, which is crucial for the subsequent analysis:
the bounds for $\upmu$, $L^i_{(Small)}$ and the wave variables $\Psi$ control all the other geometric quantities, including the transformation coefficients
between different sets of vectorfields, as well as the commutators of vectorfields.

\subsection{Some preliminary geo-analytic identities}
In this section, we provide some geo-analytic identities that we will use throughout our analysis.

We start by recalling the definition of a null form with respect to the acoustical metric
(``$g$-null form'' for short).

\begin{definition}[\textbf{$g$-Null forms}]
\label{D:NULLFORMS}
Let $\phi^{(1)}$ and $\phi^{(2)}$ be scalar functions.
We use the notation 
$\mathcal{Q}^{(g)}(\rd\phi^{(1)},\rd\phi^{(2)})$ to denote any derivative-quadratic term
of the form:
\begin{align}\label{E:Q0.def}
\mathcal{Q}^{(g)}(\rd\phi^{(1)},\rd\phi^{(2)})
& = \smoothfunction(L^i,\Psi) 
(g^{-1})^{\alpha\beta}\rd_\alpha \phi^{(1)} \rd_\beta \phi^{(2)},
\end{align}
where $\smoothfunction(\cdot)$ is a smooth function.

We use the notation 
$\mathcal{Q}_{\alpha \beta}(\rd\phi^{(1)},\rd\phi^{(2)})$
to denote any derivative-quadratic term
of the form:
\begin{align}\label{E:Qij.def}
\mathcal{Q}_{\alpha \beta}(\rd\phi^{(1)},\rd\phi^{(2)})
	= \smoothfunction(L^i, \Psi) 
		\left\lbrace
			\rd_\alpha \phi^{(1)} \rd_\beta \phi^{(2)} - \rd_\beta \phi^{(1)} \rd_\alpha \phi^{(2)}
		\right\rbrace,
\end{align}
where $\smoothfunction(\cdot)$ is a smooth function.
\end{definition}

\begin{lemma}[\textbf{Crucial structural properties of null forms}]
	\label{L:CRUCIALPROPERTIESOFNULLFORMS}
	Let $\mathcal{Q}(\rd\phi^{(1)},\rd\phi^{(2)})$ be a $g$-null form of type \eqref{E:Q0.def} or \eqref{E:Qij.def}.
	Then there exist smooth functions, all schematically denoted by ``$\smoothfunction$''
	(and which are different from the ``$\smoothfunction$'' in Definition~\ref{D:NULLFORMS}),
	such that the following identity holds:
	\begin{align} \label{E:CRUCIALPROPERTIESOFNULLFORMS}
		\upmu \mathcal{Q}(\rd \phi^{(1)},\rd\phi^{(2)})
		& = 
				\smoothfunction(L^i,\Psi)
				\Rad \phi^{(1)} \cdot \mathcal{P} \phi^{(2)}
				+
				\smoothfunction(L^i,\Psi)
				\Rad \phi^{(2)} \cdot \mathcal{P} \phi^{(1)}
				+
				\upmu 
				\smoothfunction(L^i,\Psi)
				\mathcal{P} \phi^{(1)} \cdot \mathcal{P} \phi^{(2)}.
	\end{align}
	In particular, decomposing all differentiations in the null form with respect to the $\{L,X,Y,Z\}$ frame
	leads to the \underline{absence} of all $X\phi^{(1)} \cdot X\phi^{(2)}$terms on RHS~\eqref{E:CRUCIALPROPERTIESOFNULLFORMS}.
\end{lemma}

\begin{proof}
	For null forms of type \eqref{E:Qij.def}, \eqref{E:CRUCIALPROPERTIESOFNULLFORMS} 
	follows from Lemma~\ref{lem:Cart.to.geo}
	and the fact that the Cartesian component functions $X^1,X^2,X^3$ are smooth functions of
	the $L^i$ and $\Psi$ (see \eqref{E:TRANSPORTVECTORFIELDINTERMSOFLUNITANDRADUNIT}).
	For null forms of type \eqref{E:Q0.def}, 
	\eqref{E:CRUCIALPROPERTIESOFNULLFORMS} follows from the basic identity
	$g^{-1} = - L \otimes L - (L\otimes X + X\otimes L) + \slashed g^{-1}$ 
	(see, e.g., \cite[(2.40b)]{jSgHjLwW2016})
	and Lemma~\ref{lem:induced.metric}.
\end{proof}

\begin{lemma}[\textbf{Expressions for the transversal derivatives of the transport variables in terms of tangential derivatives}]
	\label{L:TRANSPORTTRANSVERSALINTERMSOFTANGENTIAL}
	There exist smooth functions, all schematically denoted by ``$\smoothfunction$,"
	such that the following identities hold: 
	\begin{align}
		\Rad \Vr^i
		& =
			- 
			\upmu \Lunit \Vr^i
			+
			(\Vr,S) \cdot \smoothfunction(\Psi,\Lunit^i,\upmu,\Rad \Psi, \mathcal{P} \Psi),
			 \label{E:TRANSVERSALDERIVATIVEEXPRESSIONFORTRANSPORTVARIABLES1} \\
		\Rad S^i
		& =
			- 
			\upmu \Lunit S^i
			+
			(\Vr,S) \cdot \smoothfunction(\Psi,\Lunit^i,\upmu,\Rad \Psi, \mathcal{P} \Psi),
			 \label{E:TRANSVERSALDERIVATIVEEXPRESSIONFORTRANSPORTVARIABLES2} \\
		\Rad \mathcal{C}^i
		& =
			- 
			\upmu \Lunit \mathcal{C}^i
			+
			(\Vr,S,\mathcal{P} \Vr, \mathcal{P} S) 
			\cdot 
			\smoothfunction(\Psi,\Lunit^i,\upmu,\Rad \Psi, \mathcal{P} \Psi).
				\label{E:TRANSVERSALDERIVATIVEEXPRESSIONFORMODIFIEDFLUIDVARIABLES1} \\
		\Rad \mathcal{D}^i
		& =
			- 
			\upmu \Lunit \mathcal{D}^i
			+
			(\Vr,S,\mathcal{P} \Vr, \mathcal{P} S) 
			\cdot 
			\smoothfunction(\Psi,\Lunit^i,\upmu,\Rad \Psi, \mathcal{P} \Psi).
				\label{E:TRANSVERSALDERIVATIVEEXPRESSIONFORMODIFIEDFLUIDVARIABLES2}
	\end{align}
	
\end{lemma}

\begin{proof}
	\eqref{E:TRANSVERSALDERIVATIVEEXPRESSIONFORTRANSPORTVARIABLES1} and \eqref{E:TRANSVERSALDERIVATIVEEXPRESSIONFORTRANSPORTVARIABLES2} 
	follow from the transport equations \eqref{E:RENORMALIZEDVORTICTITYTRANSPORTEQUATION}
	and \eqref{E:GRADENTROPYTRANSPORT}, 
	\eqref{E:TRANSPORTVECTORFIELDINTERMSOFLUNITANDRADUNIT} 
	(which implies that $\upmu \Transport = \Rad + \upmu \Lunit$),
	and Lemma~\ref{lem:Cart.to.geo}.
	
	\eqref{E:TRANSVERSALDERIVATIVEEXPRESSIONFORMODIFIEDFLUIDVARIABLES1} and \eqref{E:TRANSVERSALDERIVATIVEEXPRESSIONFORMODIFIEDFLUIDVARIABLES2} follow from
	a similar argument based the transport equations
	\eqref{E:EVOLUTIONEQUATIONFLATCURLRENORMALIZEDVORTICITY} and \eqref{E:TRANSPORTFLATDIVGRADENT},
	where we use Lemma~\ref{L:CRUCIALPROPERTIESOFNULLFORMS} to decompose
	the null form source terms and
	\eqref{E:TRANSVERSALDERIVATIVEEXPRESSIONFORTRANSPORTVARIABLES1}--\eqref{E:TRANSVERSALDERIVATIVEEXPRESSIONFORTRANSPORTVARIABLES2}
	to re-express all $\Rad$ derivatives of $(\Vr,S)$.
\end{proof}

\begin{lemma}[\textbf{Identity for $\Rad \Lunit^i$}]
	\label{L:WEIGHTEDXLUNITIFORMLUA}
	There exist smooth functions, 
	all schematically denoted by ``$\smoothfunction$,''
	such that:
	\begin{align} \label{E:WEIGHTEDXLUNITIFORMLUA}
	\Rad \Lunit^i
	= \smoothfunction(\Psi,\Lunit^i) \Rad \Psi
		+
		\upmu \smoothfunction(\Psi,\Lunit^i) \mathcal{P} \Psi
		+
		\smoothfunction(\Psi,\Lunit^i) \mathcal{P} \upmu.
\end{align}
\end{lemma}

\begin{proof}
This was proved as \cite[(2.71)]{jSgHjLwW2016} 
(which holds in the present context with obvious modifications such as
replacing $G_{LL}\bX\Psi$ with $\vec{G}_{LL} \contr \bX \threePsi$, etc.),
where we have used that the Cartesian component functions $X^1,X^2,X^3$ are smooth functions of
the $L^i$ and $\Psi$ (see \eqref{E:TRANSPORTVECTORFIELDINTERMSOFLUNITANDRADUNIT}).
\end{proof}

\begin{lemma}[\textbf{Simple commutator identities}]
	\label{L:SIMPLECOMMUTATORIDENTITY}
	For each pair $\mathcal{P}_1, \mathcal{P}_2 \in \lbrace \Lunit, Y, Z \rbrace$,
	there exist smooth functions, all schematically denoted by ``$\smoothfunction$,''
	such that the following identity holds:
	\begin{align} \label{E:TANGENTIALSIMPLECOMMUTATORIDENTITY}
	[\mathcal{P}_1,\mathcal{P}_2] 
	& = 
	\smoothfunction(\Lunit^i,\Psi,\mathcal{P} \Lunit^i, \mathcal{P} \Psi) Y
	+
	\smoothfunction(\Lunit^i,\Psi,\mathcal{P} \Lunit^i, \mathcal{P} \Psi) Z.
	\end{align}
		
	Moreover, for each $\mathcal{P} \in \lbrace \Lunit, Y, Z \rbrace$,
	there exist smooth functions, all schematically denoted by ``$\smoothfunction$,''
	such that the following identity holds:
	\begin{align} \label{E:SIMPLECOMMUTATORIDENTITY}
	[\mathcal{P},\Rad] 
	& = 
	\smoothfunction(\upmu,\Lunit^i,\Psi,\mathcal{P} \upmu,\Rad \Psi,\mathcal{P} \Psi) Y
	+
	\smoothfunction(\upmu \Lunit^i,\Psi,\mathcal{P} \upmu, \Rad \Psi,\mathcal{P} \Psi) Z.
	\end{align}
\end{lemma}

\begin{proof}
	We first prove \eqref{E:SIMPLECOMMUTATORIDENTITY}.
	Lemma~\ref{lem:slashed} implies that $[\mathcal{P},\Rad]$ is $\ell_{t,u}$-tangent,
	i.e., that $[\mathcal{P},\Rad] t = [\mathcal{P},\Rad] u = 0$.
	Hence,
	\eqref{E:GEOMETRIC2COORDINATEPARTIALDERIVATIVESINTERMSOFOTHERVECTORFIELDS}--\eqref{E:GEOMETRIC3COORDINATEPARTIALDERIVATIVESINTERMSOFOTHERVECTORFIELDS}
	imply that this commutator can be written as a linear combination of $Y,Z$.
	Since the Cartesian component functions $X^1,X^2,X^3$ are smooth functions of
	the $L^i$ and $\Psi$ (see \eqref{E:TRANSPORTVECTORFIELDINTERMSOFLUNITANDRADUNIT}),
	the same holds for the component functions $\mathcal{P}^0,\mathcal{P}^1,\mathcal{P}^2,\mathcal{P}^3$
	(this is obvious for $\mathcal{P}=L$, while see  
	Lemmas~\ref{lem:slashed}--\ref{L:GEOMETRICCOORDINATEVECTORFIELDSINTERMSOFCARTESIANVECTORFIELDS} for $\mathcal{P}=Y,Z$).
	Also using that $\Rad^i = \upmu \Rad^i$,
	we conclude \eqref{E:SIMPLECOMMUTATORIDENTITY}
	by computing relative to the Cartesian coordinates,
	using Lemma~\ref{lem:Cart.to.geo} to express Cartesian coordinate partial derivatives
	in terms of derivatives with respect to $Y,Z$ 
	(the $X$- and $\Lunit$-derivative components of the commutator must vanish 
	since $[\mathcal{P},\Rad]$ is $\ell_{t,u}$-tangent),
	and using \eqref{E:WEIGHTEDXLUNITIFORMLUA}
	to substitute for $\Rad \Lunit^i$ factors.
	
	The identity \eqref{E:TANGENTIALSIMPLECOMMUTATORIDENTITY} can be proved through
	similar but simpler arguments that do not involve factors of $\upmu$ or $\Rad$ differentiations.
\end{proof}

\subsection{The easy $L^\infty$ estimates}

\begin{proposition}[\textbf{$L^{\infty}$ estimates for the acoustical geometry}]
\label{prop:geometric.low}
The following estimates hold for all $t\in [0,\Tboot)$:
$$\|\upmu\|_{L^\infty(\Sigma_t)} + \|L\upmu\|_{L^\infty(\Sigma_t)} \ls 1,\quad \| L_{(Small)}^i \|_{L^\infty(\Sigma_t)} \ls_{\mydiam} \mathring{\upalpha},\quad \|Y\upmu\|_{L^\infty(\Sigma_t)} + \|Z\upmu\|_{L^\infty(\Sigma_t)} \ls \epd^{\f 12},$$
$$\|\mathcal{P}^{[2,\Ntop- \toprate-4]} \upmu\|_{L^\infty(\Sigma_t)} 
+ 
\|\mathcal{P}^{{[1,\Ntop- \toprate-3]}} L^i \|_{L^\infty(\Sigma_t)} \ls \epd^{\f 12}.$$
\end{proposition}
\begin{proof}
These can be proved using the transport equations \eqref{E:UPMUFIRSTTRANSPORT} and \eqref{E:LLUNITI} (commuted with $\mathcal{P}^N$), {
the initial data size-assumptions \eqref{assumption:R+}--\eqref{assumption:small},
and the bootstrap assumptions 
\eqref{BA:LARGERIEMANNINVARIANTLARGE}--\eqref{BA:W.Li.small}.} 
See \cite[Proposition~8.10]{jSgHjLwW2016} for details of this argument. 
We note these estimates lose a slight amount of regularity compared to $\Psi$
because the transport equations \eqref{E:UPMUFIRSTTRANSPORT} and \eqref{E:LLUNITI} 
depend on the derivatives of $\Psi$.
\qedhere
\end{proof}

Our analysis also relies on the following $L^\infty$ estimates.
\begin{proposition}[\textbf{$L^{\infty}$ estimates for other geometric quantities}]
\label{prop:geometric.low.2}
The following estimates hold for all $t\in [0,\Tboot)$, where $\Speed$ denotes
the speed of sound:
$$\|X_{(Small)}^i \|_{L^\infty(\Sigma_t)} \ls_{\mydiam} \mathring{\upalpha}^{1/2}, 
\quad \| \Speed - 1 \|_{L^\infty(\Sigma_t)}\ls_{\mydiam} \mathring{\upalpha}^{1/2},$$
$$\|\mathcal{P}^{[1,\Ntop- \toprate-3]} X^i \|_{L^\infty(\Sigma_t)} \ls \epd^{\f 12}, 
\quad \| \mathcal{P}^{[1,\Ntop- \toprate-2]}\Speed \|_{L^\infty(\Sigma_t)}\ls \epd^{\f 12}.$$
\end{proposition}
\begin{proof}
The estimates for $X_{(Small)}$ follow from 
\eqref{E:LSMALLDEF}--\eqref{E:XSMALLDEF}, \eqref{eq:Cart.to.geo.0}, 
the bootstrap assumptions \eqref{BA:LARGERIEMANNINVARIANTLARGE}--\eqref{BA:W.Li.small},
and Proposition~\ref{prop:geometric.low}. 

The estimates for $\Speed$ follow from the bootstrap 
assumptions \eqref{BA:LARGERIEMANNINVARIANTLARGE}--\eqref{BA:W.Li.small}
and the fact that
$\Speed$ is a smooth function of $\rr$ and $s$ with $\Speed(0,0)=1$ (see \eqref{eq:c.normalization}). \qedhere
\end{proof}

The estimates in Propositions~\ref{prop:geometric.low} and \ref{prop:geometric.low.2}
also imply the following bounds for the commutators.
\begin{proposition}[\textbf{Pointwise bounds for vectorfield commutators}]
\label{prop:commutators.Li}
All the commutators $[L, \bX]$, $[L,Y]$, $[L, Z]$, $[\bX, Y]$, $[\bX, Z]$ and $[Y, Z]$ are $\ell_{t,u}$-tangent.

Moreover, if $\phi$ is a scalar function, then
for $0\leq N \leq \Ntop$, iterated commutators can be bounded pointwise as follows:
\begin{equation}\label{eq:commutators.Li.1}
\begin{split}
|[L, \mathcal{P}^N]\phi|
\ls &\: \epd^{\frac{1}{2}} |\mathcal{P}^{[1,N]} \phi| 
+ 
\sum_{\substack {N_1+N_2 \leq N+1 \\ N_1,\,N_2\leq N}} |\mathcal{P}^{[2,N_1]}(L^i,\Psi)| |\mathcal{P}^{[1,N_2]} \phi|,
	\\
|[\bX,\mathcal{P}^N]\phi| 
 \ls &\:  |\mathcal{P}^{[1,N]} \phi| 
 + 
 \sum_{\substack {N_1+N_2 \leq N+1 \\ N_1,\,N_2\leq N}} |\mathcal{P}^{[2,N_1]}(\upmu,L^i,\Psi)| |\mathcal{P}^{[1,N_2]} \phi| 
	\\
 &\: \ \
+ 
\sum_{\substack {N_1+N_2 \leq N \\ N_1 \leq N-1}} |\mathcal{P}^{[2,N_1]}\bX\Psi| |\mathcal{P}^{[1,N_2]} \phi|.
\end{split}
\end{equation}
In particular, 
\begin{equation}\label{eq:commutators.Li.2}
\begin{split}
|[L, \mathcal{P}^N] \phi| 
&\: \ls \epd^{\frac{1}{2}} |\mathcal{P}^{[1,N]}\phi|,
\qquad
\mbox{if } 0 \leq N \leq \Ntop -\toprate-3,
	\\
|[\bX,\mathcal{P}^N]\phi| 
&\: \ls |\mathcal{P}^{[1,N]}\phi|,
\qquad
\mbox{if } 0 \leq N \leq \Ntop -\toprate-4.
\end{split}
\end{equation}
\end{proposition}
\begin{proof}
All the commutators can be read off from Lemma~\ref{lem:slashed} (and using that coordinate vectorfields commute). In particular, since the coefficient of $\srd_t$ in $L$ and the coefficient of $\srd_u$ in $\bX$ 
both are equal to $1$, all the stated commutators are $\ell_{t,u}$-tangent.

We first prove \eqref{eq:commutators.Li.1} for $|[L, \mathcal{P}^N]\phi|$. By Lemma~\ref{lem:slashed} and the fact $L^i +X^i -v^i = 0$ (by \eqref{eq:Cart.to.geo.0}),
\begin{equation}\label{eq:commutator.easy.term}
|[L, \mathcal{P}^N]\phi| \ls \sum_{k=2}^N \sum_{\substack{ N_1+\dots+N_k = N+1 \\ 1\leq N_k\leq N}} 
\underbrace{|\mathcal{P}^{N_1}(L^i,\Psi)|\cdots |\mathcal{P}^{N_{k-1}}(L^i,\Psi)| 
|\mathcal{P}^{[1,N_k]}\phi|}_{\doteq (*)}.
\end{equation}
By \eqref{BA:LARGERIEMANNINVARIANTLARGE}--\eqref{BA:W.Li.small}, 
Propositions~\ref{prop:geometric.low}, \ref{prop:geometric.low.2} (and $N\leq \Ntop$), either 
$|\mathcal{P}^{N_j}(L^i,\Psi)| \ls \epd^{\frac{1}{2}}$
for $1 \leq j \leq k - 1$
 (in which case $(*) \ls \epd^{\frac{1}{2}}|\mathcal{P}^{[1,N]} \phi|$), 
or else there is exactly one factor 
$|\mathcal{P}^{N_j}(L^i,\Psi)|$ 
with $N_j > \Ntop -  \toprate-3$ not bounded
by $\ls \epd^{\frac{1}{2}}$
(in which case $(*) \ls \sum_{\substack {N_1+N_2 \leq N+1 \\ N_1,\,N_2\leq N}} |\mathcal{P}^{[2,N_1]}(L^i,\Psi)| |\mathcal{P}^{[1,N_2]} \phi|$). 
Hence, \eqref{eq:commutator.easy.term} is bounded above by 
the RHS of the first inequality in \eqref{eq:commutators.Li.1}. 

To bound 
$[\bX,\mathcal{P}^N]\phi$, we note that
according to Lemma~\ref{lem:slashed}, there is, in addition to \eqref{eq:commutator.easy.term}, the terms:\footnote{Importantly, one checks from Lemma~\ref{lem:slashed} that there are no terms of the form $|\mathcal{P}^{N_{k-1}} \bX \upmu|$!}
\begin{equation}\label{eq:commutator.hard.term}
 \sum_{k=2}^N \sum_{\substack{ N_1 + \dots + N_k = N \\ N_{k-1} \leq N-1 \\1\leq N_k\leq N}} 
|\mathcal{P}^{N_1}(L^i,\Psi)|\cdots|\mathcal{P}^{N_{k-2}}(L^i,\Psi)| 
|\mathcal{P}^{N_{k-1}} \bX (L^i,\Psi)| |\mathcal{P}^{N_k}\phi|.
\end{equation}
and:
\begin{equation}\label{eq:commutator.not.sohard.term} \sum_{k=2}^N \sum_{\substack{ N_1 + \dots + N_k = N+1 \\1\leq N_k\leq N}} 
|\mathcal{P}^{N_1}(L^i,\Psi)|\cdots|\mathcal{P}^{N_{k-2}}(L^i,\Psi)| 
|\mathcal{P}^{N_{k-1}} \upmu | |\mathcal{P}^{N_k}\phi|.
\end{equation}

Hence, with the help of \eqref{E:WEIGHTEDXLUNITIFORMLUA}, we can substitute
for the terms $\bX L^i$ on RHS~\eqref{eq:commutator.hard.term},
and thus RHS~\eqref{eq:commutator.hard.term}
can be bounded above by RHS of \eqref{eq:commutator.easy.term} plus \eqref{eq:commutator.not.sohard.term} and:
\begin{equation}\label{eq:commutator.hard.term.simplified}
\sum_{k=2}^N \sum_{\substack{ N_1 + \dots + N_k = N \\ N_{k-1} \leq N-1 \\1\leq N_k\leq N}} 
|\mathcal{P}^{N_1}(L^i,\Psi)|\cdots|\mathcal{P}^{N_{k-2}}(L^i,\Psi)| 
|\mathcal{P}^{N_{k-1}} \bX \Psi| |\mathcal{P}^{N_k}\phi|,
\end{equation}
both of which, by arguments similar to the ones we used to prove \eqref{eq:commutator.easy.term}, 
can be bounded above by 
the RHS of the second inequality in \eqref{eq:commutators.Li.1}.

To get from \eqref{eq:commutators.Li.1} to \eqref{eq:commutators.Li.2}, 
we use the $L^\infty$ bounds in \eqref{BA:LARGERIEMANNINVARIANTLARGE}--\eqref{BA:W.Li.small}
and
Propositions~\ref{prop:geometric.low} and \ref{prop:geometric.low.2}, which 
are applicable in the sense that they control 
a sufficient number of derivatives of all relevant quantities in $L^{\infty}$. \qedhere
\end{proof}

In the rest of the paper, we will often silently use the following simple lemma.
\begin{lemma}[\textbf{The norm of the $\ell_{t,u}$-tangent commutator vectorfields and simple comparison estimates}]
\label{L:SIMPLECOMPARISON}
The $\ell_{t,u}$-tangent commutator vectorfields $\lbrace Y,Z \rbrace$ verify the following pointwise bounds
on $\mathcal{M}_{\Tboot,U_0}$:
\begin{align} \label{E:POINTWISENORMOFELLTUTANGENTCOMMUTATORS}
	|Y| & \lesssim 1, 
	&
	|Z| & \lesssim 1.
\end{align}	

Moreover, for any $\ell_{t,u}$-tangent tensorfield $\upxi$,
the following pointwise bounds hold on $\mathcal{M}_{\Tboot,U_0}$:
\begin{align} \label{E:ELLTUCONNECTIONDERIVATIVECOMPARBLETOYZDERIVATIVES}
	|\angD \upxi|
	& \approx |\angD_Y \upxi| + |\angD_Z \upxi|.
\end{align}
\end{lemma}
\begin{proof}
To prove \eqref{E:POINTWISENORMOFELLTUTANGENTCOMMUTATORS}, we 
use Lemmas~\ref{lem:slashed} and \ref{lem:induced.metric}
and the fact that the Cartesian component functions $X^1,X^2,X^3$ are smooth functions of
the $L^i$ and $\Psi$ (see \eqref{E:TRANSPORTVECTORFIELDINTERMSOFLUNITANDRADUNIT})
to deduce that $|Y|^2= \gsphere_{AB} Y^A Y^B = \smoothfunction(L^i,\Psi)$,
where $\smoothfunction$ is a smooth function. Similar remarks hold
for $|Z|^2$. The desired estimates in \eqref{E:POINTWISENORMOFELLTUTANGENTCOMMUTATORS} 
therefore follow from 
the bootstrap assumptions \eqref{BA:LARGERIEMANNINVARIANTLARGE}--\eqref{BA:SMALLWAVEVARIABLESUPTOONETRANSVERSALDERIVATIVE}
and Proposition~\ref{prop:geometric.low}.

To prove \eqref{E:ELLTUCONNECTIONDERIVATIVECOMPARBLETOYZDERIVATIVES},
we note that the $\gsphere$-Cauchy--Schwarz inequality and 
\eqref{E:POINTWISENORMOFELLTUTANGENTCOMMUTATORS} imply that
$|\angD_Y \upxi| + |\angD_Z \upxi| \lesssim |\angD \upxi|$.
We will show how to obtain the reverse inequality when $\upxi$ is a scalar function;
the case of an arbitrary $\ell_{t,u}$-tangent tensorfield can be handled using the same arguments,
which will complete the proof.
To proceed, we note that for scalar functions $\upxi$, we have
$
|\angD \upxi|^2 = (\gsphere^{-1})^{AB} (\srd_A \upxi)(\srd_B \upxi)
$.
We now use Lemmas~\ref{L:GEOMETRICCOORDINATEVECTORFIELDSINTERMSOFCARTESIANVECTORFIELDS}
and \ref{lem:induced.metric}
and the fact that $X^1,X^2,X^3$ are smooth functions of
$L^i$ and $\Psi$ (as noted above)
to deduce that there exist smooth functions, all schematically denoted by
``$\smoothfunction$,'' such that
$
(\gsphere^{-1})^{AB} (\srd_A \upxi)(\srd_B \upxi)
= \smoothfunction(L^i,\Psi) (Y \upxi)^2 
+
\smoothfunction(L^i,\Psi) (Y \upxi) (Z \upxi)
+ 
\smoothfunction(L^i,\Psi) (Z \upxi)^2
$.
Also using the bootstrap assumptions 
\eqref{BA:LARGERIEMANNINVARIANTLARGE}--\eqref{BA:SMALLWAVEVARIABLESUPTOONETRANSVERSALDERIVATIVE},
Young's inequality,
and Proposition~\ref{prop:geometric.low},
we conclude that 
$|\angD \upxi|^2 \lesssim |Y \upxi|^2 + |Z \upxi|^2 = |\angD_Y \upxi|^2 + |\angD_Z \upxi|^2$ as desired.
\end{proof}

\subsection{$L^{\infty}$ estimates involving higher transversal derivatives}

Some aspects of our main results rely on having $L^{\infty}$ estimates for the higher transversal
derivatives of various solution variables. We provide these estimates in the next
proposition. The proofs are similar to the proofs of related estimates in
\cite{jSgHjLwW2016}.

\begin{proposition}[\textbf{$L^{\infty}$ estimates involving higher transversal derivatives}]
	\label{P:LINFTYHIGHERTRANSVERSAL}
	The following estimates hold\footnote{
	Based on our assumptions on the data (see Section~\ref{SS:FLUIDVARIABLEDATAASSUMPTIONS}),
	we could obtain $L^{\infty}$ control over additional
	$\mathcal{F}_u$-tangential derivatives of the quantities stated in the proposition 
	-- but not! additional $\Rad$ differentiations. However, for convenience, in the proposition,
	we have only derived control of a sufficient number of derivatives so that the estimates close and
	so that we can use the results in our proof of Proposition~\ref{lem:higher.transversal.for.Holder}
	and in the appendix. 
	} 
	for all $t \in [0,\Tboot)$ and $u \in [0,U_0]$,
	where in \eqref{E:RADTANGENTIALUPMULINFTY}, $\slashed{\mathcal{P}} \in \lbrace Y,Z \rbrace$:
	
	\medskip
	
	\noindent \underline{\textbf{$L^{\infty}$ estimates involving two or three transversal derivatives of 
	the wave variables}.}
	
\begin{subequations}	
\begin{align}
		\| \Lunit \mathcal{P}^{\leq 2} \Rad \Rad \Psi \|_{L^\infty(\Mtu)}
		& \leq C \epd^{\frac{1}{2}},
			\label{E:LUNITTWORADPSI} \\
		\| \mathcal{P}^{[1,2]} \Rad \Rad \Psi \|_{L^\infty(\Mtu)}
		& \leq C\epd^{\frac{1}{2}},
			\\
		\| \Rad \Rad \mathcal{R}_{(+)} \|_{L^\infty(\Mtu)}
		& \leq \| \Rad \Rad \mathcal{R}_{(+)} \|_{L^\infty(\Sigma_0)}
			+ C \epd^{\frac{1}{2}},
				\\
		\| \Rad \Rad (\mathcal{R}_{(-)},v^1,v^2,s) \|_{L^\infty(\Mtu)}
		& \leq C \epd^{\frac{1}{2}},
	\end{align}
	\end{subequations}
	
	\begin{subequations}
	\begin{align}
		\| \Lunit \Rad \Rad \Rad \Psi \|_{L^\infty(\Mtu)}
		& \leq C\epd^{\frac{1}{2}},
			\\
		\| \Rad \Rad \Rad \mathcal{R}_{(+)} \|_{L^\infty(\Mtu)}
		& \leq \| \Rad \Rad \Rad \mathcal{R}_{(+)} \|_{L^\infty(\Sigma_0)}
			+ C \epd^{\frac{1}{2}},
				\label{E:LINFINITITYTHREETRANSVERSALOFLARGERIEMANNINVARIANT} \\
		\| \Rad \Rad \Rad (\mathcal{R}_{(-)},v^1,v^2,s) \|_{L^\infty(\Mtu)}
		& \leq C \epd^{\frac{1}{2}}.
			\label{E:LINFINITITYTHREETRANSVERSALOFSMALLWAVEVARIABLES}
	\end{align}
	\end{subequations}
	
	\medskip
	
	\noindent \underline{\textbf{$L^{\infty}$ estimates involving one or two transversal derivatives of $\upmu$}.}
	\begin{subequations}
	\begin{align}
		\left\|
			\Lunit \Rad \upmu
		\right\|_{L^{\infty}(\Mtu)}
		& \leq
			\frac{1}{2}
			\left\|
				\Rad
				\left(
					\vec{G}_{\Lunit \Lunit}\contr \Rad \threePsi
				\right)
			\right\|_{L^{\infty}(\Sigma_0)}
			+
			C \epd^{\frac{1}{2}},
				\label{E:LUNITRADUPMULINFTY}
					\\
		\left\|
			\Rad \upmu
		\right\|_{L^{\infty}(\Mtu)}
		& \leq
			\left\|
				\Rad \upmu
			\right\|_{L^{\infty}(\Sigma_0)}
			+
			\mathring{\updelta}_*^{-1}
			\left\|
				\Rad
				\left(
					\vec{G}_{\Lunit \Lunit} \contr \Rad \threePsi
				\right)
			\right\|_{L^{\infty}(\Sigma_0)}
			+
			C \epd^{\frac{1}{2}},
				\label{E:RADUPMULINFTY}
\end{align}
\end{subequations}

\begin{subequations}
\begin{align}
	\left\| 
		\Lunit \Rad \mathcal{P} \upmu
	\right\|_{L^{\infty}(\Mtu)},
		\,
	\left\| 
		\Lunit \Rad \mathcal{P}^2 \upmu
	\right\|_{L^{\infty}(\Mtu)}
	& \leq C \epd^{\frac{1}{2}},
				\label{E:LUNITRADTANGENTIALUPMULINFTY}
				\\
	\left\| 
		\Rad \slashed{\mathcal{P}} \upmu
	\right\|_{L^{\infty}(\Mtu)},
		\,
	\left\| 
		\Rad \mathcal{P}^2 \upmu
	\right\|_{L^{\infty}(\Mtu)}
	& \leq C \epd^{\frac{1}{2}},
				\label{E:RADTANGENTIALUPMULINFTY}
\end{align}
\end{subequations}

\begin{subequations}
\begin{align}
		\left\|
			\Lunit \Lunit \Rad \Rad \upmu
		\right\|_{L^{\infty}(\Mtu)}
		& \leq
			C \epd^{\frac{1}{2}},
		\label{E:LUNITLUNITRADRADUPMULINFTY}
			\\
		\left\|
			\Lunit \Rad \Rad \upmu
		\right\|_{L^{\infty}(\Mtu)}
		& \leq
			\frac{1}{2}
			\left\|
				\Rad
				\Rad
				\left(
					\vec{G}_{\Lunit \Lunit}\contr \Rad \threePsi
				\right)
			\right\|_{L^{\infty}(\Sigma_0)}
			+
			C \epd^{\frac{1}{2}},
		\label{E:LUNITRADRADUPMULINFTY}
			\\
		\left\|
			\Rad \Rad \upmu
		\right\|_{L^{\infty}(\Mtu)}
		& \leq
			\left\|
				\Rad \Rad \upmu
			\right\|_{L^{\infty}(\Sigma_0)}
			+
			\mathring{\updelta}_*^{-1}
			\left\|
				\Rad
				\Rad
				\left(
					\vec{G}_{\Lunit \Lunit}\contr \Rad \threePsi
				\right)
			\right\|_{L^{\infty}(\Sigma_0)}
			+
			C \epd^{\frac{1}{2}}.
		\label{E:RADRADUPMULINFTY}
	\end{align}
	\end{subequations}

\medskip

\noindent \underline{\textbf{$L^{\infty}$ estimates involving one or two transversal derivatives of $\Lunit^i$}.}	
	\begin{subequations}
	\begin{align}
		\| \mathcal{P}^{[1,\Ntop - M_* - 5]} \Rad \Lunit^i \|_{L^\infty(\Mtu)}
		& \leq C \epd^{\frac{1}{2}},
			\\
		\| \Rad \Lunit^i \|_{L^\infty(\Mtu)}
		& \leq C,
	\end{align}
	\end{subequations}
	
	\begin{subequations}
	\begin{align}
		\| \Lunit \mathcal{P} \Rad \Rad \Lunit^i \|_{L^\infty(\Mtu)}
		& \leq C \epd^{\frac{1}{2}},
			 \label{E:LTANGENTIALTWOTRANSVERSALDERIVATIESLUNITILINFTY} \\
		\| \mathcal{P} \Rad \Rad \Lunit^i \|_{L^\infty(\Mtu)}
		& \leq C \epd^{\frac{1}{2}},
			\\
		\| \Rad \Rad \Lunit^i \|_{L^\infty(\Mtu)}
		& \leq C.
			\label{E:TWOTRANSVERSALDERIVATIESLUNITILINFTY}
	\end{align}
	\end{subequations}
	
	\medskip
	
	\noindent \underline{\textbf{$L^{\infty}$ estimates involving 
	transversal derivatives of the transported variables}.}
	\begin{align} 
	\begin{split} \label{E:VORTICITYANDENTROPYLINFTYHIGHERTRANSVERSAL}
		&
		\| \mathcal{P}^{\leq 3} \Rad^{\leq 1} (\Vr,S) \|_{L^\infty(\Mtu)}
		+
		\| \mathcal{P}^{\leq 2} \Rad \Rad (\Vr,S) \|_{L^\infty(\Mtu)}
		+
		\| \Rad^{\leq 3} (\Vr,S) \|_{L^\infty(\Mtu)}
			\\
	& \ \
		+
		\| \mathcal{P}^{\leq 2} \Rad^{\leq 1} (\mathcal{C},\mathcal{D}) \|_{L^\infty(\Mtu)}
		+
		\| \Rad^{\leq 2} (\mathcal{C},\mathcal{D}) \|_{L^\infty(\Mtu)}
		\leq C \epd.
	\end{split}
	\end{align}	
	
	Finally,
	\begin{align} 
	\begin{split} \label{E:PERMUTEDESTIMATES}
	&
	\mbox{We can permute the vectorfield operators on the LHSs of 
	\eqref{E:LUNITTWORADPSI}--\eqref{E:TWOTRANSVERSALDERIVATIESLUNITILINFTY}}
		\\
	& \mbox{up to error terms of $L^{\infty}$ size
	$\mathcal{O}(\epd^{\frac{1}{2}})$,}
	\end{split}	
		\\
	&\mbox{and on the LHS of \eqref{E:VORTICITYANDENTROPYLINFTYHIGHERTRANSVERSAL}
	up to error terms of $L^{\infty}$ size $\mathcal{O}(\epd)$}.
	\label{E:TRANSPORTPERMUTEDESTIMATES}
\end{align}
	
\end{proposition}

\begin{proof}
	To prove the lemma, we make the ``new bootstrap assumption'' that 
	the estimates in \eqref{E:VORTICITYANDENTROPYLINFTYHIGHERTRANSVERSAL}
	hold for $t \in [0,\Tboot)$ with the ``$C \epd$'' term on the RHS
	replaced by ``$\epd^{\frac{1}{2}}$,''
	and also that \eqref{E:TRANSPORTPERMUTEDESTIMATES} holds with 
	$\mathcal{O}(\epd)$ replaced by $\epd^{\frac{1}{2}}$.
	Given this new bootstrap assumption,
	to obtain \eqref{E:LUNITTWORADPSI}--\eqref{E:TWOTRANSVERSALDERIVATIESLUNITILINFTY}
	and \eqref{E:PERMUTEDESTIMATES},
	we can simply repeat\footnote{We clarify that the bootstrap parameter ``$\varepsilon$'' from 
	\cite{jSgHjLwW2016} should be identified with the quantity $\epd^{\frac{1}{2}}$
	in our bootstrap assumptions \eqref{BA:SMALLWAVEVARIABLESUPTOONETRANSVERSALDERIVATIVE}--\eqref{BA:C.D}.} 
	the proof of \cite[Lemma~9.3]{jSgHjLwW2016},
	which relies on transport-type estimates 
	that lose derivatives (in particular, one uses the transport equations 
	\eqref{E:UPMUFIRSTTRANSPORT}--\eqref{E:LLUNITI}
	and also treats the wave equation
	as a derivative-losing transport equation $\Lunit \Rad \Psi = \cdots$ by
	using \eqref{eq:wave.in.terms.of.geometric}).
	The only difference between the estimates derived in \cite[Lemma~9.3]{jSgHjLwW2016}
	and the estimates we need to derive
	is that our wave equations
	\eqref{E:VELOCITYWAVEEQUATION}--\eqref{E:ENTROPYWAVEEQUATION},
	when weighted with a factor of $\upmu$
	(so that the decomposition \eqref{eq:wave.in.terms.of.geometric}
	of $\upmu \square_g$ can be employed),
	feature some new inhomogeneous terms compared to \cite[Lemma~9.3]{jSgHjLwW2016},
	specifically, some of the ones depending on $(\mathcal{C},\mathcal{D},\Vr,S)$
	and the first derivatives of $(\Vr,S)$.
	The key point is that our new bootstrap assumption
	implies that the new inhomogeneous terms are all bounded in $L^{\infty}$ by $\lesssim \epd^{\frac{1}{2}}$,
	which is compatible with the $\mathcal{O}(\epd^{\frac{1}{2}})$-size bounds that
	one is aiming to prove; i.e., our new $\mathcal{O}(\epd^{\frac{1}{2}})$-sized error terms
	are harmless in the context of the proof.
	From this logic, it follows that
	the estimates \eqref{E:LUNITTWORADPSI}--\eqref{E:TWOTRANSVERSALDERIVATIESLUNITILINFTY}
	and \eqref{E:PERMUTEDESTIMATES}
	hold for all $t \in [0,\Tboot)$.
	We clarify that the estimates 
	\eqref{E:LUNITLUNITRADRADUPMULINFTY}
	and
	\eqref{E:LTANGENTIALTWOTRANSVERSALDERIVATIESLUNITILINFTY} 
	were not explicitly stated in \cite[Lemma~9.3]{jSgHjLwW2016}.
	However \eqref{E:LUNITLUNITRADRADUPMULINFTY}
	follows from
	commuting the transport
	equation \eqref{E:UPMUFIRSTTRANSPORT}
	with $\Lunit \Rad \Rad$
	via Lemma~\ref{L:SIMPLECOMMUTATORIDENTITY}
	and bounding the resulting algebraic expression
	for $\Lunit \Lunit \Rad \Rad \upmu$
	using the fact that the Cartesian component functions $X^1,X^2,X^3$ are smooth functions of
	the $L^i$ and $\Psi$ (see \eqref{E:TRANSPORTVECTORFIELDINTERMSOFLUNITANDRADUNIT}),
	the bootstrap assumptions \eqref{BA:LARGERIEMANNINVARIANTLARGE}--\eqref{BA:S},
	Proposition~\ref{prop:geometric.low},
	and the estimates 
	in \eqref{E:LUNITTWORADPSI}--\eqref{E:TWOTRANSVERSALDERIVATIESLUNITILINFTY}
	and \eqref{E:PERMUTEDESTIMATES}
	besides 
	\eqref{E:LUNITLUNITRADRADUPMULINFTY}
	and
	\eqref{E:LTANGENTIALTWOTRANSVERSALDERIVATIESLUNITILINFTY} .
	Similarly, \eqref{E:LTANGENTIALTWOTRANSVERSALDERIVATIESLUNITILINFTY} 
	follows from
	commuting the transport
	equation \eqref{E:LLUNITI}
	with $\mathcal{P} \Rad \Rad$.
	
	To complete the proof, it only remains for
	us to prove \eqref{E:VORTICITYANDENTROPYLINFTYHIGHERTRANSVERSAL} and \eqref{E:TRANSPORTPERMUTEDESTIMATES}
	(with the help of the already established bounds 
	\eqref{E:LUNITTWORADPSI}--\eqref{E:TWOTRANSVERSALDERIVATIESLUNITILINFTY} and \eqref{E:PERMUTEDESTIMATES});
	for if $\epd$ is sufficiently small, this yields a strict improvement
	of the new bootstrap assumption mentioned at the beginning of the proof,
	and the conclusions of the proposition then
	follow from a standard continuity argument.
	We start by noting that the bounds in \eqref{E:VORTICITYANDENTROPYLINFTYHIGHERTRANSVERSAL}
	for the pure $\mathcal{F}_u$-tangential derivatives of $(\Vr,S)$ 
	are included in the bootstrap assumptions \eqref{BA:V}--\eqref{BA:S},
	as are the following bounds:
	\begin{align} \label{E:LINFINITYBOUNDSFORONETRANSVERSALDERIVATIVEENTROPYANDVORTICITY}
		\| \mathcal{P}^{\leq 3} \Rad (\Vr,S) \|_{L^\infty(\Mtu)}
		& \lesssim \epd.
	\end{align}
	Next, we use 
	Lemma~\ref{L:SIMPLECOMMUTATORIDENTITY},
	the bootstrap assumptions \eqref{BA:LARGERIEMANNINVARIANTLARGE}--\eqref{BA:S},
	Proposition~\ref{prop:geometric.low},
	the estimates 
	\eqref{E:LUNITTWORADPSI}--\eqref{E:TWOTRANSVERSALDERIVATIESLUNITILINFTY} and \eqref{E:PERMUTEDESTIMATES},
	and the bounds 
	\eqref{E:LINFINITYBOUNDSFORONETRANSVERSALDERIVATIVEENTROPYANDVORTICITY}
	to deduce that the estimate \eqref{E:LINFINITYBOUNDSFORONETRANSVERSALDERIVATIVEENTROPYANDVORTICITY} also 
	holds for all permutations of the vectorfield operators on the LHS.
	
	We next show that:
	\begin{align} \label{E:LINFINITYBOUNDSFORTWOTRANSVERSALDERIVATIVESENTROPYANDVORTICITY}
		\| \mathcal{P}^{\leq 2} \Rad \Rad (\Vr,S) \|_{L^\infty(\Mtu)}
		& \lesssim \epd.
	\end{align}
	This estimate follows from differentiating 
	the identities 
	\eqref{E:TRANSVERSALDERIVATIVEEXPRESSIONFORTRANSPORTVARIABLES1}--\eqref{E:TRANSVERSALDERIVATIVEEXPRESSIONFORTRANSPORTVARIABLES2}
	with $\mathcal{P}^{\leq 2} \Rad$ and using
	the bootstrap assumptions \eqref{BA:LARGERIEMANNINVARIANTLARGE}--\eqref{BA:S},
	Proposition~\ref{prop:geometric.low},
	the estimates 
	\eqref{E:LUNITTWORADPSI}--\eqref{E:TWOTRANSVERSALDERIVATIESLUNITILINFTY} and \eqref{E:PERMUTEDESTIMATES},
	the estimate \eqref{E:LINFINITYBOUNDSFORONETRANSVERSALDERIVATIVEENTROPYANDVORTICITY},
	and the analog of
	\eqref{E:LINFINITYBOUNDSFORONETRANSVERSALDERIVATIVEENTROPYANDVORTICITY}
	for all permutations of the vectorfield operators on the LHS.
	(Notice that we can indeed prove \eqref{E:LINFINITYBOUNDSFORTWOTRANSVERSALDERIVATIVESENTROPYANDVORTICITY}
	with a strict improvement of our new bootstrap assumptions because the terms 
	arising from differentiating 
	\eqref{E:TRANSVERSALDERIVATIVEEXPRESSIONFORTRANSPORTVARIABLES1}--\eqref{E:TRANSVERSALDERIVATIVEEXPRESSIONFORTRANSPORTVARIABLES2}
	by $\mathcal{P}^{\leq 2} \Rad$ contain at least one factor of $(\Vr,S)$ differentiated with  
	at most one $\breve{X}$ derivative, and such factors have already been shown to bounded in the norm
	$\| \cdot \|_{L^\infty(\Mtu)}$ by $\lesssim \epd$.)
	Again using Lemma~\ref{L:SIMPLECOMMUTATORIDENTITY} to commute vectorfield derivatives,
	we also deduce that the estimate \eqref{E:LINFINITYBOUNDSFORTWOTRANSVERSALDERIVATIVESENTROPYANDVORTICITY} 
	also holds for all permutations of the vectorfield operators on the LHS. 
	
	We next show that:
	\begin{align} \label{E:LINFINITYBOUNDSFORTHREETRANSVERSALDERIVATIVESENTROPYANDVORTICITY}
		\| \Rad \Rad \Rad (\Vr,S) \|_{L^\infty(\Mtu)}
		& \lesssim \epd.
	\end{align}
	This estimate follows from differentiating 
	the identities \eqref{E:TRANSVERSALDERIVATIVEEXPRESSIONFORTRANSPORTVARIABLES1}--\eqref{E:TRANSVERSALDERIVATIVEEXPRESSIONFORTRANSPORTVARIABLES2}
	with $\Rad \Rad$ and
	using the bootstrap assumptions \eqref{BA:LARGERIEMANNINVARIANTLARGE}--\eqref{BA:S},
	Proposition~\ref{prop:geometric.low},
	the estimates 
	\eqref{E:LUNITTWORADPSI}--\eqref{E:TWOTRANSVERSALDERIVATIESLUNITILINFTY} and \eqref{E:PERMUTEDESTIMATES},
	the estimates 
	\eqref{E:LINFINITYBOUNDSFORONETRANSVERSALDERIVATIVEENTROPYANDVORTICITY}--\eqref{E:LINFINITYBOUNDSFORTWOTRANSVERSALDERIVATIVESENTROPYANDVORTICITY},
	and the analogs of
	\eqref{E:LINFINITYBOUNDSFORONETRANSVERSALDERIVATIVEENTROPYANDVORTICITY}--\eqref{E:LINFINITYBOUNDSFORTWOTRANSVERSALDERIVATIVESENTROPYANDVORTICITY}
	for all permutations of the vectorfield operators on the LHSs.
	
	Similarly, we can first prove:
	\begin{align} \label{E:LINFINITYBOUNDSFORONETRANSVERSALDERIVATIVEMODIFIEDFLUID}
		\| \mathcal{P}^{\leq 2} \Rad^{\leq 1} (\mathcal{C},\mathcal{D}) \|_{L^\infty(\Mtu)}
		& \lesssim \epd
	\end{align}
	and then:
	\begin{align} \label{E:LINFINITYBOUNDSFORTWOTRANSVERSALDERIVATIVESMODIFIEDFLUID}
		\| \Rad^{\leq 2} (\mathcal{C},\mathcal{D}) \|_{L^\infty(\Mtu)}
		& \lesssim \epd
	\end{align}
	(and that \eqref{E:LINFINITYBOUNDSFORONETRANSVERSALDERIVATIVEMODIFIEDFLUID}
	holds for all permutations of the vectorfield operators on the LHS
	all permutations of the vectorfield operators on the LHS)
	by using the identities \eqref{E:TRANSVERSALDERIVATIVEEXPRESSIONFORMODIFIEDFLUIDVARIABLES1}--\eqref{E:TRANSVERSALDERIVATIVEEXPRESSIONFORMODIFIEDFLUIDVARIABLES2}
	and arguing as above, using
	in addition the bootstrap assumption \eqref{BA:C.D}
	and the already proven estimates for $(\Vr,S)$.
	
	We have therefore established \eqref{E:VORTICITYANDENTROPYLINFTYHIGHERTRANSVERSAL}
	and \eqref{E:TRANSPORTPERMUTEDESTIMATES},
	which completes the proof of the proposition. \qedhere

	\end{proof}

\subsection{Sharp estimates for $\upmu_{\star}$}
Recall the definition of $\upmu_{\star}(t)$ in Definition~\ref{def:mustar}. In this subsection, 
in Propositions~\ref{prop:mus.int} and \ref{prop:almost.monotonicity}, 
we provide some estimates for $\upmu_{\star}(t)$ that were proved 
in \cite{jSgHjLwW2016}. We will simply cite the relevant estimates, 
noting that their proof relies only on the $L^\infty$ bounds for (lower-order derivatives of) 
the wave variables and the geometric quantities that we have already established. 
Moreover, we remark
that these estimates capture that $\upmu_{\star}(t)$ tends to $0$ linearly,
a fact that is crucial for bounding the maximum possible 
singularity strength of our high-order geometric energies 
(i.e., for controlling the blowup-rate of the energies in, for example, \eqref{BA:W1}).

Thanks to our bootstrap assumptions and the estimates of Proposition~\ref{prop:geometric.low},
the following estimates for $\upmu_{\star}(t)$ can be proved exactly as in 
\cite[(10.36), (10.39)]{jSgHjLwW2016}:
\begin{proposition}[\textbf{Control of integrals of $\upmu_{\star}$}]
\label{prop:mus.int}
Let $\toprate\in \mathbb N$ be the absolute constant appearing in
the statements of Theorem~\ref{thm:main} and
Proposition~\ref{prop:wave} below. 
For $1< b \leq 100 \toprate$, 
the quantities $\upmu_{\star}(t,u)$ and $\upmu_{\star}(t)$ from Definition~\ref{def:mustar}
obey the following estimates for every $(t,u) \in [0,\Tboot) \times [0,U_0]$:
\begin{equation}\label{eq:mus.int.1}
\int_{t'=0}^{t'=t} \upmu_{\star}^{-b} (t',u) \,dt' \ls \left(1+ \f{1}{b-1} \right) \upmu_{\star}^{-b+1}(t,u),
	\quad
\int_{t'=0}^{t'=t} \upmu_{\star}^{-b} (t') \,dt' \ls \left(1+ \f{1}{b-1} \right) \upmu_{\star}^{-b+1}(t).
\end{equation}
Moreover, for all $t\in [0,\Tboot)$,
\begin{equation}\label{eq:mus.int.2}
\int_{t'=0}^{t'=t} \upmu_{\star}^{-0.9} (t',u) \,dt' \ls 1,
	\quad
\int_{t'=0}^{t'=t} \upmu_{\star}^{-0.9} (t') \,dt' \ls 1.
\end{equation}
\end{proposition}

Thanks to our bootstrap assumptions and the estimates of Proposition~\ref{prop:geometric.low}, the
following 
``almost-monotonicity''
of $\upmu_{\star}$ can be proved as in \cite[(10.23)]{jSgHjLwW2016}:
\begin{proposition}[\textbf{The approximate monotonicity of $\upmu_{\star}$}]
\label{prop:almost.monotonicity}
For $0\leq s_1\leq s_2< \Tboot$,
$$\upmu_{\star}^{-1}(s_1) \leq 2 \upmu_{\star}^{-1}(s_2). $$
\end{proposition}

\subsection{$L^2$ estimates for the geometric quantities}
We start with a simple lemma that provides $L^2$ estimates for solutions to transport equations
along the integral curves of $\Lunit$.

\begin{lemma}[$L^2$ estimate for solutions to $\Lunit$-transport equations]
	\label{L:L2ESTIMATEFORTRANSPORT}
	Let $F$ and $f$ be smooth scalar functions on
	$[0,\Tboot)\times [0,U_0] \times \mathbb{T}^2$.
	Assume that 
	$\Lunit F(t,u,x^2,x^3) = f(t,u,x^2,x^3)$
	with initial data $F(0,u,x^2,x^3)$
	for every $(t, u,x^2,x^3) \in [0,\Tboot)\times [0,U_0] \times \mathbb{T}^2$.
	Then the following estimate holds for every $(t,u) \in [0,\Tboot)\times [0,U_0]$:
	\begin{align} \label{E:L2ESTIMATEFORTRANSPORT}
		\| F \|_{L^2(\Sigma_t^u)}
		& 
		\leq
		(1 + C \epd^{\frac{1}{2}}) 
		\| F \|_{L^2(\Sigma_0^u)}
		+
		(1 + C \epd^{\frac{1}{2}}) \int_{t'=0}^{t'=t} \| f \|_{L^2(\Sigma_{t'}^u)} \, dt'. 
	\end{align}
\end{lemma}
\begin{proof}
	Thanks to our bootstrap assumptions and the estimates of Proposition~\ref{prop:geometric.low},
	\eqref{E:L2ESTIMATEFORTRANSPORT} can be proved using essentially the same
	arguments used in the proof of \cite[Lemmas~12.2, 12.3, 13.2]{jSgHjLwW2016}.
	The only differences are 
	that we have to use the bootstrap assumptions \eqref{BA:LARGERIEMANNINVARIANTLARGE}--\eqref{BA:C.D}
	in place of the similar bootstrap assumptions from \cite{jSgHjLwW2016},
	and that different coordinates along $\ell_{t,u}$ were used in \cite{jSgHjLwW2016}
	(this is irrelevant in the sense that the estimate \eqref{E:L2ESTIMATEFORTRANSPORT}
	is independent of the coordinates on $\ell_{t,u}$).
	We clarify that the bootstrap parameter ``$\varepsilon$'' from 
	\cite{jSgHjLwW2016} should be identified with the quantity $\epd^{\frac{1}{2}}$
	in our bootstrap assumptions \eqref{BA:LARGERIEMANNINVARIANTLARGE}--\eqref{BA:C.D}.
\end{proof}

\begin{proposition}[\textbf{Easy $L^2$ estimates for the acoustical geometry}]
For $1 \leq N \leq \Ntop$, the following estimates hold for all $t\in [0,\Tboot)$:
\label{prop:geometric.top}
$$\|\mathcal{P}^{[2,N]} \upmu \|_{L^2(\Sigma_t)}^2,\, \|\mathcal{P}^{[1,N]} L^i \|_{L^2(\Sigma_t)}^2 \ls \epd \max\{1, \upmu_{\star}^{-2 \toprate+2\Ntop-2N+2.8}(t) \}.$$
\end{proposition}
\begin{proof}
In an identical manner as \cite[Lemma~14.3]{jSgHjLwW2016}, 
based on the transport equations \eqref{E:UPMUFIRSTTRANSPORT}--\eqref{E:LLUNITI}
and \eqref{E:L2ESTIMATEFORTRANSPORT}, we obtain:
$$\|\mathcal{P}^{[2,N]} \upmu \|_{L^2(\Sigma_t)},\, \|\mathcal{P}^{[1,N]} L^i \|_{L^2(\Sigma_t)} \ls \epd + \int_{s=0}^{s=t} \f{\mathbb{W}_{[1,N]}^{\f 12}(s)}{\upmu_{\star}^{\f 12}(s)}\, ds.$$
(Recall our notations in Definition~\ref{def:array.convention}, \eqref{eq:wave.energy.def.4} and Definition~\ref{def:energy.convention}.) Also using our bootstrap assumptions \eqref{BA:W1} and \eqref{BA:W2}
and Proposition~\ref{prop:mus.int}, we arrive at the desired conclusion. \qedhere
\end{proof}

In the next proposition, with the help of
of Proposition~\ref{prop:geometric.top}, 
we derive $L^2$ estimates for commutators.
\begin{proposition}[\textbf{$L^2$ estimates for commutator terms}]
\label{prop:commutators}
Let $\phi$ be a scalar function.
For $1 \leq N \leq \Ntop$, the following estimates hold for all $(t,u)\in [0,\Tboot)\times [0,U_0]$:
\begin{equation}\label{eq:commutator.L2.main}
\begin{split}
&\: \| [L, \mathcal{P}^N]\phi \|_{L^2(\Sigma_t^u)}^2,\, \| [\bX, \mathcal{P}^N]\phi \|_{L^2(\Sigma_t^u)}^2,\,\| [\upmu B, \mathcal{P}^N]\phi \|_{L^2(\Sigma_t^u)}^2 \\
\ls &\: \|\mathcal{P}^{[1,N]} \phi \|_{L^2(\Sigma_t^u)}^2 + \epd \max\{1, \upmu_{\star}^{-2 \toprate+2\Ntop-2N+2.8}(t) \} \| \mathcal{P}^{[1,\Ntop - \toprate-5]} \phi\|_{L^\infty(\Sigma_t^u)}^2.
\end{split}
\end{equation}
Moreover, we also have:
\begin{equation}\label{eq:wave.energy.commuted}
\|\mathcal{P}^N \bX\Psi\|_{L^2(\Sigma_t^u)}^2 \ls \epd \max\{1, \upmu_{\star}^{-2 \toprate+2\Ntop-2N+1.8}(t) \}.
\end{equation}
\end{proposition}
\begin{proof}
Recall the pointwise estimate \eqref{eq:commutators.Li.1}. For each of the sums in 
\eqref{eq:commutators.Li.1}, either $N_2>N_1$, in which case by \eqref{BA:W.Li.small} and Proposition~\ref{prop:geometric.low}, 
we have $|\mathcal{P}^{[2,N_1]}(\upmu,L^i,\Psi)|,\,|\mathcal{P}^{[2,N_1]}\bX\Psi| \ls 1$; or else 
$N_2 \leq N_1$, in which case (since $N\leq \Ntop$) $|\mathcal{P}^{[1,N_2]} \phi| \ls |\mathcal{P}^{[1, \Ntop - \toprate-5]} \phi|$. Hence, 
\begin{equation}\label{eq:commutator.pointwise}
\begin{split}
&\: |[L, \mathcal{P}^N]\phi|,\,|[\bX,\mathcal{P}^N]\phi| \\
\ls &\: |\mathcal{P}^{[1,N]} \phi|
+ 
\left\{ |\mathcal{P}^{[2,N]} (\upmu, L^i ,\Psi)| + |\mathcal{P}^{[2,N-1]} \bX\Psi| \right\}|\mathcal{P}^{[1,\Ntop - \toprate-5]} \phi| \\
\ls &\: |\mathcal{P}^{[1,N]} \phi|+ 
\left\{ |\mathcal{P}^{[2,N]} (\upmu, L^i ,\Psi)| + |\bX \mathcal{P}^{[2,N-1]} \Psi| + |[\bX,\mathcal{P}^{[2,N-1]}]\Psi| \right\}
|\mathcal{P}^{[1,\Ntop - \toprate-5]} \phi|.
\end{split}
\end{equation}

We first apply \eqref{eq:commutator.pointwise} to $\phi = \Psi$. Taking the $L^2(\Sigma_t^u)$ norm and introducing an induction argument in $N$ which uses \eqref{BA:W1}--\eqref{BA:W.Li.small}
and Proposition~\ref{prop:geometric.top}, we obtain: 
\begin{equation}\label{eq:commute.X.P.Psi}
\|[\bX,\mathcal{P}^N]\Psi \|_{L^2(\Sigma_t^u)}^2 \ls \epd \max\{1, \upmu_{\star}^{-2 \toprate+2\Ntop-2N+2.8}(t)\}.
\end{equation}

Taking the $L^2(\Sigma_t^u)$ norm in \eqref{eq:commutator.pointwise}, plugging in the estimate \eqref{eq:commute.X.P.Psi}, and using \eqref{BA:W1}, \eqref{BA:W2}, and Proposition~\ref{prop:geometric.top}{, we deduce} 
the desired estimates in \eqref{eq:commutator.L2.main} for $[L, \mathcal{P}^N]\phi$ and $[\bX, \mathcal{P}^N]\phi$. 

To obtain the $[\upmu B, \mathcal{P}^N]\phi$ estimate in \eqref{eq:commutator.L2.main},
{we first} note that by \eqref{E:TRANSPORTVECTORFIELDINTERMSOFLUNITANDRADUNIT},
$$[\upmu B, \mathcal{P}^N]\phi =  \upmu [L,\mathcal{P}^N] \phi + [\upmu, \mathcal{P}^N ]L\phi + [\bX, \mathcal{P}^N] \phi.$$
The first and last term can be controlled by {combining}
the commutator estimates we just established 
with the simple bound $\|\upmu\|_{L^\infty(\Sigma_t)} \ls 1$ 
from Proposition~\ref{prop:geometric.low}, while the second term can be controlled simply using the product rule and Propositions~\ref{prop:geometric.low} and \ref{prop:geometric.top}. 
We have therefore established \eqref{eq:commutator.L2.main}.

Finally, we have \eqref{eq:wave.energy.commuted} thanks to \eqref{BA:W1}, \eqref{BA:W2} and \eqref{eq:commute.X.P.Psi}. \qedhere 
\end{proof}

\section{Transport estimates for the specific vorticity and the entropy gradient}\label{sec:transport.easy.main}
We continue to work under the assumptions of Theorem~\ref{thm:bootstrap}.

In this section, we use the transport equations \eqref{E:RENORMALIZEDVORTICTITYTRANSPORTEQUATION} 
and \eqref{E:GRADENTROPYTRANSPORT} to bound $\mathcal{P}^N \Vr$ and $\mathcal{P}^N S$ for $N \leq \Ntop$.
We clarify that the ``true'' top-order estimates for the vorticity and entropy are found in
Section~\ref{sec:transport.hard}; those estimates are more involved and rely on the modified fluid variables
as well as elliptic estimates.

We will start by deriving energy estimates for general transport equations (which will also be useful in the next section). In particular, this will reduce the derivation of the energy estimates 
for $\mathcal{P}^N \Vr$ and $\mathcal{P}^N S$
to controlling the inhomogeneous terms in the transport equations and their derivatives, 
which we will carry out in Section~\ref{sec:easy.transport.homogeneous}.
The final estimates for $\mathcal{P}^N \Vr$ and $\mathcal{P}^N S$
are located in Section~\ref{sec:easy.transport.everything}.

\subsection{Estimates for general transport equations} \label{sec:transport.easy}
\begin{proposition}[\textbf{$L^2$ estimates for solutions to $\Transport$-transport equations}]
\label{prop:transport}
Let $\phi$ be a scalar function satisfying:
$$\upmu B \phi = \inhom $$
with both $\phi$ and $\inhom$ being compactly supported in $[0,U_0]\times \mathbb T^2$ for every $t\in [0,\Tboot)$.

Then the following estimate holds for every $(t,u)\in [0,\Tboot)\times [0,U_0]$:
$$\sup_{t'\in [0,t)} \|\sqrt{\upmu} \phi\|_{L^2(\Sigma_{t'}^u)}^2 + \sup_{u'\in [0,u)} \|\phi \|_{L^2(\mathcal F_{u'}^t)}^2  \ls \|\sqrt{\upmu} \phi\|_{L^2(\Sigma_0^u)}^2+ \| \inhom \|_{L^2(\Mtu)}^2.$$
\end{proposition}
\begin{proof}
In an identical manner as \cite[Proposition~3.5]{LS}, we have, for any $(t',u')  \in [0,t)\times [0,u)$, 
the identity:
\begin{equation}\label{eq:transport.basic}
\begin{split}
&\: \int_{\Sigma_{t'}^{u'}} \upmu \phi^2 \, d\tvol + \int_{\mathcal F_{u'}^{t'}} \phi^2 \,d\conevol \\
= &\: \int_{\Sigma_0^{u'} } \upmu \phi^2 \, d\tvol 
+ 
\underbrace{\int_{\mathcal F_0^{t'}} \phi^2 \,d\conevol}_{\mbox{$0$ by support assumptions}} 
+ \int_{\mathcal M_{t',u'}}  \{ 2\phi \mathfrak F + (L\upmu + \upmu \mytr \angk) \phi^2 \} \, d\vol.
\end{split}
\end{equation}
Using \eqref{eq:angk.def}, \eqref{E:UPMUFIRSTTRANSPORT}, 
Lemma~\ref{lem:induced.metric},
\eqref{BA:LARGERIEMANNINVARIANTLARGE}--\eqref{BA:W.Li.small},
and Propositions~\ref{prop:geometric.low} and \ref{prop:geometric.low.2}, 
we have $|L\upmu|,\,|\upmu\mytr \angk|\ls 1$. Thus, applying also the Cauchy--Schwarz inequality to the $2\phi\mathfrak F$ term, we have:
$$\sup_{t'\in [0,t)} \|\sqrt{\upmu} \phi\|_{L^2(\Sigma_{t'}^u)}^2 + \sup_{u'\in [0,u)} \|\phi \|_{L^2(\mathcal F_{u'}^t)}^2  \ls \|\sqrt{\upmu} \phi\|_{L^2(\Sigma_0^u)}^2+  \int_{u'=0}^{u'=u} \| \phi \|_{L^2(\mathcal F_{u'}^t)}^2\, du'+ \| \inhom \|_{L^2(\Mtu)}^2.$$
The conclusion follows from applying Gr\"onwall's inequality in $u$. \qedhere
\end{proof}

\begin{proposition}[\textbf{Higher-order $L^2$ estimates for solutions to transport equations}]
\label{prop:transport.higher.order}
Let $\phi$ be a scalar function satisfying:
$$\upmu B \phi = \inhom, $$
with both $\phi$ and $\inhom$ being compactly supported in $[0,U_0]\times \mathbb T^2$ for every $t\in [0,\Tboot)$.

Then the following estimate holds for every $(t,u)\in [0,\Tboot)\times [0,U_0]$ and $0\leq N\leq \Ntop$:
\begin{equation*}
\begin{split}
&\: \sup_{t'\in [0,t)}\|\sqrt{\upmu} \mathcal{P}^{\leq N} \phi \|_{L^2(\Sigma_{t'}^u)}^2 + \sup_{u'\in [0,u)}\| \mathcal{P}^{\leq N} \phi \|_{L^2(\mathcal F_{u'}^t)}^2 \\
\ls &\: \| \mathcal{P}^{\leq N} \phi \|_{L^2(\Sigma_0)}^2 + \| \mathcal{P}^{\leq N} \inhom \|_{L^2(\mathcal M_{t,u})}^2 + \epd \max\{ 1, \upmu_{\star}^{-2 \toprate+2\Ntop-2N+3.8}(t) \} \|\mathcal{P}^{[1, \Ntop- \toprate-5]}\phi \|_{L^\infty(\Mtu)}^2.
\end{split}
\end{equation*}

\end{proposition}
\begin{proof}
Take $0\leq N' \leq N$. We write:
$$\upmu B \mathcal{P}^{N'} \phi = \mathcal{P}^{N'} \inhom + [\upmu B, \mathcal{P}^{N'}]\phi.$$
Therefore, by Proposition~\ref{prop:transport}, 
\begin{equation}\label{eq:transport.higher.order.1}
\begin{split} 
&\: \sup_{t'\in [0,t)}\|\sqrt{\upmu} \mathcal{P}^{N'} \phi \|_{L^2(\Sigma_{t'}^u)}^2 + \sup_{u'\in [0,u)}\| \mathcal{P}^{N'} \phi \|_{L^2(\mathcal F_{u'}^t)}^2 \\
\ls &\: \| \mathcal{P}^{N'} \phi \|_{L^2(\Sigma_0)}^2 + \|\mathcal{P}^{N'} \inhom\|_{L^2(\Mtu)}^2 + \| [\upmu B, \mathcal{P}^{N'}]\phi\|_{L^2(\Mtu)}^2.
\end{split}
\end{equation}

Using Proposition~\ref{prop:commutators} and then Proposition~\ref{prop:mus.int}, we obtain:
\begin{equation}\label{eq:transport.higher.order.2}
\begin{split}
&\: \| [\upmu B, \mathcal{P}^{N'}]\phi\|_{L^2(\Mtu)}^2 \ls \int_{t'=0}^{t'=t} \| [\upmu B, \mathcal{P}^{N'}]\phi\|_{L^2(\Sigma_t^u)}^2 \, dt' \\
\ls &\: \|\mathcal{P}^{[1,N']} \phi \|_{L^2(\Mtu)}^2 + \epd \|\mathcal{P}^{[1,\Ntop- \toprate-5]} \phi \|_{L^\infty(\Mtu)}^2 \int_{t'=0}^{t'=t} \max\{ 1, \upmu_{\star}^{-2 \toprate+2\Ntop-2N'+2.8}(t') \} \, dt' \\
\ls &\: \int_{u'=0}^{u'=u} \|\mathcal{P}^{\leq N} \phi \|_{L^2(\mathcal F_u^t)}^2 du' + \epd \max\{ 1, \upmu_{\star}^{-2 \toprate+2\Ntop-2N'+3.8}(t) \} \|\mathcal{P}^{[1,\Ntop- \toprate-5} \phi \|_{L^\infty(\Mtu)}^2.
\end{split}
\end{equation}

Plugging \eqref{eq:transport.higher.order.2} into \eqref{eq:transport.higher.order.1} and summing over all $0\leq N'\leq N$, we obtain:
\begin{equation}\label{eq:transport.higher.order.3}
\begin{split} 
&\: \sup_{t'\in [0,t)}\|\sqrt{\upmu} \mathcal{P}^{\leq N} \phi \|_{L^2(\Sigma_{t'}^u)}^2 + \sup_{u'\in [0,u)}\| \mathcal{P}^{\leq N} \phi \|_{L^2(\mathcal F_{u'}^t)}^2 \\
\ls &\: \| \mathcal{P}^{\leq N} \phi \|_{L^2(\Sigma_0)}^2 + \|\mathcal{P}^{\leq N} \inhom\|_{L^2(\Mtu)}^2 + \int_{u'=0}^{u'=u} \|\mathcal{P}^{\leq N} \phi \|_{L^2(\mathcal F_{u'}^t)}^2 du' \\
&\: + \epd \max\{ 1, \upmu_{\star}^{-2 \toprate+2\Ntop-2N+3.8}(t) \} \|\mathcal{P}^{[1,\Ntop- \toprate-5} \phi \|_{L^\infty(\Mtu)}^2.
\end{split}
\end{equation}
Applying Gr\"onwall's inequality in $u$, we arrive at the desired estimate. \qedhere
\end{proof}

\subsection{Controlling the inhomogeneous terms}\label{sec:easy.transport.homogeneous}

\begin{proposition}[\textbf{Estimates tied to the inhomogeneous terms in the transport equations for $\Vr$ and $S$}]
\label{prop:Vr.S.inho}
For $0\leq N\leq \Ntop$, the following holds for every $(t,u)\in [0,\Tboot)\times [0,U_0]$:
\begin{equation}\label{eq:Vr.S.inho.weak}
\begin{split}
&\: \|\mathcal{P}^N (\upmu B \Vr)\|_{L^2(\Mtu)}^2 + \|\mathcal{P}^N (\upmu B S)\|_{L^2(\Mtu)}^2 \\
\ls &\:  \epd^3 \max\{ 1, \upmu_{\star}^{-2 \toprate + 2\Ntop -2N + 2.8}(t) \} + \int_{u' = 0}^{u'=u} (\mathbb{V}_{\leq N}(t,u') + \mathbb{S}_{\leq N}(t,u'))\, du'
\end{split}
\end{equation}
and:
\begin{equation}\label{eq:Vr.S.inho.strong}
\begin{split}
&\: \|(\mathcal{P}^N \Vr)\mathcal{P}^N (\upmu B \Vr)\|_{L^1(\Mtu)} + \|(\mathcal{P}^N S) \mathcal{P}^N (\upmu B S)\|_{L^1(\Mtu)} \\
\ls &\: \epd^3 \max\{ 1, \upmu_{\star}^{-2 \toprate + 2\Ntop -2N + 2.8}(t) \} + \int_{u' = 0}^{u'=u} (\mathbb{V}_{\leq N}(t,u') + \mathbb{S}_{\leq N}(t,u'))\, du'.
\end{split}
\end{equation}
\end{proposition}

\begin{proof}

\pfstep{Step~1: Basic pointwise estimates} We claim that the derivatives of the 
$\upmu$-weighted inhomogeneous 
terms $\upmu \mathfrak{L}^i_{(\Vr)}$ and $\upmu \mathfrak{L}^i_{(\GradEnt)}$,
which are defined respectively in
\eqref{E:SPECIFICVORTICITYLINEARORBETTER} and \eqref{E:ENTROPYGRADIENTLINEARORBETTER},
obey the following pointwise bounds:
\begin{equation}\label{eq:pointwise.of.the.most.basic.transport.inho}
\begin{split}
&\: |\mathcal{P}^N (\upmu\mathfrak L^i_{(\Vr)})| + |\mathcal{P}^N (\upmu\mathfrak L^i_{(S)})|\\
\ls &\:  \underbrace{|\mathcal{P}^{\leq N} (\Vr,S)|}_{\doteq I} + \underbrace{\epd(|\upmu\mathcal{P}^{[2,N+1]} \Psi| + |\mathcal{P}^{[1, N]} \bX \Psi|) }_{\doteq II} + \underbrace{\epd |\mathcal{P}^{[2,N]}(\upmu,\,L^i,\, \Psi)| }_{\doteq III}.
\end{split}
\end{equation}

Since this is the first instance of these kind of estimates (and we will derive similar estimates later), we give some details on how to obtain \eqref{eq:pointwise.of.the.most.basic.transport.inho}.
\begin{enumerate}
\item By Lemma~\ref{lem:Cart.to.geo}
and the fact that the Cartesian component functions $X^1,X^2,X^3$ are smooth functions of
the $L^i$ and $\Psi$ (see \eqref{E:TRANSPORTVECTORFIELDINTERMSOFLUNITANDRADUNIT}), 
the weighted Cartesian coordinate vectorfield $\upmu\rd_i$ and the transport vectorfield $\upmu B$ can be decomposed regularly (i.e., with coefficients being smooth functions of $\upmu$, $L^i$ and $\Psi$) in terms of $\bX$, $\upmu Y$, $\upmu Z$ and $\upmu L$.
\item Therefore, $\mathcal{P}^N (\upmu\mathfrak L^i_{(\Vr)})$ and $\mathcal{P}^N (\upmu\mathfrak L^i_{(S)})$ can be bounded 
as follows:
\begin{equation}\label{eq:example.product}
\begin{split}
&\: | \mathcal{P}^N (\upmu\mathfrak L^i_{(\Vr)})| + |\mathcal{P}^N (\upmu\mathfrak L^i_{(S)})|\\
\ls &\: \sum_{k=0}^N \sum_{N_1+ \cdots + N_k +n_1+n_2 = N} (1 + |\mathcal{P}^{N_1} (\upmu, L^i, \Psi)|)\cdots (1 + |\mathcal{P}^{N_k} (\upmu, L^i,\Psi)|) \\
&\: \qquad \qquad \qquad \qquad  \times  |\mathcal{P}^{n_1}(\Vr, S)| \times |\mathcal{P}^{n_2}(\upmu \mathcal{P} \Psi, \bX\Psi)| \doteq \sum_{k=0}^N \mathrm{Error}_{N_1, \dots, N_k, n_1, n_2}.
\end{split}
\end{equation}
We now bound RHS~\eqref{eq:example.product}.
\item If $N_1,\,\dots, \, N_k\leq \Ntop -  \toprate -5$ and $n_2 \leq \Ntop - \toprate-5$, we bound the terms $(1 + |\mathcal{P}^{N_j} (\upmu, L^i,\Psi)|)$ (for all $j=1,\dots,k$) and $|\mathcal{P}^{n_2}(\upmu \mathcal{P} \Psi, \bX\Psi)|$ in $L^\infty$
by $\lesssim 1$ using \eqref{BA:LARGERIEMANNINVARIANTLARGE}--\eqref{BA:W.Li.small}
and Proposition~\ref{prop:geometric.low}, which yields:
\begin{equation}\label{eq:example.product.1}
\mathrm{Error}_{N_1, \dots, N_k, n_1, n_2} \ls |\mathcal{P}^{\leq N} (\Vr,S)|.
\end{equation}
\item If $N_j > \Ntop -  \toprate - 5$ for some\footnote{Note that there can be at most one such $j$.} $j$, then all the terms $(1 + |\mathcal{P}^{N_{j'}} (\upmu, L^i, \Psi)|)$ when $j' \neq j$ and $|\mathcal{P}^{n_2}(\upmu \mathcal{P} \Psi, \bX\Psi)|$ can be bounded in $L^\infty$ by $\lesssim 1$ 
using \eqref{BA:LARGERIEMANNINVARIANTLARGE}--\eqref{BA:W.Li.small} and Proposition~\ref{prop:geometric.low}. Moreover, since it must also hold that $n_1 \leq \Ntop - \toprate -5$, we also have $|\mathcal{P}^{n_1}(\Vr, S)| \ls \epd$ by the bootstrap assumptions \eqref{BA:V} and \eqref{BA:S}. Hence, we have:
\begin{equation}\label{eq:example.product.2}
\begin{split}
\mathrm{Error}_{N_1, \dots, N_k, n_1, n_2}  \ls &\: (1+ |\mathcal{P}^{[2,N_j]}(\upmu,\,L^i,\, \Psi)|) |\mathcal{P}^{\leq n_1} (\Vr,S)| \\
\ls &\:  |\mathcal{P}^{\leq N} (\Vr,S)|+ \epd |\mathcal{P}^{[2,N]}(\upmu,\,L^i,\, \Psi)|.
\end{split}
\end{equation}
\item When $n_2 > \Ntop -  \toprate-5$, we can argue similarly as above to see that 
$(1 + |\mathcal{P}^{N_{j}}(\upmu, L^i, \Psi)|)\ls 1$ for all $j$, and 
$|\mathcal{P}^{n_1}(\Vr, S)| \ls \epd$. Notice further that since $n_2 > \Ntop -  \toprate-5$, by \eqref{BA:W.Li.small} and Proposition~\ref{prop:geometric.low} we have:
$$|\mathcal{P}^{n_2}(\upmu \mathcal{P}\Psi)| \ls |\upmu \mathcal{P}^{[2,n_2+1]} \Psi| +|\mathcal{P}^{[2, n_2]}\Psi| 
+ |\mathcal{P}^{[2,n_2]} \upmu|. $$
Hence, we have:
\begin{equation}\label{eq:example.product.3}
\mathrm{Error}_{N_1, \dots, N_k, n_1, n_2} \ls \epd(\upmu|\mathcal{P}^{N+1} \Psi| + |\mathcal{P}^{[2,N]} \Psi| + | \mathcal{P}^{[1, N]} \bX \Psi| + |\mathcal{P}^{[2,n_2]} \upmu|).
\end{equation}
\end{enumerate}

Finally, it is easy to check that \eqref{eq:example.product.1}--\eqref{eq:example.product.3} are all bounded above by the RHS of \eqref{eq:pointwise.of.the.most.basic.transport.inho}.

\pfstep{Step~2: Proof of \eqref{eq:Vr.S.inho.weak}} To derive \eqref{eq:Vr.S.inho.weak}, we control each term in \eqref{eq:pointwise.of.the.most.basic.transport.inho} in the $L^2(\Mtu)$ norm.

We begin with the term $I$ in \eqref{eq:pointwise.of.the.most.basic.transport.inho}, which we 
estimate using the definition 
of the $\mathbb{V}_{\leq N}$ and $\mathbb{S}_{\leq N}$ energies
(see Section~\ref{SSS:DEFSOFENERGIES}):
\begin{equation}\label{eq:Vr.S.inho.I}
\begin{split}
&\: \|\mathcal{P}^{\leq N} (\Vr,S)\|^2_{L^2(\Mtu)} \ls \int_{u'=0}^{u'=u} [\mathbb{V}_{\leq N} + \mathbb{S}_{\leq N}](t,u')\, du'.
\end{split}
\end{equation}

We control term $II$ in \eqref{eq:pointwise.of.the.most.basic.transport.inho} by the $\mathbb{E}_{[1,N]}$ norm, and use the bootstrap assumptions \eqref{BA:W1}, \eqref{BA:W2}, the bound \eqref{eq:wave.energy.commuted}, and Proposition~\ref{prop:mus.int} to obtain:
\begin{equation}\label{eq:Vr.S.inho.II}
\begin{split}
&\: \epd^2\|\upmu\mathcal{P}^{[2,N+1]} \Psi\|^2_{L^2(\Mtu)} 
+ 
\epd^2 \|\mathcal{P}^{[1,N]} \bX \Psi \|^2_{L^2(\Mtu)} \\
\ls &\:  \epd^2 \int_{t'=0}^{t'=t} \left[\mathbb{E}_{[1,N]}(t') + \epd^2 \max\{1, \upmu_{\star}^{-2 \toprate+2\Ntop-2N+1.8}(t') \} \right] \, dt' \\
\ls &\: \epd^3 \max\{1, \upmu_{\star}^{-2 \toprate+2\Ntop-2N+2.8}(t) \}.
\end{split}
\end{equation}

Finally, for the term $III$, we use the control for $\mathbb K_{[1,N-1]}$ and $\mathbb{F}_{[1, N-1]}$ 
provided by the bootstrap assumptions 
\eqref{BA:W1} and \eqref{BA:W2}, the bounds in Proposition~\ref{prop:geometric.top}, and Proposition~\ref{prop:mus.int} to obtain:
\begin{equation}\label{eq:Vr.S.inho.III}
\begin{split}
&\: \epd^2\|\mathcal{P}^{[2,N]}(\upmu, L^i,\Psi) \|^2_{L^2(\Mtu)} \\
\ls &\: \epd^2 \mathbb K_{[1,N-1]}(t,u) 
+ 
\epd^2 \int_{u'=0}^{u'=u} \mathbb{F}_{[1,N-1]}(t,u') \, du'  
+ 
\epd^3 \int_{t'=0}^{t'=t} \max\{1,\upmu_{\star}^{-2 \toprate+2\Ntop-2N+2.8}(t') \} \, dt' \\
\ls &\: \epd^3 \max\{1, \upmu_{\star}^{-2 \toprate+2\Ntop-2N+3.8}(t) \}.
\end{split}
\end{equation}

Combining \eqref{eq:pointwise.of.the.most.basic.transport.inho} with 
\eqref{eq:Vr.S.inho.I}--\eqref{eq:Vr.S.inho.III}, we arrive at the desired bound
\eqref{eq:Vr.S.inho.weak}.

\pfstep{Step~3: Proof of \eqref{eq:Vr.S.inho.strong}} 
The estimate \eqref{eq:Vr.S.inho.strong} 
follows as a simple consequence of 
the already obtained bound \eqref{eq:Vr.S.inho.weak}
and the Cauchy--Schwarz inequality. \qedhere

\end{proof}

\subsection{Putting everything together}\label{sec:easy.transport.everything}

\begin{proposition}[\textbf{Estimates for the specific vorticity and entropy gradient}]\label{prop:V.S}
For $0\leq  N\leq \Ntop$, the following holds for all $t \in [0,\Tboot)\times [0,U_0]$:
$$\mathbb V_N(t,u) + \mathbb S_N(t,u) \ls \epd^3 \max\{1,\,\upmu_{\star}^{-2 \toprate+2\Ntop-2N+2.8}(t)\}.$$
\end{proposition}

\begin{proof}
Using Proposition~\ref{prop:transport.higher.order} 
for $\phi = \Vr^i,S^i$,
the initial data size assumptions in \eqref{assumption:very.small}, 
the bootstrap assumptions \eqref{BA:V}--\eqref{BA:S},
and the inhomogeneous term estimates in Proposition~\ref{prop:Vr.S.inho}
for the terms on RHSs of the transport equations 
\eqref{E:RENORMALIZEDVORTICTITYTRANSPORTEQUATION} and \eqref{E:GRADENTROPYTRANSPORT}, 
we deduce:
\begin{equation*}
\begin{split}
&\:\mathbb V_{\leq N}(t,u) + \mathbb S_{\leq N}(t,u)\\
\ls &\: \epd^3 \max\{1, \upmu_{\star}^{-2 \toprate+2\Ntop-2N+2.8}(t)\} + \int_{u'=0}^{u'=u} (\mathbb V_{\leq N}(t,u') + \mathbb S_{\leq N}(t,u')) \, du'.
\end{split}
\end{equation*}
The desired estimate now follows from applying Gr\"onwall's inequality in $u$. \qedhere
\end{proof}

\section{Lower-order transport estimates for the modified fluid variables}\label{sec:transport.harder}
We continue to work under the assumptions of Theorem~\ref{thm:bootstrap}.

In this section, we derive the energy estimates for the modified fluid variables $\mathcal{C}$ and $\mathcal{D}$ \textbf{except for the top-order}. (We will derive the top-order estimates in the next section.) Thanks to Proposition~\ref{prop:transport.higher.order}, 
to obtain the desired estimates,
it remains only for us to bound the inhomogeneous terms in the transport equations \eqref{E:EVOLUTIONEQUATIONFLATCURLRENORMALIZEDVORTICITY} and \eqref{E:TRANSPORTFLATDIVGRADENT}. Before we estimate the inhomogeneous terms, we will first control the $\bX$ derivative of $\Vr$ and $S$ in Section~\ref{sec:C.D.prelim}, and give general bounds for null forms in Section~\ref{sec:null.form}. (The 
null forms will also be useful later on,
in Section~\ref{sec:wave.main.est}.) We will combine these results to control the
inhomogeneous terms in Section~\ref{sec:C.D.inhomogeneous}. 
We provide the final estimate in Section~\ref{sec:C.D.everything}.

\subsection{Preliminaries}\label{sec:C.D.prelim}
A priori, the norms $\mathbb V_N$ and $\mathbb S_N$ do not control the $\bX$ derivatives of $\Vr$ or $S$. Nonetheless, we can obtain such control in terms of the norms $\mathbb V_N$ and $\mathbb S_N$ 
by using the transport equations \eqref{E:RENORMALIZEDVORTICTITYTRANSPORTEQUATION} and \eqref{E:GRADENTROPYTRANSPORT}.
\begin{proposition}[\textbf{$L^2$ control of the transversal derivatives of the $\Vr$ and $S$}]
\label{prop:transverse.from.transport}
For $1\leq N \leq \Ntop$, the following holds for all $(t,u) \in [0,\Tboot)\times [0,U_0]$:
$$\|\mathcal{P}^{N-1} \bX (\Vr,S) \|_{L^2(\Mtu)}^2 \ls \epd^3 \max\{1, \upmu_{\star}^{-2 \toprate + 2\Ntop - 2N + 2.8}(t)\}. $$
\end{proposition}
\begin{proof}
Recalling \eqref{E:TRANSPORTVECTORFIELDINTERMSOFLUNITANDRADUNIT}, we have:
\begin{equation}\label{eq:X.in.terms.of.transport}
\mathcal{P}^{N-1} \bX \Vr = \mathcal{P}^{N-1}(\upmu B \Vr) - \mathcal{P}^{N-1}( \upmu L \Vr), \quad \mathcal{P}^{N-1} \bX S = \mathcal{P}^{N-1}(\upmu B S) - \mathcal{P}^{N-1}(\upmu L S).
\end{equation}

The terms $\mathcal{P}^{N-1}(\upmu B \Vr)$ and $\mathcal{P}^{N-1}(\upmu B S)$ can be bounded
as follows
using \eqref{eq:Vr.S.inho.weak} and Proposition~\ref{prop:V.S}:
\begin{equation}\label{eq:X.in.terms.of.transport.est.1}
\| \mathcal{P}^{N-1}(\upmu B \Vr)\|_{L^2(\Mtu)}^2 + \| \mathcal{P}^{N-1}(\upmu B S)\|_{L^2(\Mtu)}^2 \ls \epd^3 \max\{ 1, \upmu_{\star}^{-2 \toprate + 2\Ntop -2N + 2.8}(t) \}.
\end{equation}
By \eqref{BA:V}, \eqref{BA:S}, 
Propositions~\ref{prop:geometric.low}, \ref{prop:almost.monotonicity},
\ref{prop:geometric.top}, and \ref{prop:V.S}, 
we have:
\begin{equation}\label{eq:X.in.terms.of.transport.est.2}
\begin{split}
&\: \|\mathcal{P}^{N-1}(\upmu L \Vr)\|_{L^2(\Mtu)}^2 + \|\mathcal{P}^{N-1}(\upmu L S)\|_{L^2(\Mtu)}^2 \\
\ls &\: \|\mathcal{P}^{\leq N}(\Vr,S)\|_{L^2(\Mtu)}^2 + \epd^2 \|\mathcal{P}^{[2, N-1]} \upmu \|_{L^2(\Mtu)}^2 \\
\ls &\: \int_0^u [\mathbb V_{\leq N} + \mathbb S_{\leq N}](t,u') \, du' + \epd^3 \int_{t'=0}^{t'=t} \max\{1, \upmu_{\star}^{-2 \toprate+2\Ntop-2N + 4.8}(t') \} \, dt' \\
\ls &\: \epd^3 \max\{ 1, \upmu_{\star}^{-2 \toprate + 2\Ntop -2N + 2.8}(t) \}.
\end{split}
\end{equation}

Therefore, combining \eqref{eq:X.in.terms.of.transport}--\eqref{eq:X.in.terms.of.transport.est.2}, we obtain the desired conclusion. \qedhere
\end{proof}

\subsection{General estimates for null forms}\label{sec:null.form}

\begin{lemma}[\textbf{Pointwise estimates for null forms}]
\label{lem:null.form}
Suppose 
\begin{enumerate}
\item $\mathcal Q(\rd\phi^{(1)},\rd\phi^{(2)})$ is a $g$-null form,
as in Definition~\ref{D:NULLFORMS}; \\
	and
\item $\phi^{(1)}$ and $\phi^{(2)}$ obey the following $L^\infty$ estimates for some $\mathfrak d^{(1,1)} \gtrsim \mathfrak d^{(1,2)}$, $\mathfrak d^{(2,1)} \gtrsim \mathfrak d^{(2,2)}$ for all $t\in [0,\Tboot)$:
\begin{equation}\label{eq:null.form.abstract.Li}
\begin{split}
\|\mathcal{P}^{\leq \Ntop - \toprate - 5} \bX \phi^{(1)} \|_{L^\infty(\Sigma_t)} \leq \mathfrak d^{(1,1)},\\
\|\mathcal{P}^{[1,\Ntop - \toprate-5]} \phi^{(1)}\|_{L^\infty(\Sigma_t)} \leq \mathfrak d^{(1,2)},\\
\|\mathcal{P}^{\leq \Ntop - \toprate - 5} \bX \phi^{(2)} \|_{L^\infty(\Sigma_t)} \leq \mathfrak d^{(2,1)},\\
\|\mathcal{P}^{[1,\Ntop - \toprate-5]} \phi^{(2)}\|_{L^\infty(\Sigma_t)} \leq \mathfrak d^{(2,2)}.
\end{split}
\end{equation}
\end{enumerate}

Then for any $0\leq N \leq \Ntop$, the following pointwise estimate holds on $[0,\Tboot)\times \Sigma$:
\begin{equation}\label{eq:null.form.basic}
\begin{split}
&\: |\mathcal{P}^N [\upmu\mathcal Q(\rd\phi^{(1)},\rd\phi^{(2)})] | \\
 \ls &\: \mathfrak d^{(2,1)} |\mathcal{P}^{[1,N+1]} \phi^{(1)} | + \mathfrak d^{(2,2)}  |\mathcal{P}^{[1,N]} \bX \phi^{(1)}| + \mathfrak d^{(1,1)} |\mathcal{P}^{[1,N+1]} \phi^{(2)}| + \mathfrak d^{(1,2)}  | \mathcal{P}^{[1,N]} \bX \phi^{(2)}| \\
 &\: + \max\{\mathfrak d^{(1,1)} \mathfrak d^{(2,2)}, \, \mathfrak d^{(1,2)} \mathfrak d^{(2,1)}\} |\mathcal{P}^{[2,N]}(\upmu, L^i,\Psi)|,
\end{split}
\end{equation}
and for any $1 \leq N \leq \Ntop$, we have:
\begin{equation}\label{eq:null.form.basic.1}
\begin{split}
&\: |\mathcal{P}^{[1,N]} [\upmu\mathcal Q(\rd\phi^{(1)},\rd\phi^{(2)})] |\\
 \ls &\: \mathfrak d^{(2,1)} |\mathcal{P}^{[2,N+1]} \phi^{(1)} | + \mathfrak d^{(2,2)} |\mathcal{P}^{[1,N]} \bX \phi^{(1)}| + \mathfrak d^{(1,1)} |\mathcal{P}^{[2,N+1]} \phi^{(2)}| + \mathfrak d^{(1,2)} | \mathcal{P}^{[1,N]} \bX \phi^{(2)}| \\
 &\: + \underbrace{\epd^{\f 12}  (\mathfrak d^{(2,1)}|\mathcal{P} \phi^{(1)} | + \mathfrak d^{(1,1)}|\mathcal{P} \phi^{(2)} |)}_{\doteq \mathscr A} + \underbrace{\mathfrak d^{(2,2)}|\mathcal{P} \phi^{(1)} | + \mathfrak d^{(1,2)}|\mathcal{P} \phi^{(2)} |}_{\doteq\mathscr B}\\
 &\: + \max\{\mathfrak d^{(1,1)} \mathfrak d^{(2,2)}, \, \mathfrak d^{(1,2)} \mathfrak d^{(2,1)}\} |\mathcal{P}^{[2,N]}(\upmu, L^i,\Psi)|.
\end{split}
\end{equation}
\end{lemma}
\begin{proof}
Throughout this proof, $\smoothfunction(\cdot)$ denotes a smooth function of its arguments
that is free to vary from line to line. By \eqref{E:CRUCIALPROPERTIESOFNULLFORMS},
we need to control:
$$\underbrace{\mathcal{P}^N [\smoothfunction(L^i, \Psi) \upmu (\mathcal{P}\phi^{(1)}) (\mathcal{P}\phi^{(2)})]}_{\doteq I},\quad \underbrace{\mathcal{P}^N [\smoothfunction(L^i, \Psi) (\mathcal{P}\phi^{(1)}) (\bX\phi^{(2)})]}_{\doteq II},\quad  \underbrace{\mathcal{P}^N [\smoothfunction(L^i, \Psi) (\bX \phi^{(1)}) (\mathcal{P}\phi^{(2)})]}_{\doteq III}.$$
We first prove \eqref{eq:null.form.basic}. Consider term $II$. Arguing as in the proof of \eqref{eq:pointwise.of.the.most.basic.transport.inho} and then using \eqref{eq:null.form.abstract.Li}, we obtain:
\begin{equation*}
\begin{split}
&\: |\mathcal{P}^N [\smoothfunction(L^i, \Psi,\upmu) (\mathcal{P}\phi^{(1)}) (\bX\phi^{(2)})]| 
	\\
\ls &\: |\mathcal{P}^{[1, \Ntop - \toprate-5]}  \phi^{(1)}| |\mathcal{P}^{[1,N]} \bX\phi^{(2)}| + |\mathcal{P}^{[1,N+1]}  \phi^{(1)}| |\mathcal{P}^{\leq \Ntop- \toprate - 5} \bX\phi^{(2)}| \\
&\: + |\mathcal{P}^{\leq \Ntop - \toprate-5}  \phi^{(1)}| |\mathcal{P}^{\leq \Ntop - \toprate - 5} \bX\phi^{(2)}| |\mathcal{P}^{[2, N]} (\upmu, L^i, \Psi)| \\
\ls &\: \mathfrak d^{(1,2)} |\mathcal{P}^{[1,N]} \bX\phi^{(2)}| + \mathfrak d^{(2,1)} |\mathcal{P}^{[1,N+1]}  \phi^{(1)}| + \mathfrak d^{(1,2)} \mathfrak d^{(2,1)}|\mathcal{P}^{[2, N]} (\upmu, L^i, \Psi)|,
\end{split}
\end{equation*}
which is bounded from above by the RHS of \eqref{eq:null.form.basic}.

Next, we observe that the term $III$ can be
handled just like term
$II$, after we interchange the roles of $\phi^{(1)}$ and $\phi^{(2)}$. Moreover, the
term $I$ is even easier to handle
because $\mathfrak d^{(1,1)} \gtrsim \mathfrak d^{(1,2)}$, $\mathfrak d^{(2,1)} \gtrsim \mathfrak d^{(2,2)}$. 

We finally turn to the proof of \eqref{eq:null.form.basic.1}, in which we need 
to show an improvement compared to \eqref{eq:null.form.basic}
using the fact that on the LHS of the estimate, the $\upmu$-weighted null form 
is differentiated by \emph{at least one} $\mathcal{P}$. More precisely, we need to improve $\mathfrak d^{(1,1)} |\mathcal{P}^{[1,N+1]} \phi^{(2)}|$, $\mathfrak d^{(1,2)}  | \mathcal{P}^{[1, N]} \bX \phi^{(2)}|$ to $\mathfrak d^{(1,1)} |\mathcal{P}^{[2,N+1]} \phi^{(2)}|$, $\mathfrak d^{(1,2)}  | \mathcal{P}^{[2, N]} \bX \phi^{(2)}|$, at the expense of 
incurring terms of the type $\mathscr A$ and $\mathscr B$ in \eqref{eq:null.form.basic.1}. 

It is straightforward to use the arguments given in the previous paragraph to confirm that if $N\geq 2$, then $\mathfrak d^{(1,1)} |\mathcal{P}^{[1,N+1]} \phi^{(2)}|$, $\mathfrak d^{(1,2)}  | \mathcal{P}^{[1, N]} \bX \phi^{(2)}|$  
on RHS~\eqref{eq:null.form.basic} 
can be replaced by $\mathfrak d^{(1,1)} |\mathcal{P}^{[2,N+1]} \phi^{(2)}|$, $\mathfrak d^{(1,2)}  | \mathcal{P}^{[2, N]} \bX \phi^{(2)}|$. We are thus only concerned with the following terms in the case when $N=1$:
$$\underbrace{[\mathcal{P} (\smoothfunction(L^i, \Psi) \upmu)] (\mathcal{P}\phi^{(1)}) (\mathcal{P}\phi^{(2)})}_{\doteq I'},\quad \underbrace{[\mathcal{P} \smoothfunction(L^i, \Psi)] (\mathcal{P}\phi^{(1)}) (\bX\phi^{(2)})}_{\doteq II'},\quad  \underbrace{[\mathcal{P} \smoothfunction(L^i, \Psi)] (\bX \phi^{(1)}) (\mathcal{P}\phi^{(2)})}_{\doteq III'}.$$
{Next, we observe that for the terms 
$II'$ and $III'$, when the $\mathcal{P}$ derivative falls on $\smoothfunction(L^i, \Psi)$, 
\eqref{BA:W.Li.small} and Proposition~\ref{prop:geometric.low} yield
a smallness factor of $\epd^{\f 12}$. Thus,} $II'$ and $III'$ can be bounded by $\mathscr A$. 
Finally, {to handle} the term $I'${,} we can {control} either 
$\mathcal{P}\phi^{(1)}$ or $\mathcal{P}\phi^{(2)}$ in $L^\infty${, 
which allows us to bound $I'$ by} $\mathscr B$. 
\qedhere
\end{proof}

\subsection{Estimates of the inhomogeneous terms in the transport equations for \protect$\mathcal{C}$ and \protect$\mathcal{D}$}\label{sec:C.D.inhomogeneous}

\begin{proposition}[\textbf{Below-top-order estimates for the main inhomogeneous terms 
	in the transport equations for the modified fluid variables}]
\label{prop:transport.inhom.1}
For\footnote{Note that in the case $N=\Ntop$, the error terms on the RHS involving $\mathbb V_{\leq N+1}$ and $\mathbb S_{\leq N+1}$ have \emph{not} been estimated in Section~\ref{sec:transport.easy}. It is for this reason that we only consider $0\leq N \leq \Ntop-1$ at this point.} $0\leq N \leq \Ntop-1$, the main terms $\mathfrak M \in \{\mathfrak M^i_{(\mathcal{C})},\,\mathfrak M_{(\mathcal{D})}\}$ 
 (see \eqref{eq:def.M.C}--\eqref{eq:def.M.D})
can be estimated as follows for every $(t,u) \in [0,\Tboot) \times [0,U_0]$:
\begin{equation} \label{E:transport.inhom.1}
\begin{split}
\|\mathcal{P}^N (\upmu \mathfrak M) \|_{L^2(\mathcal M_{t,u})}^2
\ls &\:  \epd^3 \max\{1, \upmu_{\star}^{-2 \toprate+2\Ntop-2N+0.8}(t)\}.
\end{split}
\end{equation}
\end{proposition}
\begin{proof}
Note that $\mathfrak M^i_{(\mathcal{C})}$ consists of null forms 
(see Definition~\ref{D:NULLFORMS}) 
$\mathcal Q(\rd \Psi,\rd\Vr)$, $\mathcal Q(\rd \Psi,\rd S)$. Therefore, by Lemma~\ref{lem:null.form} 
(with $\phi^{(1)} = \Vr^i, S^i$; $\phi^{(2)} = \Psi$; 
$\mathfrak d^{(1,1)}= \mathfrak d^{(1,2)} \doteq \epd$; 
$\mathfrak d^{(2,2)} \doteq \epd^{1/2}$;
and
$\mathfrak d^{(2,1)} = \mathcal{O}(1)$ 
by virtue of the bootstrap assumptions \eqref{BA:LARGERIEMANNINVARIANTLARGE}--\eqref{BA:S}),\footnote{Note that by Lemma~\ref{lem:null.form}, 
there is also a term $\epd |\mathcal{P} \Psi|$, which we bound by $\lesssim \epd^{\f 32}$ using \eqref{BA:W.Li.small}.} 
we have:
\begin{equation}\label{eq:transport.C.D.inho.pointwise}
\begin{split}
 |\mathcal{P}^N (\upmu \mathfrak M^i_{(\mathcal{C})}) | 
 \ls &\: \epd^{\f 32} + \underbrace{|\mathcal{P}^{\leq N+1} (\Vr,S)| 
+
|\mathcal{P}^{\leq N} \bX(\Vr,S)|}_{\doteq I} 
+ 
\underbrace{\epd(|\mathcal{P}^{[2,N+1]} \Psi| + | \mathcal{P}^{[1,N]} \bX \Psi|)}_{\doteq II} \\
 &\: + \underbrace{\epd |\mathcal{P}^{[2,N]}(\upmu, L^i, \Psi)|}_{\doteq III}.
\end{split}
\end{equation}
We recall the expression for $\mathfrak M_{(\DivofEntrenormalized)}$
given by \eqref{eq:def.M.D}. The term 
$2 \exp(-2 \Densrenormalized) 
			\left\lbrace
				(\partial_a v^a) \partial_b \GradEnt^b -
				(\partial_a v^b) \partial_b \GradEnt^a
			\right\rbrace$ is a null form of type $\mathcal Q(\rd \Psi,\rd S)$. Thus,
			using the same arguments we gave when handling $\mathfrak M^i_{(\mathcal{C})}$,
			we can pointwise bound
			its $\mathcal{P}^N(\upmu \cdot)$ derivatives
			by RHS~\eqref{eq:transport.C.D.inho.pointwise}. 
			Moreover, using the same arguments given below 
			\eqref{eq:pointwise.of.the.most.basic.transport.inho},
			we see that the $\mathcal{P}^N$ derivatives of the term
			$\upmu \exp(-\Densrenormalized) \delta_{ab} (\Flatcurl \Vortrenormalized)^a \GradEnt^b$ 
			can be pointwise bounded by RHS~\eqref{eq:transport.C.D.inho.pointwise}.
			From now on, 	
			it therefore suffices to consider the terms on the RHS of \eqref{eq:transport.C.D.inho.pointwise}.

The term $I$ can be controlled using Propositions~\ref{prop:V.S} and \ref{prop:transverse.from.transport} so that:
\begin{equation}\label{eq:transport.C.D.inho.top}
\begin{split}
&\: \|\mathcal{P}^{\leq N+1} (\Vr,S)\|_{L^2(\Mtu)}^2 + \|\mathcal{P}^{\leq N} \bX(\Vr,S)\|_{L^2(\Mtu)}^2 \ls \epd^3 \max\{1, \upmu_{\star}^{-2 \toprate + 2\Ntop - 2N + 0.8}(t)\}.
\end{split}
\end{equation}

For the term $II$ in \eqref{eq:transport.C.D.inho.pointwise}, we use the bootstrap assumptions 
\eqref{BA:W1}, \eqref{BA:W2}, and \eqref{BA:W.Li.small} and the estimates of 
Propositions~\ref{prop:almost.monotonicity} and \ref{prop:commutators} 
to obtain:
\begin{equation}
\begin{split}
&\: \epd^2 \|\mathcal{P}^{[2,N+1]} \Psi\|_{L^2(\Mtu)}^2 
+ 
\epd^2
\|  \mathcal{P}^{[1,N]} \bX \Psi \|_{L^2(\Mtu)}^2 \\
\ls &\: \epd^2 \mathbb K_{[1,N]}(t,u) + \epd^2 \int_{u'=0}^{u'=u} \mathbb{F}_{[1,N]}(t,u') \, du' + \epd^2 \int_{t'=0}^{t'=t} \mathbb{E}_{[1,N]}(t',u)\, dt' \\
& \ \  + \epd^3 \max\{1, \upmu_{\star}^{-2 \toprate + 2\Ntop - 2N + 1.8}(t) \}
	\\
\ls &\: \epd^3 \max\{1, \upmu_{\star}^{-2 \toprate + 2\Ntop - 2N + 1.8}(t) \}.
\end{split}
\end{equation}

The term $III$ in \eqref{eq:transport.C.D.inho.pointwise} is the same as the term $III$ in \eqref{eq:pointwise.of.the.most.basic.transport.inho}, and can be bounded as in 
the proof of
Proposition~\ref{prop:Vr.S.inho}, which, when combined with Propositions~\ref{prop:V.S},
implies that it is bounded by $\lesssim \epd^3 \max\{1, \upmu_{\star}^{-2 \toprate + 2\Ntop - 2N + 1.8}(t) \}$.

Combining the above estimates, we conclude the desired estimate \eqref{E:transport.inhom.1}.
\qedhere


\end{proof}

\begin{proposition}[\textbf{$L^2$ control of some null forms in the modified fluid variable transport equations}]
\label{prop:transport.inhom.2}
For $0\leq N \leq \Ntop$, the terms $\mathfrak Q \in \{ \mathfrak Q_{(\mathcal{C})}^i,\,\mathfrak Q_{(\mathcal{D})}\}$ 
(see \eqref{E:RENORMALIZEDVORTICITYCURLNULLFORM}--\eqref{E:DIVENTROPYGRADIENTNULLFORM})
can be estimated as follows for all $(t,u)\in [0,\Tboot)\times [0,U_0]$:
\begin{equation}
\begin{split}
&\: \|\mathcal{P}^N (\upmu \mathfrak Q)\|_{L^2(\mathcal M_{t,u})}^2 \ls \epd^3 \max\{1, \upmu_{\star}^{-2 \toprate+2\Ntop-2N+0.8}(t)\}.
\end{split}
\end{equation}
\end{proposition}
\begin{proof}
The $\mathfrak Q$ terms can all be expressed as $S$ multiplied by a null form $\mathcal Q(\rd\Psi, \rd\Psi)$. We control the null form using \eqref{eq:null.form.basic} with $\mathfrak d^{(1,1)},\,\mathfrak d^{(1,2)} ,\,\mathfrak d^{(2,1)},\,\mathfrak d^{(2,2)} \ls 1$ (justified by \eqref{BA:LARGERIEMANNINVARIANTLARGE}--\eqref{BA:W.Li.small}) 
so that:
\begin{equation}\label{eq:transport.C.D.inho.null.pointwise}
\begin{split}
&\:  |\mathcal{P}^N (\upmu \mathfrak Q) | \\
\ls &\: \sum_{N_1+N_2\leq N}|\mathcal{P}^{\leq N_1} (\Vr,S)| ( |\mathcal{P}^{[1,N_2+1]} \Psi| 
+ 
|\mathcal{P}^{[1,N_2]}\bX \Psi| +  |\mathcal{P}^{[2,N_2]}(\upmu, L^i, \Psi)|)\\
 \ls &\: \underbrace{|\mathcal{P}^{\leq N} (\Vr,S)|}_{\doteq I} + \underbrace{\epd(|\mathcal{P}^{[2,N+1]} \Psi| + 
|\mathcal{P}^{[1, N]} \bX\Psi|)}_{\doteq II}   + \underbrace{\epd |\mathcal{P}^{[2,N]}(\upmu, L^i, \Psi)|}_{\doteq III},
\end{split}
\end{equation}
where in the last line, we used the $L^\infty$ estimates \eqref{BA:V}, \eqref{BA:S} for 
$(\Vr,S)$ if $N_1 \leq \Ntop - \toprate-5$, 
and otherwise, we used the $L^\infty$ estimates  
\eqref{BA:LARGERIEMANNINVARIANTLARGE}--\eqref{BA:W.Li.small} and
Proposition~\ref{prop:geometric.low} for $\Psi$, $\upmu$, and $L^i$.

Next, we observe that the terms $II$ and $III$ are exactly the same as $II$ and $III$ in \eqref{eq:transport.C.D.inho.pointwise} in Proposition~\ref{prop:transport.inhom.1}. We can therefore argue exactly as in Proposition~\ref{prop:transport.inhom.1} to show that these terms in $\|\cdot \|_{L^2(\Mtu)}^2$ are bounded above by $\epd^3 \max\{1, \upmu_{\star}^{-2 \toprate + 2\Ntop - 2N + 1.8}(t) \}$. Notice in particular that while Proposition~\ref{prop:transport.inhom.1} was only stated for $0\leq N\leq \Ntop -1$, the bounds for these two terms in fact also hold (and can be proved in the same way) for $N = \Ntop$.

It thus remains to consider the term $I$ in \eqref{eq:transport.C.D.inho.null.pointwise}. Importantly, notice that term $I$ in \eqref{eq:transport.C.D.inho.null.pointwise} is \emph{better} than the corresponding term $I$ in \eqref{eq:transport.C.D.inho.pointwise} because it has up to $N$, as opposed to $N+1$ derivatives. We control this term using the definition of $\mathbb V_{\leq N}$, $\mathbb S_{\leq N}$ and Proposition~\ref{prop:V.S} as follows:
\begin{equation*}
\begin{split}
&\: \| \mathcal{P}^{\leq N} (\Vr,S) \|_{L^2(\Mtu)}^2 \ls \int_{u'=0}^{u'=u} [\mathbb V_{\leq N} + \mathbb S_{\leq N}](t,u') \, du' \ls \epd^3 \max\{1, \upmu_{\star}^{-2 \toprate+2\Ntop-2N+2.8}(t)\}.
\end{split}
\end{equation*}

Combining the above estimates, we conclude the proposition. \qedhere

\end{proof}

\begin{proposition}[\textbf{$L^2$ control of some easy terms in the transport equation for $\mathcal{C}$}]
\label{prop:transport.inhom.3}
For $0\leq N\leq \Ntop$, the 
term $\mathfrak L^i_{(\mathcal{C})}$ (see \eqref{E:RENORMALIZEDVORTICITYCURLLINEARORBETTER})
can be estimated as follows for all $(t,u)\in [0,\Tboot)\times [0,U_0]$:
\begin{equation*}
\|\mathcal{P}^N (\upmu \mathfrak L^i_{(\mathcal{C})}) \|_{L^2(\mathcal M_{t,u})} \ls \epd^3 \max \{1, \upmu_{\star}^{-2 \toprate+2\Ntop-2N+0.8}(t)\}.
\end{equation*}
\end{proposition}
\begin{proof}
We begin with the following pointwise estimate:
\begin{equation*}
\begin{split}
|\mathcal{P}^{\leq N}(\upmu\mathfrak L^i_{(\mathcal{C})})| \ls &\: \epd |\mathcal{P}^{\leq N} (\Vr,S)| + \epd^2 (|\mathcal{P}^{[2,N+1]} \Psi| + | \mathcal{P}^{[1,N]} \bX\Psi|)+ \epd^2|\mathcal{P}^{[2,N]}(\upmu,L^i)|,
\end{split}
\end{equation*}
which
can be derived by using the same arguments we used to obtain
\eqref{eq:pointwise.of.the.most.basic.transport.inho}. Notice that all the above terms can be bounded above by the RHS of \eqref{eq:transport.C.D.inho.null.pointwise}. They can therefore be 
bounded in the norm $\| \cdot \|_{L^2(\mathcal M_{t,u})}$
via exactly the same arguments we used in the proof of Proposition~\ref{prop:transport.inhom.2}.
This yields the desired conclusion.
 \qedhere
\end{proof}

\subsection{Below top-order estimates for $\mathbb C$ and $\mathbb D$}\label{sec:C.D.everything}

\begin{proposition}[\textbf{Below top-order estimates for the modified fluid variables}]
\label{prop:C.D}
For $0\leq  N \leq \Ntop-1$, the following holds for $(t,u)\in [0,\Tboot)\times [0,U_0]$:
$$\mathbb C_N(t,u) + \mathbb D_N(t,u) \ls \epd^3 \max\{1,\upmu_{\star}^{-2 \toprate+2\Ntop-2N+0.8}(t)\}. $$

\end{proposition}
\begin{proof}
This follows from combining Proposition~\ref{prop:transport.higher.order} 
{for $\phi = \mathcal{C}^i,\mathcal{D}^i$}
with the initial data 
{size assumptions in \eqref{assumption:very.small.modified}, 
the bootstrap assumptions \eqref{BA:C.D},}
and the inhomogeneous term estimates (in Propositions~\ref{prop:transport.inhom.1}--\ref{prop:transport.inhom.3}) 
for the terms on the RHSs of the transport equations \eqref{E:EVOLUTIONEQUATIONFLATCURLRENORMALIZEDVORTICITY} and
\eqref{E:TRANSPORTFLATDIVGRADENT}.
\qedhere
\end{proof}

\section{Top-order transport and elliptic estimates for the specific vorticity and the entropy gradient}
\label{sec:transport.hard}
We continue to work under the assumptions of Theorem~\ref{thm:bootstrap}.

In this section, we derive top-order estimates for the modified fluid variables
$\mathcal{C}$ and $\mathcal{D}$. The key difference with the lower-order estimates 
(which we derived in Proposition~\ref{prop:C.D})
is that we can\underline{not} bound the top-order
derivatives of $\Vr$ and $S$ using the $\mathbb V$ and $\mathbb S$ norms; that approach
would lead to a loss of a derivative, which is not permissible at the top-order. 
To avoid losing a derivative, we rely on the following additional ingredient: \emph{weighted} elliptic estimates 
for the specific vorticity and entropy gradient
(recall Sections~\ref{sec:intro.elliptic}, \ref{sec:intro.transport.top.order}).

In Section~\ref{sec:tran.C.D.top}, we derive top-order transport estimates. 
The estimates are similar to the ones we derived in Section~\ref{sec:transport.harder}, 
except there are some top-order inhomogeneous terms. We derive the elliptic estimates in Sections~\ref{sec:elliptic.general} and \ref{sec:elliptic.top.order}. For the final estimate, see Section~\ref{sec:elliptic.everything}.

In our analysis, we rely on elliptic estimates relative to the 
\underline{Cartesian} spatial coordinates. In deriving these estimates,
we will use the ``Cartesian pointwise norms'' from the following definition.
\begin{definition}\label{def:spatial.derivatives}
Denote by $\rdb$ the gradient with respect to the Cartesian spatial coordinates.
For a scalar function $f$ and a one-form $\phi$, 
define respectively:
$$|\rdb f|^2 \doteq \sum_{i=1}^3 |\rd_i f|^2,\quad |\rdb \phi|^2 \doteq \sum_{i,j=1}^3 |\rd_i\phi_j|^2.$$
\end{definition}

\subsection{Top-order transport estimates for $\protect\mathbb C_{\protect\Ntop}$ and $\protect\mathbb D_{\protect\Ntop}$}\label{sec:tran.C.D.top}

\begin{proposition}[\textbf{Preliminary top-order $L^2$ estimates for the modified fluid variables}]
\label{prop:C.D.transport.main}
Let $\varsigma \in (0,1]$.
There exists a constant $C > 0$ \textbf{independent of $\varsigma$}
and a constant
$\mathfrak{c}_{\varsigma} > 0$ (depending on $\varsigma$) such that whenever 
$\mathfrak{c} \geq \mathfrak{c}_{\varsigma}$, the following estimate holds
for every $(t,u) \in [0,\Tboot)\times [0,U_0]$ (with $u'$ denoting the $u$-value of the integrand):
\begin{equation} \label{E:C.D.transport.main}
\begin{split}
&\: \| e^{-\f{\mathfrak{c} u'}2} \sqrt{\upmu} 
	\mathcal{P}^{\Ntop} (\mathcal{C},\mathcal{D})\|_{L^2(\Sigma_{t}^{u})}^2 
+ 
\| e^{-\f{\mathfrak{c} u'}2}  \mathcal{P}^{\Ntop} (\mathcal{C},\mathcal{D}) \|_{L^2(\mathcal F_{u}^{t})}^2 + \f{\mathfrak{c}}{2} \| e^{-\f{\mathfrak{c} u'}2} \mathcal{P}^{\Ntop} (\mathcal{C},\mathcal{D}) \|_{L^2(\Mtu)}^2 \\
\leq &\: C \epd^3 \upmu_{\star}^{-2 \toprate+0.8}(t) + \varsigma \int_{t'=0}^{t'=t} \f{1}{\upmu_{\star}(t')}\|e^{-\f{\mathfrak{c}u'}{2}} \sqrt \upmu \rdb \mathcal{P}^{\Ntop} (\Vr,S)\|_{L^2(\Sigma_{t'}^u)}^2 \,dt'.
\end{split}
\end{equation}
\end{proposition}
\begin{proof}
Let $\varsigma',\,\mathfrak{c}>0$ be constants to be specified later.
It is crucial that all explicit constants $C > 0$ and
implicit constants in this proof are \underline{independent} of $\varsigma'$ and $\mathfrak{c}$.
At the end of the proof, there will be a large constant $C$ such that we will choose $\varsigma'$ to satisfy
$\varsigma = C \varsigma'$, where $\varsigma > 0$ is the constant from the statement of the proposition.

\pfstep{Step~1: Transport estimate in the weighted norms} Since $\upmu B u =1$ by \eqref{E:LUNITANDRADOFUANDT}, \eqref{E:TRANSPORTVECTORFIELDINTERMSOFLUNITANDRADUNIT}, we have:
\begin{align}\label{eq:top.C.transport.weighted}
\upmu B (e^{-\f{\mathfrak{c} u}2} \mathcal{P}^{\Ntop} \mathcal{C}) 
	& = - \f {\mathfrak{c}} 2 e^{-\f{\mathfrak{c} u}2} \mathcal{P}^{\Ntop} \mathcal{C} + e^{-\f{\mathfrak{c} u}2} \upmu B (\mathcal{P}^{\Ntop} \mathcal{C}),
	\\
\upmu B (e^{-\f{\mathfrak{c} u}2} \mathcal{P}^{\Ntop} \mathcal{D}) 
& = - \f{\mathfrak{c}}2 e^{-\f{\mathfrak{c} u}2} \mathcal{P}^{\Ntop} \mathcal{D} + e^{-\f{\mathfrak{c} u}2} \upmu B (\mathcal{P}^{\Ntop} \mathcal{D}).
	\label{eq:top.D.transport.weighted}
\end{align}

Starting with \eqref{eq:top.C.transport.weighted} and \eqref{eq:top.D.transport.weighted}, we now argue using the identity \eqref{eq:transport.basic} with $\phi \doteq (\mathcal{C}^i, \mathcal{D}^i)$, 
except now, 
unlike in the proof of Proposition~\ref{prop:transport}, we do not use Gr\"onwall's inequality but instead take advantage of the good terms associated with the 
terms
$- \f {\mathfrak{c}} 2 e^{-\f{\mathfrak{c} u}2} \mathcal{P}^{\Ntop} \mathcal{C}$ and $- \f {\mathfrak{c}} 2 e^{-\f{\mathfrak{c} u}2} \mathcal{P}^{\Ntop} \mathcal{D}$
on RHSs~\eqref{eq:top.C.transport.weighted}--\eqref{eq:top.D.transport.weighted}. We thus obtain, for any $\varsigma' >0$ (here, $u'$ denotes the $u$-value of the integrand):
\begin{equation}\label{eq:top.C.D.estimate}
\begin{split}
&\: \| e^{-\f{\mathfrak{c} u'}2}  \sqrt{\upmu} \mathcal{P}^{\Ntop} \mathcal{C} \|_{L^2(\Sigma_{t}^{u})}^2 
+ 
\| e^{-\f{\mathfrak{c} u'}2}  \mathcal{P}^{\Ntop} \mathcal{C} \|_{L^2(\mathcal F_{u}^{t})}^2 + \mathfrak{c} \| e^{-\f{\mathfrak{c} u'}2} \mathcal{P}^{\Ntop} \mathcal{C} \|_{L^2(\Mtu)}^2 \\
&\: + \| e^{-\f{\mathfrak{c} u'}2}  \sqrt{\upmu} \mathcal{P}^{\Ntop} \mathcal{D} \|_{L^2(\Sigma_{t}^{u})}^2 + \| e^{-\f{\mathfrak{c} u'}2}  \mathcal{P}^{\Ntop} \mathcal{D} \|_{L^2(\mathcal F_{u}^{t})}^2 + \mathfrak{c} \| e^{-\f{\mathfrak{c} u'}2} \mathcal{P}^{\Ntop} \mathcal{D} \|_{L^2(\Mtu)}^2 \\
\ls  &\:  \| e^{-\f{\mathfrak{c}u'}{2}}  \sqrt{\upmu} \mathcal{P}^{\Ntop} (\mathcal{C}, \mathcal{D}) \|_{L^2(\Sigma_0^u)}^2 + \|e^{-\f{\mathfrak{c}u'}{2}} \mathcal{P}^{\Ntop} (\mathcal{C},\mathcal{D}) \|_{L^2(\Mtu)}^2 
	\\
&\: +  \|e^{-\f{\mathfrak{c}u'}{2}} \mathcal{P}^{\Ntop} (\mathcal{C},\mathcal{D}) \|_{L^2(\Mtu)} 
\|e^{-\f{\mathfrak{c}u'}{2}} \upmu B \mathcal{P}^{\Ntop} (\mathcal{C},\mathcal{D})\|_{L^2(\Mtu)} 
	\\
\ls &\: \| e^{-\f{\mathfrak{c}u'}{2}}  \sqrt{\upmu} \mathcal{P}^{\Ntop} (\mathcal{C}, \mathcal{D}) \|_{L^2(\Sigma_0^u)}^2 + (1+(\varsigma')^{-1}) \|e^{-\f{\mathfrak{c}u'}{2}} \mathcal{P}^{\Ntop} (\mathcal{C},\mathcal{D}) \|_{L^2(\Mtu)}^2 \\
&\: + \varsigma' \|e^{-\f{\mathfrak{c}u'}{2}} \upmu B \mathcal{P}^{\Ntop} (\mathcal{C},\mathcal{D})\|_{L^2(\Mtu)}^2. 
\end{split}
\end{equation}

\pfstep{Step~2: Estimating the easy terms} We now consider the
terms on the RHS of \eqref{eq:top.C.D.estimate}. First, the assumptions 
\eqref{assumption:very.small.modified}
on the initial data 
and the simple bound $\|\upmu\|_{L^\infty(\Sigma_0)} \ls 1$ from Proposition~\ref{prop:geometric.low}
give:
\begin{equation}\label{eq:top.C.D.data}
\| e^{-\f{\mathfrak{c}u'}{2}}  \sqrt{\upmu} 
\mathcal{P}^{\Ntop} (\mathcal{C}, \mathcal{D}) \|_{L^2(\Sigma_0^u)}^2  \ls \epd^3.
\end{equation}

Recalling the transport equations \eqref{E:EVOLUTIONEQUATIONFLATCURLRENORMALIZEDVORTICITY},
\eqref{E:TRANSPORTFLATDIVGRADENT}, we notice that the terms 
$\|e^{-\f{\mathfrak{c}u'}{2}} \upmu B \mathcal{P}^{\Ntop} \mathcal{C}\|_{L^2(\Mtu)}$ and $\|e^{-\f{\mathfrak{c}u'}{2}} \upmu B \mathcal{P}^{\Ntop} \mathcal{D}\|_{L^2(\Mtu)}$ have essentially 
been estimated in 
Propositions~\ref{prop:transport.inhom.1}--\ref{prop:transport.inhom.3}
(using $e^{-\f{\mathfrak{c}u'}{2}} \leq 1$). Crucially, however, unlike in Proposition~\ref{prop:transport.inhom.1}, 
we have not yet bounded the following terms in \eqref{eq:transport.C.D.inho.top}:
$$\|e^{-\f{\mathfrak{c}u'}{2}} \mathcal{P}^{\Ntop+1} (\Vr,S)\|_{L^2(\mathcal M_{t,u})}^2  
+ \|e^{-\f{\mathfrak{c}u'}{2}} \mathcal{P}^{\Ntop} \bX(\Vr, S)\|_{L^2(\Mtu)}^2$$
(since this is one more derivative than $\mathbb V_{\Ntop}$ and $\mathbb S_{\Ntop}$ control). In other words, simply repeating the argument in Propositions~\ref{prop:transport.inhom.1}--\ref{prop:transport.inhom.3}
and separating the error terms that depend on $\Ntop+1$ derivatives of $(\Vr,S)$,
we obtain:
\begin{equation}\label{eq:top.C.D.uncontrollable}
\begin{split}
&\: \|e^{-\f{\mathfrak{c}u'}{2}} \upmu B \mathcal{P}^{\Ntop} \mathcal{C}\|_{L^2(\Mtu)}^2 
+ 
\|e^{-\f{\mathfrak{c}u'}{2}} \upmu B \mathcal{P}^{\Ntop} \mathcal{D}\|_{L^2(\Mtu)}^2 \\
\ls &\: \epd^3 \upmu_{\star}^{-2 \toprate+0.8}(t) 
+ 
\|e^{-\f{\mathfrak{c}u'}{2}} \mathcal{P}^{ \Ntop+1} (\Vr,S)\|_{L^2(\mathcal M_{t,u})}^2  
+ 
\|e^{-\f{\mathfrak{c}u'}{2}}\mathcal{P}^{ \Ntop} \bX(\Vr, S)\|_{L^2(\Mtu)}^2.
\end{split}
\end{equation}

\pfstep{Step~3: Controlling the top-order terms} We now consider the terms on the RHS of \eqref{eq:top.C.D.uncontrollable}. First, using the commutator estimates \eqref{eq:commutator.L2.main}, 
Proposition~\ref{prop:V.S},
and 
the bootstrap assumptions \eqref{BA:V}--\eqref{BA:S} 
to control $[\mathcal{P}^{ \Ntop},\bX](\Vr,S)$ 
(and using $e^{-\f{\mathfrak{c}u'}{2}}\leq 1$), we see that:
\begin{equation}
\begin{split}
&\: \|e^{-\f{\mathfrak{c}u'}{2}} \mathcal{P}^{ \Ntop+1} (\Vr,S)\|_{L^2(\mathcal M_{t,u})}^2  
+ \|e^{-\f{\mathfrak{c}u'}{2}}\mathcal{P}^{ \Ntop} \bX(\Vr, S)\|_{L^2(\Mtu)}^2 \\
\ls &\: \epd^3 \upmu_{\star}^{-2 \toprate+2.8}(t)  + \|e^{-\f{\mathfrak{c}u'}{2}} \bX \mathcal{P}^{\Ntop} (\Vr, S)\|_{L^2(\Mtu)}^2  + \|e^{-\f{\mathfrak{c}u'}{2}} L \mathcal{P}^{\Ntop} (\Vr, S)\|_{L^2(\Mtu)}^2 \\
&\: + \|e^{-\f{\mathfrak{c}u'}{2}} Y \mathcal{P}^{\Ntop} (\Vr, S)\|_{L^2(\Mtu)}^2+ \|e^{-\f{\mathfrak{c}u'}{2}} Z \mathcal{P}^{\Ntop} (\Vr, S)\|_{L^2(\Mtu)}^2 \\
\ls &\: \epd^3  \upmu_{\star}^{-2 \toprate+2.8}(t)  
+ \|e^{-\f{\mathfrak{c}u'}{2}} B \mathcal{P}^{\Ntop} (\Vr, S)\|_{L^2(\Mtu)}^2 \\
&\: + \|e^{-\f{\mathfrak{c}u'}{2}} \rdb \mathcal{P}^{\Ntop} (\Vr, S)\|_{L^2(\Mtu)}^2,
\end{split}
\end{equation}
where we have replaced $L \mathcal{P}^{\Ntop} (\Vr, S) = B \mathcal{P}^{\Ntop} (\Vr, S) - X \mathcal{P}^{\Ntop} (\Vr, S)$ (by \eqref{E:TRANSPORTVECTORFIELDINTERMSOFLUNITANDRADUNIT}) and also used 
Lemmas~\ref{lem:slashed} and~\ref{L:GEOMETRICCOORDINATEVECTORFIELDSINTERMSOFCARTESIANVECTORFIELDS}
to express $(X,Y,Z)$ in terms of the Cartesian coordinate spatial partial derivative vectorfields,
and Propositions~\ref{prop:geometric.low} and \ref{prop:geometric.low.2} to bound the coefficients in the expressions
by $\lesssim 1$. 
Moreover, using the commutator identity
$B \mathcal{P}^{\Ntop} (\Vr, S)
=
\upmu^{-1}
\mathcal{P}^{\Ntop} [\upmu B (\Vr, S)]
+
\upmu^{-1}
[\upmu B,\mathcal{P}^{\Ntop}](\Vr, S)
$,
the commutator estimates of Proposition~\ref{prop:commutators} with $\phi \doteq (\Vr^i, S^i)$,
the bootstrap assumptions \eqref{BA:V}--\eqref{BA:S},
Proposition~\ref{prop:V.S},
the estimate \eqref{eq:Vr.S.inho.weak},
and Proposition~\ref{prop:mus.int},
we deduce (also using $e^{-\f{\mathfrak{c}u'}{2}} \leq 1$) that
$\|e^{-\f{\mathfrak{c}u'}{2}} B \mathcal{P}^{\Ntop} (\Vr, S)\|_{L^2(\Mtu)}^2
\lesssim
\epd^3 \upmu_{\star}^{-2 \toprate+0.8}(t)
$.
Combining the above results, we deduce:
\begin{equation}\label{eq:top.C.D.controlling.the.uncontrollable}
\begin{split}
&\:  \|e^{-\f{\mathfrak{c}u'}{2}} \mathcal{P}^{ \Ntop + 1} (\Vr,S)\|_{L^2(\mathcal M_{t,u})}^2  
+ \|e^{-\f{\mathfrak{c}u'}{2}}\mathcal{P}^{ \Ntop} \bX(\Vr, S)\|_{L^2(\Mtu)}^2 \\
\ls &\: \epd^3 \upmu_{\star}^{-2 \toprate + 0.8}(t) + \int_{t'=0}^{t'=t} \f{1}{\upmu_{\star}(t')}\|e^{-\f{\mathfrak{c}u'}{2}} \sqrt \upmu \rdb \mathcal{P}^{\Ntop} (\Vr,S)\|_{L^2(\Sigma_{t'}^u)}^2 \,dt'.
\end{split}
\end{equation}

\pfstep{Step~4: Putting everything together} 
Using
\eqref{eq:top.C.D.data}, \eqref{eq:top.C.D.uncontrollable} and \eqref{eq:top.C.D.controlling.the.uncontrollable} 
to control the terms on RHS~\eqref{eq:top.C.D.estimate},
we see that there is a $C>0$ such that:
\begin{equation}\label{eq:top.C.D.estimate.almost.final}
\begin{split}
&\: \| e^{-\f{\mathfrak{c} u'}2}  \sqrt{\upmu} \mathcal{P}^{\Ntop} \mathcal{C} \|_{L^2(\Sigma_{t}^{u})}^2 + \| e^{-\f{\mathfrak{c} u'}2}  \mathcal{P}^{\Ntop} \mathcal{C} \|_{L^2(\mathcal F_{u}^{t})}^2 + \mathfrak{c} \| e^{-\f{\mathfrak{c} u'}2} \mathcal{P}^{\Ntop} \mathcal{C} \|_{L^2(\Mtu)}^2 \\
&\: + \| e^{-\f{\mathfrak{c} u'}2}  \sqrt{\upmu} \mathcal{P}^{\Ntop} \mathcal{D} \|_{L^2(\Sigma_{t}^{u})}^2 + \| e^{-\f{\mathfrak{c} u'}2}  \mathcal{P}^{\Ntop} \mathcal{D} \|_{L^2(\mathcal F_{u}^{t})}^2 + \mathfrak{c} \| e^{-\f{\mathfrak{c} u'}2} \mathcal{P}^{\Ntop} \mathcal{D} \|_{L^2(\Mtu)}^2 \\
\leq &\: C (1 + \varsigma') \epd^3 \upmu_{\star}^{-2 \toprate+0.8}(t) + C(1+(\varsigma')^{-1}) \|e^{-\f{\mathfrak{c}u'}{2}} \mathcal{P}^{\Ntop} (\mathcal{C},\mathcal{D}) \|_{L^2(\Mtu)}^2 \\
&\: + C \varsigma' \int_{t'=0}^{t'=t} \f{1}{\upmu_{\star}(t')}\|e^{-\f{\mathfrak{c}u'}{2}} \sqrt \upmu \rdb \mathcal{P}^{\Ntop} (\Vr,S)\|_{L^2(\Sigma_{t'}^u)}^2 \,dt'.
\end{split}
\end{equation}

Finally, relabeling the coefficients 
$C \varsigma'$
on RHS~\eqref{eq:top.C.D.estimate.almost.final} 
by setting $\varsigma \doteq C \varsigma'$,
bounding the data term $C (1 + \varsigma') \epd^3 \upmu_{\star}^{-2 \toprate+0.8}(t)$
by a new constant ``$C$'' times $\epd^3 \upmu_{\star}^{-2 \toprate+0.8}(t)$ via the assumption
$\varsigma \in (0,1]$,
taking $\mathfrak{c}_{\varsigma}$ sufficiently large (depending on $\varsigma$) so that: 
$$C(1+(\varsigma')^{-1}) \|e^{-\f{\mathfrak{c}u'}{2}} \mathcal{P}^{\Ntop} (\mathcal{C},\mathcal{D}) \|_{L^2(\Mtu)}^2 \leq \f{\mathfrak{c}_{\varsigma}}{2} \int_{\Mtu} e^{-\mathfrak{c} u'} [|\mathcal{P}^{\Ntop} \mathcal{C}|^2 
+ |\mathcal{P}^{\Ntop} \mathcal{D}|^2 ] \, d\vol,$$
now allowing $\mathfrak{c}$ to be any constant such that $\mathfrak{c} \geq \mathfrak{c}_{\varsigma}$,
and subtracting 
$\f{\mathfrak{c}}{2} 
\int_{\Mtu} e^{-\mathfrak{c} u'} [|\mathcal{P}^{\Ntop} \mathcal{C}|^2 + |\mathcal{P}^{\Ntop} \mathcal{D}|^2 ] \, d\vol$ from both sides of \eqref{eq:top.C.D.estimate.almost.final}, 
we obtain the desired 
inequality \eqref{E:C.D.transport.main}. \qedhere
\end{proof}

\subsection{General elliptic estimates on $\mathbb R\times \mathbb T^2$}\label{sec:elliptic.general}
We begin with a standard weighted \emph{Euclidean} elliptic estimate on $\mathbb R\times \mathbb T^2$ in Proposition~\ref{prop:elliptic.Euclidean}. We then apply this in our geometric setting for general one-forms in Proposition~\ref{prop:elliptic}.

\begin{proposition}[\textbf{Weighted Euclidean elliptic estimates}]
\label{prop:elliptic.Euclidean}
Let $w:\mathbb R\times \mathbb T^2 \to \mathbb R_{>0}$ be a smooth, strictly positive, bounded weight function.

The following inequality holds for all 
one-forms $\phi = \phi_a dx^a \in C^2_c(\mathbb R \times \mathbb T^2)$:
\begin{equation*}
\begin{split}
&\: \|\sqrt{w} \rdb \phi \|_{L^2(\mathbb R\times \mathbb T^2,dx)}^2 \\
\leq &\: 4 \|\sqrt w \mathrm{curl}\, \phi\|_{L^2(\mathbb R\times \mathbb T^2, dx)}^2 + 4 \|\sqrt w \mathrm{div}\, \phi\|_{L^2(\mathbb R\times \mathbb T^2, dx)}^2 + 3 \|\rdb \log w\|_{L^\infty(\mathbb R\times \mathbb T^2)}^2 \|\sqrt w \phi\|_{L^2(\mathbb R\times \mathbb T^2, dx)}^2,
\end{split}
\end{equation*}
where $\rdb$ is as in Definition~\ref{def:spatial.derivatives}, $\| \upxi \|_{L^2(\mathbb R\times \mathbb T^2,dx)}^2 
\doteq \int_{\mathbb R\times \mathbb T^2} |\upxi|_{e}^2 \, dx$ for tensorfields $\upxi$,
$|\upxi|_e$ denotes the standard Euclidean pointwise norm of $\upxi$,
and $dx = dx^1\, dx^2\, dx^3$.
\end{proposition}
\begin{proof}
Integrating by parts and using H\"older's inequality, we find that:
\begin{equation}\label{eq:Cartesian.elliptic.1}
\begin{split}
&\: \|\sqrt w \rdb \phi\|_{L^2(\mathbb R\times \mathbb T^2, dx)}^2 = \sum_{i,j=1}^3 \int_{\mathbb R\times \mathbb T^2} w (\rd_i\phi_j)^2 \,dx \\
= &\: - \sum_{i,j=1}^3 \left\lbrace \int_{\mathbb R\times \mathbb T^2} w \phi_j (\rd_{ii}^2\phi_j) \, dx 
+
 \int_{\mathbb R\times \mathbb T^2} (\rd_i w) \phi_j (\rd_i\phi_j) \, dx \right\rbrace \\
= & - \sum_{i,j=1}^3\int_{\mathbb R\times \mathbb T^2} w \phi_j \rd_{ij}^2\phi_i \, dx + \sum_{i,j=1}^3\int_{\mathbb R\times \mathbb T^2} w \phi_j \rd_i(\rd_j\phi_i - \rd_i\phi_j) \, dx - \sum_{i,j=1}^3 \int_{\mathbb R\times \mathbb T^2} (\rd_i w) \phi_j (\rd_i\phi_j) \, dx\\
= &\: \sum_{i,j=1}^3\int_{\mathbb R\times \mathbb T^2} w (\rd_j\phi_j )(\rd_{i}\phi_i) \, dx - \sum_{i,j=1}^3\int_{\mathbb R\times \mathbb T^2} w (\rd_i\phi_j )(\rd_j\phi_i - \rd_i\phi_j) \, dx \\
&\: + \sum_{i,j=1}^3 \int_{\mathbb R\times \mathbb T^2} (\rd_j w) \phi_j (\rd_i\phi_i) \, dx - \sum_{i,j=1}^3 \int_{\mathbb R\times \mathbb T^2} (\rd_i w) \phi_j (\rd_j \phi_i - \rd_i\phi_j) \, dx- \sum_{i,j=1}^3 \int_{\mathbb R\times \mathbb T^2} (\rd_i w) \phi_j (\rd_i\phi_j) \, dx\\
\leq &\: \|\sqrt w \,\mathrm{div}\,\phi\|_{L^2(\mathbb R\times \mathbb T^2, dx)}^2 + \|\sqrt w \, \rdb\phi\|_{L^2(\mathbb R\times \mathbb T^2, dx)} \|\sqrt w \mathrm{curl}\, \phi\|_{L^2(\mathbb R\times \mathbb T^2, dx)} \\
&\: +  \|\rdb \log w\|_{L^\infty( \mathbb R\times \mathbb T^2)} \|\sqrt w \phi \|_{L^2(\mathbb R\times \mathbb T^2, dx)} 
	\left\lbrace\|\sqrt w \,\mathrm{div}\,\phi\|_{L^2(\mathbb R\times \mathbb T^2, dx)} +  \|\sqrt w \mathrm{curl}\, \phi\|_{L^2(\mathbb R\times \mathbb T^2, dx)}	\right\rbrace \\
&\: +  \|\rdb \log w\|_{L^\infty( \mathbb R\times \mathbb T^2)} \|\sqrt w \phi \|_{L^2(\mathbb R\times \mathbb T^2, dx)} \|\sqrt w \, \rdb\phi\|_{L^2(\mathbb R\times \mathbb T^2, dx)}.
\end{split}
\end{equation}

Using $|ab| \leq \f{a^2} 4+ b^2$, we find that:
\begin{equation}\label{eq:Cartesian.elliptic.2}
\begin{split}
&\: \|\sqrt w \rdb \phi\|_{L^2(\mathbb R\times \mathbb T^2, dx)}^2 \\
\leq &\: \f 12 \|\sqrt w \, \rdb\phi\|_{L^2(\mathbb R\times \mathbb T^2, dx)}^2 + 2 \|\sqrt w \,\mathrm{div}\,\phi\|_{L^2(\mathbb R\times \mathbb T^2, dx)}^2 \\
&\: + 2 \|\sqrt w \mathrm{curl}\, \phi\|_{L^2(\mathbb R\times \mathbb T^2, dx)}^2 
+ 
\f 32 \|\rdb \log w\|_{L^\infty( \mathbb R\times \mathbb T^2)}^2 \|\sqrt w \phi\|_{L^2(\mathbb R\times \mathbb T^2,dx)}^2.
\end{split}
\end{equation}

The conclusion of the lemma follows from subtracting 
$\f 12 \|\sqrt w \, \rdb\phi\|_{L^2(\mathbb R\times \mathbb T^2, dx)}^2$ 
from both sides of \eqref{eq:Cartesian.elliptic.2}. \qedhere

\end{proof}

\begin{proposition}[\textbf{Euclidean elliptic estimates with $u$-weights}]
\label{prop:elliptic}
Let $\phi = \phi_a dx^a$ be a smooth compactly supported one-form on $\Sigma_t$. Then for each $\mathfrak{c}>0$ and each $t\in [0,\Tboot)$, the following elliptic estimate holds,
where the implicit constants are \underline{independent} of $\mathfrak{c}$:
$$\|e^{-\f{\mathfrak{c} u}2}\sqrt\upmu \, \rdb\phi\|_{L^2(\Sigma_t)}\ls \|e^{-\f{\mathfrak{c} u}2} \sqrt\upmu \,\mathrm{div}\,\phi\|_{L^2(\Sigma_t)} + \|e^{-\f{\mathfrak{c} u}2} \sqrt\upmu \,\mathrm{curl}\,\phi\|_{L^2(\Sigma_t)} + \mathfrak{c} \upmu_{\star}^{-1}(t) \| e^{-\f{\mathfrak{c} u}2} \sqrt \upmu \, \phi\|_{L^2(\Sigma_t)}.$$
\end{proposition}
\begin{proof}
In this proof, the implicit constants in $\ls$ are \underline{independent} of $\mathfrak{c}$.

We apply Proposition~\ref{prop:elliptic.Euclidean} with $w \doteq e^{- \mathfrak{c} u}$. 
By 
Lemma~\ref{lem:Cart.to.geo}, 
\eqref{E:LUNITANDRADOFUANDT}, 
and Proposition~\ref{prop:geometric.low.2},
we have: $\|\rdb \log w\|_{L^\infty( \mathbb R\times \mathbb T^2)} 
\ls \mathfrak{c} \upmu_{\star}^{-1}(t)$. Hence,
$$\| e^{-\f{\mathfrak{c} u}2} \rdb\phi\|_{L^2(\Sigma_t,dx)}
\ls 
\| e^{-\f{\mathfrak{c} u}2}  \,\mathrm{div}\,\phi\|_{L^2(\Sigma_t,dx)} 
+ 
\| e^{-\f{\mathfrak{c} u}2} \,\mathrm{curl}\,\phi\|_{L^2(\Sigma_t,dx)} 
+ 
\mathfrak{c} \upmu_{\star}^{-1}(t) \|  e^{-\f{\mathfrak{c} u}2} \, \phi\|_{L^2(\Sigma_t,dx)}.$$

The conclusion thus follows from the fact
that the volume measures $\upmu dx$ and $d\tvol$ are comparable, 
which in turn follows from 
\eqref{eq:volume.form.Sigma}
and Proposition~\ref{prop:geometric.low.2}.
\qedhere
\end{proof}

\subsection{Top-order elliptic estimates for $\Vr$ and $S$}\label{sec:elliptic.top.order}
In this section, we derive top-order elliptic estimates for $\Vr$ and $\GradEnt$.

There are four main steps. 
Ultimately, our goal is to exploit the preliminary 
energy inequality for $(\mathcal{P}^{\Ntop} \mathcal{C}, \mathcal{P}^{\Ntop} \mathcal{D})$
that we derived in in Proposition~\ref{prop:C.D.transport.main}, 
and to do this, we have to control the integrand term 
$\|e^{-\f{\mathfrak{c}u'}{2}} \sqrt \upmu \rdb \mathcal{P}^{\Ntop} (\Vr,S)\|_{L^2(\Sigma_{t'}^u)}^2$ 
on RHS~\eqref{E:C.D.transport.main} with the help of elliptic estimates. 
To achieve this, we first
commute the top-order operators
$\mathcal{P}^{\Ntop}$ through the Euclidean operators $\mathrm{div}$ and $\mathrm{curl}$.
To avoid uncontrollable commutator terms, we introduce a $\upmu$ weight into the commutators.
In the second step, we have to control
$(\mathrm{div} \mathcal{P}^{\Ntop} \Vr,\mathrm{div} \mathcal{P}^{\Ntop} \GradEnt)$
and
$(\mathrm{curl} \mathcal{P}^{\Ntop} \Vr,\mathrm{curl} \mathcal{P}^{\Ntop} \GradEnt)$
in terms of the modified fluid variables
$(\mathcal{P}^{\Ntop} \mathcal{C}, \mathcal{P}^{\Ntop} \mathcal{D})$ 
from \eqref{E:RENORMALIZEDCURLOFSPECIFICVORTICITY}--\eqref{E:RENORMALIZEDDIVOFENTROPY}
plus simpler error terms.
The first and second steps are carried out in Lemmas~\ref{lem:curlVr}--\ref{lem:divS}.

Next, in Proposition~\ref{prop:C.D.elliptic.final},
we use the weighted elliptic estimates
on $\Sigma_t$ provided by Proposition~\ref{prop:elliptic}
and the results of the first two steps 
to obtain:
$
\|e^{-\f{\mathfrak{c} u}2} \sqrt{\upmu} ( \rdb \mathcal{P}^{\Ntop}\Vr, \rdb \mathcal{P}^{\Ntop} \GradEnt) \|_{L^2(\Sigma_t)}^2
\lesssim 
\|e^{-\f{\mathfrak{c} u}2}\sqrt\upmu  (\mathcal{P}^{\Ntop}\mathcal{C},\mathcal{P}^{\Ntop}\mathcal{D})\|_{L^2(\Sigma_t)}^2
+
\cdots
$, 
where ``$\cdots$'' denotes simpler error terms for which we already have an independent bound.
Finally, in Proposition~\ref{prop:elliptic.putting.everything.together},
we combine all of these results to obtain our main $L^2$ estimate\footnote{We clarify that although the estimate for $(\mathcal{P}^{\Ntop} \mathcal{C},\mathcal{P}^{\Ntop} \mathcal{D})$
and the aforementioned estimates
$
\|e^{-\f{\mathfrak{c} u}2}\sqrt{\upmu}  
(\rdb \mathcal{P}^{\Ntop} \Vr,\rdb \mathcal{P}^{\Ntop}  \GradEnt) \|_{L^2(\Sigma_t)}^2
\lesssim 
\|e^{-\f{\mathfrak{c} u}2}\sqrt\upmu  (\mathcal{P}^{\Ntop} \mathcal{C},\mathcal{P}^{\Ntop} \mathcal{D})\|_{L^2(\Sigma_t)}^2
+
\cdots
$
together imply a top-order $L^2$ estimate for 
$(\rdb \mathcal{P}^{\Ntop} \Vr,\rdb \mathcal{P}^{\Ntop} \GradEnt)$,
we do not explicitly state such an estimate in the paper because we do not need it for our main results. 
} for
$(\mathcal{P}^{\Ntop} \mathcal{C},\mathcal{P}^{\Ntop} \mathcal{D})$.

\subsubsection{Controlling $\mathrm{curl}\,\mathcal{P}^{\Ntop} \Vr$ and $\mathrm{div}\,\mathcal{P}^{\Ntop} \Vr$}
We start with a simple commutation lemma.

\begin{lemma}[\textbf{Commuting geometric vectorfields with $\upmu$-weighted Cartesian vectorfields}]
\label{lem:commute.with.Cartesian}
Let $\phi$ be a smooth function such that:
$$\|\mathcal{P}^{\leq \Ntop- \toprate-5} \phi\|_{L^\infty(\Sigma_t)} \leq \epd, \quad  
\| \mathcal{P}^{\leq \Ntop- \toprate - 5} \bX\phi\|_{L^\infty(\Sigma_t)} \leq \epd$$ 
for all $t \in [0,\Tboot)$.

Then, for $0\leq N \leq \Ntop$, the following holds in $\mathcal M_{\Tboot,U_0}$:
$$| [\upmu \rd_i, \mathcal{P}^N]\phi |\ls |\mathcal{P}^{[1,N]} \phi| + |\mathcal{P}^{\leq N-1} \bX \phi| +  \epd ( | \mathcal{P}^{[2,N]} (\upmu, L^i,\Psi) | +  |\mathcal{P}^{[2,N-1]} \bX\Psi |).$$
\end{lemma}
\begin{proof}
We first use Lemma~\ref{lem:Cart.to.geo} to express $\upmu\rd_i$ in terms of the geometric vectorfields and then argue as in Proposition~\ref{prop:commutators.Li}. \qedhere
\end{proof}

\begin{lemma}[\textbf{$L^2$ estimates for the Euclidean curl of the derivatives of 
$\Vr$ in terms of the derivatives of $\mathcal{C}$}]
\label{lem:curlVr}
Let $\mathfrak{c} \geq 0$ be a real number.
The following estimate holds for all $t\in [0,\Tboot)$,
where the implicit constants are \underline{independent} of $\mathfrak{c}$:
$$\|e^{-\f{\mathfrak{c}u'}2} \sqrt\upmu \, \mathrm{curl}\, \mathcal{P}^{\Ntop} \Vr \|^2_{L^2(\Sigma_t)} \ls \epd^3 \upmu_{\star}^{-2 \toprate + 0.8}(t) + \|e^{-\f{\mathfrak{c}u'}2} \sqrt\upmu \,\mathcal{P}^{\Ntop} \mathcal{C}\|_{L^2(\Sigma_t)}^2.$$
\end{lemma}

\begin{proof}
We first compute the commutator $[\upmu \curl, \mathcal{P}^{\Ntop}]$ using Lemma~\ref{lem:commute.with.Cartesian} and the bootstrap assumption \eqref{BA:V}:
\begin{equation}\label{eq:PN.mu.curl.Vr.commutator}
\begin{split}
&\: | [\upmu \,\mathrm{curl}, \mathcal{P}^{\Ntop}] \Vr |\\
\ls &\: | \mathcal{P}^{\leq \Ntop} \Vr| + | \mathcal{P}^{\leq {\Ntop}-1} \bX \Vr| 
+
\epd \left( | \mathcal{P}^{[2,\Ntop]} (\upmu,L^i,\Psi) | + |\mathcal{P}^{[2,\Ntop-1]}\bX\Psi | \right).
\end{split}
\end{equation}
On the other hand, by \eqref{E:RENORMALIZEDCURLOFSPECIFICVORTICITY},
Lemma~\ref{lem:Cart.to.geo},
the bootstrap assumptions \eqref{BA:LARGERIEMANNINVARIANTLARGE}--\eqref{BA:C.D}, 
and Propositions~\ref{prop:geometric.low} and 
\ref{prop:geometric.low.2}, we have:
\begin{equation}\label{eq:PN.mu.curl.Vr.main.computation}
\begin{split}
&\:| \mathcal{P}^{\Ntop} (\upmu \,\mathrm{curl}\,\Vr)| \\
= &\: \left|\mathcal{P}^{\Ntop} \left\{ \upmu \left[ \exp(\rr) \mathcal{C} - \exp(-2\rr) c_s^{-2} \f{p_{;s}}{\bar{\varrho}} \GradEnt^a \rd_a v + \exp(-2\rr) c_s^{-2} \f{p_{;s}}{\bar{\varrho}} (\rd_av^a)S \right] \right\} \right| \\
\ls &\: \upmu |\mathcal{P}^{\Ntop} \mathcal{C}| + |\mathcal{P}^{\leq {\Ntop}-1}\mathcal{C}| + |\mathcal{P}^{\leq \Ntop} S| \\
&\: + \epd \left(|\mathcal{P}^{[2, \Ntop]} (\upmu, L^i)| + \upmu |\mathcal{P}^{\Ntop+1} \Psi| + |\mathcal{P}^{[2, \Ntop]} \Psi| + |\mathcal{P}^{[1, \Ntop]} \bX \Psi| \right).
\end{split}
\end{equation}
We stress that on RHS~\eqref{eq:PN.mu.curl.Vr.main.computation}, it is important that the top-order 
terms $\mathcal{P}^{\Ntop} \mathcal{C}$ and $\mathcal{P}^{\Ntop+1} \Psi$ 
are accompanied by a factor of $\upmu$. 

We can therefore use \eqref{eq:PN.mu.curl.Vr.commutator} and \eqref{eq:PN.mu.curl.Vr.main.computation} (to write $\upmu \mathrm{curl} \mathcal{P}^{\Ntop} \Vr = [\upmu \,\mathrm{curl}, \mathcal{P}^{\Ntop}] \Vr + \mathcal{P}^{\Ntop} (\upmu \,\mathrm{curl}\,\Vr)$), multiply by $e^{-\f{\mathfrak{c} u}{2}} \upmu^{-\f 12}$, take the $L^2(\Sigma_t)$ norm, and then use $e^{-\f{\mathfrak{c} u}{2}}\leq 1$ to obtain:
\begin{equation}\label{eq:top.order.elliptic.error.example}
\begin{split}
&\: \|e^{-\f{\mathfrak{c} u'}{2}} \sqrt\upmu\, \mathrm{curl} \, (\mathcal{P}^{\Ntop} \Vr)\|_{L^2(\Sigma_t)} \\
\ls &\: \|e^{-\f{\mathfrak{c} u'}{2}} \sqrt\upmu \mathcal{P}^{\Ntop} \mathcal{C}\|_{L^2(\Sigma_t)} + \upmu_{\star}^{-1}(t) \|\sqrt\upmu \mathcal{P}^{\leq \Ntop-1} \mathcal{C} \|_{L^2(\Sigma_t)} +  \upmu_{\star}^{-1}(t) \|\sqrt\upmu \mathcal{P}^{\leq \Ntop} (\Vr,S) \|_{L^2(\Sigma_t)} \\
&\: +  \upmu_{\star}^{-1}(t) \|\sqrt\upmu  \mathcal{P}^{\leq \Ntop-1} \bX \Vr \|_{L^2(\Sigma_t)} +\epd \upmu_{\star}^{-\f 12}(t) \| \mathcal{P}^{[2,\Ntop]} (\upmu, L^i) \|_{L^2(\Sigma_t)}  \\
&\: + \epd ( \|\sqrt\upmu \mathcal{P}^{\Ntop+1} \Psi\|_{L^2(\Sigma_t)} + \upmu_{\star}^{-\f 12}(t) \|\mathcal{P}^{[1,\Ntop]} \bX \Psi\|_{L^2(\Sigma_t)} + \upmu_{\star}^{-1}(t) \|\sqrt\upmu \mathcal{P}^{[2,\Ntop]} \Psi\|_{L^2(\Sigma_t)}) \\
\ls &\: \|e^{-\f{\mathfrak{c} u'}{2}} \sqrt\upmu \mathcal{P}^{\Ntop} \mathcal{C}\|_{L^2(\Sigma_t)} + \epd^{\f 32} \upmu_{\star}^{- \toprate+0.4}(t),
\end{split}
\end{equation}
where we have used Proposition~\ref{prop:C.D} to bound $\upmu_{\star}^{-1}(t) \|\sqrt\upmu \mathcal{P}^{\leq \Ntop-1} \mathcal{C} \|_{L^2(\Sigma_t)}$; Proposition~\ref{prop:V.S} to bound  $\upmu_{\star}^{-1}(t) \|\sqrt\upmu \mathcal{P}^{\leq \Ntop} (\Vr,S) \|_{L^2(\Sigma_t)}$; Proposition~\ref{prop:transverse.from.transport} to bound $\upmu_{\star}^{-1}(t) \|\sqrt\upmu  \mathcal{P}^{\leq \Ntop-1} \bX \Vr \|_{L^2(\Sigma_t)}$; Proposition~\ref{prop:geometric.top} to bound $\epd \upmu_{\star}^{-\f 12}(t) \| \mathcal{P}^{[2,\Ntop]} (\upmu, L^i) \|_{L^2(\Sigma_t)}$; and the bootstrap assumptions \eqref{BA:W1}, \eqref{BA:W2}, and \eqref{eq:wave.energy.commuted} to estimate all the remaining terms. (We remark that the worst terms are $\upmu_{\star}^{-1}(t) \|\sqrt\upmu \mathcal{P}^{\leq \Ntop-1} \mathcal{C} \|_{L^2(\Sigma_t)}$, $\upmu_{\star}^{-1}(t) \|\sqrt\upmu \mathcal{P}^{\leq \Ntop} (\Vr,S) \|_{L^2(\Sigma_t)}$, $\upmu_{\star}^{-1}(t) \|\sqrt\upmu  \mathcal{P}^{\leq \Ntop-1} \bX \Vr \|_{L^2(\Sigma_t)}$, and $\upmu_{\star}^{-\f 12}(t) \| \mathcal{P}^{[1,\Ntop]} \bX \Psi\|_{L^2(\Sigma_t)}$, 
which determine the blowup-exponent $-\toprate +0.4$ for $\upmu_{\star}$ on RHS~\eqref{eq:top.order.elliptic.error.example}). Squaring \eqref{eq:top.order.elliptic.error.example}, we arrive at the desired result.
 \qedhere
\end{proof}

\begin{lemma}[\textbf{$L^2$ estimates for the Euclidean divergence of the derivatives of $\Vr$}]
\label{lem:divVr}
Let $\mathfrak{c} \geq 0$ be a real number.
The following estimate holds for all $t\in [0,\Tboot)$,
where the implicit constant is \underline{independent} of $\mathfrak{c}$:
$$\|e^{-\f{\mathfrak{c} u'}2}\sqrt\upmu  \mathrm{div}\, \mathcal{P}^{\Ntop} \Vr \|^2_{L^2(\Sigma_t)} \ls \epd^3 \upmu_{\star}^{-2 \toprate + 0.8}(t).$$
\end{lemma}
\begin{proof}
The commutator $[\upmu \mathrm{div}, \mathcal{P}^{\Ntop}]  \Vr$ can be computed 
in an identical manner as \eqref{eq:PN.mu.curl.Vr.commutator}. Thus, we have:
\begin{equation}\label{eq:divVr.1}
\begin{split}
 |[\upmu \mathrm{div}, \mathcal{P}^{\Ntop}] \Vr | \ls \mbox{RHS of \eqref{eq:PN.mu.curl.Vr.commutator}}.
\end{split}
\end{equation}
We also use Lemma~\ref{lem:Cart.to.geo},
the fact that the Cartesian component functions $X^1,X^2,X^3$ are smooth functions of
	the $L^i$ and $\Psi$ (see \eqref{E:TRANSPORTVECTORFIELDINTERMSOFLUNITANDRADUNIT}),
\eqref{E:FLATDIVOFRENORMALIZEDVORTICITY}, and the $L^\infty$ bounds 
in \eqref{BA:LARGERIEMANNINVARIANTLARGE}--\eqref{BA:V}
and Proposition~\ref{prop:geometric.low}
to deduce:
\begin{equation}\label{eq:divVr.2}
\begin{split}
&\:| \mathcal{P}^{\Ntop} (\upmu \,\mathrm{div}\,\Vr)|  =| \mathcal{P}^{\Ntop} (\upmu \, \Vr^a \rd_a \rr)| \\
\ls &\: |\mathcal{P}^{\leq \Ntop} \Vr| + \epd (|\mathcal{P}^{[2,\Ntop]} \upmu| + \upmu|\mathcal{P}^{\Ntop+1} \Psi | + |\mathcal{P}^{[2,\Ntop]} \Psi| + |\mathcal{P}^{[1,\Ntop]} \bX \Psi| ). 
\end{split}
\end{equation}

Notice that every term on the RHS of \eqref{eq:divVr.2} has already appeared on the RHSs of \eqref{eq:PN.mu.curl.Vr.commutator} and \eqref{eq:PN.mu.curl.Vr.main.computation}. Hence, 
with the help of the simple identity 
$ \upmu \mathrm{div}\, \mathcal{P}^{\Ntop} \Vr
=
\mathcal{P}^{\Ntop} (\upmu \,\mathrm{div}\,\Vr)
+ [\upmu \mathrm{div}, \mathcal{P}^{\Ntop}] \Vr$
and the estimates obtained above,
we can argue exactly as in Lemma~\ref{lem:curlVr} to obtain the same estimate. (Note that here there are no $\mathcal{C}$ terms and so we do not have the term $\|e^{-\f{\mathfrak{c}u'}2} \sqrt\upmu \,\mathcal{P}^{\Ntop} \mathcal{C}\|_{L^2(\Sigma_t)}^2$.) \qedhere

\end{proof}

\subsubsection{Controlling $\mathrm{curl} \,\mathcal{P}^{\Ntop} S$ and $\mathrm{div}\,\mathcal{P}^{\Ntop} S$}

\begin{lemma}[\textbf{$L^2$ estimates for the Euclidean curl of the derivatives of $S$}]
\label{lem:curlS}
Let $\mathfrak{c} \geq 0$ be a real number.
The following estimate holds for all $t\in [0,\Tboot)$,
where the implicit constant is \underline{independent} of $\mathfrak{c}$:
$$\|e^{-\f{\mathfrak{c} u'}2} \sqrt\upmu \mathrm{curl} \mathcal{P}^{\Ntop} S\|_{L^2(\Sigma_t)}^2 \ls \epd^3 \upmu_{\star}^{-2 \toprate+0.8}(t).$$
\end{lemma}
\begin{proof}
By \eqref{E:CURLGRADENT}, $\mathrm{curl} S = 0$. Hence, using Lemma~\ref{lem:commute.with.Cartesian} and the bootstrap assumption \eqref{BA:S}, 
\begin{equation}\label{eq:PN.mu.curl.S.commutator}
\begin{split}
&\:  |\upmu\,\mathrm{curl}\, \mathcal{P}^{\Ntop} S |= | [\upmu \,\mathrm{curl}, \mathcal{P}^{\Ntop}] S |\\
\ls &\: | \mathcal{P}^{\leq {\Ntop} } S| + | \mathcal{P}^{\leq {\Ntop}-1} \bX S| + \epd ( | \mathcal{P}^{[2,\Ntop]} (\upmu,L^i,\Psi) | + |\mathcal{P}^{[2,\Ntop-1]}\bX\Psi |).
\end{split}
\end{equation}

The only new terms here compared to \eqref{eq:PN.mu.curl.Vr.commutator} and \eqref{eq:PN.mu.curl.Vr.main.computation} 
are 
$| \mathcal{P}^{\leq {\Ntop} } S|$
and
$|\mathcal{P}^{\leq {\Ntop}-1} \bX S|$, 
which can be handled using 
Propositions~\ref{prop:V.S} and \ref{prop:transverse.from.transport} 
in the same way that we handled the corresponding terms
$\upmu_{\star}^{-1} \|\sqrt\upmu \mathcal{P}^{\leq \Ntop} \Vr \|_{L^2(\Sigma_t)}$
and
$\upmu_{\star}^{-1} \|\sqrt\upmu \mathcal{P}^{\leq \Ntop-1} \bX \Vr \|_{L^2(\Sigma_t)}$ in
the proof of 
Lemma~\ref{lem:curlVr}. \qedhere 
\end{proof}

\begin{lemma}[\textbf{$L^2$ estimates for the Euclidean divergence of the derivatives of 
$S$ in terms of the derivatives of $\mathcal{D}$}]
\label{lem:divS}
Let $\mathfrak{c} \geq 0$ be a real number.
The following estimate holds for all $t\in [0,\Tboot)$,
where the implicit constants are \underline{independent} of $\mathfrak{c}$:
$$\|e^{-\f{\mathfrak{c} u'}2}\sqrt\upmu \mathrm{div} \mathcal{P}^{\Ntop} S\|_{L^2(\Sigma_t)}^2 \ls \epd^3 \upmu_{\star}^{-2 \toprate+0.8} + \|e^{-\f{\mathfrak{c}u'}2} \sqrt\upmu \,\mathcal{P}^{\Ntop} \mathcal{D}\|_{L^2(\Sigma_t)}^2.$$
\end{lemma}

\begin{proof}
Using Lemma~\ref{lem:commute.with.Cartesian} and the bootstrap assumption \eqref{BA:S}, 
we find that:
\begin{equation*}
\begin{split}
&\: | [\upmu \,\mathrm{div}, \mathcal{P}^{\Ntop}] S |
\ls \mbox{RHS of \eqref{eq:PN.mu.curl.S.commutator}}.
\end{split}
\end{equation*}
Therefore, we can therefore handle $|[\upmu \,\mathrm{div}, \mathcal{P}^{\Ntop}] S|$ 
by using the same arguments we gave in the proof of Lemma~\ref{lem:curlS}.

We then express $\mathrm{div}S$ in terms of $\mathcal{D}$ using \eqref{E:RENORMALIZEDDIVOFENTROPY}
and use Lemma~\ref{lem:Cart.to.geo}{,
the fact that the Cartesian component functions $X^1,X^2,X^3$ are smooth functions of
the $L^i$ and $\Psi$ (see \eqref{E:TRANSPORTVECTORFIELDINTERMSOFLUNITANDRADUNIT}),}
and the $L^\infty$ bounds in 
{\eqref{BA:LARGERIEMANNINVARIANTLARGE}--\eqref{BA:W.Li.small},} \eqref{BA:S}, {\eqref{BA:C.D},}
and {Proposition~\ref{prop:geometric.low}} to deduce:
\begin{equation*}
\begin{split}
&\: |\mathcal{P}^{\Ntop} (\upmu \, \mathrm{div} \, S)| \leq  |\mathcal{P}^{\Ntop} (\upmu \exp(2\rr) \mathcal{D})| + |\mathcal{P}^{\Ntop} (\upmu\exp(2\rr) \GradEnt^a\rd_a\rr)| \\
\ls &\: \upmu |\mathcal{P}^{\Ntop} \mathcal{D}| + |\mathcal{P}^{\leq \Ntop -1} \mathcal{D}| + |\mathcal{P}^{\leq \Ntop} S| \\
&\: + \epd ( |\mathcal{P}^{[2, \Ntop]} (\upmu, L^i)| + \upmu |\mathcal{P}^{\Ntop+1} \Psi| + |\mathcal{P}^{[2, \Ntop]} \Psi| + |\mathcal{P}^{[1, \Ntop]} \bX \Psi| ).
\end{split}
\end{equation*}
The new terms here compared to \eqref{eq:PN.mu.curl.Vr.commutator} and \eqref{eq:PN.mu.curl.Vr.main.computation} are 
{$|\mathcal{P}^{\leq \Ntop} S|$, which we handled just below
\eqref{eq:PN.mu.curl.S.commutator}, and}
$\upmu |\mathcal{P}^{\Ntop} \mathcal{D}|$ and $|\mathcal{P}^{\leq \Ntop -1} \mathcal{D}|$, which can be treated 
{using the same arguments we used to handle the terms}
$\upmu |\mathcal{P}^{\Ntop} \mathcal{C}|$ and $|\mathcal{P}^{\leq \Ntop -1} \mathcal{C}|$ in 
{our proof of}
Lemma~\ref{lem:curlVr}{. Hence, the weighted{, squared $L^2(\Sigma_t)$} norms 
corresponding to these new terms are}
bounded above by $\epd^3 \upmu_{\star}^{-2 \toprate+0.8} + \|e^{-\f{\mathfrak{c}u'}2} \sqrt\upmu \,\mathcal{P}^{\Ntop} \mathcal{D}\|_{L^2(\Sigma_t)}^2$. \qedhere
\end{proof}

\subsubsection{Proving the elliptic estimates}
We now combine Lemmas~\ref{lem:curlVr}--\ref{lem:divS} and the elliptic estimates in Proposition~\ref{prop:elliptic} to obtain the following proposition.
\begin{proposition}[\textbf{Preliminary top-order elliptic estimates for $\Vr$ and $S$}]
\label{prop:C.D.elliptic.final}
Let $\mathfrak{c} \geq 0$ be a real number.
The following estimates hold for all $t\in [0,\Tboot)$,
where the implicit constants are \underline{independent} of $\mathfrak{c}$:
\begin{equation}\label{eq:C.elliptic.final}
\|e^{-\f{\mathfrak{c} u}2}\sqrt{\upmu} \rdb \mathcal{P}^{\Ntop} \Vr\|_{L^2(\Sigma_t)}^2\ls \epd^3 (1+\mathfrak{c}^2) \upmu_{\star}^{-2 \toprate + 0.8}(t) + \|e^{-\f{\mathfrak{c} u}2}\sqrt\upmu \mathcal{P}^{\Ntop} \mathcal{C}\|_{L^2(\Sigma_t)}^2,
\end{equation}
and: 
\begin{equation}\label{eq:D.elliptic.final}
\|e^{-\f{\mathfrak{c} u}2}\sqrt{\upmu} \rdb \mathcal{P}^{\Ntop} S \|_{L^2(\Sigma_t)}^2 \ls \epd^3 (1+\mathfrak{c}^2) \upmu_{\star}^{-2 \toprate + 0.8}(t) + 
\|e^{-\f{\mathfrak{c} u}2}\sqrt\upmu \mathcal{P}^{\Ntop} \mathcal{D}\|_{L^2(\Sigma_t)}^2.
\end{equation}
\end{proposition}
\begin{proof}
Applying first Proposition~\ref{prop:elliptic}, and then Lemmas~\ref{lem:curlVr}, \ref{lem:divVr}, Proposition~\ref{prop:V.S} (and using $e^{-\f{\mathfrak{c} u}2}\leq 1$,) we obtain:
\begin{equation*}
\begin{split}
&\: \|e^{-\f{\mathfrak{c} u}2}\sqrt{\upmu} \rdb \mathcal{P}^{\Ntop} \Vr\|_{L^2(\Sigma_t)}^2 \\
\ls &\: \|e^{-\f{\mathfrak{c} u}2} \sqrt\upmu \,\mathrm{div}\,\mathcal{P}^{\Ntop} \Vr \|_{L^2(\Sigma_t)}^2 + \|e^{-\f{\mathfrak{c} u}2} \sqrt\upmu \,\mathrm{curl}\,\mathcal{P}^{\Ntop} \Vr\|_{L^2(\Sigma_t)}^2 + \mathfrak{c}^2 \upmu_{\star}^{-2}(t) \| e^{-\f{\mathfrak{c} u}2} \sqrt \upmu \, \mathcal{P}^{\Ntop} \Vr \|_{L^2(\Sigma_t)}^2 \\
\ls &\: \epd^3 (1 + \mathfrak{c}^2) \upmu_{\star}^{-2 \toprate + 0.8}(t) + \|e^{-\f{\mathfrak{c} u}2}\sqrt\upmu \mathcal{P}^{\Ntop} \mathcal{C}\|_{L^2(\Sigma_t)}^2,
\end{split}
\end{equation*}
which proves \eqref{eq:C.elliptic.final}. The proof of \eqref{eq:D.elliptic.final} is similar, except we use
Lemmas~\ref{lem:curlS}, \ref{lem:divS} instead of Lemmas~\ref{lem:curlVr}, \ref{lem:divVr}. \qedhere
\end{proof}

\subsection{Putting everything together}\label{sec:elliptic.everything}
\begin{proposition}[\textbf{The main top-order estimates for the modified fluid variables}]
\label{prop:elliptic.putting.everything.together}
The following estimate holds for every $(t,u)\in [0,\Tboot)\times [0,U_0]$:
$$\mathbb C_{\Ntop}(t,u) + \mathbb D_{\Ntop}(t,u) 
\ls  \epd^3 \upmu_{\star}^{-2 \toprate+0.8}(t). $$
\end{proposition}
\begin{proof}

\pfstep{Step~1: Controlling $\|e^{-\f{\mathfrak{c} u}2}\sqrt{\upmu} \rdb \mathcal{P}^{\Ntop} (\Vr, S) \|_{L^2(\Sigma_t)}^2$ via Gr\"onwall-type argument} Given $\varsigma>0$, we first apply Proposition~\ref{prop:C.D.elliptic.final} and then use\footnote{Here, we again relabeled the $\varsigma$ from Proposition~\ref{prop:C.D.transport.main}} Proposition~\ref{prop:C.D.transport.main} (for\footnote{Note that in view of the fact that $\Vr$, $S$ are compactly supported in $u\in [0, U_0]$ (by Lemma~\ref{lem:localization}), it follows that the integral on $\Sigma_t^{U_0}$ is the same as the integral on $\Sigma_t$.} $u = U_0$) 
 to deduce that if $\mathfrak{c} > 0$ is sufficiently large (depending on $\varsigma$), then
the following estimate holds, where the constants $C > 0$ and $C_* > 0$
are independent of $\mathfrak{c}$ and $\varsigma$:
\begin{equation}\label{eq:elliptic.putting.everything.together.main}
\begin{split}
&\: \|e^{-\f{\mathfrak{c} u}2}\sqrt{\upmu} \rdb \mathcal{P}^{\Ntop} (\Vr, S) \|_{L^2(\Sigma_t)}^2\\
\leq &\: C\epd^3 (1+\mathfrak{c}^2) \upmu_{\star}^{-2 \toprate + 0.8}(t) + C\|e^{-\f{\mathfrak{c} u}2}\sqrt\upmu \mathcal{P}^{\Ntop} (\mathcal{C}, \mathcal{D})\|_{L^2(\Sigma_t)}^2 \\
\leq &\: C_* \epd^3 (1+\mathfrak{c}^2) \upmu_{\star}^{-2 \toprate+0.8}(t) + \varsigma \int_{t'=0}^{t'=t} \f{1}{\upmu_{\star}(t')}\|e^{-\f{\mathfrak{c}u}{2}} \sqrt \upmu \rdb \mathcal{P}^{\Ntop} (\Vr,S)\|_{L^2(\Sigma_{t'})}^2 \,dt'.
\end{split}
\end{equation}
We clarify that it is only for notational convenience for the argument in \eqref{eq:elliptic.putting.everything.together.goal}--\eqref{eq:elliptic.putting.everything.together.contradiction.2} below that
we have used the symbol $C_* > 0$ to denote the fixed constant on the last line of 
\eqref{eq:elliptic.putting.everything.together.main}.

We now argue by a continuity argument to show that, after choosing $\varsigma$ smaller 
 and $\mathfrak{c}$ larger
if necessary, \eqref{eq:elliptic.putting.everything.together.main} implies the following estimate:
\begin{equation}\label{eq:elliptic.putting.everything.together.goal}
\|e^{-\f{\mathfrak{c} u}2}\sqrt{\upmu} \rdb \mathcal{P}^{\Ntop} (\Vr, S) \|_{L^2(\Sigma_t)}^2 \leq 2 C_* \epd^3 (1+\mathfrak{c}^2) \upmu_{\star}^{-2 \toprate+0.8}(t).
\end{equation}

If it is not the case that \eqref{eq:elliptic.putting.everything.together.goal} holds on $[0,\Tboot)$, 
then by continuity, there exists $T_* \in [0, \Tboot)$ such that \eqref{eq:elliptic.putting.everything.together.goal} holds for all 
$t\in [0,T_*]$ and that: 
\begin{equation}\label{eq:elliptic.putting.everything.together.contradiction}
\|e^{-\f{\mathfrak{c} u}2}\sqrt{\upmu} \rdb \mathcal{P}^{\Ntop} (\Vr, S) \|_{L^2(\Sigma_{T_*})}^2 
= 
2 C_* \epd^3 (1+\mathfrak{c}^2) \upmu_{\star}^{-2 \toprate+0.8}(T_*).
\end{equation}

However, plugging the estimate \eqref{eq:elliptic.putting.everything.together.goal}
(which by assumption holds for $t \in [0,T_*]$)
into the integral in
\eqref{eq:elliptic.putting.everything.together.main}, 
using Proposition~\ref{prop:mus.int} 
(and $\toprate \geq 1$) to integrate away a negative power of $\upmu_{\star}$, and finally choosing $\varsigma$ sufficiently small, we obtain that for $t\in [0,T_*]$, we have:
\begin{equation}\label{eq:elliptic.putting.everything.together.contradiction.2}
\|e^{-\f{\mathfrak{c} u}2}\sqrt{\upmu} \rdb \mathcal{P}^{\Ntop} (\Vr, S) \|_{L^2(\Sigma_t)}^2 \leq \f 32 C_* \epd^3 (1+\mathfrak{c}^2) \upmu_{\star}^{-2 \toprate+0.8}(t),
\end{equation}
which obviously contradicts \eqref{eq:elliptic.putting.everything.together.contradiction} when $t = T_*$. 
It therefore follows that our desired estimate \eqref{eq:elliptic.putting.everything.together.goal} holds
for all $t\in [0,\Tboot)$.

\pfstep{Step~2: Deducing the estimates for $\mathbb C_{\Ntop}(t,u)$ and $\mathbb D_{\Ntop}(t,u)$} At this point, we can fix the constants $\mathfrak{c}$, $\varsigma$, which we will absorb into the 
ensuring generic constants ``$C$''. 
Moreover, since $u\in [0,U_0]$ on the support of $\Omega$ and $S$ (by Lemma~\ref{lem:localization}), we will also absorb the weights $e^{-\f{\mathfrak{c} u}2}$ into the constants. Hence, plugging \eqref{eq:elliptic.putting.everything.together.goal} 
into RHS~\eqref{E:C.D.transport.main} and then using Proposition~\ref{prop:mus.int}, we obtain:
\begin{equation}
\begin{split}
&\: \mathbb C_{\Ntop}(t,u) + \mathbb D_{\Ntop}(t,u) \\
\ls &\: \epd^3 \upmu_{\star}^{-2 \toprate+0.8}(t) +  \int_{t'=0}^{t'=t} \f{1}{\upmu_{\star}(t')}\|e^{-\f{\mathfrak{c}u'}{2}} \sqrt \upmu \rdb \mathcal{P}^{\Ntop} (\Vr,S)\|_{L^2(\Sigma_{t'})}^2 \,dt' \\
\ls &\: \epd^3 \upmu_{\star}^{-2 \toprate+0.8}(t) + \epd^3 \int_{t'=0}^{t'=t} \upmu_{\star}^{-2 \toprate - 0.2}(t') \,dt' \ls \epd^3 \upmu_{\star}^{-2 \toprate+0.8}(t),
\end{split}
\end{equation}
as desired. \qedhere

\end{proof}

\section{Wave estimates for the fluid variables}\label{sec:wave.main.est}
We continue to work under the assumptions of Theorem~\ref{thm:bootstrap}.

In this section, we 
derive a priori energy estimates for the wave variables,
which will in particular yield strict improvements of the bootstrap assumptions \eqref{BA:W1}--\eqref{BA:W2}.
In Section~\ref{sec:wave.black.box},
we start by providing a somewhat general\footnote{Using a slight reorganization of the paper, 
these estimates could be upgraded so that they are ``black box'' estimates for inhomogeneous wave equations.
Given the setup of this paper, they are not quite black box estimates because the proofs
rely on the estimates of Section~\ref{sec:geometry}, some of which 
(e.g., some of the estimates in Proposition~\ref{P:LINFTYHIGHERTRANSVERSAL}) 
depend on the structure of the inhomogeneous terms in the wave equations.
\label{FN:NOTQUITEBLACKBOX}}  
``auxiliary'' proposition, which yields
energy estimates for solutions to \emph{inhomogeneous} quasilinear wave equations 
\emph{in terms of norms of the inhomogeneity}. The difficult aspect of the proof
is that we have to close the estimates even though $\upmu$ can be tending towards $0$,
that is, even though the shock may be forming.
We delay discussing the proof of the auxiliary proposition until Appendix~\ref{app:elliptic};
as we will explain, 
modulo small modifications based on established techniques, 
the proposition was proved as
\cite[Proposition~14.1]{LS} (see also \cite[Proposition~14.1]{jSgHjLwW2016}).
Then, in Section~\ref{sec:wave.inho}, we bound the specific inhomogeneous terms 
that are are relevant for our main results, that is, the inhomogeneous
terms on the RHSs of the fluid wave equations \eqref{E:VELOCITYWAVEEQUATION}--\eqref{E:ENTROPYWAVEEQUATION}.
Finally, in Section~\ref{sec:wave.everything}, we prove the final a priori energy estimates.

\subsection{The main estimates for inhomogeneous covariant wave equations}\label{sec:wave.black.box}
In this section, we state the ``auxiliary'' Proposition~\ref{prop:wave}, 
which yields energy estimates for solutions to the fluid wave equations. In this section, 
we ignore the precise structure of the inhomogeneous terms and
simply denote them by $\mathfrak{G}$. That is, we state the estimates 
of Proposition~\ref{prop:wave} in terms of various norms of $\mathfrak{G}$.
Later on, in Proposition~\ref{prop:wave.final}, we will control the relevant norms of $\mathfrak{G}$
to obtain our final a priori energy estimates for the wave variables.
Proposition~\ref{prop:wave} is of independent interest in the sense that
with small modifications, it could be used to study shock formation 
for compressible Euler flow with given smooth forcing terms.

\begin{proposition}[The main estimates for the inhomogeneous geometric wave equations]
\label{prop:wave}
Let 
		$
			\threePsi
			\doteq 
			(\Psi_1,\Psi_2,\Psi_3,\Psi_4,\Psi_5)
			\doteq
			(\mathcal{R}_{(+)},
			 \mathcal{R}_{(-)},
				v^2, v^3,s)$,
				as in \eqref{E:THREEVECPSI}. 
Recall that the $\Psi_{\imath}$ are solutions
to the inhomogeneous covariant wave system $$\upmu \square_{g(\vec{\Psi})}\Psi_{\imath} = \mathfrak G_{\imath},$$
where $\vec{\mathfrak G} = (\mathfrak{G}_1,\mathfrak{G}_2,\mathfrak{G}_3,\mathfrak{G}_4,\mathfrak{G}_5)$ 
is the array whose entries are the product of $\upmu$ and 
the inhomogeneous terms on the
RHSs of the five scalar wave equations 
\eqref{E:VELOCITYWAVEEQUATION}--\eqref{E:ENTROPYWAVEEQUATION}.
Assume that the following smallness bound holds:\footnote{
We clarify that in our main results,
			in the proof of Proposition~\ref{prop:wave.final},
			we will show that the smallness assumption \eqref{eq:F.smallness}
			is satisfied for the particular inhomogeneous terms $\vec{\mathfrak G}$ 
			stated in the hypotheses of the proposition.
			However, for the purposes of proving Proposition~\ref{prop:wave},
			the precise structure of $\vec{\mathfrak G}$ is not important.}
\begin{equation}\label{eq:F.smallness}
 \|\mathcal{P}^{\leq\lceil \f{\Ntop}{2} \rceil} \mathfrak G\|_{L^\infty(\Mtu)} \leq \epd^{\f 12}.
 \end{equation}

Then there exists an \textbf{absolute constant} $\toprate \in \mathbb N$, 
independent of the equation of state and all other
parameters in the problem, such that the following hold.
As in Theorem~\ref{thm:bootstrap},
let $\Tboot \in [0, 2\mathring{\updelta}_*^{-1}]$,
and assume that:
\begin{enumerate}
\item The bootstrap assumptions \eqref{BA:W1}--\eqref{BA:C.D} all hold for all $t\in [0,\Tboot)$,
	where we recall that in the bootstrap assumptions, $\Ntop$ is any integer
	satisfying $\Ntop \geq 2 \toprate + 10$;
\item  In \eqref{BA:LARGERIEMANNINVARIANTLARGE}, the parameter 
$\mathring{\upalpha}$ is sufficiently small in a manner 
only on the equation of state and $\bar{\varrho}$;
\item The parameter $\epd > 0$ in \eqref{BA:W1}--\eqref{BA:C.D}
satisfies $\epd^{\frac{1}{2}} \leq \mathring{\upalpha}$ and
is sufficiently small in
a manner that depends only on 
the equation of state, 
$\Ntop$,
$\bar{\varrho}$, 
$\mathring{\upsigma}$, 
$\mathring{\updelta}$, 
and $\mathring{\updelta}_*^{-1}$;
	\\
and
\item The soft bootstrap assumptions stated in Section~\ref{SSS:SOFTBOOTSTRAP} hold 
	(including $\upmu>0$ in $[0,\Tboot) \times \mathbb{R} \times \mathbb{T}^2$).
\end{enumerate}

Then the following estimates hold for every 
$(t,u)\in [0,\Tboot)\times [0,U_0]$,
where $\upmu_{\star}$ is defined in Definition~\ref{def:mustar}:

\begin{enumerate}
\item The top- and penultimate-order wave energies defined in \eqref{eq:wave.energy.def.4}
obey the following estimates:
\begin{equation}\label{eq:main.EE.prop.top}
\begin{split}
&\: \sup_{\hat{t} \in [0,t]}\upmu_{\star}^{2\toprate -1.8}(\hat{t})\mathbb{W}_{[1,\Ntop]}(\hat{t},u)  + \sup_{\hat{t} \in [0,t]}\upmu_{\star}^{2\toprate -3.8}(\hat{t})\mathbb{W}_{[1,\Ntop-1]}(\hat{t},u) \\
\ls &\:  \epd^2 +  \sup_{\hat{t}\in [0,t]} \upmu_{\star}^{2\toprate-1.8}(\hat{t}) \int_{t'=0}^{t' = \hat{t}} \upmu_{\star}^{-\f 32}(t') \left\{\int_{s=0}^{s=t'} \|\mathcal{P}^{[1,\Ntop]}\mathfrak G\|_{L^2(\Sigma_s^u)} \, ds\right\}^2 \, dt'  \\
&\: + \sup_{\hat{t}\in [0,t]} \upmu_{\star}^{2\toprate-1.8}(\hat{t}) \|(|L\mathcal{P}^{[1,\Ntop]} \Psi| + |\bX \mathcal{P}^{[1,\Ntop]} \Psi|) |\mathcal{P}^{[1,\Ntop]} \mathfrak G| \|_{L^1(\mathcal M_{\hat{t},u})} \\
&\: + \sup_{\hat{t}\in [0,t]} \upmu_{\star}^{2\toprate-3.8}(\hat{t}) \|(|L\mathcal{P}^{[1,\Ntop-1]} \Psi| + |\bX \mathcal{P}^{[1,\Ntop-1]} \Psi|) |\mathcal{P}^{[1,\Ntop-1]} \mathfrak G| \|_{L^1(\mathcal M_{\hat{t},u})};
\end{split}
\end{equation}
\item  For $1\leq N\leq \Ntop-1$, the lower-order wave energies $\mathbb{W}_{[1,N]}$ 
	defined in \eqref{eq:wave.energy.def.3}
	obey the following estimates:
\begin{equation}\label{eq:main.EE.prop.low}
\begin{split}
&\: \mathbb{W}_{[1,N]}(t,u) \\
\ls &\:  \epd^2 + \max\{1, \upmu_{\star}^{-2 \toprate+2\Ntop-2N+1.8}(t) \} \left(\sup_{s\in [0,t]} \min\{1,\upmu_{\star}^{2\toprate-2\Ntop+2N+0.2}(s)\} \mathbb{Q}_{[1,N+1]}(s)\right)  \\
&\: + \| (|L\mathcal{P}^{[1,N]} \Psi| + |\bX \mathcal{P}^{[1,N]} \Psi|) |\mathcal{P}^{[1,N]} \mathfrak G| \|_{L^1(\Mtu)}.
\end{split}
\end{equation}
\end{enumerate}
\end{proposition}

\begin{remark}
The proof of Proposition~\ref{prop:wave} follows from almost exactly the same arguments used in the proof of
\cite[Proposition~14.1]{LS}. 
The only differences are the following two changes:
\begin{enumerate}
\item We have to track the influence of the inhomogeneous terms $\mathfrak G$ 
	on the estimates.
\item In $3$D, the second fundamental form of the null hypersurfaces of the acoustical metric has three (as opposed to one) independent components. This necessitates an additional elliptic estimate
that was not needed in the $2$D case treated in \cite{LS}. 
This elliptic estimate is standard; 
see \cites{dCsK1993,sKiR2003,dC2007}.
\end{enumerate}
These differences necessitate minor modifications to the proof of \cite[Proposition~14.1]{LS}. 
We will sketch them in Appendix~\ref{app:elliptic}.
\end{remark}

\begin{remark}[\textbf{Additional term in the top-order estimate}]
In Proposition~\ref{prop:wave}, the inhomogeneous term $\vec{\mathfrak G}$ makes an additional appearance in the top- and penultimate-order estimates as compared to the estimates of all the lower orders.
By ``additional appearance,'' we are referring to the double time integral,
which comes from a difficult top-order commutator term that depends on the acoustic geometry;
this difficult term has to be controlled by first integrating a transport equation, 
which explains the double time-integration; see Appendix~\ref{app:elliptic}.
\end{remark}

\subsection{Estimates for the inhomogeneous terms}\label{sec:wave.inho}
We start by controlling the null forms in the wave equations.
\begin{proposition}[\textbf{Control of wave equation error terms involving null forms}]
\label{prop:wave.inho.null}
For $\mathfrak Q\in \{\mathfrak{Q}_{(v)}^i,\mathfrak Q_{(\pm)}\}$ (see \eqref{E:VELOCITYNULLFORM}, \eqref{E:DENSITYNULLFORM}) and $1\leq N \leq \Ntop$, the following holds for all $(t,u) \in [0,\Tboot)\times [0,U_0]$ and for all $\varsigma \in (0,1]$, where the implicit constants are independent of $\varsigma$:
\begin{equation}\label{eq:Psi.null.condition}
\begin{split}
&\: \| (|L\mathcal{P}^{[1,N]} \Psi| + |\bX \mathcal{P}^{[1,N]} \Psi|)\mathcal{P}^{[1,N]} (\upmu \mathfrak Q)\|_{L^1(\Mtu)}\\
\ls &\:  \epd^2 \max\{1, \upmu_{\star}^{-2\toprate + 2\Ntop -2N +1.8}(t)\} 
	\\
&\: +\varsigma \mathbb K_{[1,N]}(t,u) + (1+\varsigma^{-1})
\left(\int_{u' =0}^{u' = u} \mathbb{F}_{[1,N]}(t,u') \, du' +  \int_{t'=0}^{t'=t} \mathbb{E}_{[1,N]}(t',u)\, dt' \right),
\end{split}
\end{equation}
and: 
\begin{equation}\label{eq:Psi.null.condition.top}
\begin{split}
&\:  \int_{t'=0}^{t'=t} \upmu_{\star}^{-\f 32}(t') \left\{ \int_{s=0}^{s=t'} 
\|\mathcal{P}^{[1,\Ntop]} (\upmu \mathfrak Q)\|_{L^2(\Sigma_s^u)}\, ds\right\}^2 \, dt' \\
\ls &\: \epd^2 \upmu_{\star}^{-2\toprate +1.8}(t)
 + \int_{t'=0}^{t'=t} \upmu_{\star}^{-\f 32}(t') \left\{ \int_{s=0}^{s=t'} \upmu_{\star}^{-\f 12}(s)\mathbb{E}_{[1,\Ntop]}^{\f 12}(s)  \, ds\right\}^2 \, dt'.
\end{split}
\end{equation}

\end{proposition}
\begin{proof}
\pfstep{Step~1: Proof of \eqref{eq:Psi.null.condition}} To bound the LHS of \eqref{eq:Psi.null.condition}, we use the Cauchy--Schwarz and the Young inequalities to obtain, 
for any $\varsigma>0$:
\begin{equation}\label{eq:Psi.condition.pf.1}
\begin{split}
&\: \| (|L\mathcal{P}^{[1,N]} \Psi| + |\bX \mathcal{P}^{[1,N]} \Psi|)\mathcal{P}^{[1,N]} (\upmu \mathfrak Q)\|_{L^1(\Mtu)} \\
\ls &\: (1 + \varsigma^{-1})  \left(\int_{u' =0}^{u' = u} \mathbb{F}_{[1,N]}(t,u') \, du' +  \int_{t'=0}^{t'=t} \mathbb{E}_{[1,N]}(t',u)\, dt' \right) 
+ \varsigma \|\mathcal{P}^{[1,N]} (\upmu \mathfrak Q)\|_{L^2(\Mtu)}^2.
\end{split}
\end{equation}
 
By inspection, 
it can be checked that $\mathfrak Q$ is a 
$g$-null form (see Definition~\ref{D:NULLFORMS}) that is 
quadratic in the wave variables. 
Hence, applying \eqref{eq:null.form.basic.1} with $\phi^{(1)},\,\phi^{(2)} = \Psi$, $\mathfrak d^{(1,1)},\, \mathfrak d^{(2,1)} \ls 1$, $\mathfrak d^{(1,2)},\, \mathfrak d^{(2,2)} \ls \epd^{\f 12}$ 
(which is justified by the bootstrap assumptions 
\eqref{BA:LARGERIEMANNINVARIANTLARGE}--\eqref{BA:W.Li.small}), we obtain:
\begin{equation}\label{eq:derivative.of.wave.null}
\begin{split}
 |\mathcal{P}^{[1,N]} (\upmu \mathfrak Q)| 
\ls &\: |\mathcal{P}^{[2,N+1]}\Psi| + \epd^{\f 12} \{ |\mathcal{P}^{[1,N]} \bX \Psi| + |\mathcal{P} \Psi| + |\mathcal{P}^{[2,N]} (\upmu, L^i)| \}.
\end{split}
\end{equation}

To bound \eqref{eq:derivative.of.wave.null} in $L^2(\Mtu)$, we control $|\mathcal{P}^{[2,N+1]}\Psi|$ by the energies \eqref{eq:wave.energy.def.1}--\eqref{eq:wave.energy.def.3}, control  $|\mathcal{P}^{[1,N]} \bX \Psi|$ by 
\eqref{eq:wave.energy.commuted}, bound $|\mathcal{P}\Psi|$ by \eqref{BA:W.Li.small}, and $|\mathcal{P}^{[2,N]} (\upmu, L^i)|$ by Proposition~\ref{prop:geometric.top}. We thus obtain the following bound for any
$\varsigma \in (0,1]$, where the implicit constants are independent of $\varsigma$:
\begin{equation}\label{eq:Psi.condition.pf.2}
\begin{split}
&\: \varsigma \|\mathcal{P}^{[1,N]} (\upmu \mathfrak Q)\|_{L^2(\Mtu)}^2 \\
\ls &\: \varsigma 
				\left\lbrace
				\mathbb K_{[1,N]}(t,u) + \int_{u' =0}^{u' = u} \mathbb{F}_{[1,N]}(t,u') \, du' +  \int_{t'=0}^{t'=t} \mathbb{E}_{[1,N]}(t',u)\, dt'\right\rbrace \\
&\: + \epd^2 \int_{t'=0}^{t'=t} \max\{1,\upmu_{\star}^{-2\toprate+2\Ntop-2N + 1.8}(t')\}  \, dt' + \epd^2 \\
\ls &\: \epd^2 \max\{1, \upmu_{\star}^{-2\toprate + 2\Ntop -2N +1.8}(t)\}\\
&\: +\varsigma \left\lbrace
	\mathbb K_{[1,N]}(t,u) + \int_{u' =0}^{u' = u} \mathbb{F}_{[1,N]}(t,u') \, du' 
+  
\int_{t'=0}^{t'=t} \mathbb{E}_{[1,N]}(t',u)\, dt'
\right\rbrace,
\end{split}
\end{equation}
where in the last line, we have used Proposition~\ref{prop:almost.monotonicity}.

Putting \eqref{eq:Psi.condition.pf.1}--\eqref{eq:Psi.condition.pf.2} together, we obtain \eqref{eq:Psi.null.condition}.

\pfstep{Step~2: Proof of \eqref{eq:Psi.null.condition.top}} We begin with \eqref{eq:derivative.of.wave.null} when $N = \Ntop$. Notice that unlike in Step~1, we now have to control $|\mathcal{P}^{[2,N+1]}\Psi|$ only with the $\mathbb{E}$ (but not $\mathbb{F}$ and $\mathbb K$) energy (since we need an estimate on a fixed-$t$ hypersurface). This gives a $\upmu_{\star}^{-\f 12}$ degeneration. The other terms can be controlled by using arguments similar to the ones we used in Step~1.
In total, we have:
\begin{equation}
\begin{split}
 \|\mathcal{P}^{[1,\Ntop]} (\upmu \mathfrak Q)\|_{L^2(\Sigma_s^u)} 
\ls &\: \upmu_{\star}^{-\f 12}(s)\mathbb{E}_{[1,\Ntop]}^{\f 12}(s) + \epd \max\{1, \upmu_{\star}^{-\toprate + 0.9}(t)\}.
\end{split}
\end{equation}
Finally, integrating with respect to time and using
Proposition~\ref{prop:mus.int}, we obtain \eqref{eq:Psi.null.condition.top}. \qedhere
\end{proof}

Next, we control the easy linear terms in the wave equations.

\begin{proposition}[\textbf{Control of wave equation error terms involving easy linear inhomogeneous terms}]
\label{prop:wave.inho.linear}
For $\mathfrak L \in \{ \mathfrak L_{(v)}^i,\,\mathfrak L_{(\pm)},\,\mathfrak L_{(s)} \}$ (see \eqref{E:VELOCITYILINEARORBETTER}, \eqref{E:PMLINEARORBETTER}, \eqref{E:ENTROPYLINEARORBETTER}) and $1\leq N \leq \Ntop$, the following holds for all $(t,u) \in [0,\Tboot)\times [0,U_0]$:
\begin{equation}\label{eq:wave.inho.linear.1}
 \| (|L\mathcal{P}^{[1,N]} \Psi| + |\bX \mathcal{P}^{[1,N]} \Psi|)\mathcal{P}^{[1,N]} (\upmu \mathfrak L)\|_{L^1(\Mtu)} \ls \mbox{\upshape RHS~\eqref{eq:Psi.null.condition}},
\end{equation}
and:
\begin{equation}\label{eq:wave.inho.linear.2}
\int_{t'=0}^{t'=t} \upmu_{\star}^{-\f 32}(t') \left\{ \int_{s=0}^{s=t'} \|\mathcal{P}^{[1,\Ntop]} (\upmu \mathfrak L)\|_{L^2(\Sigma_s^u)}\, ds\right\}^2 \, dt' \ls \mbox{\upshape RHS~\eqref{eq:Psi.null.condition.top}}.
\end{equation}
\end{proposition}
\begin{proof}
We first pointwise
bound $\upmu \mathfrak L \in 
\{ \upmu \mathfrak L_{(v)}^i,\, \upmu \mathfrak L_{(\pm)},\, \upmu \mathfrak L_{(s)}\}$ in a similar manner\footnote{In fact, we can even do better than terms in \eqref{eq:derivative.of.wave.null} because of the extra smallness in $\epd$ we have from the bootstrap assumptions. However, we do not need this improvement for our proof.}  
as \eqref{eq:derivative.of.wave.null}:
\begin{equation}\label{eq:wave.inho.linear.3}
|\mathcal{P}^{[1,N]} (\upmu \mathfrak L) | \ls |\mathcal{P}^{\leq N}(\Vr,S)|+  \mbox{terms already in \eqref{eq:derivative.of.wave.null}}.
\end{equation}

\pfstep{Proof of \eqref{eq:wave.inho.linear.1}} The terms in \eqref{eq:wave.inho.linear.3} that are already in \eqref{eq:derivative.of.wave.null} can of course be controlled as in Proposition~\ref{prop:wave.inho.null}. We therefore focus on $|\mathcal{P}^{\leq N}(\Vr,S)|$, for which we have the following estimate using the Cauchy--Schwarz and H\"older inequalities and Proposition~\ref{prop:V.S}:
\begin{equation}
\begin{split}
&\: \| (|L\mathcal{P}^{[1,N]} \Psi| + |\bX \mathcal{P}^{[1,N]} \Psi |)\mathcal{P}^{\leq N}(\Vr,S)\|_{L^1(\Mtu)} \\
\ls &\: \|L\mathcal{P}^{[1,N]} \Psi \|_{L^2(\Mtu)}^2 + \|\bX \mathcal{P}^{[1,N]} \Psi \|_{L^2(\Mtu)}^2 + \int_0^u \|\mathcal{P}^{\leq N}(\Vr,S) \|_{L^2(\mathcal F_{u'}^t)}^2 \, du' \\
\ls &\: \|L\mathcal{P}^{[1,N]} \Psi \|_{L^2(\Mtu)}^2 + \|\bX \mathcal{P}^{[1,N]} \Psi \|_{L^2(\Mtu)}^2 
+ 
\epd^3 \max\{1, \upmu_{\star}^{-2\toprate+2\Ntop + 2N +2.8}(t)\},
\end{split}
\end{equation}
which can indeed be bounded above by RHS of \eqref{eq:Psi.null.condition} as claimed.

\pfstep{Proof of \eqref{eq:wave.inho.linear.2}} Again, we only focus on the $|\mathcal{P}^{\leq \Ntop}(\Vr,S)|$ term in \eqref{eq:wave.inho.linear.3}. Using the definitions of the $\mathbb V$ and $\mathbb S$ norms and Propositions~\ref{prop:mus.int} and ~\ref{prop:V.S}, we deduce:
\begin{equation}
\begin{split}
&\: \int_{t'=0}^{t'=t} \upmu_{\star}^{-\f 32}(t') \left\{ \int_{s=0}^{s=t'} 
\|\mathcal{P}^{\leq \Ntop} (\Vr,S)\|_{L^2(\Sigma_s^u)}\, ds\right\}^2 \, dt' \\
\ls &\: \int_{t'=0}^{t'=t} \upmu_{\star}^{-\f 32}(t') \left\{ \int_{s=0}^{s=t'} \upmu_{\star}^{-\f 12}(s) [\mathbb V_{\leq \Ntop}^{\f 12}(s) + \mathbb S_{\leq \Ntop}^{\f 12}(s)]\, ds\right\}^2 \, dt' \\
\ls &\: \epd^3 \int_{t'=0}^{t'=t} \upmu_{\star}^{-\f 32}(t') \left\{ \int_{s=0}^{s=t'} 
\upmu_{\star}^{-\toprate+ 0.9}(s)\, ds\right\}^2 \, dt' \\
\ls &\: \epd^3 \max\{1, \upmu_{\star}^{-2\toprate +3.3}(t) \}
\ls \epd^2  \max\{1, \upmu_{\star}^{-2\toprate +1.8}(t) \},
\end{split}
\end{equation}
which can indeed be bounded above by RHS of \eqref{eq:Psi.null.condition.top} as claimed. \qedhere
\end{proof}

Finally, we consider the linear terms involving $\mathcal{C}$ and $\mathcal{D}$.
\begin{proposition}[\textbf{Control of wave equation error terms involving $\mathcal{C}$ and $\mathcal{D}$}]
\label{prop:wave.inho.main}
For $\mathfrak M \in \{ c^2 \exp(2\Densrenormalized) \mathcal{C}^i, \, \Speed \exp(\rr) \frac{p_{;s}}{\bar{\varrho}} \mathcal{D},\,\Speed^2 \exp(2\rr) \DivofEntrenormalized,\, F_{;\Ent } \Speed^2 \exp(2\rr) \DivofEntrenormalized \}$ (cf.~main terms in \eqref{E:VELOCITYWAVEEQUATION}--\eqref{E:ENTROPYWAVEEQUATION}) and $1\leq N \leq \Ntop$, the following holds for all $(t,u)\in [0,\Tboot)\times [0,U_0]$:
\begin{equation}\label{eq:wave.inho.main.1}
\|( |L\mathcal{P}^{[1,N]}\Psi| + |\bX\mathcal{P}^{[1,N]}\Psi |)\mathcal{P}^{[1,N]} (\upmu \mathfrak M) \|_{L^1(\Mtu)} \ls \mbox{\upshape RHS~\eqref{eq:Psi.null.condition}},
\end{equation}
and:
\begin{equation}\label{eq:wave.inho.main.2}
\int_{t'=0}^{t'=t} \upmu_{\star}^{-\f 32}(t') 
\left\{ \int_{s=0}^{s=t'} \|\mathcal{P}^{[1,\Ntop]} (\upmu \mathfrak M)\|_{L^2(\Sigma_s^u)}\, ds\right\}^2 \, dt' \ls \mbox{\upshape RHS~\eqref{eq:Psi.null.condition.top}}.
\end{equation}
\end{proposition}
\begin{proof}

We first use the 
bootstrap assumptions \eqref{BA:LARGERIEMANNINVARIANTLARGE}--\eqref{BA:W.Li.small} and \eqref{BA:C.D}
and Proposition~\ref{prop:geometric.low} to deduce:
\begin{equation}\label{eq:est.C.in.wave.main}
\begin{split}
 |\mathcal{P}^N (\upmu \mathfrak M)| 
\ls &\: \underbrace{\upmu |\mathcal{P}^{N} (\mathcal{C}, \mathcal{D})|}_{\doteq I} + \underbrace{|\mathcal{P}^{\leq N-1} (\mathcal{C}, \mathcal{D})|}_{\doteq II} + \mbox{ terms already in \eqref{eq:derivative.of.wave.null}}.
\end{split}
\end{equation}

\pfstep{Step~1: Proof of \eqref{eq:wave.inho.main.1}} 
The terms already in \eqref{eq:derivative.of.wave.null} were handled in the proof of \eqref{eq:Psi.null.condition}, so we only have to handle $I$ and $II$ in \eqref{eq:est.C.in.wave.main}. 
We will use slightly different arguments for each of these two terms. For $I$, we have\footnote{Note that it is only at the top $N = \Ntop$ level that 
$\mathbb C^{\f 12}_{\leq N}$  and $\mathbb D^{\f 12}_{\leq N}$ is only bounded by 
$\upmu_{\star}^{-\toprate+\Ntop-N+0.4}(t')$. 
For $N< \Ntop$, we have the stronger estimates in Proposition~\ref{prop:C.D}, which in principle 
would allow us to avoid controlling the term $I$ separately.} 
by the Cauchy--Schwarz inequality, Propositions~\ref{prop:C.D}, 
\ref{prop:elliptic.putting.everything.together}, the bootstrap assumptions \eqref{BA:W1}, \eqref{BA:W2}, 
and Propositions~\ref{prop:geometric.low} and \ref{prop:mus.int} that:
\begin{equation}\label{eq:est.C.in.wave.1}
\begin{split}
&\: \|( |L\mathcal{P}^{[1,N]}\Psi| + |\bX\mathcal{P}^{[1,N]}\Psi |)\upmu \mathcal{P}^{N} (\mathcal{C}, \mathcal{D}) \|_{L^1(\Mtu)} \\
\ls &\: \int_{t'=0}^{t'=t} \mathbb{E}_{[1,N]}^{\f 12}(t',u)[\mathbb C^{\f 12}_{\leq N} + \mathbb D^{\f 12}_{\leq N}](t',u) \, dt' \\
\ls &\: \epd^{\f 12} \epd^{\f 32} \int_{t'=0}^{t'=t} \max\{1,\upmu_{\star}^{-\toprate+\Ntop-N+0.9}(t')\} \max\{ 1, \upmu_{\star}^{-\toprate+\Ntop-N+0.4}(t')\} \, dt' \\
\ls &\: \epd^2 \max\{1, \upmu_{\star}^{-2\toprate+2\Ntop-2N+2.3}(t)\}.
\end{split}
\end{equation}

For $II$ in \eqref{eq:est.C.in.wave.main}, we use Cauchy--Schwarz and Proposition~\ref{prop:C.D} to obtain:
\begin{equation}\label{eq:est.C.in.wave.2}
\begin{split}
&\: \|( |L\mathcal{P}^{[1,N]}\Psi| + |\bX\mathcal{P}^{[1,N]}\Psi |) \mathcal{P}^{\leq N-1} (\mathcal{C}, \mathcal{D}) \|_{L^1(\Mtu)} \\
\ls &\: \|L\mathcal{P}^{[1,N]}\Psi\|_{L^2(\Mtu)}^2 + \|\bX \mathcal{P}^{[1,N]}\Psi\|_{L^2(\Mtu)}^2 + \|\mathcal{P}^{\leq N-1} (\mathcal{C}, \mathcal{D}) \|_{L^2(\Mtu)}^2 \\
\ls &\: \|L\mathcal{P}^{[1,N]}\Psi\|_{L^2(\Mtu)}^2 + \|\bX \mathcal{P}^{[1,N]}\Psi\|_{L^2(\Mtu)}^2 + \int_{u'=0}^{u'=u} [\mathbb C_{\leq N-1} + \mathbb D_{\leq N-1}](t,u')\, du' \\
\ls &\: \|L\mathcal{P}^{[1,N]}\Psi\|_{L^2(\Mtu)}^2 + \|\bX \mathcal{P}^{[1,N]}\Psi\|_{L^2(\Mtu)}^2 + \epd^3 \max\{1, \upmu_{\star}^{-2\toprate + 2\Ntop - 2N + 2.8}(t) \}.
\end{split}
\end{equation}

Finally, we observe that RHS~\eqref{eq:est.C.in.wave.1} and RHS~\eqref{eq:est.C.in.wave.2} are 
$\leq {\mbox{\upshape RHS}~\eqref{eq:Psi.null.condition}}$. We have therefore proved \eqref{eq:wave.inho.main.1}.

\pfstep{Step~2: Proof of \eqref{eq:wave.inho.main.2}} Returning to \eqref{eq:est.C.in.wave.main}, 
we again note that we only have to consider terms 
not already controlled in Proposition~\ref{prop:wave.inho.null}. 

Applying Propositions~\ref{prop:geometric.low}, \ref{prop:mus.int}, \ref{prop:C.D}, and \ref{prop:elliptic.putting.everything.together}, 
we have:
\begin{equation*}
\begin{split}
&\: \int_{t'=0}^{t'=t} \upmu_{\star}^{-\f 32}(t') \left\{ \int_{s=0}^{s=t'} [\|\upmu \mathcal{P}^{\leq \Ntop} (\mathcal{C}, \mathcal{D})\|_{L^2(\Sigma_s)} +  \|\mathcal{P}^{\leq \Ntop-1} (\mathcal{C}, \mathcal{D})\|_{L^2(\Sigma_s)}]\, ds\right\}^2 \, dt' \\
\ls &\:  \int_{t'=0}^{t'=t} \upmu_{\star}^{-\f 32}(t') 
\left\{ \int_{s=0}^{s=t'} [\mathbb C_{\leq \Ntop}^{\f 12}
+ 
\mathbb D_{\leq \Ntop}^{\f 12}](s) 
+ 
\f1{\upmu_{\star}^{\f 12}(s)}[\mathbb C_{\leq \Ntop-1}^{\f 12} 
+ 
\mathbb D_{\leq \Ntop-1}^{\f 12}](s)\, ds
\right\}^2 \, dt' \\
\ls &\:  \epd^3\int_{t'=0}^{t'=t} \upmu_{\star}^{-\f 32}(t') \left\{ \int_{s=0}^{s=t'} 
 \upmu_{\star}^{-\toprate+0.4}(s) \, ds \right\}^2 \, dt' \\
\ls &\: \epd^3 \upmu_{\star}^{-2\toprate+2.3}(t) 
\ls 
\epd^2 \upmu_{\star}^{-2\toprate+1.8}(t),
\end{split}
\end{equation*}
which is therefore bounded above by the RHS of \eqref{eq:Psi.null.condition.top}. \qedhere

\end{proof}

\subsection{Putting everything together}\label{sec:wave.everything}

\begin{proposition}[\textbf{Main $L^2$ estimates for the wave variables}]\label{prop:wave.final}
For $1\leq N \leq \Ntop$, the following holds for all $(t,u) \in [0,\Tboot)\times [0,U_0]$:
\begin{equation}\label{eq:wave.est.goal}
\begin{split}
\mathbb{W}_{[1,N]}(t,u) \ls \epd^2 \max\{1, \upmu_{\star}^{-2 \toprate + 2\Ntop-2N+1.8}(t)\}.
\end{split}
\end{equation}
\end{proposition}
\begin{proof}
We first use the pointwise bounds \eqref{eq:derivative.of.wave.null}, \eqref{eq:wave.inho.linear.3}, \eqref{eq:est.C.in.wave.main}, 
the bootstrap assumptions \eqref{BA:W.Li.small}--\eqref{BA:C.D},
and Proposition~\ref{prop:geometric.low}
to deduce that the assumption \eqref{eq:F.smallness} in Proposition~\ref{prop:wave} on the inhomogeneous terms $\vec{\mathfrak G}$,
i.e., the terms on the RHSs of \eqref{E:VELOCITYWAVEEQUATION}--\eqref{E:ENTROPYWAVEEQUATION},
is satisfied. 
Hence, the results of Proposition~\ref{prop:wave} are valid, and 
we will use them throughout the rest of this proof.
We will also silently use the basic fact that 
$\upmu_{\star}(t,u) \leq 1$ and $\upmu_{\star}(t) \leq 1$; 
see Definition~\ref{def:mustar}.

\pfstep{Step~1: $N = \Ntop$} 
By the top- and penultimate- order general wave estimates \eqref{eq:main.EE.prop.top} in Proposition~\ref{prop:wave}, the initial data assumptions in \eqref{assumption:support}, \eqref{assumption:R+}--\eqref{assumption:small},  
and the bounds for the inhomogeneous terms in Propositions~\ref{prop:wave.inho.null}--\ref{prop:wave.inho.main}, we obtain
the following bound for any $\varsigma \in (0,1]$
(with implicit constants that are independent of $\varsigma$): 
\begin{equation}\label{eq:wave.est.final.0}
\begin{split}
&\: \sup_{\hat{t} \in [0,t]} \upmu_{\star}^{2\toprate - 1.8}(\hat{t})
\left(\mathbb{E}_{[1,\Ntop]}(\hat{t},u) + \mathbb{F}_{[1,\Ntop]}(\hat{t},u) + \mathbb K_{[1,\Ntop]}(\hat{t},u) \right) \\
&\: + \sup_{\hat{t} \in [0,t]} \upmu_{\star}^{2\toprate - 3.8}(\hat{t})
\left(\mathbb{E}_{[1,\Ntop-1]}(\hat{t},u) + \mathbb{F}_{[1,\Ntop-1]}(\hat{t},u) + \mathbb K_{[1,\Ntop-1]}(\hat{t},u) \right)  \\
\ls &\:  \epd^2 +  \sup_{\hat{t}\in [0,t]} \upmu_{\star}^{2\toprate-1.8}(\hat{t}) \int_{t'=0}^{t'=\hat{t}} \upmu_{\star}^{-\f 32}(t') \left\{\int_{s=0}^{s=t'} \|\mathcal{P}^{[1,\Ntop]}\mathfrak G\|_{L^2(\Sigma_s)} \, ds\right\}^2 \, dt'  \\
&\: + \sup_{\hat{t}\in [0,t]} \upmu_{\star}^{2\toprate-1.8}(\hat{t}) \|(|L\mathcal{P}^{[1,\Ntop]} \Psi| + |\bX \mathcal{P}^{[1,\Ntop]} \Psi|) |\mathcal{P}^{[1,\Ntop]} \mathfrak G| \|_{L^1(\mathcal M_{\hat{t},u})} \\
&\:  + \sup_{\hat{t}\in [0,t]} \upmu_{\star}^{2\toprate-3.8}(\hat{t}) \|(|L\mathcal{P}^{[1,\Ntop-1]} \Psi| + |\bX \mathcal{P}^{[1,\Ntop-1]} \Psi|) |\mathcal{P}^{[1,\Ntop-1]} \mathfrak G| \|_{L^1(\mathcal M_{\hat{t},u})} \\
\ls &\: \epd^2  + \sup_{\hat{t}\in [0,t]} \upmu_{\star}^{2\toprate-1.8}(\hat{t}) \int_{t'=0}^{t'=t} \upmu_{\star}^{-\f 32}(t',u) \left\{ \int_{s=0}^{s=t'} \upmu_{\star}^{-\f 12}(s)\mathbb{E}^{\f 12}_{[1,\Ntop]}(s)  \, ds\right\}^2 \, dt' \\
&\: + \sup_{\hat{t}\in [0,t]} \upmu_{\star}^{2\toprate - 1.8}(\hat{t}) \Bigg\{ \varsigma \mathbb K_{[1,\Ntop]}(\hat{t},u)  \\
&\: \qquad\qquad\qquad\qquad + (1+\varsigma^{-1}) \left(\int_{t'=0}^{t'=\hat{t}} \mathbb{E}_{[1,\Ntop]}(t',u) \, dt'  + \int_{u'=0}^{u'=u} \mathbb{F}_{[1,\Ntop]}(\hat{t},u')\, du'\right) \Bigg\} \\
&\:  + \sup_{\hat{t}\in [0,t]} \upmu_{\star}^{2\toprate - 3.8}(\hat{t}) 
\Bigg\{ \varsigma \mathbb K_{[1,\Ntop-1]}(\hat{t},u) \\
&\: \qquad\qquad\qquad + (1+\varsigma^{-1}) \left(\int_{t'=0}^{t'=\hat{t}} \mathbb{E}_{[1,\Ntop-1]}(t',u) \, dt'  + \int_{u'=0}^{u'=u} \mathbb{F}_{[1,\Ntop-1]}(\hat{t},u')\, du'\right) \Bigg\}.
\end{split}
\end{equation}

We now argue as follows using \eqref{eq:wave.est.final.0}:
\begin{itemize}
\item We choose $\varsigma>0$ sufficiently small and absorb the terms 
$$\varsigma \sup_{\hat{t}\in [0,t]} \upmu_{\star}^{2\toprate - 1.8}(\hat{t})\mathbb K_{[1,\Ntop]}(\hat{t}, u),\quad \varsigma \sup_{\hat{t}\in [0,t]} \upmu_{\star}^{2\toprate - 3.8}(\hat{t})\mathbb K_{[1,\Ntop-1]}(\hat{t}, u)$$ 
appearing on the RHS 
by the terms 
$$\sup_{\hat{t}\in [0,t]} \upmu_{\star}^{2\toprate - 1.8}(\hat{t})\mathbb K_{[1,\Ntop]}(\hat{t}, u),\quad \sup_{\hat{t}\in [0,t]} \upmu_{\star}^{2\toprate - 3.8}(\hat{t})\mathbb K_{[1,\Ntop-1]}(\hat{t}, u)$$ 
on the LHS.
\item We then apply Proposition~\ref{prop:almost.monotonicity}
(using that the exponents $2\toprate - 1.8$ and $2\toprate - 3.8$ are positive) and
Gr\"onwall's inequality to handle the terms involving the integrals of $\mathbb{E}$ and $\mathbb{F}$.
\end{itemize}

This leads to the following estimate (where on the LHS, we have dropped the below-top-order energies):
\begin{align} 
\begin{split} \label{ANOTHEReq:wave.est.final.0}
	&
	\sup_{\hat{t} \in [0,t]} \upmu_{\star}^{2\toprate - 1.8}(\hat{t})\left(\mathbb{E}_{[1,\Ntop]}
	(\hat{t},u) + \mathbb{F}_{[1,\Ntop]}(\hat{t},u) + \mathbb K_{[1,\Ntop]}(\hat{t},u) \right)
		\\
	& \lesssim \epd^2  + 
	\sup_{\hat{t}\in [0,t]} \upmu_{\star}^{2\toprate-1.8}(\hat{t}) \int_{t'=0}^{t'=t} 
	\upmu_{\star}^{-\f 32}(t',u) \left\{ \int_{s=0}^{s=t'} 
	\upmu_{\star}^{-\f 12}(s)\mathbb{E}^{\f 12}_{[1,\Ntop]}(s)  
	\, ds\right\}^2 \, dt'.
\end{split}
\end{align}

We will now apply a further Gr\"onwall-type argument to \eqref{ANOTHEReq:wave.est.final.0}.
Define:
$$\iota(t) \doteq \exp\left( \int_{s=0}^{s=t} \upmu_{\star}^{-0.9}(s)\, ds \right),$$
and, for a large $\mathfrak C>0$ to be chosen later:
$$H(t)\doteq \sup_{\hat t \in [0,t]} \iota^{-2\mathfrak C}(\hat t) \upmu_{\star}^{2\toprate-1.8}(\hat t) \mathbb{E}_{[1,\Ntop]}(\hat t).$$

{From the definitions 
of 
$\mathbb{E}_{[1,\Ntop]}$,
$\iota$, 
and
$H$,
the fact that $\iota$ is increasing,
and the estimate \eqref{ANOTHEReq:wave.est.final.0}, we find that} 
there exists a constant\footnote{We call the constant $C_{**}$ so as to make the notation clearer later in the proof.} 
$C_{**}>0$ \textbf{independent of} $\mathfrak C > 0$ so that:
\begin{equation}\label{eq:wave.est.final.1}
\begin{split}
H(t)
\leq &\:  C_{**} \left(\epd^2  + \sup_{\hat t\in [0,t]}\upmu_{\star}^{2\toprate -1.8}(\hat t) \iota^{-2\mathfrak C}(\hat{t})\int_{t'=0}^{t'=\hat t} \upmu_{\star}^{-\f 32}(t') \left\{ \int_{s=0}^{s=t'} 
	 \upmu_{\star}^{-\f 12}(s)
	\mathbb{E}^{\f 12}_{[1,\Ntop]}(s) \, ds\right\}^2 \, dt' \right).
\end{split}
\end{equation}

Before we proceed, note that for $n=1,2$, an easy change of variables gives:
\begin{equation}\label{eq:stupid.change.of.variables}
\int_{s=0}^{s=t'} \iota^{n \mathfrak C}(s) \upmu_{\star}^{-0.9}(s)\,ds =\int_{y=0}^{y=\int_{\tau=0}^{\tau=t'} \upmu_{\star}^{-0.9}(\tau)\, d\tau}  e^{n \mathfrak C y}\, dy \leq 
\f{1}{n \mathfrak C}\iota^{n \mathfrak C}(t').
\end{equation}

Fix $t\in [0,\Tboot)$ and $\hat t\in [0,t]$. Since $\iota^{-\mathfrak C}$ is decreasing and $\upmu_{\star}$ is almost decreasing by Proposition~\ref{prop:almost.monotonicity}, we have, using \eqref{eq:stupid.change.of.variables}
and the estimate \eqref{eq:wave.est.final.1} for $H$, the following bound
for the terms under the sup on RHS~\eqref{eq:wave.est.final.1}:
\begin{equation}\label{eq:wave.est.final.2}
\begin{split}
&\: \upmu_{\star}^{2\toprate - 1.8}(\hat t) \iota^{- 2\mathfrak C}(\hat t) \int_{t'=0}^{t'= \hat t} \upmu_{\star}^{-\f 32}(t') \left\{ \int_{s=0}^{s=t'} \upmu_{\star}^{-\f 12}(s) \mathbb{E}^{\f 12}_{[1,\Ntop]}(s)  \, ds\right\}^2 \, dt' \\
\leq &\: \upmu_{\star}^{2\toprate - 1.8}(\hat t)\iota^{-2 \mathfrak C}(\hat t) \\
&\: \times \int_{t'=0}^{t'=\hat t} 
\upmu_{\star}^{-\f 32}(t') \left\{ \int_{s=0}^{s=t'} \upmu_{\star}^{- \toprate + 1.3}(s) 
\left[\iota^{\mathfrak C}(s) \upmu_{\star}^{-0.9}(s)\right] \left[\iota^{- \mathfrak C}(s) \upmu_{\star}^{\toprate - 0.9}(s)
	\mathbb{E}^{\f 12}_{[1,\Ntop]}(s)\right]  \, ds\right\}^2 \, dt' \\
\leq &\: 
2^{2 \toprate - 2.6}
\upmu_{\star}^{2\toprate - 1.8}(\hat t) \iota^{-2 \mathfrak C}(\hat t) 
\int_{t'=0}^{t'=\hat t} 
\upmu_{\star}^{-2 \toprate + 1.1}(t') 
	\left\{ \int_{s=0}^{s=t'}  	
		\iota^{\mathfrak C}(s) \upmu_{\star}^{-0.9}(s) 
		H^{\f 12}(s)  
	\, ds\right\}^2 \, dt' \\
\leq &\: 
2^{2 \toprate - 2.6}
\upmu_{\star}^{2\toprate - 1.8}(\hat t) 
\iota^{-2 \mathfrak C}(\hat t) 
H(t)
\int_{t'=0}^{t'=\hat t} 
\upmu_{\star}^{-2 \toprate + 1.1}(t') 
	\left\{ \int_{s=0}^{s=t'}  	
		\iota^{\mathfrak C}(s) \upmu_{\star}^{-0.9}(s) 
	\, ds\right\}^2 \, dt' \\
\leq &\: 
2^{2 \toprate - 2.6}
\upmu_{\star}^{2\toprate - 1.8}(\hat t) 
 \iota^{-2 \mathfrak C}(\hat t)
\frac{H(t)}{\mathfrak{C}^2}
\int_{t'=0}^{t'=\hat t} 
	\left[\iota^{2 \mathfrak C}(t') \upmu_{\star}^{-0.9}(t') \right]
	\upmu_{\star}^{-2 \toprate + 2}(t') 
\, dt'
	\\
\leq &\: 
2^{4 \toprate - 4.6}
\upmu_{\star}^{0.2}(\hat t) 
\iota^{-2 \mathfrak C}(\hat t) 
\frac{H(t)}{\mathfrak{C}^2}
\int_{t'=0}^{t'=\hat t} 
	\iota^{2 \mathfrak C}(t') \upmu_{\star}^{-0.9}(t')
\, dt'
	\\
\leq &\: 
2^{4 \toprate - 5.6}
\upmu_{\star}^{0.2}(\hat{t}) 
\frac{H(t)}{\mathfrak{C}^3}
\leq
2^{4 \toprate - 5.6}
\frac{H(t)}{\mathfrak{C}^3}.
\end{split}
\end{equation}

Plugging \eqref{eq:wave.est.final.2} into \eqref{eq:wave.est.final.1}, 
we obtain:
\begin{equation}\label{eq:wave.est.final.3}
\begin{split}
H(t)
\leq &\:  C_{**} 
\left\lbrace
\epd^2 
+ 
2^{4 \toprate - 5.6}
\frac{H(t)}{\mathfrak{C}^3}
\right\rbrace.
\end{split}
\end{equation}
Choosing $\mathfrak C>0$ sufficiently large 
such that $\frac{2^{4 \toprate - 5.6}}{\mathfrak{C}^3} \leq \f 12$,
we immediately infer
from \eqref{eq:wave.est.final.3} that $H(t) \leq 2C_{**} \epd^2$. 
From this estimate, 
\eqref{eq:wave.est.final.2}
the definition of $\iota(t)$,
and the estimate \eqref{eq:mus.int.2},
we find that 
$\mbox{RHS~\eqref{ANOTHEReq:wave.est.final.0}} \leq C \epd^2$,
where $C$ is allowed to depend on $\mathfrak{C}$.
From this estimate and
the definition of
$\mathbb{W}_{[1,N]}(t,u)$, we conclude
\eqref{eq:wave.est.goal} in the case $N = \Ntop$. 

\pfstep{Step~2: $1\leq N \leq \Ntop -1$} Let $1\leq N \leq \Ntop-1$. Arguing
like we did at the 
beginning of Step~1, except for using \eqref{eq:main.EE.prop.low} instead of \eqref{eq:main.EE.prop.top}, we obtain:
\begin{equation}\label{eq:wave.est.for.induction}
\begin{split}
&\: \mathbb{E}_{[1,N]}(t,u) + \mathbb{F}_{[1,N]}(t,u) + \mathbb K_{[1,N]}(t,u) \\
\ls &\:  \epd^2 \max\{1, \upmu_{\star}^{-2\toprate + 2\Ntop -2N +1.8} (t)\} 
	\\
&\: + \max\{1, \upmu_{\star}^{-2\toprate+2\Ntop-2N+1.8}(t) \} \left(\sup_{s\in [0,t]} 
\min\{1,\upmu_{\star}^{2 \toprate-2\Ntop+2N+0.2}(s)\} \mathbb{Q}_{[1,N+1]}(s)\right) \\ 
&\: + \varsigma \mathbb K_{[1,N]} + (1+\varsigma^{-1}) \left(\int_{t'=0}^{t'=t} \mathbb{E}_{[1,N]}(t',u) \, dt'  + \int_{u'=0}^{u'=u} \mathbb{F}_{[1,N]}(t,u')\, du' \right) \\
\ls &\: \epd^2 \max\{1, \upmu_{\star}^{-2\toprate + 2\Ntop -2N +1.8} (t)\} 
	\\
 &\: + \max\{1, \upmu_{\star}^{-2\toprate+2\Ntop-2N+1.8}(t) \} \left(\sup_{s\in [0,t]} \min\{1,\upmu_{\star}^{2\toprate-2\Ntop+2N+0.2}(s)\} \mathbb{Q}_{[1,N+1]}(s)\right),
\end{split}
\end{equation}
where to obtain the last inequality, we first took $\varsigma$ to be sufficiently small to absorb $\varsigma \mathbb K_{[1,N]}$, and then used Gr\"onwall's inequality.

Using \eqref{eq:wave.est.for.induction}, we easily obtain \eqref{eq:wave.est.goal} by induction in decreasing $N$. Notice in particular that the base case $N = \Ntop$ has already been proven in Step~1. \qedhere
\end{proof}

\section{Proving the $L^\infty$ estimates}\label{sec:Linfty}
We continue to work under the assumptions of Theorem~\ref{thm:bootstrap}.

In this section{, we derive $L^\infty$ estimates that in particular yield an improvement 
over the bootstrap assumptions we made in Section~\ref{SS:BOOTSTRAPASSUMPTIONS}.}
This is the final section in which we derive PDE estimates that are 
needed for the proof of Theorem~\ref{thm:bootstrap};
aside from the appendix, the rest of the paper (i.e., Section~\ref{sec:everything}) 
entails deriving consequences of the estimates and assembling the logic of the proof.} 
We first bound (in Propositions~\ref{prop:Linfty.est}, \ref{prop:Linfty.est.1}) the $L^\infty$ norm of the fluid variables, specific vorticity, entropy gradient and modified fluid variables and their $\mathcal{P}$ derivatives using the energy estimates we have already obtained and Sobolev embedding (Lemma~\ref{lem:Sobolev}). 
Then, in Propositions~\ref{prop:Linfty.est.1} and \ref{prop:trans.V.S}, 
we control derivatives of these variables  
that involve one factor of $\bX$ 
by combining the just-obtained 
$L^\infty$-estimates for
$\mathcal{P}$-derivatives
with the (wave or transport) equations.

\begin{lemma}[\textbf{Sobolev embedding estimates}]
\label{lem:Sobolev}
Suppose $\phi$ is a smooth function with $u$-support in 
$[0,U_0]$. Then, for every $t\in [0,\Tboot)$,
we have the following estimate:
\begin{align}\label{eq:stupid.embedding}
\|\phi\|_{L^\infty(\Sigma_t)}
&
\ls
\sup_{u \in [0,U_0]} \| \mathcal{P}^{\leq 2} \phi\|_{L^2(\ell_{0,u})} 
+ 
\sup_{u \in [0,U_0]} \|\Lunit \mathcal{P}^{\leq 2} \phi\|_{L^2(\mathcal F_u^t)}. 
\end{align}
\end{lemma}
\begin{proof}
First, using standard Sobolev embedding on $\mathbb T^2$, 
using 
\eqref{E:GEOMETRIC2COORDINATEPARTIALDERIVATIVESINTERMSOFOTHERVECTORFIELDS}--\eqref{E:GEOMETRIC3COORDINATEPARTIALDERIVATIVESINTERMSOFOTHERVECTORFIELDS}
to express $\srd_2,\srd_3$ in terms of derivatives with respect to $\{Y,Z\}$,
comparing the volume forms using Definition~\ref{D:NONDEGENERATEVOLUMEFORMS},
and using the estimates of Proposition~\ref{prop:geometric.low.2}, 
we deduce: 
\begin{align}\label{E:TORUSSOB}
 \|\phi\|_{L^\infty(\ell_{t,u})} 
	& 
	\ls
	\sum_{i+j\leq 2}\left(\int_{\ell_{t,u}} |\slashed \rd_2^i \slashed \rd_3^j  \phi|^2\, dx^2 \, dx^3 \right)^{\f 12} 
	\ls 
	\left(\int_{\ell_{t,u}} |\mathcal{P}^{\leq 2} \phi|^2 \, d \lambda_{\slashed g} \right)^{\f 12}
	\doteq 
	\| \mathcal{P}^{\leq 2} \phi \|_{L^2(\ell_{t,u})}.
 \end{align}
To complete the proof of \eqref{eq:stupid.embedding}, 
it remains only for us to control RHS~\eqref{E:TORUSSOB} by
showing that
for any smooth function $\varphi$ 
(where the role of $\varphi$ will be played by $\mathcal{P}^{\leq 2} \phi$), we have:
\begin{align} \label{E:L2FTCALONGNULLHYPERSURFACES}
	\| \varphi \|_{L^2(\ell_{t,u})}
	& 
	\leq 
	C
	\| \varphi \|_{L^2(\ell_{0,u})}
	+
	C \| L \varphi\|_{L^2(\mathcal F_u^t)}.
\end{align}
To prove \eqref{E:L2FTCALONGNULLHYPERSURFACES},
we start by using the identity $\srd_t = \Lunit - \Lunit^A \srd_A$ 
(see \eqref{eq:slashes.1})
to deduce that:
\begin{align} \label{E:FIRSTSTEPL2FTCALONGNULLHYPERSURFACES}
	\frac{\partial}{\partial t}
	\int_{\ell_{t,u}} \varphi^2 \, dx^2 \, dx^3
	& =
	2
	\int_{\ell_{t,u}} \varphi \srd_t \varphi \, dx^2 \, dx^3
	=
	2
	\int_{\ell_{t,u}} \varphi \Lunit \varphi \, dx^2 \, dx^3
	-
	2
	\int_{\ell_{t,u}} \varphi \Lunit^A \srd_A \varphi \, dx^2 \, dx^3
		\\
& = 2
		\int_{\ell_{t,u}} \varphi \Lunit \varphi \, dx^2 \, dx^3
		+
		\int_{\ell_{t,u}} \varphi^2 (\srd_A \Lunit^A) \, dx^2 \, dx^3,
	\notag
\end{align}
where in the last step, we integrated the geometric coordinate partial derivatives $\srd_A$ by parts
(and we recall that capital Latin indices vary over $2,3$).
Again using 
\eqref{E:GEOMETRIC2COORDINATEPARTIALDERIVATIVESINTERMSOFOTHERVECTORFIELDS}--\eqref{E:GEOMETRIC3COORDINATEPARTIALDERIVATIVESINTERMSOFOTHERVECTORFIELDS} 
to express $\srd_2,\srd_3$ in terms of derivatives with respect to $\{Y,Z\}$,
and using the estimates of Propositions~\ref{prop:geometric.low} and \ref{prop:geometric.low.2},
we find that $|\srd_A \Lunit^A| \leq C$.
From this estimate, 
\eqref{E:FIRSTSTEPL2FTCALONGNULLHYPERSURFACES},
and Young's inequality,
we deduce that:
\begin{align} \label{E:SECONDSTEPL2FTCALONGNULLHYPERSURFACES}
	\left|
		\frac{\partial}{\partial t}
		\int_{\ell_{t,u}} \varphi^2 \, dx^2 \, dx^3
	\right|
	& \leq
		C
		\int_{\ell_{t,u}} |\Lunit \varphi|^2 \, dx^2 \, dx^3
		+
		C
		\int_{\ell_{t,u}} \varphi^2 \, dx^2 \, dx^3.
\end{align}
Integrating \eqref{E:SECONDSTEPL2FTCALONGNULLHYPERSURFACES} 
with respect to time, using the fundamental theorem of calculus,
and then applying Gr\"onwall's inequality, 
we find that:
\begin{align} \label{E:THIRDSTEPL2FTCALONGNULLHYPERSURFACES}
	\int_{\ell_{t,u}} \varphi^2 \, dx^2 \, dx^3
	& \leq 
		C \int_{\ell_{0,u}} \varphi^2 \, dx^2 \, dx^3
		+
		C
		\int_{t'=0}
			\int_{\ell_{t',u}} |\Lunit \varphi|^2 \, dx^2 \, dx^3
		\, dt'.
\end{align}
Again comparing the volume forms using Definition~\ref{D:NONDEGENERATEVOLUMEFORMS}
and using the estimates of Proposition~\ref{prop:geometric.low.2},
we arrive at the desired bound \eqref{E:L2FTCALONGNULLHYPERSURFACES}. \qedhere
\end{proof}

\begin{proposition}\label{prop:Linfty.est}
The following $L^\infty$ estimates hold for all $t\in [0,\Tboot)$:
\begin{equation}\label{eq:Linfty.est.1}
\|\mathcal{P}^{[1,\Ntop - \toprate -2]} \Psi \|_{L^\infty(\Sigma_t)} \ls \epd, 
\end{equation}
\begin{equation}\label{eq:Linfty.est.2}
\|\mathcal{P}^{\leq \Ntop - \toprate -2} (\Vr, S) \|_{L^\infty(\Sigma_t)} 
+ 
\|\mathcal{P}^{\leq \Ntop - \toprate -3} (\mathcal{C}, \mathcal{D}) \|_{L^\infty(\Sigma_t)} \ls \epd^{\f 32}.
\end{equation}
\end{proposition}
\begin{proof}
{These two estimates follow as immediate consequences}
of the energy estimates (respectively for $(\mathbb V, \mathbb S)$, $(\mathbb C, \mathbb D)$ and $\mathbb{W}$) in Propositions~\ref{prop:V.S}, \ref{prop:C.D}, and \ref{prop:wave.final}{, Lemma~\ref{lem:Sobolev}, and the initial data size-assumptions \eqref{assumption:small}--\eqref{assumption:very.small.modified}.} 
\qedhere
\end{proof}

\begin{proposition}\label{prop:Linfty.est.1}
The following $L^\infty$ estimates hold for all $t\in [0,\Tboot)$:
\begin{subequations}
\begin{align}
\label{eq:Psi.itself.Li}
 \|\mathcal R_{(+)}\|_{L^\infty(\Sigma_t)} \ls_{\mydiam} \mathring{\upalpha},\quad 
\|(\mathcal R_{(-)},v^2,v^3,s)\|_{L^\infty(\Sigma_t)} \ls \epd, \\
\label{eq:XPsi.Li}
\| \bX \mathcal R_{(+)} \|_{L^\infty(\Sigma_t)} \leq 2\mathring{\updelta}, 
\quad \| \bX (\mathcal R_{(-)},v^2,v^3,s) \|_{L^\infty(\Sigma_t)} \ls \epd, \\
\label{eq:PXPsi.Li}
 \|\mathcal{P}^{ [1, \Ntop - \toprate-4]} \bX \Psi\|_{L^\infty(\Sigma_t)} \ls \epd.
\end{align}
\end{subequations}
\end{proposition}
\begin{proof}
\pfstep{Step~1: Proof of \eqref{eq:Psi.itself.Li}} 

Since $\Lunit t = 1$, we can apply the fundamental theorem of calculus along the integral curves of $\Lunit$ to deduce
that for any scalar function $\phi$, we have:
\begin{align} \label{E:SIMPLEFTCESTIMATE}
	\|\phi \|_{L^\infty(\Sigma_t)} 
	& \leq
		\|\phi \|_{L^\infty(\Sigma_0)} 
	+
	\int_{t'=0}^t
		\|L \phi \|_{L^\infty(\Sigma_{t'})}
	\, dt'.
\end{align}
By Proposition~\ref{prop:Linfty.est}, we have $\|L \Psi\|_{L^\infty(\Sigma_t)} \ls \epd$. 
From this estimate,
the data assumptions \eqref{assumption:R+} and \eqref{assumption:small.trans},
and \eqref{E:SIMPLEFTCESTIMATE} with $\phi \doteq \Psi$,
we conclude the desired bounds in \eqref{eq:Psi.itself.Li}.

\pfstep{Step~2: An auxiliary estimate for $\mytr \upchi$} We need an auxiliary estimate before proving \eqref{eq:XPsi.Li}.
To start, we note that the same arguments used to prove
Proposition~\ref{prop:geometric.low}, based on the transport equation\footnote{Note importantly that the RHS of \eqref{E:LLUNITI} does not contain an $\bX\Psi$ term!} \eqref{E:LLUNITI}, but now with the 
estimate \eqref{eq:Linfty.est.1} in place of the $L^{\infty}$ bootstrap assumptions 
for $\|\mathcal{P}^{[1,\Ntop - \toprate -2]} \Psi \|_{L^\infty(\Sigma_t)}$
in \eqref{assumption:small},
yield the following estimate:
\begin{equation}\label{eq:L.improved}
\|\mathcal{P}^{[1,\Ntop- \toprate-3]} L^i\|_{L^\infty(\Sigma_t)}\ls \epd.
\end{equation}
We next use 
Lemma~\ref{lem:slashed},
Lemma~\ref{lem:induced.metric},
and the fact that the Cartesian component functions $X^1,X^2,X^3$ are smooth functions of
the $L^i$ and $\Psi$ (see \eqref{E:TRANSPORTVECTORFIELDINTERMSOFLUNITANDRADUNIT})
to write the identity \eqref{E:TRCHIINTERMSOFOTHERVARIABLES}
in the following form, where ``$\smoothfunction$'' schematically denotes smooth functions:
$\mytr \upchi 
= 
\smoothfunction(\Lunit^i,\Psi)\mathcal{P} \Lunit^i + \smoothfunction(\Lunit^i,\Psi)\mathcal{P} \Psi$.
From this equation,
the estimates of Proposition~\ref{prop:Linfty.est},
\eqref{eq:Psi.itself.Li},
and
\eqref{eq:L.improved},
we obtain the desired auxiliary estimate:
\begin{equation}\label{eq:chi.smaller}
\|\mathcal{P}^{\leq \Ntop- \toprate-4} \mytr \upchi \|_{L^\infty(\Sigma_t)} \ls \epd.
\end{equation}

\pfstep{Step~3: Controlling $\mathcal{P}^{\leq \Ntop- \toprate-4} L \bX \Psi$} By \cite[Proposition~2.16]{jSgHjLwW2016}, the wave operator is given by:\footnote{Here, $\slashed{\Delta}$ is the Laplace--Beltrami operator on $\ell_{t,u}$, which can be expressed as a second order differential operator in $Y$ and $Z$ with regular coefficients.} 
\begin{equation}\label{eq:wave.in.terms.of.geometric}
\begin{split}
\upmu \square_{g(\threePsi)} f 
			& = - \Lunit(\upmu \Lunit f + 2 \Rad f)
				+ \upmu \angLap f
				- \mytr \upchi \Rad f
				- \upmu \mytr \angk \Lunit f
				- 2 \upmu \upzeta^{\#} \cdot \angdiff f.
				\end{split}
				\end{equation}
				
Consider now the wave equations \eqref{E:VELOCITYWAVEEQUATION}--\eqref{E:ENTROPYWAVEEQUATION}. We will now bound the inhomogeneous terms in these equations. For $\mathfrak Q\in \{\mathfrak{Q}_{(v)}^i,\mathfrak Q_{(\pm)},\mathfrak Q_{(s)}\}$, we first apply \eqref{eq:null.form.basic} with $\phi^{(1)},\,\phi^{(2)} = \Psi$, $\mathfrak d^{(1,1)},\, \mathfrak d^{(2,1)} \ls 1$, $\mathfrak d^{(1,2)},\, \mathfrak d^{(2,2)} \ls \epd$ (which is justified by Proposition~\ref{prop:Linfty.est} 
and the bootstrap assumptions \eqref{BA:LARGERIEMANNINVARIANTLARGE}--\eqref{BA:W.Li.small}), and then use \eqref{BA:LARGERIEMANNINVARIANTLARGE}--\eqref{BA:W.Li.small} and Propositions~\ref{prop:geometric.low} and \ref{prop:Linfty.est} to obtain:
\begin{equation}\label{eq:derivative.of.wave.null.Linfty}
\begin{split}
&\:  |\mathcal{P}^{\leq \Ntop- \toprate-4} (\upmu \mathfrak Q)| \\
\ls &\: |\mathcal{P}^{[1,\Ntop- \toprate-3]}\Psi| 
+ 
\epd \left\{ |\mathcal{P}^{[1,\Ntop- \toprate-4]} \bX \Psi| 
+ 
|\mathcal{P}^{[2,\Ntop- \toprate-4]} (\upmu, L^i)| \right\} \ls \epd.
\end{split}
\end{equation}
For $\mathfrak L \in \{ \mathfrak L_{(v)}^i,\,\mathfrak L_{(\pm)},\,\mathfrak L_{(s)} \}$ and $\mathfrak M \in \{ c^2 \exp(2\Densrenormalized) \mathcal{C}^i, \, \Speed \exp(\rr) \frac{p_{;s}}{\bar{\varrho}} \mathcal{D},\,\Speed^2 \exp(2\rr) \DivofEntrenormalized \}$, we use the pointwise bounds \eqref{eq:wave.inho.linear.3}, \eqref{eq:est.C.in.wave.main} together with \eqref{BA:LARGERIEMANNINVARIANTLARGE}--\eqref{BA:C.D}
and Propositions~\ref{prop:geometric.low} and \ref{prop:Linfty.est}
to obtain:
\begin{equation}\label{eq:derivative.of.wave.other.Linfty}
\begin{split}
  |\mathcal{P}^{\leq \Ntop- \toprate-4} (\upmu \mathfrak L)| 
	+ 
	|\mathcal{P}^{\leq \Ntop- \toprate-4} (\upmu \mathfrak M)|  
	\ls 
	\epd.
\end{split}
\end{equation}
Combining \eqref{eq:derivative.of.wave.null.Linfty} and \eqref{eq:derivative.of.wave.other.Linfty}, we thus obtain: 
\begin{equation}\label{eq:Box.Psi.est.Li}
|\mathcal{P}^{\leq \Ntop- \toprate-4} (\upmu  \square_g \Psi)| \ls \epd.
\end{equation}
We now use \eqref{eq:Box.Psi.est.Li} together with \eqref{eq:wave.in.terms.of.geometric} to control 
$\mathcal{P}^{\leq \Ntop- \toprate-4} L\bX\Psi$. 
 The key point is that every term in 
$\mathcal{P}^{\leq \Ntop- \toprate-4} \mbox{\eqref{eq:wave.in.terms.of.geometric}}$
except for $\mathcal{P}^{\leq \Ntop- \toprate-4} (-2L\bX \Psi)$
is already known to be bounded in $L^{\infty}$ by $\mathcal{O}(\epd)$.
More precisely, we express the Ricci coefficients 
{on RHS~\eqref{eq:wave.in.terms.of.geometric}}
using \eqref{E:TRCHIINTERMSOFOTHERVARIABLES}--\eqref{eq:upzeta.def} and $\slashed\Delta$ using 
{Lemmas~\ref{L:GEOMETRICCOORDINATEVECTORFIELDSINTERMSOFCARTESIANVECTORFIELDS} and \ref{lem:induced.metric}. We also use the transport equation \eqref{E:UPMUFIRSTTRANSPORT} 
to express\footnote{This step is needed to avoid having to control $\Ntop - \toprate-3$ $\mathcal{P}$-derivatives
of $\upmu$ in $L^{\infty}$, since Proposition~\ref{prop:geometric.low} does not yield $L^{\infty}$
control of that many derivatives of $\upmu$.} 
the factor of $L \upmu$ on RHS~\eqref{eq:wave.in.terms.of.geometric} as RHS~\eqref{E:UPMUFIRSTTRANSPORT}.}
Then using 
Propositions~\ref{prop:geometric.low}, 
\ref{prop:geometric.low.2}, and \ref{prop:Linfty.est},
the estimates \eqref{eq:Psi.itself.Li} and \eqref{eq:L.improved}--\eqref{eq:chi.smaller},
and the bootstrap assumptions \eqref{BA:LARGERIEMANNINVARIANTLARGE}--\eqref{BA:W.Li.small} 
(to control all $\Rad \Psi$-involving products on RHS~\eqref{eq:wave.in.terms.of.geometric} except $-2L\bX\Psi$),
we obtain 
$|\mathcal{P}^{\leq \Ntop- \toprate-4} L\bX\Psi| \ls \epd$. 
Also using the first commutator estimate in \eqref{eq:commutators.Li.2}
with $\phi \doteq \bX \Psi$
and the bootstrap assumption \eqref{BA:W.Li.small},
we further deduce that:
\begin{equation}\label{eq:Linfty.LbXPsi}
\|L \mathcal{P}^{\leq \Ntop- \toprate-4} \bX \Psi\|_{L^\infty(\Sigma_t)} 
\ls 
|\mathcal{P}^{\leq \Ntop- \toprate-4} \Lunit \bX \Psi|
+
\epd^{\frac{1}{2}} 
|\mathcal{P}^{[1, \Ntop- \toprate-4]} \bX \Psi|
\ls
\epd.
\end{equation}

\pfstep{Step~4: Proof of \eqref{eq:XPsi.Li} and \eqref{eq:PXPsi.Li}} We finally conclude \eqref{eq:XPsi.Li} and \eqref{eq:PXPsi.Li} using \eqref{E:SIMPLEFTCESTIMATE} and \eqref{eq:Linfty.LbXPsi}, together with the initial data bounds \eqref{assumption:R+.trans}, \eqref{assumption:small.trans} and \eqref{assumption:small}. \qedhere
\end{proof}

\begin{proposition}\label{prop:trans.V.S}
The following $L^\infty$ estimates hold for all $t\in [0,\Tboot)$:
$$\|\mathcal{P}^{\leq \Ntop - \toprate -4} \bX (\Vr, S)\|_{L^\infty(\Mtu)} \ls \epd^{\f 32}.$$
\end{proposition}
\begin{proof}
We apply $\mathcal{P}^{\leq \Ntop - \toprate -4}$ 
to \eqref{E:TRANSVERSALDERIVATIVEEXPRESSIONFORTRANSPORTVARIABLES1}--\eqref{E:TRANSVERSALDERIVATIVEEXPRESSIONFORTRANSPORTVARIABLES2}
and then bound all terms on the RHS in $L^\infty$ using 
Propositions~\ref{prop:geometric.low},
\ref{prop:Linfty.est}, and \ref{prop:Linfty.est.1}. 
\qedhere
\end{proof}

\section{Putting everything together}\label{sec:everything}
This is the concluding section. First, in Section~\ref{sec:end.of.bootstrap}, we use the estimates derived in Sections~\ref{sec:trivial}--\ref{sec:Linfty} to conclude our main a priori estimates, i.e., to prove Theorem~\ref{thm:bootstrap}. 

With the help of Theorem~\ref{thm:bootstrap}, all of the main results stated
in Section~\ref{sec:statement} are quite easy to prove. We will prove Theorems~\ref{thm:main} and \ref{thm:shock} in Section~\ref{sec:proof.of.main.theorem}, Corollary~\ref{cor:stupid.nonvanishing} 
in Section~\ref{sec:stupid.nonvanishing}, and finally, Corollary~\ref{cor:stupid.Holder} 
in Section~\ref{sec:stupid.Holder}.

\subsection{Proof of the main a priori estimates}\label{sec:end.of.bootstrap}

\begin{proof}[Proof of Theorem~\ref{thm:bootstrap}] We prove each of the four
conclusions asserted by Theorem~\ref{thm:bootstrap}.
\begin{enumerate}
\item By Proposition~\ref{prop:wave.final}, for $1\leq N \leq \Ntop$, the following wave estimates hold:
$$\mathbb{W}_N(t) \ls \epd^2 \max\{1, \upmu_{\star}^{-2 \toprate+2\Ntop-2N+1.8}(t) \}.$$
	Hence, the inequalities in \eqref{BA:W1}--\eqref{BA:W2} hold with $\epd$ replaced by $C\epd^2$.
\item By \eqref{eq:Psi.itself.Li}--\eqref{eq:XPsi.Li}, the {inequalities}
in \eqref{BA:LARGERIEMANNINVARIANTLARGE} hold with $\mathring{\upalpha}^{\f 12}$ replaced by $C_{\mydiam}\mathring{\upalpha}$ and $3\mathring{\updelta}$ replaced by $2\mathring{\updelta}$.
\item By {\eqref{eq:Linfty.est.1} and \eqref{eq:Psi.itself.Li}--\eqref{eq:PXPsi.Li},
the inequalities in 
\eqref{BA:SMALLWAVEVARIABLESUPTOONETRANSVERSALDERIVATIVE}--\eqref{BA:W.Li.small}} 
hold with $\epd^{\f 12}$ replaced by $C\epd$. 
\item By \eqref{eq:Linfty.est.2} and Proposition~\ref{prop:trans.V.S}, 
the {inequalities}  
\eqref{BA:V}--\eqref{BA:C.D} hold with $\epd$ replaced by $C\epd^{\f 32}$. \qedhere
\end{enumerate}
\end{proof}

\subsection{Proof of the main theorems}\label{sec:proof.of.main.theorem}

\begin{proof}[Proof of the regularity theorem (Theorem~\ref{thm:main})]
By the main a priori estimates (Theorem~\ref{thm:bootstrap}) and a standard continuity argument, all the estimates established in the proof of Theorem~\ref{thm:bootstrap} hold 
on $[0,T) \times \Sigma$.
As a consequence, the energy estimates 
\eqref{conclusion:W.energy}, \eqref{conclusion:V.energy} and \eqref{conclusion:C.energy} follow from Propositions~\ref{prop:wave.final}, \ref{prop:V.S}, \ref{prop:C.D}, and \ref{prop:elliptic.putting.everything.together}. 
As for the $L^\infty$ estimates, \eqref{conclusion:W.P.Li} holds thanks to \eqref{eq:Linfty.est.1} and \eqref{eq:PXPsi.Li}; \eqref{conclusion:Psi.itself.Li} and \eqref{conclusion:XPsi.Li} hold thanks to \eqref{eq:Psi.itself.Li} and \eqref{eq:XPsi.Li} respectively; and \eqref{conclusion:V.S.C.D} holds thanks to \eqref{eq:Linfty.est.2} and Proposition~\ref{prop:trans.V.S}.

Moreover, Lemma~\ref{L:GEOMETRICCOORDINATEVECTORFIELDSINTERMSOFCARTESIANVECTORFIELDS},
the identity $\srd_t = \Lunit - \Lunit^A \srd_A$ 
(see \eqref{eq:slashes.1}),
and
the $L^{\infty}$ estimates mentioned above,
together with those of Propositions~\ref{prop:geometric.low}, \ref{prop:geometric.low.2}, 
and \ref{P:LINFTYHIGHERTRANSVERSAL},
imply that the solution can be smoothly extended\footnote{Note that these estimates
imply that the $\srd_t$ derivatives of many geometric coordinate partial derivatives of the solution
are uniformly bounded on $[0,T] \times \mathbb{R} \times \mathbb{T}^2$, which leads to their
extendibility to $[0,T] \times \mathbb{R} \times \mathbb{T}^2$.} 
to 
$[0,T] \times \mathbb{R} \times \mathbb{T}^2$
as a function of the geometric coordinates
$(t,u,x^2,x^3)$.

It remains for us to show that the solution can be extended 
as a smooth solution of both the geometric and the Cartesian coordinates
as long as $\inf_{t\in [0,T)} \upmu_{\star}(t) >0$.
Now the estimates \eqref{conclusion:W.P.Li}--\eqref{conclusion:XPsi.Li}, Lemma~\ref{lem:Cart.to.geo}, 
and the assumed lower bound on $\upmu_{\star}$ together imply that the fluid variables and their first partial derivatives with respect to the Cartesian coordinates remain bounded. 
Standard local existence results/continuation criteria then imply that the solution can be 
smoothly extended in the Cartesian coordinates
to a Cartesian slab $[0,T + \epsilon] \times \Sigma$ for some $\epsilon > 0$.
Finally, within this Cartesian slab, one can solve the eikonal equation \eqref{E:INTROEIKONAL} 
such that the map $(t,u,x^2,x^3) \rightarrow (t,x^1,x^2,x^3)$
is a diffeomorphism from $[0,T + \epsilon] \times \mathbb{R} \times \mathbb{T}^2$
onto $[0,T + \epsilon] \times \Sigma$; the diffeomorphism property of this map
follows easily from the identity $\srd_u x^1 = \f{\upmu \Speed^2}{X^1}$
(see \eqref{E:GEOMETRICUCOORDINATEPARTIALDERIVATIVESINTERMSOFOTHERVECTORFIELDS})
and the fact that $\f{\upmu \Speed^2}{X^1} < 0$ in $[0,T + \epsilon] \times \mathbb{R} \times \mathbb{T}^2$
whenever $\epsilon$ is small enough, thanks to $\upmu > 0$, \eqref{E:XSMALLDEF}, 
and the estimates of Proposition~\ref{prop:geometric.low.2}
for $X_{(Small)}^i$ and $\Speed - 1$. 
This implies that the solution can also be smoothly extended in the geometric coordinates $(t,u,x^2,x^3)$.
\qedhere
\end{proof}

\begin{proof}[Proof of the shock formation theorem (Theorem~\ref{thm:shock})]
\pfstep{Step~1: Vanishing of $\upmu_{\star}$}
First, we will show that:
\begin{align} \label{E:MUSTARFINALEST}
\upmu_{\star}(t) =  
1 
+ 
\mathcal{O}_{\mydiam}(\mathring{\upalpha}) 
+ 
\mathcal{O}(\epd)
- 
\mathring{\updelta}_*t.
\end{align}

To prove \eqref{E:MUSTARFINALEST},
we start by using \eqref{E:UPMUFIRSTTRANSPORT}, 
\eqref{E:KEYLARGETERMEXPANDED}, and the $L^\infty$ estimates established in 
Propositions~\ref{prop:geometric.low} and \ref{prop:geometric.low.2} and
Theorem~\ref{thm:main}
to deduce that:
\begin{equation}\label{eq:Lmu.main}
L\upmu = -
			\frac{1}{2}
			\Speed^{-1}
			(\Speed^{-1} \Speed_{;\rr} + 1)
			\Rad \mathcal{R}_{(+)} 
			+ 
			\mathcal{O}(\epd)
\end{equation}
and: 
\begin{align}\label{eq:LLmu.main}
L\left\lbrace
			\frac{1}{2}
			\Speed^{-1}
			(\Speed^{-1} \Speed_{;\rr} + 1)
	\right\rbrace 
	& 
	= 
	\mathcal{O}(\epd),
	&
	L\left\lbrace
			\frac{1}{2}
			\Speed^{-1}
			(\Speed^{-1} \Speed_{;\rr} + 1)
			\Rad \mathcal{R}_{(+)} 
	\right\rbrace 
	& 
	= 
	\mathcal{O}(\epd).
\end{align}
Moreover, from \eqref{E:INTROEIKONAL}, \eqref{E:FIRSTUPMU}, and our data assumptions \eqref{assumption:R+} and
\eqref{assumption:small.trans}, we have the following initial condition estimate for $\upmu$:
\begin{align} \label{E:MUDATA}
	\upmu \restriction_{\Sigma_0}
	& = 1 + \mathcal{O}_{\mydiam}(\mathring{\upalpha}) + \mathcal{O}(\epd). 
\end{align}	
From \eqref{eq:Lmu.main}--\eqref{E:MUDATA},
\eqref{assumption:lower.bound},  
and the fundamental theorem of calculus along the integral curves of $\Lunit$
(and recalling that $\Lunit t = 1$), we conclude \eqref{E:MUSTARFINALEST}.

\pfstep{Step~2: Proof of (1), (2), and (3)} Define: 
\begin{equation}\label{eq:def.Tsing}
T_{(Sing)} \doteq \sup \{T\in [0,2\mathring{\updelta}_*^{-1}]: \mbox{a smooth solutions exists with }\upmu> 0 \mbox{ on $[0,T)\times \Sigma$} \}.
\end{equation}
From Theorem~\ref{thm:main}, it follows that either $T_{(Sing)} = 2\mathring{\updelta}_*^{-1}$
or $\liminf_{t\to T^-_{(Sing)}} \upmu_{\star}(t) = 0$.

Using \eqref{eq:Lmu.main}, we infer that $\upmu_{\star}(t)$ first vanishes
at a time equal to 
$\left 
\lbrace 1 
+ 
\mathcal{O}_{\mydiam}(\mathring{\upalpha}) 
+ 
\mathcal{O}(\epd) 
\right\rbrace
\mathring{\updelta}_*^{-1}
$.
From this fact, the definition of $T_{(Sing)}$, and the above discussion, 
it follows that this time of first vanishing of $\upmu_{\star}(t)$ is equal to $T_{(Sing)}$,
which implies \eqref{eq:Tsing.est}.
Using Theorem~\ref{thm:main} again, we have therefore proved parts (1), (2) and (3) of Theorem~\ref{thm:shock}.

\pfstep{Step~3: Proof of (4)} In the next step, we will show that the vanishing of  
$\upmu_{\star}$ along $\Sigma_{T_{(Sing)}}$ coincides with the blowup of $|\partial_1 \mathcal{R}_{(+)}|$
at one or more points in $\Sigma_{T_{(Sing)}}$;
that will show that
$T_{(Sing)}$ is indeed the time of first singularity formation and in particular yields the conclusion (4)
stated in Theorem~\ref{thm:shock}.

\pfstep{Step~4: Proof of (5)}  We now prove that $\mathscr{S}_{Blowup} = \mathscr{S}_{Vanish}$. 
This in particular also implies the blowup-claim in conclusion (4) of Theorem~\ref{thm:shock}. 
We first prove $\subseteq$. 
If $(u,x^2,x^3) \notin \mathscr{S}_{Vanish}$, then $\upmu$ has a lower bound away from $0$ near 
$(T_{(Sing)},u,x^2,x^3)$ and thus the estimates in Theorem~\ref{thm:main} (and Lemma~\ref{lem:Cart.to.geo}) imply that the fluid variables are $C^1$ functions of the geometric coordinates and the Cartesian coordinates
near the point with geometric coordinates $(T_{(Sing)},u,x^2,x^3)$, i.e., $(u,x^2,x^3) \notin \mathscr{S}_{Blowup}$. 

To show $\supseteq$, suppose $(u,x^2,x^3) \in \mathscr{S}_{Vanish}$. 
Let $\beta(t)$ denote the $t$-parameterized integral curve of $L$ emanating from
$(T_{(Sing)},u,x^2,x^3)$. Note in particular that $\upmu \circ \beta(T_{(Sing)}) \doteq \upmu(T_{(Sing)},u,x^2,x^3) = 0$,
and recall that $\Lunit t = 1$.
We next use
\eqref{eq:Lmu.main}--\eqref{E:MUDATA}, \eqref{assumption:lower.bound}, 
\eqref{eq:Tsing.est},
and the fundamental theorem of calculus along the integral curve $\beta(t)$
to deduce that for $0 \leq t \leq T_{(Sing)}$, 
we have
$\frac{1}{2}
			\left|
			\Speed^{-1}
			(\Speed^{-1} \Speed_{;\rr} + 1)
			\right|
			\circ \beta(0) 
			\times
			|\Rad \mathcal{R}_{(+)}|\circ \beta(t) 
				\geq  \f{3 \mathring{\updelta}_*}{4}$
(for otherwise, $\upmu \circ \beta(T_{(Sing)}) = 0$ would not be possible).
Also using \eqref{eq:Cart.to.geo.1},
Propositions~\ref{prop:geometric.low}, \ref{prop:geometric.low.2}, and the $L^{\infty}$ estimates
of Theorem~\ref{thm:bootstrap},
we find that the following estimate holds for
$0 \leq t \leq T_{(Sing)}$:
$\frac{1}{2}
			\left|
			\Speed^{-1}
			(\Speed^{-1} \Speed_{;\rr} + 1)
			\right|
			\circ \beta(0) 
			\times
			|\upmu \partial_1 \mathcal{R}_{(+)}| \circ \beta(t) 
			\geq \f{\mathring{\updelta}_*}{2}$.
In particular, 
also considering Remark~\ref{R:NONVANISHINGFACTOR},
we deduce that
$\limsup_{t \uparrow T_{(Sing)}^-} 
|\partial_1 \mathcal R_{(+)})| \circ \beta(t) 
	\geq 
			\f{\mathring{\updelta}_*}
			{2|\Speed^{-1}
			(\Speed^{-1} \Speed_{;\rr} + 1)| \circ \beta(0)}
			\limsup_{t \uparrow T_{(Sing)}^-}  
			\frac{1}{\upmu \circ \beta(t)}
			= \infty$. Hence 
			$(u,x^2,x^3) \in \mathscr{S}_{Blowup}$,
			which finishes the proof that $\mathscr{S}_{Blowup} = \mathscr{S}_{Vanish}$.

Finally, we prove that $\mathscr{S}_{Vanish} = \mathbb{R} \times \mathbb{T}^2 \setminus \mathscr{S}_{Regular}$. 
The direction $\subseteq$ holds since $\mathscr{S}_{Vanish} = \mathscr{S}_{Blowup}$ and obviously 
$\mathscr{S}_{Blowup} \subseteq \mathbb{R} \times \mathbb{T}^2 \setminus \mathscr{S}_{Regular}$. We now show the direction $\supseteq$. Suppose that $(u,x^2,x^3) \notin \mathscr{S}_{Blowup}$, i.e.,  
$\upmu(T_{(Sing)},u,x^2,x^3) > 0$. Then 
the estimates with respect to the geometric vectorfields established in Theorem~\ref{thm:main} 
and Lemma~\ref{lem:Cart.to.geo} imply that in a neighborhood of $(T_{(Sing)},u,x^2,x^3)$ 
intersected with $\lbrace t \leq T_{(Sing)} \rbrace$,
the fluid variables remain $C^1$
functions of the geometric coordinates and Cartesian coordinates. 
We have therefore proved part (5) of Theorem~\ref{thm:shock}, which completes its proof. \qedhere
\end{proof}

\subsection{Non-triviality of $\Vr$ and $S$ (Proof of Corollary~\ref{cor:stupid.nonvanishing})}\label{sec:stupid.nonvanishing}

\begin{proof}[Proof of Corollary~\ref{cor:stupid.nonvanishing}]
Using equations \eqref{eq:Lmu.main}--\eqref{E:MUDATA},
we deduce (recalling that $\epd^{1/2} \leq \mathring{\upalpha}$ by assumption)
that along any $t$-parameterized integral curve $\beta(t)$ of $L$ emanating from $\Sigma_0$
(i.e., $\beta^0(0)=0$, where $\beta^{\alpha}$ denotes the Cartesian components of $\beta$),
we have
$\upmu 
\circ 
\beta(t) 
= 
1 
-
			\frac{1}{2}
			t
			[\Speed^{-1}
			(\Speed^{-1} \Speed_{;\rr} + 1)
			\Rad \mathcal{R}_{(+)}] \circ \beta(0) 
+
\mathcal{O}_{\mydiam}(\mathring{\upalpha})$.
From this bound,
\eqref{eq:Tsing.est} (which implies that
$0 \leq t \leq T_{(Sing)} = \left\lbrace 1 + \mathcal{O}_{\mydiam}(\mathring{\upalpha}) + \mathcal{O}(\epd) \right\rbrace \mathring{\updelta}_*^{-1}$),
\eqref{assumption:lower.bound}, 
and
the assumption \eqref{eq:problem.concentrated.somewhere},
we see that if
$|u \circ \beta(0) - \mathring{\upsigma} + \mathring{\updelta}_*^{-1}| 
\geq 3 \mathring{\upalpha} \mathring{\updelta}_*^{-1}$
(where $u \circ \beta(0)$ is the value of the $u$ coordinate at $\beta(0)$),
then $\upmu \circ \beta(t) \geq 3/8$ for $0 \leq t \leq T_{(Sing)}$
(assuming that $\mathring{\upalpha}$ and $\epd$ are sufficiently small).

Now fix any $(u_*,x_*^2,x_*^3) \in \mathscr{S}_{Vanish}$ (that is, $\upmu(T_{(Sing)},u_*,x_*^2,x_*^3) = 0$).
We will show that under the assumptions of the corollary, there is a constant $C > 1$ such that:
\begin{align} \label{E:QUANTITATIVENONVANISHING}
	C^{-1} \epd^3 \leq |S(T_{(Sing)},u_*,x_*^2,x_*^3)| \leq C \epd^3,
		\quad
	C^{-1} \epd^2 \leq |\Vr(T_{(Sing)},u_*,x_*^2,x_*^3)| \leq C \epd^2.
\end{align}
Clearly, the bounds \eqref{E:QUANTITATIVENONVANISHING} imply the desired conclusion of the corollary.

To initiate the proof of \eqref{E:QUANTITATIVENONVANISHING},
we let $\beta_{(Sing)}(t)$ denote the $t$-parameterized integral curve of $L$ passing through
$(T_{(Sing)},u_*,x_*^2,x_*^3)$. 
Then since \eqref{E:LUNITANDRADOFUANDT} implies that
the coordinate function $u$ is constant along $\beta_{(Sing)}$ 
(and thus $u \circ \beta_{(Sing)}(0) = u_*$),
the results derived two paragraphs above guarantee that
$|u_* - \mathring{\upsigma} + \mathring{\updelta}_*^{-1}| 
\leq 3 \mathring{\upalpha} \mathring{\updelta}_*^{-1}$.
In particular, in view of the initial condition \eqref{E:INTROEIKONAL} for $u$ along $\Sigma_0$,
we see that
$|\beta_{(Sing)}^1(0) - \mathring{\updelta}_*^{-1}| 
\leq 3 \mathring{\upalpha} \mathring{\updelta}_*^{-1}$,
where $\beta_{(Sing)}^1(0) \doteq x^1 \circ \beta_{(Sing)}(0)$ 
is the $x^1$ coordinate of the point $\beta_{(Sing)}(0) \in \Sigma_0$.
Then, since Proposition~\ref{prop:geometric.low}
yields that
$\frac{d}{dt} \beta^1 = L \beta^1 = L^1 = 1 + L_{(Small)}^1 = 1 + \mathcal{O}_{\mydiam}(\mathring{\upalpha})$,
we can integrate in time and use \eqref{eq:Tsing.est} to deduce
that 
$\beta^1(T_{(Sing)}) 
= 
\beta^1(0) + T_{(Sing)} + \mathcal{O}_{\mydiam}(\mathring{\upalpha}) T_{(Sing)}
=
- 
\mathring{\updelta}_*^{-1}
+ 
T_{(Sing)} 
+ 
\mathcal{O}_{\mydiam}(\mathring{\upalpha}) 
T_{(Sing)}
=
\mathcal{O}_{\mydiam}(\mathring{\upalpha})
\mathring{\updelta}_*^{-1}
$.
That is, the $x^1$ coordinate of the singular point $(T_{(Sing)},u_*,x_*^2,x_*^3)$
is of size $\mathcal{O}_{\mydiam}(\mathring{\upalpha}) \mathring{\updelta}_*^{-1}$.

Let now $\gamma_{(Sing)}$ be the integral curve of $B$ passing through the singular point 
$(T_{(Sing)},u_*,x_*^2,x_*^3)$ as above. 
Since \eqref{E:TRANSPORTVECTORFIELDINTERMSOFLUNITANDRADUNIT} and \eqref{conclusion:Psi.itself.Li}
imply that
$\Transport = \rd_t +\mathcal{O}_{\mydiam}(\mathring{\upalpha}) \rd$,
we can integrate with respect to time along $\gamma_{(Sing)}$ and use
\eqref{eq:Tsing.est} and the bound on the $x^1$ coordinate of the singular point $(T_{(Sing)},u_*,x_*^2,x_*^3)$ proved above
to deduce that $\gamma_{(Sing)}$ intersects $\Sigma_0$ at a point $q$ with
$x^1$ coordinate $q^1$ of size $q^1 = \mathcal{O}_{\mydiam}(\mathring{\upalpha}) \mathring{\updelta}_*^{-1}$.
In view of the initial condition \eqref{E:INTROEIKONAL} for $u$ along $\Sigma_0$,
we see that the $u$ coordinate of $q^1$, which we denote by $u|_q$,
verifies $\Big|u|_q - \mathring{\upsigma} \Big|= \mathcal{O}_{\mydiam}(\mathring{\upalpha}) \mathring{\updelta}_*^{-1}$.
From this bound and the assumption \eqref{eq:V.S.nonvanishing.assumption}, we see that:
\begin{equation}\label{RECALLLEDeq:V.S.nonvanishing.assumption}
\f 12\epd^2 \leq \Big|\Vr|_q \Big|\leq \epd^2,
\quad 
\f 12 \epd^3\leq \Big|S|_q \Big|\leq \epd^3.
\end{equation}
To complete the proof, we need to use \eqref{RECALLLEDeq:V.S.nonvanishing.assumption}
to prove \eqref{E:QUANTITATIVENONVANISHING}.
To this end, we find it convenient to parameterize $\gamma_{(Sing)}$ by the eikonal function.
Since \eqref{E:TRANSPORTVECTORFIELDINTERMSOFLUNITANDRADUNIT} and \eqref{E:LUNITANDRADOFUANDT}
guarantee that $\upmu \Transport u = 1$, this is equivalent to studying integral curves of $\upmu \Transport$.
That is, we slightly abuse notation by denoting the re-parameterized integral curve by the same symbol
$\gamma_{(Sing)}$; i.e., $\gamma_{(Sing)}$ solves the integral curve ODE
$\frac{d}{d u} \gamma_{(Sing)}(u) = \upmu \Transport \circ \gamma_{(Sing)}(u)$.
To proceed, we multiply the transport equations 
\eqref{E:RENORMALIZEDVORTICTITYTRANSPORTEQUATION} and \eqref{E:GRADENTROPYTRANSPORT}
by $\upmu$ and use \eqref{E:TRANSPORTVECTORFIELDINTERMSOFLUNITANDRADUNIT},
\eqref{E:LUNITANDRADOFUANDT},
Lemma~\ref{lem:Cart.to.geo},
Propositions~\ref{prop:geometric.low}, \ref{prop:geometric.low.2}, and the $L^{\infty}$ estimates
of Theorem~\ref{thm:bootstrap}
to deduce that along $\gamma_{(Sing)}$, 
\eqref{E:RENORMALIZEDVORTICTITYTRANSPORTEQUATION} and \eqref{E:GRADENTROPYTRANSPORT}
imply the following evolution equations, expressed in schematic form:
\begin{align}
\frac{d}{du} \Vr \circ \gamma_{(Sing)}(u) 
& = \mathcal{O}(1) \Vr \circ \gamma_{(Sing)}(u) 
	+ \mathcal{O}(1) S \circ \gamma_{(Sing)}(u),
		\label{E:SCHEMATICUPARAMETERIZEDVORTICITYTRANSPORT} \\
\frac{d}{du} S \circ \gamma_{(Sing)}(u) 
& = \mathcal{O}(1) S \circ \gamma_{(Sing)}(u).
\label{E:SCHEMATICUPARAMETERIZEDENTROPYGRADIENTTRANSPORT}
\end{align}
From the evolution equations 
\eqref{E:SCHEMATICUPARAMETERIZEDVORTICITYTRANSPORT}--\eqref{E:SCHEMATICUPARAMETERIZEDENTROPYGRADIENTTRANSPORT}, 
the initial conditions 
\eqref{RECALLLEDeq:V.S.nonvanishing.assumption},
and the fact that $0 \leq u \leq U_0$ in the support of the solution (see Section~\ref{sec:trivial}),
we conclude that if $\epd$ is sufficiently small, 
then there is a $C > 1$ such that \eqref{E:QUANTITATIVENONVANISHING} holds. \qedhere

\end{proof}

\subsection{H\"older estimates (Proof of Corollary~\ref{cor:stupid.Holder})}\label{sec:stupid.Holder}
Throughout this section, we work under the assumptions of Corollary~\ref{cor:stupid.Holder}.


\begin{lemma}[\textbf{A simple calculus lemma}]
\label{lem:trivial.calculus}
Let $J \subseteq \mathbb R$ be an interval. Suppose $f:J\to \mathbb R$ is a $C^3$ function such that:
\begin{enumerate}
\item $f$ is increasing, i.e., $f' \geq 0$;
\item There exists $\mathring{b}>0$ such that
$ f^{(3)}(y) \geq \mathring{b}$ for every $y\in J$,
where $f^{(3)}$ denotes the third derivative of $f$.
\end{enumerate}

Then for any $y_1,\, y_2 \in J$,  the following estimate holds:
$$|f(y_1) -f(y_2)| \geq \f{\mathring{b}}{48} |y_1 - y_2|^3.$$
\end{lemma}
\begin{proof}
First, note that the assumption on $f^{(3)}$ implies that $f''$ is strictly increasing. 
In particular, $f''$ can at most change sign once.

Without loss of generality, assume $y_1 \neq y_2$. We consider three cases: 
the first two are such that $f''(y_1)$ and $f''(y_2)$ are of the same sign,
while the third is such that they have opposite sign.

\textbf{Case~1: $y_1 < y_2$ and $f''(y_1) < f''(y_2) \leq 0$.} By Taylor's theorem, 
\begin{equation*}
\begin{split}
f(y_1) = &\: f(y_2) - f'(y_2) (y_2-y_1) + \f 12 f''(y_2) (y_2 -y_1)^2 \\
&\: - \f 12 (y_2 -y_1)^3 \int_0^1 (1-\tau)^2 f^{(3)}(y_2 + \tau (y_1-y_2)) \, d\tau  \leq  f(y_2) - \f {\mathring{b}}6 (y_2 -y_1)^3,
\end{split}
\end{equation*}
where we have used $f'(y_2) \geq 0$, $f''(y_2)\leq 0$ and $ f^{(3)}(y) \geq \mathring{b}$.

Therefore,
$$|f(y_1) - f(y_2)| = f(y_2) -f(y_1) \geq \f {\mathring{b}}6 (y_2 -y_1)^3.$$

\textbf{Case~2: $y_2 < y_1$ and $f''(y_1)  > f''(y_2) \geq 0$.} 
This can be treated in the same way as Case~1 so that we have:
$$|f(y_1) - f(y_2)| = f(y_1) -f(y_2) \geq \f {\mathring{b}}6 (y_1 -y_2)^3.$$

\textbf{Case~3: $y_1 < y_2$, $f''(y_1) < 0 < f''(y_2)$.} 
Since $f''$ is strictly increasing, there
exists a unique $z\in (y_1,y_2)$ such that $f''(z) = 0$. 
Therefore, using Case~1 (for $y_1$ and $z$) and Case~2 (for $y_2$ and $z$), we have:
$$|f(y_1) - f(y_2)| = f(y_2) - f(z) + f(z) -f(y_1) 
\geq 
\f{\mathring{b}}6 (| y_2-z|^3 + | y_1-z|^3) 
\geq \f{\mathring{b}}{2^3\cdot 6} (y_2 - y_1)^3,$$
where in the very last inequality we have used $y_2 - y_1 \leq 2 \max\{|y_1 - z|,\, |y_2-z|\}$.
 
Combining all three cases, we conclude the desired inequality. \qedhere
\end{proof}

\begin{lemma}[\textbf{Quantitative negativity of $\srd_u^3 x^1$}]
\label{lem:higher.transversal.for.Holder}
Under the assumptions of Corollary~\ref{cor:stupid.Holder},
the following holds at all points such that
$(t,u) \in [\frac{3}{4}T_{(Sing)},T_{(Sing)}) 
\times 
[\f{\mathring{\upsigma}}{2}, \f{3\mathring{\upsigma}}{2}]$:
$$\srd_u^3 x^1 \leq - \mathring{\upbeta}.$$
\end{lemma}

\begin{proof}
In this proof, we will silently use the fact that the Cartesian component functions $X^1,X^2,X^3$ are smooth functions of
the $L^i$ and $\Psi$ (see \eqref{E:TRANSPORTVECTORFIELDINTERMSOFLUNITANDRADUNIT}) 
and the fact that
$\Speed$ is a smooth function of $\Psi$.

By \eqref{eq:srd.x1}, to prove the lemma,
we need to estimate $\srd_u^3 x^1 = \srd_u^2 (\f{\upmu \Speed^2}{X^1})$.
To proceed, we use \eqref{E:GEOMETRICUCOORDINATEPARTIALDERIVATIVESINTERMSOFOTHERVECTORFIELDS} 
(in particular, the fact that $\srd_u - \Rad$ is $\ell_{t,u}$-tangent) and 
the $L^{\infty}$ estimates of Propositions~\ref{prop:geometric.low}, \ref{prop:geometric.low.2},
and \ref{P:LINFTYHIGHERTRANSVERSAL} and
Theorem~\ref{thm:bootstrap} 
to deduce that:
\begin{align}\label{eq:du3.x1.1}
\srd_u^3 x^1 
& = \bX \bX \left(\f{\upmu \Speed^2}{X^1} \right) 
	+ 
	\mathcal{O}(\epd).
\end{align}

We will now estimate the term $\bX \bX (\f{\upmu \Speed^2}{X^1})$ on RHS~\eqref{eq:du3.x1.1}.
We start by noting that the $L^{\infty}$ estimates of Propositions~\ref{prop:geometric.low}, \ref{prop:geometric.low.2},
and \ref{P:LINFTYHIGHERTRANSVERSAL}
and Theorem~\ref{thm:bootstrap} together imply that
$|L L \bX \bX (\f{\upmu \Speed^2}{X^1}) | = \mathcal{O}(\epd)$.
Therefore, letting $\gamma(t)$ be any integral curve of $L$ parametrized by Cartesian time $t$ 
(with $\gamma(0) \in \Sigma_0$)
and recalling that $Lt=1$,
we integrate this estimate twice in time to deduce that for $t \in [0,T_{(Sing)})$,
we have:
\begin{align}
\begin{split} \label{eq:FIRSTest.LXXupmu.2}
\bX \bX \left(\f{\upmu \Speed^2}{X^1} \right) \circ \gamma(t)
& = 
\left[\bX \bX \left(\f{\upmu \Speed^2}{X^1}\right) \right] \circ \gamma(0)
+
t \left[L \bX \bX \left(\f{\upmu \Speed^2}{X^1}\right) \right] \circ \gamma(0)
+ 
\mathcal{O}(\epd)
	\\
& = 
\left[\bX \bX \left(\f{\upmu \Speed^2}{X^1}\right) \right] \circ \gamma(0)
+
t \left[\bX \bX L \left(\f{\upmu \Speed^2}{X^1}\right)\right]\circ \gamma(0)
+ 
\mathcal{O}(\epd),
\end{split}
\end{align}
where to deduce the last equality, we used in particular \eqref{E:PERMUTEDESTIMATES}.

Next, using the transport equation \eqref{E:UPMUFIRSTTRANSPORT}, 
\eqref{E:KEYLARGETERMEXPANDED}, 
the fact that
$X\restriction_{\Sigma_0} = - \Speed \rd_1$ (by \eqref{E:ACOUSTICALMETRIC},
\eqref{E:INTROEIKONAL}, \eqref{eq:Cart.to.geo.1}, and the normalization condition $g(X,X)=1$), 
and the $L^{\infty}$ estimates mentioned above,
we deduce that:
\begin{align}
\begin{split} \label{eq:est.LXXupmu.2}
\left[\bX \bX L \left(\f{\upmu \Speed^2}{X^1} \right) \right] \circ \gamma(0)
& 
= 
\left[\bX \bX \left\{ (L\upmu) \f{\Speed^2}{X^1} \right\} \right] \circ \gamma(0) + \mathcal{O}(\epd) 
	\\
&
=  
\f 12 \left[\bX \bX  \left\{\left(\Speed^{-1} \Speed_{;\rr} +1 \right)(\bX\mathcal R_{(+)}) \right\} \right] \circ \gamma(0)
+ 
\mathcal{O}(\epd).
\end{split}
\end{align}

Next, using that $X\restriction_{\Sigma_0} = - \Speed \rd_1$,
and using that
$\upmu\restriction_{\Sigma_0} = \frac{1}{\Speed}$
(this follows from the initial condition in \eqref{E:INTROEIKONAL}
and the fact that \eqref{E:LUNITANDRADOFUANDT} implies that $\Radunit u = \frac{1}{\upmu}$),
we deduce:
\begin{align} \label{E:VANISHINGINITIALCONDITION}
	\bX \bX\left(\f{\upmu \Speed^2}{X^1} \right)\restriction_{\Sigma_0} 
	& 
	=
	- 
	\bX \bX (1)
	= 0.
\end{align}
Combining \eqref{eq:FIRSTest.LXXupmu.2}--\eqref{E:VANISHINGINITIALCONDITION},
we find that:
\begin{align}\label{eq:SECONDest.LXXupmu.2}
\bX \bX \left(\f{\upmu \Speed^2}{X^1} \right) \circ \gamma(t)
& = 
\frac{t}{2} 
\left[\bX \bX  \left\{ \left(\Speed^{-1} \Speed_{;\rr} +1 \right)(\bX\mathcal R_{(+)}) \right\} \right] \circ \gamma(0)
+ 
\mathcal{O}(\epd).
\end{align}
From \eqref{eq:SECONDest.LXXupmu.2} and our assumption \eqref{eq:nondegeneracy}, 
we deduce that at any point whose corresponding $u$ coordinate\footnote{Recall that $u \restriction_{\Sigma_0} = \mathring{\upsigma} - x^1$ and that the $u$-value is constant along the integral curves of $L$ by virtue of the first equation in \eqref{E:LUNITANDRADOFUANDT}.}
satisfies $u\in [\f{\mathring{\upsigma}}{2}, \f{3\mathring{\upsigma}}{2}]$,
we have:
\begin{align}\label{eq:THIRDest.LXXupmu.2}
\bX \bX \left(\f{\upmu \Speed^2}{X^1} \right) \circ \gamma(t)
& \leq
- 2 t \mathring{\updelta}_* \mathring{\upbeta}
+ 
\mathcal{O}(\epd).
\end{align}
In particular, for points whose corresponding $u$ and $t$ coordinates
satisfy, respectively, $u\in [\f{\mathring{\upsigma}}{2}, \f{3\mathring{\upsigma}}{2}]$ 
and 
$t\in [\frac{3}{4}T_{(Sing)},T_{(Sing)})$,
we have, in view of \eqref{eq:Tsing.est},
the following estimate:
\begin{align}\label{eq:du3.x1.2}
\bX \bX \left(\f{\upmu \Speed^2}{X^1} \right) \circ \gamma(t) 
& 
\leq 
-
\f{3\mathring{\upbeta}}{2}
+ 
\mathcal{O}_{\mydiam}(\mathring{\upalpha}) 
\mathring{\updelta}_* \mathring{\upbeta}
+
\mathcal{O}(\epd).
\end{align}
Combining \eqref{eq:du3.x1.1} and \eqref{eq:du3.x1.2}, we conclude the lemma. \qedhere
\end{proof}

\begin{lemma}[\textbf{The main H\"{o}lder estimate for the eikonal function}]
\label{lem:Lipschitz.cov}
Under the assumptions of Corollary~\ref{cor:stupid.Holder},
the following holds for $t \in [\frac{3}{4}T_{(Sing)},T_{(Sing)})$:
$$\sup_{ \substack{ p_1, p_2 \in \Sigma_t, \, p_1 \neq p_2 
	\\ u(p_i)\in [\f{\mathring{\upsigma}}{2}, \f{3\mathring{\upsigma}}2] }} \f{|u(p_1) - u(p_2)|}{\mathrm{dist}_{Euc}(p_1,p_2)^{\f 13}}  \leq 5 \mathring{\upbeta}^{-\f 13}.$$
Above, $u(p_i)$ denotes the value of the eikonal function at $p_i$,
$x(p_i)$ denotes the Cartesian spatial coordinates of $p_i$, and 
$\mathrm{dist}_{Euc}(p_1,p_2)$ denotes the Euclidean distance in $\Sigma_t$ between $p_1$ and $p_2$.
\end{lemma}

\begin{proof}

\pfstep{Step~1: Estimating $\min_{u(p_i) = u_i}  \mathrm{dist}_{Euc}(p_1, p_2)$ by carefully choosing two points} Consider two distinct values $u_1$, $u_2$ which obey $u_i\in [\f{\mathring{\upsigma}}{2}, \f{3\mathring{\upsigma}}2]$. By compactness of the constant-$u$ hypersurfaces in $\Sigma_t$, 
there exist points ${\bf p}_1, {\bf p}_2  \in \Sigma_t$ with 
$u({\bf p}_i) = u_i$ and $\mathrm{dist}_{Euc}({\bf p}_1,{\bf p}_2) = \min_{u(p_i) = u_i}  \mathrm{dist}_{Euc}(p_1, p_2)$. In particular, ${\bf p}_1$ and ${\bf p}_2$ are connected by a Euclidean straight line $L_{{\bf p}_1, {\bf p_2}}$ which is Euclidean-perpendicular to $\{u = u_i\}$ at the point ${\bf p}_i$ for $i=1,2$.

Now by Lemma~\ref{lem:Cart.to.geo} and \eqref{E:LUNITANDRADOFUANDT}, the Euclidean gradient of $u$ satisfies:
\begin{equation}\label{eq:compute.u.gradient}
	\upmu \rd_i u = c^{-2} X^i,\;\, i=1,2,3.
\end{equation}
Recalling (by Proposition~\ref{prop:geometric.low.2} and conclusions (2) and (3) 
of Theorem~\ref{thm:bootstrap}) that 
$\Speed^{-2} X^1 = -1+ \mathcal{O}_{\mydiam}(\mathring{\upalpha})$, 
$\Speed^{-2} X^2,\,\Speed^{-2} X^3 = \mathcal{O}_{\mydiam}(\mathring{\upalpha})$, 
we deduce from \eqref{eq:compute.u.gradient}
that $L_{{\bf p}_1, {\bf p_2}}$ makes a Euclidean
angle of $\mathcal{O}_{\mydiam}(\mathring{\upalpha})$ with respect to $\rd_1$. 
Therefore, using \eqref{eq:compute.u.gradient} again (which implies that constant-$u$ hypersurfaces 
in $\Sigma_t$
make an angle $\mathcal{O}(\mathring{\upalpha})$ with constant-$x^1$ planes), 
we infer that there exist\footnote{We can, for instance, 
take $\mfp_1 = {\bf p}_1$ and let $\mfp_2$ be the unique point in both the level set $\{u=u_2\}$ and the line passing through $\mfp_1$ with tangent vector everywhere equal to $\rd_1$.} $\mfp_1$, $\mfp_2$ such that:
\begin{enumerate}
\item $u(\mfp_i) = u_i$.
\item $\rd_1$ is tangent to the Euclidean line $\mathfrak L$ connecting $\mfp_1$ and $\mfp_2$.
\item $\min_{u(p_i) = u_i}  \mathrm{dist}_{Euc}(p_1, p_2) = \mathrm{dist}_{Euc}({\bf p}_1,{\bf p}_2) \geq \f 12\mathrm{dist}_{Euc}(\mfp_1,\, \mfp_2) = \f12 |x^1(\mfp_1) - x^1(\mfp_2)|$.
\end{enumerate}
We fix such a choice of $(\mfp_1, \mfp_2)$ for any given $(u_1, u_2)$ (with $u_1\neq u_2$).

\pfstep{Step~2: Estimating $|x^1(\mfp_1) - x^1(\mfp_2) |$} By \eqref{eq:srd.x1}, 
Proposition~\ref{prop:geometric.low.2},
and conclusions (2) and (3) of Theorem~\ref{thm:bootstrap}, we have:
$$\srd_u x^1 = \upmu (-1+ \mathcal{O}_{\mydiam}(\mathring{\upalpha})).$$
Hence, for every fixed $(x^2, x^3)$, $x^1$ is a strictly decreasing function in $u$. Moreover, by Lemma~\ref{lem:higher.transversal.for.Holder},
$\srd_u^3 x^1 \leq - \mathring{\upbeta}.$ 
Hence, we are exactly in the setting to apply Lemma~\ref{lem:trivial.calculus} (for the $1$-variable function $f(u) = -x^1(u)$, where $(x^2,x^3)$ is fixed, and $\mathring{b} = \mathring{\upbeta}$) to obtain:
\begin{equation}\label{eq:lower.bd.x.diff}
|x^1(\mfp_1) - x^1(\mfp_2)| \geq \f{\mathring{\upbeta}}{48} |u_1 - u_2|^3.
\end{equation}

In view of our choice of $\mfp_1$ and $\mfp_2$ in Step~1, 
we conclude from \eqref{eq:lower.bd.x.diff} that:
\begin{equation*}
\begin{split}
\sup_{ \substack{p_1,p_2 \in \Sigma_t, p_1 \neq p_2 \\ u(p_i)\in [\f{\mathring{\upsigma}}{2}, \f{3\mathring{\upsigma}}2] }} \f{|u(p_1) - u(p_2)|}{\mathrm{dist}_{Euc}(p_1,p_2)^{\f 13}} \leq &\: \sup_{\substack{ u_1\neq u_2 \\ u_i\in [\f{\mathring{\upsigma}}{2}, \f{3\mathring{\upsigma}}2] }} \f{|u_1 - u_2|}{\inf_{p_1,p_2 \in \Sigma_t, u(p_i) = u_i} \mathrm{dist}_{Euc}(p_1,p_2)^{\f 13}} \\
\leq &\: 2^{1/3} \sup_{\substack{p_1,p_2 \in \Sigma_t, p_1\neq p_2 \\ u(p_i)\in [\f{\mathring{\upsigma}}{2}, \f{3\mathring{\upsigma}}2] }} \f{|u_1-u_2|}{|x^1(\mfp_1) - x^1(\mfp_2)|^{\f 13}} 
\leq 96^{1/3} \mathring{\upbeta}^{-\f 13}
\leq 5 \mathring{\upbeta}^{-\f 13}.
\end{split}
\end{equation*}

\end{proof}

We are now ready to conclude the proof of Corollary~\ref{cor:stupid.Holder}.
\begin{proof}[Proof of Corollary~\ref{cor:stupid.Holder}]
Our starting point is the observation that
the estimates in Theorem~\ref{thm:main}
guarantee that for at each fixed $t$ with $0 \leq t \leq T_{(Sing)}$, the 
fluid variables and higher-order variables
$\rr$, $v^i$, $\Vr^i$, $S^i$, $\mathcal{C}^i$, and $\mathcal{D}$ are all uniformly Lipschitz 
\emph{when viewed as functions of the $(u,x^2,x^3)$ coordinates}. Therefore, the key to proving 
Corollary~\ref{cor:stupid.Holder} is to understand the regularity of the map $(x^1,x^2,x^3)\mapsto (u,x^2,x^3)$.

To this end, we first note that by the assumption (1) in Corollary~\ref{cor:stupid.Holder}, the equations \eqref{eq:Lmu.main}--\eqref{E:MUDATA}, 
\eqref{eq:Tsing.est},
and the arguments given in the proof of Corollary~\ref{cor:stupid.nonvanishing}, it follows that
away from $u\in [\f{3\mathring{\upsigma}}4, \f{5\mathring{\upsigma}}4]$,
we have $\upmu > \f 12$
From this lower bound, 
Lemma~\ref{lem:Cart.to.geo}, 
and the estimates of
Proposition~\ref{prop:geometric.low.2},
we see that
when $u \notin [\f{3\mathring{\upsigma}}4, \f{5\mathring{\upsigma}}4]$,
the map $(x^1,x^2,x^3)\mapsto (u,x^2,x^3)$ remains uniformly Lipschitz 
(in fact, we could prove that it is even more regular). 
Combined with the aforementioned fact that $\rr$, $v^i$, $\Vr^i$, $S^i$, $\mathcal{C}^i$ and $\mathcal{D}$ 
are uniformly Lipschitz in the $(u,x^2,x^3)$ coordinates, 
we see that at each fixed $t$ with $0 \leq t \leq T_{(Sing)}$, 
$\rr$, $v^i$, $\Vr^i$, $S^i$, $\mathcal{C}^i$, and $\mathcal{D}$ 
{are also uniformly Lipschitz in the $(x^1,x^2,x^3)$ coordinates}
away from $u\in [\f{3\mathring{\upsigma}}4, \f{5\mathring{\upsigma}}4]$.
Moreover, 
\eqref{E:MUSTARFINALEST} guarantees that 
in the region $\lbrace 0 \leq t \leq \frac{3}{4}T_{(Sing)} \rbrace$,
we have $\upmu > \frac{1}{8}$. Thus, for the same reasons given above,
the map $(x^1,x^2,x^3)\mapsto (u,x^2,x^3)$ is uniformly Lipschitz in 
$\lbrace 0 \leq t \leq \frac{3}{4}T_{(Sing)} \rbrace$,
and thus $\rr$, $v^i$, $\Vr^i$, $S^i$, $\mathcal{C}^i$, and $\mathcal{D}$ 
also remain uniformly Lipschitz in the $(x^1,x^2,x^3)$ coordinates in this region.

It remains for us to consider the difficult region in which
$u\in [\f{3\mathring{\upsigma}}4, \f{5\mathring{\upsigma}}4] \subseteq [\f{\mathring{\upsigma}}2, \f{3\mathring{\upsigma}}2]$
and $t \in [\frac{3}{4}T_{(Sing)},T_{(Sing)})$.
Using Lemma~\ref{lem:Lipschitz.cov},
we see that the map
$(x^1,x^2,x^3) \mapsto (u,x^2,x^3)$ is uniformly $C^{1/3}$ in this difficult region.
Hence, $\rr$, $v^i$, $\Vr^i$, $S^i$, $\mathcal{C}^i$, 
and $\mathcal{D}$ all have uniformly bounded Cartesian spatial $C^{1/3}$ norms 
in this region as well. \qedhere
\end{proof}

\appendix

\setcounter{section}{0}
\setcounter{subsection}{0}
\setcounter{equation}{0}

\section{Proof of the wave estimates}\label{app:elliptic}
In this appendix, we sketch the proof of the wave equation estimates, 
that is, of Proposition~\ref{prop:wave}.
As we already discussed in Section~\ref{sec:wave.black.box}, 
although the wave equation estimates that we need 
are almost identical to the ones derived in \cite{LS}, there are
two differences:
\begin{enumerate}
\item The wave equations in Proposition~\ref{prop:wave} feature 
the inhomogeneous terms ``$\mathfrak{G}$,''
and we need to track the influence of these inhomogeneous terms
on the estimates. Recall that the precise
inhomogeneous terms are located on
RHSs~\eqref{E:VELOCITYWAVEEQUATION}--\eqref{E:ENTROPYWAVEEQUATION},
but for purposes of proving Proposition~\ref{prop:wave}, we do not need to
 know their precise structure.
\item Recall that our commutation vectorfields $\lbrace L,Y,Z \rbrace$
are constructed out of the acoustic eikonal function $u$, and hence
the commuted wave equations feature error terms that depend on the acoustic geometry.
In $3$D, some additional arguments are needed (compared to the $2$D case treated in \cite{LS})
to control the top-order derivatives of some of these error terms.
\end{enumerate}
The issue (2) is tied to the fact that the null second fundamental form of null hypersurfaces in $1+3$ dimensions has now $3$ independent components, which stands in contrast to the case of $1+2$ dimensions,
where it has only a single component (i.e., it is trace-free in $1+2$ dimensions). This issue is by now very 
well-understood, and it can be resolved by using an elliptic estimate. For completeness, we will nonetheless sketch the main points needed for the argument in this appendix.

We now further discuss the issue (2).
In $1+2$ dimensions, $\mytr \upchi$ satisfies a transport equation
known as the Raychaudhuri equation\footnote{Note that this is a purely differential geometric identity that is independent of the compressible Euler equations.} (see \cite[(6.2.5)]{jSgHjLwW2016}):
\begin{equation}\label{eq:2D.Raychaudhuri}
\upmu L \mytr \upchi = (L\upmu)\mytr \upchi - \upmu (\mytr \upchi)^2 - \upmu \mathrm{Ric}_{LL},
\end{equation}
where $\mathrm{Ric}$ is the Ricci curvature of the acoustical metric $g$
and $\mathrm{Ric}_{LL} \doteq \mathrm{Ric}_{\alpha \beta}\Lunit^{\alpha} \Lunit^{\beta}$.
In contrast, in $1+3$ dimensions, RHS~\eqref{eq:2D.Raychaudhuri} 
features some additional terms. Specifically, in $1+3$ dimensions, the Raychaudhuri equation 
takes the following form (see \cite[(11.23)]{jS2016b}):
\begin{equation}\label{eq:3D.Raychaudhuri}
\upmu L \mytr \upchi = (L\upmu)\mytr \upchi - \upmu | \upchi |^2 - \upmu \mathrm{Ric}_{LL} = (L\upmu)\mytr \upchi - \upmu (\mytr \upchi)^2 -\upmu | \chih |^2 - \upmu \mathrm{Ric}_{LL},
\end{equation}
where $\chih$ is the traceless part of $\upchi$, i.e., it 
can be defined by imposing the following identity:
$\upchi = \chih + \f 12 (\mytr\upchi) \slashed g$. 
In other words,
\eqref{eq:3D.Raychaudhuri} has an additional $-\upmu |\chih|^2$ term 
compared to \eqref{eq:2D.Raychaudhuri}, 
and this additional term cannot be bounded using the only the transport 
equation \eqref{eq:3D.Raychaudhuri}, 
(since LHS~\eqref{eq:3D.Raychaudhuri} features a transport operator acting only on the component $\mytr\upchi$, as opposed to 
the full second fundamental form $\upchi$). 
The saving grace, however, as already noticed in \cite{dC2007} 
(see also \cites{dCsK1993,sKiR2003}), 
is that one can use geometric identities (specifically, the famous Codazzi equation)
and elliptic estimates to control $\angD \chih$ 
in terms of $\angdiff \mytr \upchi$ plus simpler error terms.
A top-order version of this kind of argument allows one to control
the difficult top-order derivatives of the term $-\upmu | \chih |^2$ on RHS~\eqref{eq:3D.Raychaudhuri}; 
see Section~\ref{sec:elliptic.on.spheres} for the details.
We remark that for the solutions under study, the $-\upmu | \chih |^2$ term
is quadratically small and, as it turns out, 
it does not have much effect on the dynamics.

\subsection{Running assumptions in the appendix and the dependence of constants and parameters}
\label{SS:ASSUMPTIONSANDCONSTANTS}
Throughout the entire appendix,
we work in the setting of Proposition~\ref{prop:wave}. 
In particular, we make the same assumptions as we did in Theorem~\ref{thm:bootstrap} 
(which provides the main a priori estimates) 
as well as the smallness assumption \eqref{eq:F.smallness} for the inhomogeneous terms $\mathfrak{G}$.

Our analysis involves various constants and parameters that play distinct roles in the proof.
We have already introduced these quantities earlier in the article.
For the reader's convenience, we again provide a brief description of these quantities in order
to help the reader understand their role in our subsequent arguments in the appendix.
\begin{itemize}
		\item The background density constant $\bar{\varrho} > 0$ was fixed at the beginning of the paper.
			The parameters 
			$\mathring{\upsigma}$, $\mathring{\updelta}_*$, $\mathring{\updelta}$, $\mathring{\upalpha}$ and $\epd$
			measure the size of the $x^1$-support and various norms of the initial data; 
			see Section~\ref{SS:FLUIDVARIABLEDATAASSUMPTIONS}.
		\item As in the rest of the paper, the positive integer $\Ntop$ denotes the maximum number 
			of times that we commute the equations for the purpose of obtaining $L^2$-type energy estimates.
		\item $M_{\mathrm{abs}}$ denotes an \textbf{absolute constant}, that is,
			a constant that can be chosen to be independent of
			$\Ntop$,
			the equation of state, 
			$\bar{\varrho}$, 
			$\mathring{\upsigma}$, 
			$\mathring{\updelta}$, 
			and
			$\mathring{\updelta}_*^{-1}$,
			as long as 
			$\mathring{\upalpha}$ and
			$\epd$ are sufficiently small.
			The constants $M_{\mathrm{abs}}$ arise as numerical coefficients that multiply the borderline
			energy error integrals; see in particular RHS~\eqref{E:JAREDTOPORDERINTEGRALINEQUALITY}.
			The universality of the $M_{\mathrm{abs}}$ is crucial since, 
			as the next two points clarify, they drive the blowup-rate
			of the top-order energies, which in turn controls the size of largeness of $\Ntop$
			needed to close the proof.
		\item As in the rest of the paper, 
			the positive integer $\toprate$ controls the blowup-rate of the high-order energies.
			The following point is crucial: 
			\textbf{for the proof to close, we need to choose $\toprate$ to be sufficiently large in a
			manner that depends \underline{only} on the absolute constants $M_{\mathrm{abs}}$. In particular,
			$\toprate$ does not depend on $\Ntop$.} 
		\item Once $\toprate$ has been chosen to be sufficiently large (as described in the previous point),
			for the proof to close, \textbf{we need to choose $\Ntop$ to be sufficiently large
			in a manner that depends \underline{only} on the integer $\toprate$ fixed in the previous step}.
		\item Once $\Ntop$ has been chosen to be sufficiently large (as described in the previous point),
			to close the proof, we must choose
			$\epd$ to be sufficiently small
			in a manner that is allowed to depend on all other parameters and constants.
			We must also choose $\mathring{\upalpha}$ to be sufficiently small
			in a manner that depends only on the equation of state and 
			$\bar{\varrho}$.
			We always assume that $\epd^{\frac{1}{2}} \leq \mathring{\upalpha}$.
		\item In contrast to $M_{\mathrm{abs}}$, 
			the constants 
			$\Ccrit$ are less delicate and
			are allowed to depend on
			the equation of state, 
			$\bar{\varrho}$, 
			$\mathring{\upsigma}$, 
			$\mathring{\updelta}$, 
			and
			$\mathring{\updelta}_*^{-1}$.
			We use the notation ``$\Ccrit$'' to emphasize that these constants multiply difficult, borderline 
			energy estimate error terms,
			but we could have just as well denoted these constants by ``$C$''
			(where $C$ has the properties described in the next point),
			and the proof would go through.
		\item Unless otherwise stated, 
		``general'' constants $C$ are allowed 
			to depend on
			$\Ntop$,
			$M_{\mathrm{abs}}$,
			the equation of state, 
			$\bar{\varrho}$, 
			$\mathring{\upsigma}$, 
			$\mathring{\updelta}$, 
			and
			$\mathring{\updelta}_*^{-1}$.
			When we write $A \lesssim B$, it means that there exists a $C > 0$ with the above dependence properties 
			such that $A \leq CB$.
			Moreover, $A \approx B$ means that $A \lesssim B$ and $B \lesssim A$.
\end{itemize}

\subsection{An outline of the rest of the appendix}
In sections \ref{sec:appendix.top.order}--\ref{SS:MAININTEGRALINEQUALITIES},
we will derive the estimates we need to prove Proposition~\ref{prop:wave}.
The conclusion of the proof of Proposition~\ref{prop:wave} 
is located in Section~\ref{sec:sketch.prop.wave}. 

Proposition~\ref{prop:wave} is an analog of the similar result \cite[Proposition~14.1]{LS}.
In fact, in our proof of the proposition, we will exactly follow the strategy from \cite{LS}. 
For this reason, we will only focus on terms which did not already appear in \cite{LS}. 
We begin by identifying the most difficult wave equation error terms
in Section~\ref{sec:appendix.top.order}. As in \cites{jSgHjLwW2016,LS}, these hardest terms are commutator terms
involving the top-order derivatives of $\mytr\upchi$, which we control using the following steps:
\begin{itemize}
\item In Section~\ref{sec:mod.quan.eq}, we write down the transport equations satisfied by the important \emph{modified quantities}. The modified quantities are special combinations of solution variables 
	involving $\mytr\upchi$.
	With the help of the Raychaudhuri equation \eqref{eq:3D.Raychaudhuri}, 
	the modified quantities will allow us to avoid the loss of a derivative at the top-order 
	and/or allow us to avoid fatal borderline error integrals.
\item In Section~\ref{sec:elliptic.on.spheres}, we use elliptic estimates on $\ell_{t,u}$ to control the top-order derivatives of $\hat{\upchi}$ in terms of the modified quantities.
\item In Section~\ref{SS:SPLITENERGIES}, we define partial energies, which are similar
	to the energies we defined in Section~\ref{SS:DEFSOFENANDFLUX},
	but they control all wave variables \underline{except for} the ``difficult'' one $\mathcal{R}_{(+)}$
	(which is such that $|\partial_1 \mathcal{R}_{(+)}|$ blows up as the shock forms).
	As in \cite{LS}, the partial energies play an important role in allowing us
	to close the proof using a universal number of derivatives, that
	is, a number $\Ntop$ that is independent of the equation of state and all parameters
	in the problem; the role of these partial energies will be made clear in
	Section~\ref{sec:sketch.prop.wave}.
	\item In Section~\ref{sec:est.for.modified}, we use the transport equations in Section~\ref{sec:mod.quan.eq} and the 	
	estimates in Section~\ref{sec:elliptic.on.spheres} to obtain the bounds for the top-order derivatives of $\mytr\upchi$.
\end{itemize}
At this point in the proof, we will have obtained all of the main new estimates we need to prove
Proposition~\ref{prop:wave}. In Section~\ref{SS:MAININTEGRALINEQUALITIES},
	we use our estimates for the top-order derivatives of $\mytr\upchi$ to derive preliminary energy integral inequalities
	for the wave equation solutions.
	These are the same integral inequalities that were derived in \cite[Proposition~14.3]{LS},
	except they include the new terms generated by the inhomogeneous terms $\mathfrak{G}$ 
	featured in the statement of Proposition~\ref{prop:wave}.
	Finally, in Section~\ref{sec:sketch.prop.wave}, we use these integral inequalities
	and a slightly modified version of the Gr\"onwall-type argument used in the proof
	of \cite[Proposition~14.1]{LS}, carefully tracking the different kinds of constants,
	thereby obtaining a priori estimates for the energies 
	and concluding the proof of Proposition~\ref{prop:wave}.

We close this section with three remarks to help the reader
understand how we use cite/use results that were proved in \cite{LS}.

\begin{remark}[\textbf{Implicit reliance on results we have already proved}]
	The estimates in this appendix rely, 
	in addition to the bootstrap assumptions, 
	on many of the estimates 
	that we independently derived in Section~\ref{sec:geometry},
	such as the results
	of Propositions~\ref{prop:geometric.low}, 
	\ref{prop:geometric.low.2}, 
	\ref{P:LINFTYHIGHERTRANSVERSAL},
	\ref{prop:mus.int},
	\ref{prop:almost.monotonicity},
	and \ref{prop:geometric.top}.
	Many of the results that we cite from \cite{LS} rely on
	these propositions, and we will not always explicitly indicate
	the dependence of the results of \cite{LS} on these propositions.
\end{remark}

\begin{remark}[$\varepsilon$ \textbf{vs} $\epd^{\frac{1}{2}}$]
	\label{R:EPSILONVSSQRRTDATAEPSILON}
	The bootstrap smallness parameter ``$\varepsilon$'' from \cite{LS} 
	should be identified with the quantity $\epd^{\frac{1}{2}}$
	in our bootstrap assumptions \eqref{BA:SMALLWAVEVARIABLESUPTOONETRANSVERSALDERIVATIVE}--\eqref{BA:C.D}.
	For this reason, various error terms from \cite{LS}
	reappear in the present paper, but with
	the factors of $\varepsilon$ replaced by $\epd^{\frac{1}{2}}$.
	This minor point has no substantial effect on our analysis,
	and we will often avoid explicitly
	pointing out that the error terms from \cite{LS} need to be modified as such.
\end{remark}

\begin{remark}[\textbf{Vorticity terms have been soaked up into $\mathfrak{G}$}]
	\label{R:VORTICITYSOAKEDUPINTOINHOMOGENEOUSTERMS}
	Many error terms in the estimates of \cite{LS} involve vorticity terms
	that are generated by the vorticity terms
	on the RHS of the wave equations.
	However, in this appendix, we have soaked these
	error terms up into our definition of the inhomogeneous terms $\mathfrak{G}$
	in Proposition~\ref{prop:wave}. For this reason,
	it is to be understood that many of the estimates
	cited from \cite{LS} have to be modified so that these vorticity terms
	are absent and are instead replaced with analogous error terms that depend on $\mathfrak{G}$
	(where throughout the appendix, we carefully explain how the term $\mathfrak{G}$ appears in various estimates).
\end{remark}

\subsection{The top-order commutator terms that require the modified quantities}\label{sec:appendix.top.order}
To begin, we recall that $\{Y,Z\}$ 
denotes the commutation vectorfields tangent to $\ell_{t,u}$,
and that we use the notation $\slashed{\mathcal{P}}$ to denote 
a generic element of this set. 
In the following proposition, we identify the most difficult error terms in the
top-order commuted wave equations.

\begin{proposition}[Identifying the most difficult commutator terms]
\label{prop:identity.main.wave.commutator.terms}
Let $\mathfrak G$ denote the inhomogeneous terms in the wave equations
from Proposition~\ref{prop:wave}. Then solutions to the wave equations of
Proposition~\ref{prop:wave} satisfy the following top-order wave equations
(which identify the most difficult commutator terms):
\begin{align}
\upmu \square_g (\slashed{\mathcal{P}}^{\Ntop-1} L \Psi) =  &\:(\slashed d^\sharp \Psi)(\upmu \slashed d \slashed{\mathcal{P}}^{\Ntop-1} \mytr \upchi)  + \slashed{\mathcal{P}}^{\Ntop-1} L\mathfrak G + \mbox{Harmless}, \label{eq:hard.commutator.L}\\
\upmu \square_g (\slashed{\mathcal{P}}^{\Ntop-1} Y\Psi) =  &\:(\bX \Psi) (\slashed{\mathcal{P}}^{\Ntop-1} Y \mytr \upchi) + \Speed^{-2} X^2 (\slashed d^\sharp \Psi)(\upmu \slashed d \slashed{\mathcal{P}}^{\Ntop-1} \mytr \upchi) \notag \\
&\: + \slashed{\mathcal{P}}^{\Ntop-1} Y\mathfrak G+ \mbox{Harmless}, \label{eq:hard.commutator.Y}\\
\upmu \square_g (\slashed{\mathcal{P}}^{\Ntop-1} Z\Psi) = &\: (\bX \Psi) (\slashed{\mathcal{P}}^{\Ntop-1}Z \mytr \upchi) + \Speed^{-2} X^3 (\slashed d^\sharp \Psi)(\upmu \slashed d \slashed{\mathcal{P}}^{\Ntop-1} \mytr \upchi). \notag \\
&\: + \slashed{\mathcal{P}}^{\Ntop-1} Z\mathfrak G+ \mbox{Harmless}. \label{eq:hard.commutator.Z}
\end{align}
Above, the terms ``Harmless'' are precisely the 
$\mbox{Harmless}_{(Wave)}^{\leq \Ntop}$ terms defined in \cite[Definition~13.1]{LS},
except here we do not need to allow for the presence of vorticity-involving terms
in the definition of $\mbox{Harmless}_{(Wave)}^{\leq \Ntop}$
because we have soaked these terms up into our definition
of the wave equation inhomogeneous term $\mathfrak{G}$.

Moreover, for any \underline{other} top-order operator $\mathcal{P}^{\Ntop}$ 
(i.e., a top-order operator featuring at least two copies of $L$ or featuring only a single $L$ 
but in an order different from \eqref{eq:hard.commutator.L}),
there are no difficult commutator terms in the sense that the following equation holds:
\begin{align} \label{E:EASTYTOPORDERCOMMUTATOR}
	\upmu  \square_g (\mathcal{P}^{\Ntop} \Psi)
	& = \mathcal{P}^{\Ntop} \mathfrak G + \mbox{Harmless}.
\end{align}

\end{proposition}

\begin{proof}
This is exactly the same as \cite[Proposition~13.2]{LS} with the obvious modifications: 
(1) we have $\{L, Y, Z\}$ (as opposed to just $\{L, Y\}$) as commutation vectorfields, 
and (2) we have accounted for the presence of the inhomogeneous terms $\mathfrak{G}$.
We stress that even in three spatial dimensions,
the top-order derivatives of $\upchi$ that appear on 
RHSs~\eqref{eq:hard.commutator.L}--\eqref{eq:hard.commutator.Z}
only involve its trace-part $\mytr \upchi$, as opposed to involving the full tensor $\upchi$.
Roughly speaking, this follows from three basic facts: 
\textbf{i)} all of these top-order terms are generated when
all $\Ntop + 1$ derivatives (including the two coming from $\square_g$) 
on the LHSs fall on the components $\mathcal{P}^i$ (where $\mathcal{P} \in \{L, Y, Z\}$);
\textbf{ii)} all $\mathcal{P}^i$ can be expressed as functions $\Psi$ and $L^1,L^2,L^3$;
and \textbf{iii)} Lemma~\ref{L:VFBASIC} 
and \eqref{eq:wave.in.terms.of.geometric} with $f \doteq u$ together 
imply that $\upmu \square_g u = - \mytr \upchi$.
Hence, considering also \eqref{E:FIRSTUPMU}, we have, schematically,  
that $\upmu \square_g \partial u = - \partial \mytr \upchi + \cdots$,
where $\cdots$ denotes terms that involve lower-order derivatives 
(i.e., up to second-order derivatives)
of the eikonal function $u$
and/or derivatives of $\Psi$.
Thus, \eqref{E:FIRSTUPMU}, \eqref{E:LGEOEQUATION}, \eqref{E:LUNITDEF}
imply that the scalar functions $\mathcal{P}^i$ satisfy, schematically,\footnote{Of course,
careful geometric decompositions are needed to obtain the precise form of the terms on 
RHSs~\eqref{eq:hard.commutator.L}--\eqref{eq:hard.commutator.Z}; here we are simply emphasizing
that the dependence of the top-order terms is through the derivatives of $\mytr \upchi$.}
$\square_g \mathcal{P}^i = \partial \mytr \upchi + \cdots$.

\qedhere
\end{proof}

\begin{remark}
Notice that in \cite[Proposition~13.2]{LS}, there is an additional difficult commutator term coming from (in the language of the present paper) the commutation with $\bX$. Since in this paper,
we use only the subset of energy estimates in \cite{LS} that 
\emph{avoid commutations with $\bX$}, 
an added benefit of our approach here is that we do not need to handle these additional terms.\footnote{Of course, even if these terms had been present in our work here, 
we could have handled them in the same way they were handled in \cite{LS}.}
\end{remark}

\subsection{The modified quantities and the additional terms in the transport equations}\label{sec:mod.quan.eq}
In order to control the top-order commutator terms from Proposition~\ref{prop:identity.main.wave.commutator.terms}, the idea from \cite{dC2007} is to introduce modified quantities, which are corrected versions 
of $\mytr\upchi$. The ``fully modified quantities'' 
solve transport equations with source terms that enjoy improved regularity, 
thus allowing us to avoid a loss of regularity at the top-order. 
The ``partially modified quantities'' lead to cancellations in the energy identities that allow us to avoid
error integrals whose singularity strength would have been too severe for us to control.

\begin{definition}[\textbf{Modified versions of the derivatives of} $\mytr \upchi$]
\label{D:TRANSPORTRENORMALIZEDTRCHIJUNK}
We define, for every\footnote{In practice, we need these quantities only to handle the difficult
terms from Proposition~\ref{prop:identity.main.wave.commutator.terms}, which
involve purely $\ell_{t,u}$-tangential derivatives of $\mytr \upchi$.
Put differently, in practice, we only need to use the quantities$\upchifullmodarg{\slashed{\mathcal{P}}^N}$.
\label{FN:FULLYMODIFIEDONLYNEEDLTUTANGENTIAL}} 
fixed string of order $N$ commutators $\mathcal{P}^N \in \mathscr{P}^{(N)}$, the \textbf{fully modified quantity} $\upchifullmodarg{\mathcal{P}^N}$ as follows:
\begin{subequations}
\begin{align}
	\upchifullmodarg{\mathcal{P}^N}
	& \doteq \upmu \mathcal{P}^N \mytr \upchi 
			 + 
			 \mathcal{P}^N \upchifullmodinhom,
		\label{E:TRANSPORTRENORMALIZEDTRCHIJUNK} 
			\\
	\upchifullmodinhom
	& \doteq - \vec{G}_{\Lunit \Lunit}\contr\Rad \threePsi
				- \frac{1}{2} \upmu \mytr \angG\contr\Lunit \threePsi
				- \frac{1}{2} \upmu \vec{G}_{\Lunit \Lunit}\contr\Lunit \threePsi 
				+ \upmu \angGmixedarg{\Lunit}{\#}\contr \angdiff \threePsi.
			\label{E:LOWESTORDERTRANSPORTRENORMALIZEDTRCHIJUNKDISCREPANCY}
\end{align}
\end{subequations}

We define, for every\footnote{As in Footnote~\ref{FN:FULLYMODIFIEDONLYNEEDLTUTANGENTIAL},
in practice, we only need to use the quantities$\upchipartialmodarg{\slashed{\mathcal{P}}^N}$.
\label{FN:PARTIALLYMODIFIEDONLYNEEDLTUTANGENTIAL}}  
fixed string of order $\mathcal{P}^N \in \mathscr{P}^{(N)}$, the \textbf{partially modified quantity} $\upchipartialmodarg{\mathcal{P}^N}$ as follows:
\begin{subequations}
\begin{align}
	\upchipartialmodarg{\mathcal{P}^N}
	& \doteq \mathcal{P}^N \mytr \upchi 
		+ \upchipartialmodinhomarg{\mathcal{P}^N},
		\label{E:TRANSPORTPARTIALRENORMALIZEDTRCHIJUNK} \\
	\upchipartialmodinhomarg{\mathcal{P}^N}
	& \doteq 
		- 
		\frac{1}{2} \mytr \angG\contr\Lunit \mathcal{P}^N \threePsi
		+ 
		\angGmixedarg{\Lunit}{\#}\contr \angdiff \mathcal{P}^N \threePsi.
			\label{E:TRANSPORTPARTIALRENORMALIZEDTRCHIJUNKDISCREPANCY}
\end{align}
\end{subequations}
\end{definition}

\begin{proposition}[Transport equations satisfied by the modified quantities]
\label{P:FULLYMODIFIEDTRANSPORTWITHEXTRATERMS}
The fully modified quantities solve the following modified version of equation \cite[(6.9)]{LS},
where $\vec{\mathfrak G}$ denotes the array of inhomogeneous terms in the wave equations
from Proposition~\ref{prop:wave}:
\begin{equation}\label{eq:fully.modified.transport}
\begin{split}
L \upchifullmodarg{\mathcal{P}^{\Ntop}}  
- 
\left( 2\f{L\upmu}{\upmu} - 2\mytr\upchi \right) \upchifullmodarg{\mathcal{P}^{\Ntop}} 
&
=  
\mbox{Non-vorticity-involving terms in \cite[(6.9)]{LS}} 
	\\
& \ \
- 
\mathcal{P}^{\Ntop} (\upmu |\chih|^2)
+ 
\frac 12 \mathcal{P}^{\Ntop} (\vec{G}_{\Lunit \Lunit}\contr \vec{\mathfrak G}).
\end{split}
\end{equation}

Moreover, the partially modified quantities
solve the following modified version of equation \cite[(6.10)]{LS}:
\begin{equation}\label{eq:partially.modified.transport}
L\upchipartialmodarg{\mathcal{P}^{\Ntop-1}} 
=  
\mbox{Terms in \cite[(6.10)]{LS}} 
- 
\mathcal{P}^{\Ntop-1} (|\chih|^2).
\end{equation}
\end{proposition}

\begin{remark}
\label{R:NOVORTICITYTERMSALGEBRA}
We clarify that the vorticity-involving terms in \cite[(6.9)]{LS} are absent from
RHS~\eqref{eq:fully.modified.transport}
because we have soaked these terms up into our definition
of the wave equation inhomogeneous term $\mathfrak{G}$. 
\end{remark}

\begin{proof}[Proof of Proposition~\ref{P:FULLYMODIFIEDTRANSPORTWITHEXTRATERMS}]
The key point is that the derivations of both \cite[(6.9), (6.10)]{LS} 
used the Raychaudhuri transport equation satisfied by $\mytr \upchi$, 
and thus we need to take into account the additional $-\upmu |\chih|^2$ term in \eqref{eq:3D.Raychaudhuri} as compared to \eqref{eq:2D.Raychaudhuri}.

The derivation of \cite[(6.9)]{LS} consists of two steps. 
First, in \cite[Lemma~6.1]{LS}, one expresses $\upmu \mathrm{Ric}_{LL}$ in terms of a sum of two terms: one term is a total $L$ derivative, and the other term is of lower order; see \cite[(6.1)]{LS}. 
Step~1 in particular uses the wave equations $\upmu \square_{g(\threePsi)} \Psi_{\iota} = \cdots$. 
In the second step, one combines the result of \cite[Lemma~6.1]{LS} 
with the $1+2$-dimensional Raychaudhuri equation \eqref{eq:2D.Raychaudhuri} and 
then commutes the resulting equation to obtain \cite[(6.9)]{LS}. 
In our setting, each step requires a small modification.
\begin{itemize}
\item In the first step, instead of $\upmu \square_{g(\threePsi)} \Psi_{\iota} = \cdots$, 
we have $\upmu \square_{g(\threePsi)} \Psi_\iota = \mathfrak G_{\iota}$. Thus,
we get an additional term $\frac 12 \mathcal{P}^{\Ntop} (\vec{G}_{\Lunit \Lunit}\contr \vec{\mathfrak G})$
on RHS~\eqref{eq:fully.modified.transport}.
\item In the second step, we need to use the $1+3$-dimensional Raychaudhuri equation
\eqref{eq:3D.Raychaudhuri} instead of \eqref{eq:2D.Raychaudhuri} and get the extra term 
$-\mathcal{P}^{\Ntop} (\upmu |\chih|^2)$
on RHS~\eqref{eq:fully.modified.transport}.
\end{itemize}
We thus obtain \eqref{eq:fully.modified.transport}.

The derivation of \cite[(6.10)]{LS} is simpler because its proof relies only on 
the $1+2$-dimensional Raychaudhuri equation \eqref{eq:2D.Raychaudhuri} 
(in particular, it does not rely on the wave equations $\upmu \square_{g(\threePsi)} \Psi_{\iota} = \cdots$). 
Thus,
to obtain \eqref{eq:partially.modified.transport}, 
we simply replace the application of \eqref{eq:2D.Raychaudhuri} from \cite[(6.10)]{LS} 
by an application of \eqref{eq:3D.Raychaudhuri}. 
The additional term in \eqref{eq:partially.modified.transport} is a result of the extra 
$- \upmu |\chih|^2$ term in 
\eqref{eq:3D.Raychaudhuri} compared to \eqref{eq:2D.Raychaudhuri}. 
\qedhere
\end{proof}

\subsection{Control of the geometry of $\ell_{t,u}$ and the elliptic estimates for $\chih$}\label{sec:elliptic.on.spheres}
The following elliptic estimate is standard; see \cite[Lemma~8.8]{dC2007}.

\begin{lemma}[\textbf{Elliptic estimate for symmetric, trace-free tensorfields}]
\label{lem:elliptic}
Let $(\mathcal{M}_2, {\slashed{\gamma}})$ 
be a closed, orientable Riemmanian manifold, and let $\mu$ be a non-negative function on $\mathcal{M}_2$. 
Then the following estimate 
holds for all trace-free symmetric covariant $2$-tensorfield $\xi$ belonging to 
$W^{1,2}(\mathcal{M}_2,{\slashed{\gamma}})$:
\begin{equation}\label{eq:Bochner}
\int_{\mathcal{M}_2} \mu^2 (\f 12 |\nab\xi|_{\slashed{\gamma}}^2 + 2 \mathfrak K_{\slashed{\gamma}} |\xi|_{{\slashed{\gamma}}}^2) \,\mathrm{dA}_{{\slashed{\gamma}}} \leq 3 \int_{\mathcal{M}_2} \mu^2 |\slashed{\mathrm{div}}_{\slashed{\gamma}}\xi|_{\slashed{\gamma}}^2 \,\mathrm{dA}_{{\slashed{\gamma}}} + 3 \int_{\mathcal{M}_2} |\nab \mu|^2_{\slashed{\gamma}} |\xi|_{\slashed{\gamma}}^2 \,\mathrm{dA}_{{\slashed{\gamma}}},
\end{equation}
where $\nab$, $\slashed{\mathrm{div}}_{\slashed{\gamma}}$, $\mathfrak K_{\slashed{\gamma}}$ and $\mathrm{dA}_{{\slashed{\gamma}}}$ are respectively the Levi-Civita connection, divergence operator, Gaussian curvature and induced area measure associated with ${\slashed{\gamma}}$. 
\end{lemma}

In order to use Lemma~\ref{lem:elliptic}, we need an $L^\infty$ estimate for the Gaussian curvature
of the tori $(\ell_{t,u},\slashed{g})$. 
We provide this basic estimate in the following proposition.
\begin{proposition}\label{prop:Gauss.curv.est}
The Gaussian curvature $\mathfrak K_{\slashed{ g}}$ of $(\ell_{t,u},\slashed{g})$
satisfies the following estimate for every $(t,u) \in [0,\Tboot)\times [0,U_0]$:
$$\| \mathfrak K_{\slashed{ g}} \|_{L^\infty(\Mtu)} \ls \epd^{\f 12}.$$
\end{proposition}
\begin{proof}
It is a standard fact that at fixed $(t,u)$, 
$\mathfrak K_{\slashed{ g}}$ can be expressed in terms of
the components of
$\slashed{g}$, $\slashed{g}^{-1}$ 
with respect to the coordinate system $(x^2,x^3)$ on $\ell_{t,u}$
and their first and second partial derivatives with respect to the geometric coordinate vectorfields $\srd_2, \srd_3$.
Schematically, we have 
$\mathfrak{K}_{\slashed{g}} 
= 
\slashed{g}^{-1} \cdot \slashed{g}^{-1} \cdot \srd^2 \slashed{g} 
+
\slashed{g}^{-1} \cdot \slashed{g}^{-1} \cdot \slashed{g}^{-1} \cdot \srd \slashed{g} \cdot \srd \slashed{g}$,
where $\srd \in \lbrace \srd_2,\srd_3 \rbrace$.

Recalling the expression for the induced metric $\slashed{g}$ in Lemma~\ref{lem:induced.metric} and 
the relations between the vectorfields in 
Lemma~\ref{L:GEOMETRICCOORDINATEVECTORFIELDSINTERMSOFCARTESIANVECTORFIELDS},
we see that 
the desired estimate for $\mathfrak K_{\slashed{ g}}$ follows from 
Proposition~\ref{prop:geometric.low.2}. \qedhere
\end{proof}

We now apply the elliptic estimate in Lemma~\ref{lem:elliptic} to control the top-order 
derivatives of $\chih$ in terms of the top-order 
pure $\ell_{t,u}$-tangential derivatives of $\mytr\upchi$.
\begin{proposition}\label{prop:Codazzi}
The following estimate holds for\footnote{Recall that 
$
\slashed{\mathcal L}_{\slashed{\mathcal{P}}}$
denotes Lie differentiation with respect to elements
$\slashed{\mathcal{P}} \in \lbrace Y, Z \rbrace$,
followed by projection onto $\ell_{t,u}$.
} 
the $\Ntop$-th $\ell_{t,u}$-tangential derivatives of $\chih$ for every $(t,u) \in [0,\Tboot) \times [0,U_0]$:
$$\| \upmu (\slashed{\mathcal L}_{\slashed{\mathcal{P}}})^{\Ntop} \chih \|_{L^2(\Sigma_t^u)} 
\ls 
\| \upmu \slashed{\mathcal{P}}^{\Ntop} \mytr\upchi \|_{L^2(\Sigma_t^u)} 
+ 
\epd^{\f 12} \upmu_{\star}^{-\toprate+0.9}(t).$$
\end{proposition}
\begin{proof}

\pfstep{Step~0: Preliminaries}
Throughout the proof, we will silently use the following observations,
valid for $\mathcal{P} \in \lbrace L,Y,Z \rbrace$
and
$\slashed{\mathcal{P}} \in \{Y,Z\}$,
where $\smoothfunction(\cdot)$ denotes a generic smooth function of its arguments that is allowed
to vary from line to line.
\begin{itemize}
	\item The component functions $X^1,X^2,X^3$ are smooth functions of
		the $L^i$ and $\Psi$; see \eqref{E:TRANSPORTVECTORFIELDINTERMSOFLUNITANDRADUNIT}.
		The same holds for the component functions $\mathcal{P}^0,\mathcal{P}^1,\mathcal{P}^2,\mathcal{P}^3$;
		this is obvious for $\mathcal{P}=L$, while see Lemma~\ref{lem:slashed} for $\mathcal{P}=Y,Z$.
		Similarly, the geometric coordinate component functions $\gsphere_{AB}$ and $(\gsphere^{-1})^{AB}$
		are smooth functions of
		the $L^i$ and $\Psi$; see Lemma~\ref{lem:induced.metric}.
	\item For $\srd \in \lbrace \srd_2, \srd_3 \rbrace$, we have the following schematic identity:
	$
\srd
=
\smoothfunction(L^i,\Psi)
Y
+
\smoothfunction(L^i,\Psi)
Z
$; see Lemma~\ref{L:GEOMETRICCOORDINATEVECTORFIELDSINTERMSOFCARTESIANVECTORFIELDS}.
\item For $\ell_{t,u}$-tangent one-forms $\upxi$, we have
	$|\upxi| \approx \sum_{A= 2,3} |\upxi_A| \approx |\upxi_Y| + |\upxi_Z| \doteq |\upxi(Y)| + |\upxi(Z)|$;
	this follows from the discussion in the previous two points,
	the bootstrap assumptions \eqref{BA:LARGERIEMANNINVARIANTLARGE}--\eqref{BA:W.Li.small},
	and the $L^{\infty}$ estimates for $L_{(Small)}^i$ from Proposition~\ref{prop:geometric.low}.
	In particular, for scalar functions $\phi$, we have 
	$|\angD \phi| \approx \sum_{A= 2,3} |\srd_A \phi| \approx |Y \phi| + |Z \phi|$.
	Analogous estimates hold for $\ell_{t,u}$-tangent tensorfields of any order.	
\item For type $\binom{0}{n}$ tensorfields, we have the following covariant identity,
		expressed schematically:
		$[\nab,\slashed{\mathcal L}_{\mathcal{P}}] \upxi
		=
		(\nab \slashed{\mathcal L}_{\mathcal{P}} \gsphere)
		\cdot \upxi
		$. It is straightforward to check 
		that $\slashed{\mathcal L}_{\mathcal{P}} \gsphere$
		is in fact equal to the $\ell_{t,u}$-projection of the deformation tensor of $\mathcal{P}$
		(the deformation tensor itself is equal to $\mathcal{L}_{\mathcal{P}} g$, where $g$ is the acoustical metric).
	\item Relative to the geometric coordinates $(t,u,x^2,x^3)$,
we have
$
\slashed{\mathcal L}_{\mathcal{P}} \gsphere
=
\smoothfunction(L^i,\Psi)
(\mathcal{P} L^i,\mathcal{P} \Psi)
$
(where the $\mathcal{P}$'s on the LHS and RHS can be different).
\item For $\ell_{t,u}$-tangent tensorfields $\upxi$, we have the following schematic
		identity, valid relative to the geometric coordinates, where $\srd \in \lbrace \srd_2, \srd_3 \rbrace$:
		$\nab \upxi - \srd \upxi = \smoothfunction(L^i,\Psi) 
		\upxi
		\cdot (\slashed{\mathcal{P}} L^i,\slashed{\mathcal{P}} \Psi)$;
		this follows from expressing $\nab$ in terms of geometric coordinate partial derivatives
		and the Christoffel symbols of $\gsphere$ and then expressing
		$\srd = \smoothfunction(L^i,\Psi) \slashed{\mathcal{P}}$ on the RHS.
	\item For $\ell_{t,u}$-tangent tensorfields $\upxi$, we have the following schematic
		identity, valid relative to the geometric coordinates, where $\srd \in \lbrace \srd_2, \srd_3 \rbrace$:
		$\slashed{\mathcal L}_{\slashed{\mathcal{P}}} \upxi 
		= 
\slashed{\mathcal{P}}^A
\srd_A \upxi
+
\smoothfunction(L^i,\Psi)
\upxi
\cdot
(\slashed{\mathcal{P}} L^i, \slashed{\mathcal{P}} \Psi)
$
(where the $\slashed{\mathcal{P}}$'s on the LHS and RHS can be different);
this formula is straightforward to verify relative to the geometric coordinates.
\item  For $\ell_{t,u}$-tangent tensorfields $\upxi$, we have the following schematic
		identity, valid relative to the geometric coordinates, where $\srd \in \lbrace \srd_2, \srd_3 \rbrace$:
$\nab_{\slashed{\mathcal{P}}} \xi 
-
\slashed{\mathcal L}_{\slashed{\mathcal{P}}} \upxi
= 
\smoothfunction(L^i,\threePsi)
\upxi
\cdot
(\slashed{\mathcal{P}} L^i, \slashed{\mathcal{P}} \Psi)
$
(where the $\slashed{\mathcal{P}}$'s on the LHS and RHS can be different);
this formula is straightforward to verify relative to the geometric coordinates.
\item If $f$ is a scalar function, then $\mathcal{L}_{\mathcal{P}} \slashed{d} f = \slashed{d} \mathcal{P}f$,
where $\slashed{d}$ denotes $\ell_{t,u}$-gradient of $f$;
this formula is straightforward to verify relative to the geometric coordinates.
\end{itemize}

\pfstep{Step~1: Codazzi equation\footnote{We use the phrase ``Codazzi equation''
because the equations we use in this analysis are
closely related to the classical Codazzi equation,
which links $\angdiv \upchi$, $\angD \mytr\upchi$, and
the curvature components of the acoustical metric.}} 
We compute $(\slashed{\mathcal L}_{\mathcal{P}})^{\Ntop-1} \nab^A \upchi_{BA}$ by differentiating 
\eqref{E:CHIINTERMSOFOTHERVARIABLES} with the operator
$(\slashed{\mathcal L}_{\mathcal{P}})^{\Ntop-1}\slashed{\mathrm{div}}_{\slashed{g}}$ and treating all capital Latin indices as tensorial indices, while treating all lowercase Latin indices as corresponding to scalar functions. We clarify that the tensor on LHS~\eqref{E:CHIINTERMSOFOTHERVARIABLES}
is symmetric, while the first, third, and fourth products on RHS~\eqref{E:CHIINTERMSOFOTHERVARIABLES} 
are not. Hence, for clarity, we emphasize that
when we write ``differentiating \eqref{E:CHIINTERMSOFOTHERVARIABLES} with 
$(\slashed{\mathcal L}_{\mathcal{P}})^{\Ntop-1} \nab^A \upchi_{BA}$,'' 
it is to be understood that the corresponding first term on the RHS is an $\ell_{t,u}$-tangent one-form
with index ``$B$'' whose top-order part (in the sense of the number of derivatives that fall on $L^a$)
is
$
(\slashed{\mathcal L}_{\mathcal{P}})^{\Ntop-1}
(g_{ab} ( (\slashed{g}^{-1})^{AC} \nab_C \srd_B L^a)\otimes \srd_A x^b))
=
(\slashed{\mathcal L}_{\mathcal{P}})^{\Ntop-1}
(g_{ab} ( (\slashed{g}^{-1})^{AC} \nab_B \srd_C L^a)\otimes \srd_A x^b))
$,
where to obtain the last equality,
we used
the commutation identity $\nab_C \srd_B L^a = \nab_B \srd_C L^a$,
which is a consequence of the torsion-free
property of $\nab$ and the fact that we are viewing
the Cartesian components $L^a$ as scalar functions.
Notice that unless all the $\Ntop$ derivatives fall
on the factor $\slashed{d} L^a$ in the first product on RHS~\eqref{E:CHIINTERMSOFOTHERVARIABLES}, 
the expression involves at most $\Ntop$ derivatives on $L$ and 
$\Psi$, and we can control such terms using the bounds we have obtained thus far. 
In total, 
using the symmetry property $\upchi_{BA} = \upchi_{AB}$,
isolating the terms featuring the top-order derivatives of the components $L^a$,
and estimating the remaining terms with 
\eqref{BA:W1}--\eqref{BA:W.Li.small} and
Propositions~\ref{prop:geometric.low} and \ref{prop:geometric.top}, we 
obtain:
\begin{equation}\label{eq:to.derive.Codazzi.1}
\begin{split}
&\: \|\upmu (\slashed{\mathcal L}_{\slashed{\mathcal{P}}})^{\Ntop-1} \nab^A \upchi_{AB} -  \upmu(\slashed{\mathcal L}_{\slashed{\mathcal{P}}})^{\Ntop-1} \{ g_{ab} (\slashed{g}^{-1})^{AC} (\nab_C \srd_B L^a )(\srd_A x^b) \} \|_{L^2(\Sigma_t^u)} \\
\ls &\: 
 \|\upmu \mathcal{P}^{\Ntop+1} \Psi\|_{L^2(\Sigma_t^u)} 
+
\|\upmu \mathcal{P}^{[1,\Ntop]} (L^i,\Psi)\|_{L^2(\Sigma_t^u)} 
\ls \epd^{\f 12} \upmu_{\star}^{-\toprate +0.9}(t).
\end{split}
\end{equation}
We then compute $(\slashed{\mathcal L}_{\slashed{\mathcal{P}}})^{\Ntop-1} \srd_B \mytr\upchi$ in a similar manner using 
\eqref{E:TRCHIINTERMSOFOTHERVARIABLES} to obtain:
\begin{equation}\label{eq:to.derive.Codazzi.2}
\begin{split}
&\: \|\upmu (\slashed{\mathcal L}_{\slashed{\mathcal{P}}})^{\Ntop-1} \srd_B \mytr\upchi - \upmu (\slashed{\mathcal L}_{\slashed{\mathcal{P}}})^{\Ntop-1} \{ g_{ab} (\slashed{g}^{-1})^{AC} (\nab_B \srd_C L^a )(\srd_A x^b) \} \|_{L^2(\Sigma_t^u)}  \ls \epd^{\f 12} \upmu_{\star}^{-\toprate +0.9}(t).
\end{split}
\end{equation}
In view of the commutation identity $\nab_C \srd_B L^a = \nab_B \srd_C L^a$ mentioned above
(which implies that the second terms on LHSs of \eqref{eq:to.derive.Codazzi.1} and \eqref{eq:to.derive.Codazzi.2} coincide), we can use \eqref{eq:to.derive.Codazzi.1}, 
\eqref{eq:to.derive.Codazzi.2},
and the triangle inequality to obtain:
\begin{equation}\label{eq:Codazzi!}
\begin{split}
 \|\upmu (\slashed{\mathcal L}_{\slashed{\mathcal{P}}})^{\Ntop-1} \nab^A \upchi_{AB}\|_{L^2(\Sigma_t^u)} \ls &\: \| \upmu (\slashed{\mathcal L}_{\slashed{\mathcal{P}}})^{\Ntop-1} \srd_B \mytr\upchi \|_{L^2(\Sigma_t^u)} + \epd^{\f 12} \upmu_{\star}^{-\toprate +0.9}(t) \\
\ls &\: \| \upmu \slashed{\mathcal{P}}^{\Ntop} \mytr\upchi \|_{L^2(\Sigma_t^u)} + \epd^{\f 12} \upmu_{\star}^{-\toprate +0.9}(t),
\end{split}
\end{equation}
where to obtain the last line, 
we used the commutation identity 
$(\slashed{\mathcal L}_{\slashed{\mathcal{P}}})^{\Ntop-1} \srd_B \mytr\upchi 
= 
\srd_B \slashed{\mathcal{P}}^{\Ntop-1} \mytr\upchi$
(in which we are thinking of both sides as $\ell_{t,u}$-tangent one-forms 
with components corresponding to the index ``$B$''),
the schematic identity 
$
\srd
=
\smoothfunction(L^i,\Psi)
Y
+
\smoothfunction(L^i,\Psi)
Z
$,
and Proposition~\ref{prop:geometric.low}.

Now since $(\slashed{\mathrm{div}}_{\slashed{g}} \chih)_B = (\slashed{\mathrm{div}}_{\slashed{g}} \upchi)_B - \f 12 \srd_B \mytr\upchi =  \nab^A \upchi_{AB} - \f 12 \srd_B \mytr\upchi$, we deuce from the estimate \eqref{eq:Codazzi!} 
that:
\begin{equation}\label{eq:Codazzi!!}
\begin{split}
\|\upmu (\slashed{\mathcal L}_{\slashed{\mathcal{P}}})^{\Ntop-1} \slashed{\mathrm{div}}_{\slashed{g}} \chih \|_{L^2(\Sigma_t^u)} 
\ls \| \upmu \slashed{\mathcal{P}}^{\Ntop} \mytr\upchi \|_{L^2(\Sigma_t^u)} + \epd^{\f 12} \upmu_{\star}^{-\toprate +0.9}(t).
\end{split}
\end{equation}

\pfstep{Step~2: Commuting $\slashed{\mathrm{div}}_{\slashed{g}}$ with $\slashed{\mathcal L}_{\slashed{\mathcal{P}}}$ derivatives} We now deduce from \eqref{eq:Codazzi!!} an estimate for $\slashed{\mathrm{div}}_{\slashed{g}} (\slashed{\mathcal L}_{\slashed{\mathcal{P}}})^{\Ntop-1} \chih$. For this, 
we simply note that the commutator 
$[\slashed{\mathrm{div}}_{\slashed{g}}, (\slashed{\mathcal L}_{\slashed{\mathcal{P}}})^{\Ntop-1}]\chih$ 
can be controlled by up to $\Ntop$ $\slashed{\mathcal{P}}$ derivatives of $\Psi$ and $L^i$. 
Hence, by \eqref{eq:Codazzi!!},  
\eqref{BA:W1}--\eqref{BA:W.Li.small}, 
and
Propositions~\ref{prop:geometric.low} and \ref{prop:geometric.top}, we have:
\begin{equation}\label{eq:commuted.Codazzi}
\begin{split}
&\: \| \upmu \slashed{\mathrm{div}}_{\slashed{g}} (\slashed{\mathcal L}_{\slashed{\mathcal{P}}})^{\Ntop-1} \chih \|_{L^2(\Sigma_t^u)}\\
\ls &\: \| \upmu  (\slashed{\mathcal L}_{\slashed{\mathcal{P}}})^{\Ntop-1} \slashed{\mathrm{div}}_{\slashed{g}} \chih \|_{L^2(\Sigma_t^u)} + \|\upmu \slashed{\mathcal{P}}^{[1,\Ntop]} (\Psi, L^i)\|_{L^2(\Sigma_t^u)} \\
\ls &\: \| \upmu  \slashed{\mathcal{P}}^{\Ntop} \mytr\upchi \|_{L^2(\Sigma_t^u)} 
+ 
\epd^{\f 12}  \upmu_{\star}^{-\toprate+0.9}(t).
\end{split}
\end{equation}

\pfstep{Step~3: Bounding the trace-part of $(\slashed{\mathcal L}_{\slashed{\mathcal{P}}})^{\Ntop-1} \chih$} By definition, $\mytr \chih = 0$. Note that the commutator 
$[\slashed{g}^{-1}, (\slashed{\mathcal L}_{\slashed{\mathcal{P}}})^{\Ntop-1}] \chih$ 
can be controlled by up to $\Ntop-1$ $\slashed{\mathcal{P}}$ derivatives of $\Psi$ and $L^i$. Hence, this commutator can be treated in the same way we treated the commutator term in Step~2, which yields the bound:
\begin{equation}\label{eq:tr.commuted.chih.est}
\begin{split}
&\: \|\mytr (\slashed{\mathcal L}_{\slashed{\mathcal{P}}})^{\Ntop-1} \chih\|_{L^2(\Sigma_t^u)} 
\ls \|[\slashed{g}^{-1}, (\slashed{\mathcal L}_{\slashed{\mathcal{P}}})^{\Ntop-1}] \chih\|_{L^2(\Sigma_t^u)} \ls \epd^{\f 12} \upmu_{\star}^{-\toprate+0.9}(t).
\end{split}
\end{equation}

Moreover, we can take a further $\slashed{\mathcal{P}}$-derivative of 
$\mytr (\slashed{\mathcal L}_{\slashed{\mathcal{P}}})^{\Ntop-1} \chih$, and the resulting term can be controlled by up to $\Ntop$ 
$\slashed{\mathcal{P}}$ derivatives of $\Psi$ and $L^i$. Therefore, using 
\eqref{BA:W1}--\eqref{BA:W.Li.small} and 
Propositions~\ref{prop:geometric.low} and \ref{prop:geometric.top}, we obtain:
\begin{equation}\label{eq:div.tr.commuted.chih.est}
\begin{split}
&\: \|\upmu \angD (\mytr (\slashed{\mathcal L}_{\slashed{\mathcal{P}}})^{\Ntop-1} \chih)\|_{L^2(\Sigma_t^u)} 
\ls \| \upmu \slashed{\mathcal{P}}^{[1,\Ntop]} (\Psi, L^i)\|_{L^2(\Sigma_t^u)} \ls \epd^{\f 12} \upmu_{\star}^{-\toprate+0.9}(t).
\end{split}
\end{equation}

\pfstep{Step~4: Elliptic estimates} Define $\xi$ to be the $\slashed{g}$-trace-free part of $(\slashed{\mathcal L}_{\slashed{\mathcal{P}}})^{\Ntop-1} \chih$, i.e.,
\begin{equation}\label{eq:def.xi.for.elliptic}
\xi_{AB} \doteq (\slashed{\mathcal L}_{\slashed{\mathcal{P}}})^{\Ntop-1} \chih_{AB} - \f 12 \slashed{g}_{AB} \mytr (\slashed{\mathcal L}_{\slashed{\mathcal{P}}})^{\Ntop-1} \chih.
\end{equation}
The term $(\slashed{\mathcal L}_{\slashed{\mathcal{P}}})^{\Ntop-1} \chih_{AB}$ 
on RHS~\eqref{eq:def.xi.for.elliptic}
can be written using \eqref{E:CHIINTERMSOFOTHERVARIABLES}, \eqref{E:TRCHIINTERMSOFOTHERVARIABLES} as an expression of up to $\Ntop$ 
$\mathcal{P}$ derivatives of $\Psi$ and $L^i$. Hence, by \eqref{E:CHIINTERMSOFOTHERVARIABLES}, \eqref{E:TRCHIINTERMSOFOTHERVARIABLES}, \eqref{BA:W1}--\eqref{BA:W.Li.small}, and
Propositions~\ref{prop:geometric.low} and \ref{prop:geometric.top}, we obtain:
\begin{equation}\label{eq:xi.L2.prelim}
\| (\slashed{\mathcal L}_{\slashed{\mathcal{P}}})^{\Ntop-1} \chih \|_{L^2(\Sigma_t^u)} 
\ls \| \mathcal{P}^{[1,\Ntop]} (\Psi, L^i)\|_{L^2(\Sigma_t^u)}  \ls \epd^{\f 12}\upmu_{\star}^{-\toprate+0.9}(t).
\end{equation}
Combining \eqref{eq:xi.L2.prelim} with \eqref{eq:tr.commuted.chih.est}, we find that:
\begin{equation}\label{eq:xi.L2.est}
\|\xi\|_{L^2(\Sigma_t^u)}\ls \epd^{\f 12} \upmu_{\star}^{-\toprate+0.9}(t).
\end{equation}

Moreover, in view of the algebraic relation
$$\slashed{\mathrm{div}}_{\slashed{g}} \xi = \slashed{\mathrm{div}}_{\slashed{g}} (\slashed{\mathcal L}_{\slashed{\mathcal{P}}})^{\Ntop-1} \chih - \frac{1}{2} \angD (\mytr (\slashed{\mathcal L}_{\slashed{\mathcal{P}}})^{\Ntop-1} \chih) $$
and the estimates \eqref{eq:commuted.Codazzi} and \eqref{eq:div.tr.commuted.chih.est}, we have:
\begin{equation}\label{eq:div.xi.est}
\| \upmu \slashed{\mathrm{div}}_{\slashed{g}} \xi \|_{L^2(\Sigma_t^u)} \ls \|\upmu \slashed{\mathcal{P}}^{\Ntop} \mytr\upchi \|_{L^2(\Sigma_t^u)} + \epd^{\f 12} \upmu_{\star}^{-\toprate+0.9}(t).
\end{equation}

Therefore, applying the elliptic estimates in Lemma~\ref{lem:elliptic} on $\ell_{t,u}$ with $\xi$ as in \eqref{eq:def.xi.for.elliptic} and $\mu = \upmu$, integrating over $u\in [0,U_0]$, and using 
\eqref{eq:div.tr.commuted.chih.est},
\eqref{eq:xi.L2.prelim}, 
\eqref{eq:xi.L2.est}, 
and \eqref{eq:div.xi.est}, 
as well as the Gauss curvature estimate in Proposition~\ref{prop:Gauss.curv.est} 
and the estimates of Proposition~\ref{prop:geometric.low}
(including the bound $|\nab \upmu|  \ls \epd^{\frac{1}{2}}$ that it implies), 
we obtain:
\begin{equation*}
\begin{split}
\|\upmu (\slashed{\mathcal L}_{\slashed{\mathcal{P}}})^{\Ntop} \chih \|_{L^2(\Sigma_t^u)} 
& 
\ls \|\upmu \nab (\slashed{\mathcal L}_{\slashed{\mathcal{P}}})^{\Ntop-1} \chih \|_{L^2(\Sigma_t^u)} + \|\upmu (\slashed{\mathcal L}_{\slashed{\mathcal{P}}})^{\Ntop-1} \chih \|_{L^2(\Sigma_t^u)}\\
\ls &\: \|\upmu \nab \xi \|_{L^2(\Sigma_t^u)} + \epd^{\f 12} \upmu_{\star}^{-\toprate+0.9}(t) \\
\ls &\: \|\upmu \slashed{\mathrm{div}}_{\slashed{g}} \xi \|_{L^2(\Sigma_t^u)} + (\|\mathfrak K_{\slashed{g}}\|_{L^\infty(\Sigma_t)}^{\f 12} + \|\nab \upmu\|_{L^\infty(\Sigma_t)}) \|\xi\|_{L^2(\Sigma_t^u)} + \epd^{\f 12} \upmu_{\star}^{-\toprate+0.9}(t) \\
\ls &\: \|\upmu \slashed{\mathcal{P}}^{\Ntop} \mytr\upchi \|_{L^2(\Sigma_t^u)} + \epd^{\f 12} \upmu_{\star}^{-\toprate+0.9}(t),
\end{split}
\end{equation*}
which is what we wanted to prove. \qedhere
\end{proof}

\subsection{The partial energies}
\label{SS:SPLITENERGIES}
To derive our top-order energy estimates for the wave equations, 
we will use the approach of \cite{LS},
which relies on distinguishing the ``full energies''
featured in definitions \eqref{eq:wave.energy.def.1}--\eqref{eq:wave.energy.def.4}
(which control all wave variables) 
from the ``partial energies,'' which are captured by the next definition.
The main point is that the partial energies do not control the difficult
almost Riemann invariant $\mathcal{R}_{(+)}$ 
(it is difficult in the sense that the shock formation is driven by the relative largeness 
of $|\Rad \mathcal{R}_{(+)}|$), and it turns out that this leads to 
easier error terms in the corresponding energy identities.
Importantly, we need to distinguish
the partial energies from the full energies
in order to close the proof using a uniform number of derivatives\footnote{We could close the proof without introducing the partial energies, but those simpler, less precise arguments would allow for the possibility
that the number of derivatives needed to close the estimates 
might depend on the equation of state, 
$\bar{\varrho}$, 
$\mathring{\upsigma}$, $\mathring{\updelta}$ and $\mathring{\updelta}_*^{-1}$.} 
$\Ntop$, that is, a number derivatives that does not depend on the equation of state
or any parameters in the problem; see the arguments in Section~\ref{sec:sketch.prop.wave}
for clarification on the role played by the partial energies in allowing us to close
the proof using a number of derivatives that is independent of the equation of state
and all parameters in the problem.

\begin{definition}[The partial energies]
\label{D:JAREDPARTIALENERGIES}
At the top-order, we define the partial energy by:
$$\mathbb{E}_{\Ntop}^{(Partial)}(t,u) 
\doteq \sup_{t'\in [0,t)}\sum_{\widetilde{\Psi}\in \{\mathcal R_{(-)}, v^2, v^3, s\}} \left(\|\bX \mathcal{P}^{\Ntop} \widetilde{\Psi}\|_{L^2(\Sigma_{t'}^u)}^2+ \|\sqrt{\upmu} \mathcal{P}^{\Ntop+1} \widetilde{\Psi}\|_{L^2(\Sigma_{t'}^u)}^2 \right).$$
Similarly, we separate the contribution of $\mathcal R_{(+)}$ from that of other components of $\Psi$ and define $\mathbb{F}_{\Ntop}^{(Partial)}$, $\mathbb K_{\Ntop}^{(Partial)}$, $\mathbb{Q}_{\Ntop}^{(Partial)}$ in an analogous way,
that is, as in Section~\ref{SS:DEFSOFENANDFLUX}, but without the $\mathcal R_{(+)}$-involving terms.
\end{definition}

\subsection{$L^2$ estimates for the top-order derivatives of $\mytr \upchi$ tied to the modified quantities}\label{sec:est.for.modified}

\begin{proposition}[$L^2$ estimates for the top-order derivatives of $\mytr \upchi$ tied to the fully modified quantities]
\label{prop:trch.top}
There exists an absolute positive constant $M_{\mathrm{abs}} \in \mathbb{N}$,
a positive constant
$\Ccrit \in \mathbb N$,
and a constant $C > 0$
(each having the properties described in Section~\ref{SS:ASSUMPTIONSANDCONSTANTS})  
such that the following estimate 
(whose right-hand side involves the wave energies \eqref{eq:wave.energy.def.1}--\eqref{eq:wave.energy.def.4}
as well as the partial energies of Definition~\ref{D:JAREDPARTIALENERGIES})
holds for every $(t,u) \in [0,\Tboot) \times [0,U_0]$:
\begin{equation}\label{eq:trch.top.sharp}
\begin{split}
&\: \|(\bX \mathcal{R}_{(+)}) \slashed{\mathcal{P}}^{\Ntop} \mytr \upchi \|_{L^2(\Sigma_t^u)} \\
\leq &\: \mbox{Non-vorticity-involving terms on RHS~\cite[(14.27)]{LS} with the boxed constants replaced by 
$\boxed{M_{\mathrm{abs}}}$} 	
	\\
&\: \ \ \mbox{and the constant $C_*$ replaced by $\boxed{\Ccrit}$} \\
&\: \ \ 
+ C \epd \upmu_{\star}^{-\toprate+ 0.9}(t) + C\upmu_{\star}^{-1}(t)\int_{t'=0}^{t'=t}  \| \mathcal{P}^{[1, \Ntop]}  \mathfrak G \|_{L^2(\Sigma_{t'}^u)} \, dt',
\end{split}
\end{equation}
and:
\begin{equation}\label{eq:trch.top.non.sharp}
\begin{split}
\| \upmu \slashed{\mathcal{P}}^{\Ntop} \mytr \upchi \|_{L^2(\Sigma_t^u)} 
\ls &\: 
\mbox{Non-vorticity-involving terms on RHS~\cite[(14.28)]{LS}} 
	\\
&\: \ \ + 
\epd \upmu_{\star}^{-\toprate+ 1.9}(t)  
+ 
\int_{t'=0}^{t'=t}  \| \mathcal{P}^{[1, \Ntop]}  \mathfrak G \|_{L^2(\Sigma_{t'}^u)} \, dt'.
\end{split}
\end{equation}
\end{proposition}

\begin{remark}
\label{R:VORTICITYTERMSHAVEBEENSOAKEDINTOGENERALINHOMTERM}
We clarify that in the proofs of \cite[(14.27)]{LS} 
and \cite[(14.28)]{LS}, the vorticity-involving inhomogeneous terms in the wave
equations led to error integrals on RHSs~\cite[(14.27)]{LS} and \cite[(14.28)]{LS} 
that involved the vorticity energies;
in contrast, on RHSs~\eqref{eq:trch.top.sharp}--\eqref{eq:trch.top.non.sharp}, 
the vorticity-involving terms are not explicitly indicated 
because we have soaked them up into our definition of 
the wave equation inhomogeneous term $\mathfrak{G}$.
\end{remark}

\begin{proof}
The proofs of both estimates are similar. 
We first discuss the proof of \eqref{eq:trch.top.non.sharp} in Steps~1--2.
In Step~3, we describe the changes we need in order to obtain \eqref{eq:trch.top.sharp}.
Throughout this proof, we freely use the observations made in Step~0 of the proof of
Proposition~\ref{prop:Codazzi}.

Following \cites{jSgHjLwW2016,LS}, 
in order to bound $\upmu \slashed {\mathcal{P}}^{\Ntop} \mytr \upchi$, 
we first control the fully modified quantity (recall the definition in \eqref{E:TRANSPORTRENORMALIZEDTRCHIJUNK}), and then bound the difference of $\upmu \slashed{\mathcal{P}}^{\Ntop} \mytr \upchi$ 
and 
the fully modified quantity. See the corresponding estimates in 
\cite[Lemma~13.9, Proposition~13.11, Lemma~14.14]{LS}.

\pfstep{Step~1: Controlling the inhomogeneous terms in \eqref{eq:fully.modified.transport}} 
We first estimate the two new terms on RHS~\eqref{eq:fully.modified.transport} in the following norms
(recall that here we are assuming that in \eqref{eq:fully.modified.transport}, 
$\mathcal{P}^{\Ntop}$ is equal to a pure $\ell_{t,u}$-tangential operator $\slashed{\mathcal{P}}^{\Ntop}$):
\begin{equation}\label{eq:to.be.controlled.for.modified}
\int_{t'=0}^{t'=t} \|\slashed{\mathcal{P}}^{\Ntop} (\upmu |\chih|^2) \|_{L^2(\Sigma_{t'}^u)}\, dt',\quad \int_{t'=0}^{t'=t} \| \slashed{\mathcal{P}}^{\Ntop}  
(\vec{G}_{\Lunit \Lunit}\contr \vec{\mathfrak G})\|_{L^2(\Sigma_{t'}^u)} \, dt'.
\end{equation}

\pfstep{Step~1(a): The $\slashed{\mathcal{P}}^{\Ntop} (\upmu |\chih|^2)$ term} 
For the first term in \eqref{eq:to.be.controlled.for.modified}, the most (and indeed only) difficult 
contribution arises when all operators $\slashed{\mathcal{P}}^{\Ntop}$ 
fall on one factor of $\chih$. For the lower-order terms, we use 
the identities \eqref{E:CHIINTERMSOFOTHERVARIABLES}, 
\eqref{E:TRCHIINTERMSOFOTHERVARIABLES},
and 
$\chih_{AB} = \upchi_{AB} - \frac{1}{2} \gsphere_{AB} \mytr \upchi$,
\eqref{BA:LARGERIEMANNINVARIANTLARGE}--\eqref{BA:W.Li.small}, and 
Proposition~\ref{prop:geometric.low} to obtain the following pointwise estimates:
\begin{equation}\label{eq:chih.not.top.pointwise}
\begin{split}
&\: \left| \slashed{\mathcal{P}}^{\Ntop} (\upmu |\chih|^2) 
- 2 \upmu \chih^{\sharp\sharp} (\slashed{\mathcal L}_{\slashed{\mathcal{P}}})^{\Ntop} \chih \right| 
\ls \epd^{\f 12} |\slashed{\mathcal{P}}^{[1,\Ntop]}(\Psi,L^i,\upmu)|.
\end{split}
\end{equation}
From \eqref{BA:W1}, \eqref{BA:W2} and Proposition~\ref{prop:geometric.top}, 
and
the estimate \eqref{eq:chih.not.top.pointwise}, we see that:
\begin{equation}\label{eq:top.chih.but.not.really.top}
\begin{split}
&\: \| \slashed{\mathcal{P}}^{\Ntop} (\upmu |\chih|^2) 
- 
2 \upmu \chih^{\sharp\sharp} (\slashed{\mathcal L}_{\slashed{\mathcal{P}}})^{\Ntop} \chih \|_{L^2(\Sigma_t^u)} 
\ls \epd^{\f 12} \|\slashed{\mathcal{P}}^{[1, \Ntop]} (\Psi,L^i,\upmu) \|_{L^2(\Sigma_t^u)} 
\ls  \epd \upmu_{\star}^{-\toprate +1.4}(t).
\end{split}
\end{equation}

On the other hand, the top-order derivative $\upmu (\slashed{\mathcal L}_{\slashed{\mathcal{P}}})^{\Ntop}\chih$ can be bounded using Proposition~\ref{prop:Codazzi}, while the low-order factor $\chih^{\sharp\sharp}$
can be bounded in $L^{\infty}$ by $\lesssim \epd^{\frac{1}{2}}$ by virtue of 
the bootstrap assumptions \eqref{BA:LARGERIEMANNINVARIANTLARGE}--\eqref{BA:W.Li.small}
and the estimates of Proposition~\ref{prop:geometric.low}. 
Therefore, combining \eqref{eq:top.chih.but.not.really.top} and Proposition~\ref{prop:Codazzi}, 
and then using Proposition~\ref{prop:mus.int}, 
we bound the first term in \eqref{eq:to.be.controlled.for.modified} 
as follows:
\begin{equation}\label{eq:chih.for.top.trch}
\begin{split}
&\: \int_{t'=0}^{t'=t} \|\slashed{\mathcal{P}}^{\Ntop} (\upmu |\chih|^2) \|_{L^2(\Sigma_{t'}^u)}\, dt' \\
\ls &\: \epd^{\f 12}\int_{t'=0}^{t'=t} \| \upmu \slashed{\mathcal{P}}^{\Ntop} \mytr\upchi \|_{L^2(\Sigma_{t'}^u)} \, dt' + \epd \int_{t'=0}^{t'=t} \upmu_{\star}^{-\toprate + 0.9}(t') \, dt' \\
\ls &\: \epd^{\f 12}\int_{t'=0}^{t'=t} \| \upmu \slashed{\mathcal{P}}^{\Ntop} \mytr\upchi \|_{L^2(\Sigma_{t'}^u)} \, dt' + \epd \upmu_{\star}^{-\toprate + 1.9}(t).
\end{split}
\end{equation}

\pfstep{Step~1(b): The $\slashed{\mathcal{P}}^{\Ntop} ( \vec{G}_{\Lunit \Lunit}\contr \vec{\mathfrak G})$ term}  
To handle the second term in \eqref{eq:to.be.controlled.for.modified}, 
we simply use H\"older's inequality together with 
\eqref{BA:W1}--\eqref{BA:W.Li.small}, Propositions~\ref{prop:geometric.low}, \ref{prop:geometric.top}, the assumption \eqref{eq:F.smallness},
and Proposition~\ref{prop:mus.int} to obtain the following bound:
\begin{equation}\label{eq:G.for.top.trch}
\begin{split}
&\: \int_{t'=0}^{t'=t} \| \slashed{\mathcal{P}}^{\Ntop} (\vec{G}_{\Lunit \Lunit}\contr \vec{\mathfrak G}) \|_{L^2(\Sigma_{t'}^u)} \, dt' \\
\ls &\: \int_{t'=0}^{t'=t} \left\lbrace \|\slashed{\mathcal{P}}^{[2, \Ntop]} (\Psi, L^i) \|_{L^2(\Sigma_{t'}^u)}  \|\slashed{\mathcal{P}}^{ \leq \lceil \f {\Ntop}{2} \rceil } \mathfrak G \|_{L^\infty(\Sigma_{t'})} + \| \slashed{\mathcal{P}}^{[1, \Ntop]}  \mathfrak G \|_{L^2(\Sigma_{t'}^u)} \right\rbrace \, dt'  \\
\ls &\: \int_{t'=0}^{t'=t} \left\lbrace \epd \upmu_{\star}^{-\toprate+1.4}(t') + \| \slashed{\mathcal{P}}^{[1,\Ntop]}  \mathfrak G \|_{L^2(\Sigma_{t'}^u)} \right\rbrace \, dt' \\
\ls &\: \epd \upmu_{\star}^{-\toprate+2.4}(t) + \int_{t'=0}^{t'=t}  \| \mathcal{P}^{[1,\Ntop]}  \mathfrak G \|_{L^2(\Sigma_{t'}^u)} \, dt'.
\end{split}
\end{equation}

\pfstep{Step~2: Bounding the fully modified quantity} 
The fully modified quantity $\upchifullmodarg{\slashed{\mathcal{P}}^{\Ntop}}$ satisfies the transport equation 
\eqref{eq:fully.modified.transport} in the $L$ direction. 
We use the arguments given in \cite[Proposition~13.11]{LS} to
integrate the transport equation to obtain a pointwise estimate for $\upchifullmodarg{\slashed{\mathcal{P}}^{\Ntop}}$.
On the right-hand side of the pointwise estimate
there appears, in particular, 
the time integral of the
new terms
$
\slashed{\mathcal{P}}^{\Ntop} (\upmu |\chih|^2)
$
and 
$
\frac 12 \slashed{\mathcal{P}}^{\Ntop} (\vec{G}_{\Lunit \Lunit}\contr \vec{\mathfrak G})
$
on RHS~\eqref{eq:fully.modified.transport}.
We then take the $L^2(\Sigma_t^u)$ norm of the resulting pointwise inequality, 
as in the proof of \cite[Lemma~14.14]{LS}.
This yields an $L^2(\Sigma_t^u)$ estimate for $\upchifullmodarg{\slashed{\mathcal{P}}^{\Ntop}}$.
We next use \eqref{E:TRANSPORTRENORMALIZEDTRCHIJUNK} to
algebraically express $\upmu \slashed{\mathcal{P}}^{\Ntop} \mytr \upchi$
in terms of $\upchifullmodarg{\slashed{\mathcal{P}}^{\Ntop}}$ plus a remainder term, and then use
the triangle inequality to obtain an $L^2(\Sigma_t^u)$ estimate for $\upmu \slashed{\mathcal{P}}^{\Ntop} \mytr \upchi$.
One of the remainder terms is $\slashed{\mathcal{P}}^{\Ntop} \upchifullmodinhom$, 
and it can be estimated exactly as in \cite[Lemma~14.14]{LS}. 
In total, we find that:
\begin{equation*}
\begin{split}
&\: \| \upmu \slashed{\mathcal{P}}^{\Ntop} \mytr \upchi\|_{L^2(\Sigma_t^u)} \\
\ls &\: \|\upchifullmodarg{\slashed{\mathcal{P}}^{\Ntop}} \|_{L^2(\Sigma_t^u)} 
+ 
\|\slashed{\mathcal{P}}^{\Ntop} \upchifullmodinhom\|_{L^2(\Sigma_t^u)} \\
\ls &\: \mbox{Terms on RHS~\cite[(14.28)]{LS}}
	+ 
	\int_{t'=0}^{t'=t} 
		\| \slashed{\mathcal{P}}^{\Ntop} (\upmu |\chih|^2)  \|_{L^2(\Sigma_{t'}^u)} 
	\, dt' 
	+ 
	\int_{t'=0}^{t'=t}  \| \slashed{\mathcal{P}}^{\Ntop} ( \vec{G}_{\Lunit \Lunit}\contr \vec{\mathfrak G}) \|_{L^2(\Sigma_{t'}^u)} \, dt' 
\\
\ls &\: 
	\mbox{Terms on RHS~\cite[(14.28)]{LS}}
	+ 
	\epd \upmu_{\star}^{-\toprate+ 1.9}(t) 
	+ 
	\epd^{\f 12} \int_{t'=0}^{t'=t} \| \upmu \slashed{\mathcal{P}}^{\Ntop} \mytr \upchi \|_{L^2(\Sigma_{t'}^u)} \, dt' 
	\\
&\: 
+ 
\int_{t'=0}^{t'=t}  \| \slashed{\mathcal{P}}^{[1,\Ntop]} \mathfrak G \|_{L^2(\Sigma_{t'}^u)} \, dt' 
\\
\ls 
&\: \mbox{Terms on RHS~\cite[(14.28)]{LS}}
+ 
\epd \upmu_{\star}^{-\toprate+ 1.9}(t)
+ 
\int_{t'=0}^{t'=t}  \| \mathcal{P}^{[1,\Ntop]}  \mathfrak G \|_{L^2(\Sigma_{t'}^u)} \, dt',
\end{split}
\end{equation*}
where to obtain the next-to-last line, 
we used the estimates \eqref{eq:chih.for.top.trch} and \eqref{eq:G.for.top.trch},
and to obtain the last line, we used
Gr\"onwall's inequality to eliminate the factor
$
\epd^{\f 12}\int_{t'=0}^{t'=t} \| \upmu \slashed{\mathcal{P}}^{\Ntop} \mytr \upchi \|_{L^2(\Sigma_{t'}^u)} \, dt'
$
on the RHS. We have therefore proved \eqref{eq:trch.top.non.sharp}.

\pfstep{Step~3: Proof of \eqref{eq:trch.top.sharp}} 
\eqref{eq:trch.top.sharp} can be proved using arguments that are very similar
to the ones we used in the proof of \eqref{eq:trch.top.non.sharp}, 
except that we need to keep track of the constants in the borderline terms,
i.e., the absolute constant $M_{\mathrm{abs}}$ 
(whose precise value we do not bother to estimate here)
and the parameter-dependent constant
$\Ccrit$. This can be done exactly as in the proof of \cite[(14.27)]{LS}.
The only terms which are not already present in \cite[(14.27)]{LS} 
are exactly those we encountered already in Steps~1--2. These new terms 
can be treated exactly as in the proof of \eqref{eq:trch.top.non.sharp}, 
since we do not have to keep track of the sharp constants for these new terms
(we instead allow a general constant $C$). \qedhere
\end{proof}

\begin{proposition}[$L^2$ estimates for the partially modified quantities]
\label{prop:partially.modified}
There exists an absolute positive constant $M_{\mathrm{abs}} \in \mathbb{N}$,
a positive constant
$\Ccrit \in \mathbb N$,
and a constant $C > 0$
(each having the properties described in Section~\ref{SS:ASSUMPTIONSANDCONSTANTS})  
such that the partially modified quantity $\upchipartialmodarg{\slashed{\mathcal{P}}^{\Ntop-1}}$ obeys the
following estimates 
(whose right-hand side involves the wave energies \eqref{eq:wave.energy.def.1}--\eqref{eq:wave.energy.def.4}
as well as the partial energies of Definition~\ref{D:JAREDPARTIALENERGIES})
for every $(t,u) \in [0,\Tboot) \times [0,U_0]$:
\begin{equation}\label{eq:partiall.modified.1}
\begin{split}
 &\: \| \f{1}{\sqrt{\upmu}} (\bX \mathcal{R}_{(+)}) L\upchipartialmodarg{\slashed{\mathcal{P}}^{\Ntop-1}} \|_{L^2(\Sigma_t^u)} \\
 \leq &\: 
\mbox{Terms on RHS~\cite[(14.32a)]{LS} with the boxed constants replaced by 
$\boxed{M_{\mathrm{abs}}}$} 	
	\\
 &\: \ \ \mbox{and the constant $C_*$ replaced by $\boxed{\Ccrit}$} \\ 
 &\: \ \
	+ C \epd \upmu_{\star}^{-\toprate + 0.9}(t),
\end{split}
\end{equation}
\begin{equation}\label{eq:partiall.modified.2}
\begin{split}
 &\: \| \f{1}{\sqrt{\upmu}} (\bX \mathcal{R}_{(+)}) \upchipartialmodarg{\slashed{\mathcal{P}}^{\Ntop-1}} \|_{L^2(\Sigma_t^u)} \\
 \leq &\: 
\mbox{Terms on RHS~\cite[(14.32b)]{LS} with the boxed constants replaced by 
$\boxed{M_{\mathrm{abs}}}$} 	
	\\
 &\: \ \ \mbox{and the constant $C_*$ replaced by $\boxed{\Ccrit}$} \\ 
 &\: \ \ + C \epd \upmu_{\star}^{-\toprate + 1.9}(t),
\end{split}
\end{equation}
\begin{equation}\label{eq:partiall.modified.3}
 \| L \upchipartialmodarg{\slashed{\mathcal{P}}^{\Ntop-1}} \|_{L^2(\Sigma_t^u)} 
\ls  
\mbox{Terms in \cite[(14.33a)]{LS}} 
+ 
\epd \upmu_{\star}^{-\toprate+1.4}(t),
 \end{equation}
 and
\begin{equation}\label{eq:partiall.modified.4}
 \| \upchipartialmodarg{\slashed{\mathcal{P}}^{\Ntop-1}} \|_{L^2(\Sigma_t^u)} \ls 
	\mbox{Terms in \cite[(14.33b)]{LS}} 
	+ 
	\epd \upmu_{\star}^{-\toprate+2.4}(t).
 \end{equation}
\end{proposition}
\begin{proof}
To control $L\upchipartialmodarg{\slashed{\mathcal{P}}^{\Ntop-1}}$, we bound the terms on the RHS of
the transport equation \eqref{eq:partially.modified.transport}. 
Note that for this estimate, the only term not already found in \cite{LS} is the term 
$-\slashed{\mathcal{P}}^{\Ntop - 1} (|\chih|^2)$. 
Compared to the estimates for the fully modified quantity that we derived in Proposition~\ref{prop:trch.top}, 
the estimates for the partially modified quantity is simpler in two ways:
i)~the transport equation \eqref{eq:partially.modified.transport} does not feature wave equation
inhomogeneous term $\mathfrak G$, 
and ii)~the additional term only has up to $\Ntop-1$ derivatives of $\chih$,
and thus elliptic estimates are not necessary to control this term.

We now estimate $-\slashed{\mathcal{P}}^{\Ntop - 1}(|\chih|^2)$. 
By \eqref{E:CHIINTERMSOFOTHERVARIABLES}, 
\eqref{BA:W1}--\eqref{BA:W.Li.small},
and
Propositions~\ref{prop:geometric.low} and \ref{prop:geometric.top}, we have:
\begin{equation}\label{eq:chih.one.lower.than.top}
\| \slashed{\mathcal{P}}^{\Ntop - 1}(|\chih|^2)\|_{L^2(\Sigma_t^u)} 
\ls \epd^{\f 12} \|\mathcal{P}^{[1,\Ntop]} (\Psi, L^i)\|_{L^2(\Sigma_t^u)} 
\ls \epd \upmu_{\star}^{-\toprate+1.4}(t).
\end{equation}
We now recall \eqref{eq:partially.modified.transport}. 
The terms that are already in \mbox{Terms in \cite[(6.10)]{LS}}  
can be treated using the same arguments that were used to prove \cite[(14.32a)]{LS}
and \cite[(14.33a)]{LS},
except here we do not bother to
estimate the absolute constant $M_{\mathrm{abs}}$ that arises in the arguments,
and we have renamed the constant $C_*$ as $\boxed{\Ccrit}$.
From this fact, the estimate \eqref{eq:chih.one.lower.than.top},
and the bootstrap assumption \eqref{BA:LARGERIEMANNINVARIANTLARGE} 
for $\Rad \mathcal{R}_{(+)}$, we deduce
\eqref{eq:partiall.modified.1} and \eqref{eq:partiall.modified.3}. 

To obtain \eqref{eq:partiall.modified.4},
we use the transport equation estimate
provided by Lemma~\ref{L:L2ESTIMATEFORTRANSPORT},
the estimate \eqref{eq:partiall.modified.3} for the source term,
Proposition~\ref{prop:mus.int},
and the initial data bound 
$\| \upchipartialmodarg{{\slashed{\mathcal{P}}}^{\Ntop-1}}(0,\cdot) \|_{L^2(\Sigma_t^u)}
\lesssim \epd$
obtained in the proof of \cite[(14.33b)]{LS}.

Similarly, \eqref{eq:partiall.modified.2}
can be proved using the same arguments used in the proof of \cite[(14.32b)]{LS}.
The estimate is based on integrating the transport equation \eqref{eq:partially.modified.transport}
along the integral curves of $\Lunit$
and using Lemma~\ref{L:L2ESTIMATEFORTRANSPORT}.
The only new term we have to handle 
comes from the $- \mathcal{P}^{\Ntop - 1} (|\chih|^2)$ term on RHS~\eqref{eq:partially.modified.transport},
and by Lemma~\ref{L:L2ESTIMATEFORTRANSPORT}, this term leads to the following
additional term that has to be controlled:
$ 
\f{1}{\sqrt{\upmu_{\star}(t)}} 
\| \bX \mathcal{R}_{(+)} \|_{L^{\infty}(\Sigma_t^u)}
\int_{t'=0}^{t'=t}
\| 
	\mathcal{P}^{\Ntop - 1} (|\chih|^2)
\|_{L^2(\Sigma_{t'}^u)}
\, dt'
$.
In view of the bootstrap assumption \eqref{BA:LARGERIEMANNINVARIANTLARGE},
the estimate \eqref{eq:chih.one.lower.than.top},
and Proposition~\ref{prop:mus.int},
we bound this additional term
by $\lesssim \upmu_{\star}^{-\toprate+1.9}(t)$, 
which is $\leq \mbox{RHS~\eqref{eq:partiall.modified.2}}$ as desired.

\qedhere
\end{proof}

\subsection{The main integral inequalities for the energies}
\label{SS:MAININTEGRALINEQUALITIES}
Our main goal in this section is to prove Proposition~\ref{P:JAREDimproved.critical.constant},
which provides integral inequalities for the various wave energies at various derivative levels.
Most of the analysis is the same as in \cite{LS}. In the next definition, we highlight
the error terms in the energy estimates that are new in the present paper compared to \cite{LS}.
The new terms stem from the inhomogeneous term $\mathfrak{G}$ in the wave equations
as well as the $- \upmu |\chih|^2$ term on the RHS of the $3$D Raychaudhuri equation \eqref{eq:3D.Raychaudhuri}.

\begin{definition}[New energy estimate error terms]
\label{D:NEWENERGYESTIMATE}
We use the notation $\mbox{\upshape NewError}_{\Ntop}^{(Top)}(t,u)$ to denote any term that obeys the following bound
for every $(t,u) \in [0,\Tboot) \times [0,U_0]$:
\begin{align}\label{E:NEWTOPORDER}
\begin{split}
\mbox{\upshape NewError}_{\Ntop}^{(Top)}(t,u) 
& \leq C \epd^2 \upmu_{\star}^{-2\toprate+1.8}(t) 
+ 
C\int_{t'=0}^{t'=t} \upmu_{\star}^{-\f 32}(t') 
\left\{\int_{s=0}^{s=t'} \|\mathcal{P}^{[1,\Ntop]} \mathfrak G\|_{L^2(\Sigma_s^u)} \, ds \right\}^2 \, dt' 
	\\
&\: 
+ 
C \| (|L\mathcal{P}^{[1,\Ntop]} \Psi| + |\bX \mathcal{P}^{[1,\Ntop]} \Psi|) |\mathcal{P}^{[1,\Ntop]} \mathfrak G| \|_{L^1(\Mtu)},
\end{split}
\end{align}
where $C > 0$ is a constant of the type described in Section~\ref{SS:ASSUMPTIONSANDCONSTANTS}.

Similarly, 
we use the notation $\mbox{\upshape NewError}_{N-1}^{(Below-Top)}(t,u)$
to denote any term that obeys the following bound
for every $(t,u) \in [0,\Tboot) \times [0,U_0]$:
\begin{align}\label{E:NEWBELOWTOPORDER}
\begin{split}
\mbox{\upshape NewError}_{N-1}^{(Below-Top)}(t,u)
& \leq
C \| (|L\mathcal{P}^{[1,N-1]} \Psi| + |\bX \mathcal{P}^{[1,N-1]} \Psi|) |\mathcal{P}^{[1,N-1]} \mathfrak G| \|_{L^1(\Mtu)}.
\end{split}
\end{align}

\end{definition}

\begin{proposition}[The main integral inequalities for the energies]
\label{P:JAREDimproved.critical.constant}
Let $\mathbb{Q}_{[1,N]}(t,u), \mathbb K_{[1,N]}(t,u)$ be the wave energies from Section~\ref{SSS:DEFSOFENERGIES},
and let $\mathbb{Q}^{(Partial)}_{[1,N]}(t,u), \mathbb K^{(Partial)}_{[1,N]}(t,u)$
be the partial wave energies from Section~\ref{SS:SPLITENERGIES}.
There exist \textbf{an absolute constant} $M_{\mathrm{abs}}\in \mathbb N$ and a constant $\Ccrit \in \mathbb N$ 
depending on the equation of state, $\bar{\varrho}$, $\mathring{\upsigma}$, $\mathring{\updelta}$ and 
$\mathring{\updelta}_*^{-1}$, such that the following estimates, which are
modified versions of \cite[(14.3)]{LS},
hold for every $(t,u) \in [0,\Tboot) \times [0,U_0]$:
\begin{align} \label{E:JAREDTOPORDERINTEGRALINEQUALITY}
\begin{split}
&\: \max\{ \mathbb{Q}_{[1,\Ntop]}(t,u), \mathbb K_{[1,\Ntop]}(t,u)\}\\
\leq &\: \boxed{M_{\mathrm{abs}}} \int_{t'=0}^t \f{\|[L\upmu]_-\|_{L^\infty(\Sigma_{t'}^u)}}{\upmu_{\star}(t',u)} \mathbb{Q}_{[1,\Ntop]}(t',u)\, dt' \\
&\: + \boxed{M_{\mathrm{abs}}} \int_{t'=0}^{t'=t}  \f{\|[L\upmu]_-\|_{L^\infty(\Sigma^u_{t'})}}{\upmu_{\star}(t',u)} \sqrt{\mathbb{Q}_{[1,\Ntop]}}(t',u) \int_{s=0}^{s=t'} \f{\|[L\upmu]_-\|_{L^\infty(\Sigma^u_s)}}{\upmu_{\star}(s,u)} \sqrt{\mathbb{Q}_{[1,\Ntop]}}(s,u) \, ds\, dt' \\
&\: + \boxed{M_{\mathrm{abs}}} \f{\|L\upmu\|_{L^\infty(^{(-)}\Sigma^u_{t;t})}}{\upmu_{\star}^{\f 12}(t,u)} \sqrt{\mathbb{Q}_{[1,\Ntop]}}(t,u) \int_{t'=0}^{t'=t} \f{1}{\upmu_{\star}^{\f 12}(t',u)}\sqrt{\mathbb{Q}_{[1,\Ntop]}}(t',u)\,dt' \\
&\: + \boxed{\Ccrit} \int_{t'=0}^t \f{\|[L\upmu]_-\|_{L^\infty(\Sigma_{t'}^u)}}{\upmu_{\star}(t',u)} \sqrt{\mathbb{Q}_{[1,\Ntop]}}(t',u) \sqrt{\mathbb{Q}^{(Partial)}_{[1,\Ntop]}}(t',u)\, dt' \\
&\: + \boxed{\Ccrit} \int_{t'=0}^{t'=t}  \f{\|[L\upmu]_-\|_{L^\infty(\Sigma^u_{t'})}}{\upmu_{\star}(t',u)} \sqrt{\mathbb{Q}_{[1,\Ntop]}}(t',u) \int_{s=0}^{s=t'} \f{\|[L\upmu]_-\|_{L^\infty(\Sigma^u_s)}}{\upmu_{\star}(s,u)} \sqrt{\mathbb{Q}_{[1,\Ntop]}^{(Partial)}}(s,u) \, ds\, dt' \\
&\: + \boxed{\Ccrit} \f{\|L\upmu\|_{L^\infty(^{(-)}\Sigma^u_{t;t})}}{\upmu_{\star}(t,u)^{\f 12}} \sqrt{\mathbb{Q}_{[1,\Ntop]}}(t,u) \int_{t'=0}^{t'=t} \f{1}{\upmu_{\star}^{\f 12}(t',u)}\sqrt{\mathbb{Q}^{(Partial)}_{[1,\Ntop]}}(t',u)\,dt' 
	\\
&\: + \mbox{the error terms $\mbox{\upshape Error}_{\Ntop}^{(Top)}(t,u)$ defined by \cite[(14.4)]{LS}} 
	\\
&\: + \mbox{the error terms $\mbox{\upshape NewError}_{\Ntop}^{(Top)}(t,u)$ defined by \eqref{E:NEWTOPORDER}}.
\end{split}
\end{align}
The set $^{(-)}\Sigma^u_{t;t}$ appearing on RHS~\eqref{E:JAREDTOPORDERINTEGRALINEQUALITY} 
is defined in\footnote{We have no need to state its precise definition here; 
later, we will simply quote the relevant estimates from \cite{LS} that are tied to this set.} 
\cite[Definition~10.4]{LS}.

Moreover, the partial wave energies 
obey the following estimates,
which are modified versions of \cite[(14.5)]{LS}:
\begin{align} \label{E:JAREDPARTIALTOPORDERINTEGRALINEQUALITY}
\begin{split}
&\: \max\{ \mathbb{Q}_{[1,\Ntop]}^{(Partial)}(t,u), \mathbb K_{[1,\Ntop]}^{(Partial)}(t,u)\}\\
\leq &\: \mbox{the error terms $\mbox{\upshape Error}_{\Ntop}^{(Top)}(t,u)$ defined by \cite[(14.4)]{LS}} 
	\\
&\: + \mbox{the error terms $\mbox{\upshape NewError}_{\Ntop}^{(Top)}(t,u)$ defined by \eqref{E:NEWTOPORDER}}.
\end{split}
\end{align}

Finally, we have the following below-top-order estimates, 
which are modified versions\footnote{Note that the lower-order estimate \cite[(14.6)]{LS} 
is easier and has fewer additional terms. 
This is because to obtain the top-order estimates 
\cite[(14.3), (14.5)]{LS}, one needs to bound all of the commutator 
terms, including the difficult ones identified in Proposition~\ref{prop:identity.main.wave.commutator.terms},
without losing derivatives.
In contrast, to obtain the lower-order estimates \cite[(14.6)]{LS}, 
one is allowed to lose a derivative, as is manifested by
the double-time-integral term on RHS~\eqref{E:JAREDBELOWTOPORDERINTEGRALINEQUALITY}.
This double-time-integral will eventually be responsible for the coupling between the energies of different orders;
see in particular the estimates \eqref{eq:main.EE.prop.low} in the statement of Proposition~\ref{prop:wave}.} 
of \cite[(14.6)]{LS}:
\begin{align} \label{E:JAREDBELOWTOPORDERINTEGRALINEQUALITY}
\begin{split}
&\: \max\{ \mathbb{Q}_{[1,N-1]}(t,u), \mathbb K_{[1,N-1]}(t,u)\} 
\\
\leq &\: 
	C
	\int_{t'=0}^t
				\frac{1}{\upmu_{\star}^{\frac{1}{2}}(t',u)} 
						\sqrt{\mathbb{Q}_{[1,N-1]}}(t',u) 
						\int_{s=0}^{t'}
							\frac{1}{\upmu_{\star}^{\frac{1}{2}}(s,u)} 
							\sqrt{\mathbb{Q}_{[1,N]}}(s,u)  
						\, ds
				\, dt'
	\\
&\:+ \mbox{the error terms $\mbox{\upshape Error}_{N-1}^{(Below-Top)}(t,u)$ defined by \cite[(14.7)]{LS}} 
	\\
&\:+ \mbox{the error terms $\mbox{\upshape NewError}_{N-1}^{(Below-Top)}(t,u)$ defined by \eqref{E:NEWBELOWTOPORDER}}.
\end{split}
\end{align}

\end{proposition}

\begin{proof}
\pfstep{Step~1: Proof of \eqref{E:JAREDBELOWTOPORDERINTEGRALINEQUALITY}} 
We begin with \eqref{E:JAREDBELOWTOPORDERINTEGRALINEQUALITY}, 
which is the easier estimate since it is below top-order. 
Here, we use that \cite[(14.6)]{LS} is proved by differentiating the wave equation 
$\upmu  \square_{g(\vec{\Psi})}\Psi = \cdots$ 
with $\mathcal{P}^{N'}$, computing the commutator $[\upmu \square_{g(\vec{\Psi})}, \mathcal{P}^{N'}]$,
multiplying the commuted equation by $(1+2\upmu) L \mathcal{P}^{N'} \Psi + \bX \mathcal{P}^{N'} \Psi$, 
and then integrating (with respect to the volume form $d \vol$ Definition~\ref{D:NONDEGENERATEVOLUMEFORMS}) 
by parts over the spacetime region $\Mtu$ (for $1\leq N' \leq N-1$).
Hence, to prove \eqref{E:JAREDBELOWTOPORDERINTEGRALINEQUALITY},
we repeat the argument in \cite{LS}, except that here we simply denote all of the inhomogeneous terms
in the wave equations as ``$\mathfrak G$.'' 
That is, we start with the wave equations 
$\upmu \square_{g(\vec{\Psi})} \Psi_{\iota} = \mathfrak G_{\iota}$
and commute them to obtain the wave equations 
$\upmu \square_{g(\vec{\Psi})} \mathcal{P}^{N'} \Psi_{\iota} = [\upmu \square_{g(\vec{\Psi})}, \mathcal{P}^{N'}] \Psi_{\iota} 
+ 
\mathcal{P}^{N'} \mathfrak G_{\iota}$.
The main point is that for the below-top-order estimates,
all commutator terms 
$[\upmu \square_{g(\vec{\Psi})}, \mathcal{P}^{N'}] \Psi_{\iota}$
can be handled exactly as in \cite{LS}. 
These commutator terms lead to the presence of the first term
on RHS~\eqref{E:JAREDBELOWTOPORDERINTEGRALINEQUALITY} as well as the error term
$\mbox{\upshape Error}_{N-1}^{(Below-Top)}(t,u)$ on RHS~\eqref{E:JAREDBELOWTOPORDERINTEGRALINEQUALITY}.
We clarify that in the proof of \cite[(14.6)]{LS}, the vorticity-involving inhomogeneous terms in the wave
equation led to error integrals on RHS~\cite[(14.6)]{LS} that involved the vorticity energies;
in contrast, on RHS~\eqref{E:JAREDBELOWTOPORDERINTEGRALINEQUALITY}, 
the vorticity-involving terms are not explicitly indicated 
because we have soaked them up into our definition of $\mathfrak G_{\iota}$.
Thus, to complete the proof of \eqref{E:JAREDBELOWTOPORDERINTEGRALINEQUALITY},
we only have to discuss the contribution of the inhomogeneous term $\mathfrak G_{\iota}$.
From the above discussion, it follows that we only have to show that
the following energy identity error integrals are bounded in magnitude
by $\leq \mbox{RHS~\eqref{E:JAREDBELOWTOPORDERINTEGRALINEQUALITY}}$
when $1\leq N' \leq N-1$ and $(t,u)\in [0,\Tboot)\times [0,U_0]$:
$$
\int_{\Mtu}  
	\left\lbrace
		(1+2\upmu) L\mathcal{P}^{N'}\Psi + \bX \mathcal{P}^{N'}\Psi
	\right\rbrace
	\mathcal{P}^{N'} \mathfrak G
\, d \vol.
$$
The desired estimate is simple;
in view of the $L^{\infty}$ estimates for $\upmu$ provided by Proposition~\ref{prop:geometric.low},
we see that these error integrals are all bounded by
$
C \| (|L\mathcal{P}^{[1,N-1]} \Psi| + |\bX \mathcal{P}^{[1,N-1]} \Psi|) 
|\mathcal{P}^{[1,N-1]} \mathfrak G| \|_{L^1(\Mtu)}
$,
which are exactly the error terms we have defined in \eqref{E:NEWBELOWTOPORDER}.

\medskip

\pfstep{Step~2: Proof of \eqref{E:JAREDTOPORDERINTEGRALINEQUALITY}}

\pfstep{Step~2(a): Preliminaries} 
As in our proof of \eqref{E:JAREDBELOWTOPORDERINTEGRALINEQUALITY},
to prove \eqref{E:JAREDTOPORDERINTEGRALINEQUALITY}, 
the only new step compared to \cite{LS} is tracking the contribution of the
wave equation inhomogeneous terms $\mathfrak G_{\iota}$ to the energy estimates.
As in Step 1, one way in which this inhomogeneous term contributes to the energy estimates is
through the error terms
$C \| (|L\mathcal{P}^{[1,\Ntop]} \Psi| 
+ 
|\bX \mathcal{P}^{[1,\Ntop]} \Psi|) |\mathcal{P}^{[1,\Ntop]} \mathfrak G| \|_{L^1(\Mtu)}$,
which are found on RHS~\eqref{E:NEWTOPORDER}. 
However, in the top-order case, there is a second way in which 
$\mathfrak G_{\iota}$ contributes to the top-order energy estimates.
To explain this contribution, we first note that, as in the proof of \cite[(14.3)]{LS},
we have to handle some additional difficult top-order commutator terms 
involving the top-order derivatives of $\mytr\upchi$.
Specifically, these difficult top-order commutator terms are
explicitly listed on RHSs~\eqref{eq:hard.commutator.L}--\eqref{eq:hard.commutator.Z}.
Recalling that we multiply the wave equation by $(1+2\upmu) L\mathcal{P}^{N'}\Psi + \bX \mathcal{P}^{N'}\Psi$
to derive the wave equation energy estimates at level $N'$, we see that 
up to harmless factors that are $\mathcal{O}(1)$ by virtue of the estimates of Proposition~\ref{prop:geometric.low.2},
these difficult commutator terms lead to the following three error integrals in the top-order energy estimates:
$$
\int_{\Mtu}  
	(\bX \mathcal{P}^{\Ntop} \Psi)
	(\bX \Psi) 
	\slashed{\mathcal{P}}^{\Ntop} \mytr\upchi
\, d \vol,
$$
$$
\int_{\Mtu}  
	\left\lbrace
		(1+2\upmu) L\mathcal{P}^{\Ntop} \Psi
	\right\rbrace
	(\bX \Psi) \slashed{\mathcal{P}}^{\Ntop} \mytr\upchi
\, d \vol,
$$
$$
\int_{\Mtu}  
	\left\lbrace
		(1+2\upmu) L\mathcal{P}^{\Ntop} \Psi + \bX \mathcal{P}^{\Ntop} \Psi
	\right\rbrace
	(\slashed{\mathcal{P}} \Psi) \upmu \slashed{\mathcal{P}}^{\Ntop} \mytr\upchi
\, d \vol.
$$

We will control these three terms, respectively, in Steps~2(b)--(d) below.

\pfstep{Step~2(b): Contributions from 
$\int_{\Mtu}  
	(\bX \mathcal{P}^{\Ntop} \Psi)
	(\bX \Psi) 
	\slashed{\mathcal{P}}^{\Ntop} \mytr\upchi
\, d \vol$} 
We first consider the case $\Psi = \mathcal{R}_{(+)}$,
which is by far the most difficult case.
Using H\"older's inequality and the estimate 
\eqref{eq:trch.top.sharp} in Proposition~\ref{prop:trch.top}, 
we deduce that:
\begin{equation}\label{eq:hardest.commutator.term}
\begin{split}
&\: 
\left|
\int_{\Mtu}  
	(\bX \mathcal{P}^{\Ntop} \mathcal{R}_{(+)})
	(\bX \mathcal{R}_{(+)}) 
	\slashed{\mathcal{P}}^{\Ntop} \mytr\upchi
\, d \vol
\right|
	\\
&\leq
\: 
\int_{t'=0}^{t'=t} 
	\|\bX\mathcal{P}^{\Ntop} \mathcal{R}_{(+)} \|_{L^2(\Sigma_{t'}^u)}
	\|(\bX \mathcal{R}_{(+)}) \slashed{\mathcal{P}}^{\Ntop} \mytr\upchi \|_{L^2(\Sigma_{t'}^u)}
	\, dt'
	\\
\leq &\: \mbox{Terms on RHS~\cite[(14.3)]{LS} with the boxed constants replaced by $\boxed{M_{\mathrm{abs}}}$}
	\\
&\: \ \ \mbox{and the constant $C_*$ replaced by $\boxed{\Ccrit}$} 
	\\
&\: + C \underbrace{ \int_{t'=0}^{t'=t} \| \bX\mathcal{P}^{\Ntop} \mathcal{R}_{(+)}\|_{L^2(\Sigma_{t'}^u)} \epd \upmu_{\star}^{-\toprate + 0.9}(t') \, dt' }_{\doteq \mathrm{I}}\\
&\: + C \underbrace{ \int_{t'=0}^{t'=t} \upmu_{\star}^{-1}(t) \|\bX\mathcal{P}^{\Ntop}
\mathcal{R}_{(+)}\|_{L^2(\Sigma_{t'}^u)} \left\{\int_{s=0}^{s=t'} \|\mathcal{P}^{[1,\Ntop]} \mathfrak G\|_{L^2(\Sigma_s^u)} \, ds \right\}\, dt' }_{\doteq \mathrm{II}}.
\end{split}
\end{equation}
We clarify that Remark~\ref{R:VORTICITYTERMSHAVEBEENSOAKEDINTOGENERALINHOMTERM}
also applies to the terms on RHS~\cite[(14.3)]{LS}
(some of which also appear on RHS~\eqref{eq:hardest.commutator.term}).

To handle the term $\mathrm{I}$ in \eqref{eq:hardest.commutator.term}, 
we use Cauchy--Schwarz inequality in $t'$ and Proposition~\ref{prop:mus.int} 
to deduce:
\begin{equation}\label{eq:hardest.commutator.term.1}
\begin{split}
&\: \int_{t'=0}^{t'=t} \| \bX\mathcal{P}^{\Ntop} \mathcal{R}_{(+)}\|_{L^2(\Sigma_{t'}^u)} \epd \upmu_{\star}^{-\toprate 
+ 0.9}(t') \, dt' \\
\ls &\: \int_{t'=0}^{t'=t} \upmu_{\star}^{-\f 12}(t) \|\bX\mathcal{P}^{\Ntop}\mathcal{R}_{(+)}\|_{L^2(\Sigma_{t'}^u)}^2 \, dt' + \epd^2 \int_{t'=0}^{t'=t} \upmu_{\star}^{-2\toprate + 2.3}(t')\, dt' \\
\ls &\: \int_{t'=0}^{t'=t} \upmu_{\star}^{-\f 12}(t) \|\bX\mathcal{P}^{\Ntop}\mathcal{R}_{(+)}\|_{L^2(\Sigma_{t'}^u)}^2 \, dt' + \epd^2 \upmu_{\star}^{-\toprate + 3.3}(t).
\end{split}
\end{equation}

For the term $\mathrm{II}$ in \eqref{eq:hardest.commutator.term}, 
we apply first the Cauchy--Schwarz inequality in $t'$ and then Young's inequality to obtain:
\begin{equation}\label{eq:hardest.commutator.term.2}
\begin{split}
&\: \int_{t'=0}^{t'=t} \upmu_{\star}^{-1}(t) \|\bX\mathcal{P}^{\Ntop} \mathcal{R}_{(+)}\|_{L^2(\Sigma_{t'}^u)} 
\left\{\int_{s=0}^{s=t'} \|\mathcal{P}^{[1,\Ntop]} \mathfrak G\|_{L^2(\Sigma_s^u)} \, ds \right\}\, dt' \\
\ls &\: \int_{t'=0}^{t'=t} \upmu_{\star}^{-\f 12}(t) \|\bX\mathcal{P}^{\Ntop} \mathcal{R}_{(+)} \|_{L^2(\Sigma_{t'}^u)}^2 \, dt' + \int_{t'=0}^{t'=t} \upmu_{\star}^{-\f 32}(t') \left\{\int_{s=0}^{s=t'} \|\mathcal{P}^{[1,\Ntop]} \mathfrak G\|_{L^2(\Sigma_s^u)} \, ds \right\}^2 \, dt'.
\end{split}
\end{equation}
Notice that the term $\int_{t'=0}^{t'=t} \upmu_{\star}^{-\f 12}(t) \|\bX\mathcal{P}^{\Ntop}\mathcal{R}_{(+)}\|_{L^2(\Sigma_{t'}^u)}^2 \, dt' $ appearing on the RHSs of both \eqref{eq:hardest.commutator.term.1} and \eqref{eq:hardest.commutator.term.2} is bounded above by $\int_{t'=0}^{t'=t} \upmu_{\star}^{-\f 12}(t',u) \mathbb{Q}_{\Ntop}(t',u) \, dt'$, 
which is among the error terms $\mbox{\upshape Error}_{\Ntop}^{(Top)}(t,u)$ defined by \cite[(14.4)]{LS}.
Therefore,
combining \eqref{eq:hardest.commutator.term}--\eqref{eq:hardest.commutator.term.2} 
and taking into account definition~\ref{E:NEWTOPORDER},
we obtain that: 
\begin{align}\label{eq:hardest.commutator.term.final}
\|(\bX \mathcal{P}^{\Ntop} \mathcal{R}_{(+)}) (\bX \mathcal{R}_{(+)}) \slashed{\mathcal{P}}^{\Ntop} \mytr\upchi \|_{L^1(\Mtu)} 
& \leq \mbox{RHS~\eqref{E:JAREDTOPORDERINTEGRALINEQUALITY}}
\end{align}
as desired.

We also need to bound the integral
$\int_{\Mtu}  
	(\bX \mathcal{P}^{\Ntop} \Psi)
	(\bX \Psi) 
	\slashed{\mathcal{P}}^{\Ntop} \mytr\upchi
\, d \vol$
in the remaining cases $\Psi \in \lbrace R_{(-)}, v^2, v^3, s \rbrace$.
As we further explain below in Step~3, 
a similar argument allows us to bound these error integrals 
by exploiting one crucial simplifying feature:
these error integrals are bounded by RHS~\eqref{E:JAREDTOPORDERINTEGRALINEQUALITY},
but \emph{without the difficult boxed-constant-involving integrals}
on the right-hand side.
The difference is that we can take advantage of the 
smallness of the factor $\| \bX \Psi \|_{L^\infty(\Sigma_t)} \leq \epd^{\frac{1}{2}}$
(valid for $\Psi \in \lbrace \mathcal{R}_{(-)}, v^2, v^3, s \rbrace$ -- but not for $\mathcal{R}_{(+)}$!),
which is provided by the bootstrap assumption \eqref{BA:SMALLWAVEVARIABLESUPTOONETRANSVERSALDERIVATIVE};
this allows us to avoid the error terms with large boxed constants 
and thus allows us to relegate the contribution of these error integrals
to the error term $\mbox{\upshape Error}_{\Ntop}^{(Top)}(t,u)$ on RHS~\eqref{E:JAREDTOPORDERINTEGRALINEQUALITY};
we refer to \cite[pg.~154]{LS} for further details.

\pfstep{Step~2(c): Contributions from 
$\int_{\Mtu}  
	\left\lbrace
		(1+2\upmu) L\mathcal{P}^{\Ntop} \Psi
	\right\rbrace
	(\bX \Psi) \slashed{\mathcal{P}}^{\Ntop} \mytr\upchi
\, d \vol$} 
We first consider the case $\Psi = \mathcal{R}_{(+)}$,
which is by far the most difficult case.
Unlike the error integral we controlled in Step~2(b), 
as in \cite{LS}, this error integral can be controlled
by first using the definition \eqref{E:TRANSPORTPARTIALRENORMALIZEDTRCHIJUNK}
of the partially modified quantities to algebraically replace the factor
$\slashed{\mathcal{P}}^{\Ntop} \mytr\upchi$
with a $\slashed{\mathcal{P}}$ derivative of
$\upchipartialmodarg{\slashed{\mathcal{P}}^{\Ntop-1}}$
plus remainder terms (that one controls separately), 
and then using integration by parts
to swap the $L$ and $\slashed{\mathcal{P}}$ derivatives.
Notice that by Proposition~\ref{prop:partially.modified}, 
the partially modified quantity obeys the same bounds as in \cite[Lemma~14.19]{LS},
except the estimates of Proposition~\ref{prop:partially.modified}
feature $\epd$-multiplied terms such as
$C \epd \upmu_{\star}^{-\toprate + 1.4}(t)$ on the RHSs,
which can be handled using arguments of the type we used to control the
error term \eqref{eq:hardest.commutator.term.1}.
In particular, the right-hand sides of the estimates in Proposition~\ref{prop:partially.modified}
do not involve the wave equation inhomogeneity $\mathfrak{G}$.
Hence, the error integral
$\int_{\Mtu}  
	\left\lbrace
		(1+2\upmu) L\mathcal{P}^{\Ntop} \mathcal{R}_{(+)}
	\right\rbrace
	(\bX \mathcal{R}_{(+)}) \slashed{\mathcal{P}}^{\Ntop} \mytr\upchi
\, d \vol$ 
can be bounded using exactly the same arguments given in \cite[Lemma~14.17]{LS} and \cite[Lemma~14.12]{jSgHjLwW2016},
except with the boxed constants from \cite{LS}
replaced by 
$\boxed{M_{\mathrm{abs}}}$ 
and the constant $C_*$ from \cite{LS} replaced by $\boxed{\Ccrit}$.
As a consequence, the error integral under consideration
is bounded in magnitude by $\leq \mbox{RHS~\eqref{E:JAREDTOPORDERINTEGRALINEQUALITY}}$.

To bound the integral
$\int_{\Mtu}  
	\left\lbrace
		(1+2\upmu) L\mathcal{P}^{\Ntop} \Psi
	\right\rbrace
	(\bX \Psi) \slashed{\mathcal{P}}^{\Ntop} \mytr\upchi
\, d \vol$
in the remaining cases $\Psi \in \lbrace R_{(-)}, v^2, v^3, s \rbrace$,
we can again (as in Step~2(b)) take advantage of the smallness
$\| \bX \Psi \|_{L^\infty(\Sigma_t)} \leq \epd^{\frac{1}{2}}$
(valid for $\Psi \in \lbrace \mathcal{R}_{(-)}, v^2, v^3, s \rbrace$ -- but not for $\mathcal{R}_{(+)}$!),
which is provided by the bootstrap assumption \eqref{BA:SMALLWAVEVARIABLESUPTOONETRANSVERSALDERIVATIVE}.
This again allows us to relegate the contribution of these integrals
to the error term $\mbox{\upshape Error}_{\Ntop}^{(Top)}(t,u)$ on RHS~\eqref{E:JAREDTOPORDERINTEGRALINEQUALITY};
see \cite[pg.~154]{LS} for further details.

\pfstep{Step~2(d): Contributions from 
$
\int_{\Mtu}  
	\left\lbrace
		(1+2\upmu) L\mathcal{P}^{\Ntop} \Psi + \bX \mathcal{P}^{\Ntop} \Psi
	\right\rbrace
	(\slashed{\mathcal{P}} \Psi) \upmu \slashed{\mathcal{P}}^{\Ntop} \mytr\upchi
\, d \vol
$} This error integral is similar to the one we treated in Step~2(b), but easier. Here are the differences:
\begin{itemize}
\item There is an additional $\upmu$ factor.
\item There is a $L \mathcal{P}^{\Ntop} \Psi$ term, in addition to a $\bX \mathcal{P}^{\Ntop} \Psi$ term.
\item There is a factor of $\slashed{\mathcal{P}} \Psi$ instead of $\bX \Psi$.
\end{itemize} 
Notice that due to the additional factor of $\upmu$, 
we can control the $L^2(\Sigma_t^u)$ norm of $\sqrt{\upmu} L \mathcal{P}^{\Ntop} \Psi$ by the $\mathbb{Q}_{\Ntop}$ energy (recall the definition \eqref{eq:wave.energy.def.1} for the energy). Moreover, comparing \eqref{BA:W.Li.small} with \eqref{BA:LARGERIEMANNINVARIANTLARGE}, we see that 
the factor $\slashed{\mathcal{P}} \Psi$ gives an additional $\epd^{\f12}$ $L^{\infty}$-smallness factor 
compared to $\bX \mathcal{R}_{(+)}$. 
Therefore, we can use H\"older's inequality, 
\eqref{BA:W.Li.small},
the $L^{\infty}$ bound for $\upmu$ in Proposition~\ref{prop:geometric.low},
\eqref{eq:trch.top.non.sharp} in Proposition~\ref{prop:trch.top}, 
and Proposition~\ref{prop:mus.int}
and argue as in Step~2(b)
(taking into account definition~\ref{E:NEWTOPORDER})
to obtain:
\begin{equation*}
\begin{split}
&\: 
\left|
\int_{\Mtu}  
	\left\lbrace
		(1+2\upmu) L\mathcal{P}^{\Ntop} \Psi + \bX \mathcal{P}^{\Ntop} \Psi
	\right\rbrace
	(\slashed{\mathcal{P}} \Psi) \upmu \slashed{\mathcal{P}}^{\Ntop} \mytr\upchi
\, d \vol
\right|
\\
\ls &\: \epd^{\f 12} 
\int_{t'=0}^{t'=t}
\upmu_{\star}^{-\f 12}(t',u)
\left(\|\bX\mathcal{P}^{\Ntop}\Psi\|_{L^2(\Sigma_{t'}^u)} + \|\sqrt{\upmu} L\mathcal{P}^{\Ntop}\Psi\|_{L^2(\Sigma_{t'}^u)}\right) 
\|\upmu \slashed{\mathcal{P}}^{\Ntop} \mytr\upchi \|_{L^2(\Sigma_{t'}^u)}
\,dt' \\
\ls &\: \epd^{\f 12} \int_{t'=0}^{t'=t} \upmu_{\star}^{-\f 12}(t',u) \mathbb{Q}_{\Ntop}(t',u) \, dt' 
+ 
\epd^{\f 52} \upmu_{\star}^{-2 \toprate + 4.3}(t) \\
&\: 
+ 
\epd^{\f 12} 
\int_{t'=0}^{t'=t} 
	\upmu_{\star}^{-\f 12}(t',u) 
		\left\{
		\int_{s=0}^{s=t'}  \| \mathcal{P}^{[1, \Ntop]}  \mathfrak G \|_{L^2(\Sigma_s^u)} \, ds 
		\right\}^2 \, dt' 
			\\
\leq &\: \mbox{Non-boxed-constant-involving terms on RHS~\eqref{E:JAREDTOPORDERINTEGRALINEQUALITY}}.
\end{split}
\end{equation*}
Combining Steps~2(a)--2(d), we arrive at the desired bound \eqref{E:JAREDTOPORDERINTEGRALINEQUALITY}.

\medskip

\pfstep{Step~3: Proof of \eqref{E:JAREDPARTIALTOPORDERINTEGRALINEQUALITY}}
In this step, we only have to derive top-order energy estimates for
$\mathcal{R}_{(-)}, v^2, v^3, s$. This is in contrast to Step~2, in which we
also had to derive energy estimates for $\mathcal{R}_{(+)}$.
The proof of \eqref{E:JAREDPARTIALTOPORDERINTEGRALINEQUALITY} is 
the same as the proof of \cite[(14.5)]{LS},
except we have to account for the contribution of the inhomogeneous terms $\mathfrak G_{\iota}$
in the wave equations satisfied by 
$\widetilde{\Psi} 
\in \lbrace \mathcal{R}_{(-)}, v^2, v^3, s \rbrace$. For the same reason as in Step~2, these inhomogeneous terms lead to
error integrals that are controlled by the terms 
$\mbox{\upshape NewError}_{\Ntop}^{(Top)}(t,u)$ on
RHS~\eqref{E:JAREDPARTIALTOPORDERINTEGRALINEQUALITY}.
We clarify that the proof of \eqref{E:JAREDPARTIALTOPORDERINTEGRALINEQUALITY}
requires that we control the difficult error integrals
$$
\int_{\Mtu}  
	(\bX \mathcal{P}^{\Ntop} \widetilde{\Psi})
	(\bX \widetilde{\Psi}) 
	\slashed{\mathcal{P}}^{\Ntop} \mytr\upchi
\, d \vol,
$$
$$
\int_{\Mtu}  
	\left\lbrace
		(1+2\upmu) L\mathcal{P}^{\Ntop} \widetilde{\Psi}
	\right\rbrace
	(\bX \widetilde{\Psi}) \slashed{\mathcal{P}}^{\Ntop} \mytr\upchi
\, d \vol,
$$
$$
\int_{\Mtu}  
	\left\lbrace
		(1+2\upmu) L\mathcal{P}^{\Ntop} \widetilde{\Psi} + \bX \mathcal{P}^{\Ntop} \widetilde{\Psi}
	\right\rbrace
	(\mathcal{P} \widetilde{\Psi}) \upmu \slashed{\mathcal{P}}^{\Ntop} \mytr\upchi
\, d \vol,
$$
as in Step~2. In Step~2, the first two of these error integrals led to error terms that are 
controlled by the boxed-constant-involving terms on RHS~\eqref{E:JAREDTOPORDERINTEGRALINEQUALITY}.
However, in Step~3, we can take advantage of the smallness
of the factors $\bX \widetilde{\Psi}$ in these integrals.
That is, we can exploit the smallness
estimate $\| \bX \widetilde{\Psi} \|_{L^\infty(\Sigma_t)} \leq \epd^{\frac{1}{2}}$
(valid for $\widetilde{\Psi} \in \lbrace \mathcal{R}_{(-)}, v^2, v^3, s \rbrace$ -- but not for $\mathcal{R}_{(+)}$!),
which is provided by the bootstrap assumption \eqref{BA:SMALLWAVEVARIABLESUPTOONETRANSVERSALDERIVATIVE};
this allows us to avoid the error terms with large boxed constants 
(which are found on RHS~\eqref{E:JAREDTOPORDERINTEGRALINEQUALITY}),
and allow us to relegate the contribution of the corresponding error integrals
to the error term $\mbox{\upshape Error}_{\Ntop}^{(Top)}(t,u)$ on RHS~\eqref{E:JAREDPARTIALTOPORDERINTEGRALINEQUALITY}.
See \cite[pg.~154]{LS} for further details.

\qedhere
\end{proof}

\subsection{Sketch of the proof of Proposition~\ref{prop:wave}}
\label{sec:sketch.prop.wave}
The argument here is the same as in the proof of \cite[Proposition~14.1]{LS}, 
except we have to handle the additional terms in Proposition~\ref{P:JAREDimproved.critical.constant}.
\begin{proof}[Sketch of proof of Proposition~\ref{prop:wave}]
\pfstep{Step~1: The top- and penultimate- orders (Proof of \eqref{eq:main.EE.prop.top})} 
It turns out that the top-order energies are heavily coupled to the penultimate-order energies.
In turn, this forces us to perform a Gr\"onwall-type argument that simultaneously handles
the top- and penultimate-order energy estimates at the same time. 
For these reasons, 
we follow the notation of \cite[Proposition~14.1]{LS} and define:\footnote{For easy comparisons with the proof of
\cite[Proposition~14.1]{LS}, we are using the notations ``$F$,'' ``$G$,'' and ``$H$'' here. The reader should be careful
to distinguish these functions 
from the different functions $F$ and $G$ in Definitions~\ref{def:variables.fluid} and \ref{D:BIGGANDBIGH}.}
\begin{align}
F(t,u) \doteq &\: \sup_{(\hat{t}, \hat{u}) \in [0,t]\times [0,u]} \iota_F^{-1}(\hat{t},\hat{u}) \max\{\mathbb{Q}_{[1,\Ntop]}(\hat{t},\hat{u}),\, \mathbb K_{[1,\Ntop]}(\hat{t},\hat{u})\}, \label{eq:F.def.in.appendix}\\
G(t,u) \doteq &\: \sup_{(\hat{t}, \hat{u}) \in [0,t]\times [0,u]} \iota_G^{-1}(\hat{t},\hat{u}) 
\max\{\mathbb{Q}_{[1,\Ntop]}^{(Partial)}(\hat{t},\hat{u}),\, \mathbb K_{[1,\Ntop]}^{(Partial)}(\hat{t},\hat{u})\}, \label{eq:G.def.in.appendix}\\
H(t,u) \doteq &\: \sup_{(\hat{t}, \hat{u}) \in [0,t]\times [0,u]} \iota_H^{-1}(\hat{t},\hat{u}) \max\{\mathbb{Q}_{[1,\Ntop-1]}(\hat{t},\hat{u}),\, \mathbb K_{[1,\Ntop-1]}(\hat{t},\hat{u})\}, \label{eq:H.def.in.appendix}
\end{align}
where:
\begin{align*}
\iota_1(t) \doteq &\: \int_{t'=0}^{t'=t} \f 1{\sqrt{\Tboot -t'}}\, dt', \\
\iota_2(t) \doteq &\: \int_{t'=0}^{t'=t} \upmu_{\star}^{- 0.9}(t')\, dt', \\
\iota_F(t,u) = \iota_G(t,u) \doteq &\: \upmu_{\star}^{-2\toprate+1.8}(t)\iota_1^c(t)\iota_2^c(t) e^{ct} e^{cu}, \\
\iota_H(t,u) \doteq &\: \upmu_{\star}^{-2\toprate+3.8}(t) \iota_1^c(t)\iota_2^c(t) e^{ct} e^{cu}.
\end{align*}

Following exactly the same\footnote{Here we note one minor difference
compared to \cite[Proposition~14.1]{LS}: that proposition was more precise with respect to $u$ in
the sense that it yielded a priori estimates in terms of powers of
$\upmu_{\star}(t,u)$, rather than $\upmu_{\star}(t)$ (see Definition~\ref{def:mustar}).
For this reason, in the proof \cite[Proposition~14.1]{LS}, the definition of the analog of
$\iota_2$ involved $\upmu_{\star}(t,u)$, and similarly for the $\upmu_{\star}$-dependent
factors on the RHSs of the analogs of $\iota_F$, $\iota_G$, and $\iota_H$. 
The change we have made in this paper 
has no substantial effect on the analysis; at the relevant points
in the proof of \cite[Proposition~14.1]{LS}, all of the needed estimates
hold true with $\upmu_{\star}(t)$ in place of $\upmu_{\star}(t,u)$.
}
argument\footnote{The detailed argument relies on some extensions and
sharpened versions of the estimates
of Proposition~\ref{prop:mus.int}.
Given the estimates of Section~\ref{sec:geometry}, such as
Propositions~\ref{prop:geometric.low}, 
	\ref{prop:geometric.low.2}, 
	and
	\ref{P:LINFTYHIGHERTRANSVERSAL},
	the needed estimates can be proved using the same arguments given in \cite{LS}.
} 
used in the proof of \cite[Proposition~14.1]{LS} 
(see in particular \cite[(14.64)--(14.66)]{LS}), 
but taking into account the additional terms in Proposition~\ref{P:JAREDimproved.critical.constant}, 
we can choose $\toprate \in \mathbb N$ and $c>0$ sufficiently large depending on the absolute 
constant $M_{\mathrm{abs}}$ in Proposition~\ref{P:JAREDimproved.critical.constant}
so that the following holds\footnote{The inequality \cite[(14.64)]{LS} 
featured a term $C F^{\frac{1}{2}}(t,u)G^{\frac{1}{2}}(t,u)$ on the right-hand side. We used Young's inequality
to bound this term by $\leq a F(t,u) + \upalpha_3 G(t,u)$, where $\upalpha_3 \doteq C^2/a$ and we have chosen
$a$ to be small, which allows us to soak up $a F(t,u)$ into the term $\upalpha_1 F(t,u)$.} 
for every $(\hat{t},\hat{u})\in [0,t]\times [0,u]$:
\begin{align}
F(\hat{t},\hat{u}) \leq &\: C \epd^2 + \upalpha_1 F(t,u) + \upalpha_2 H(t,u) 
+ \upalpha_3 G(t,u)
\notag \\
&\: + C \iota_F^{-1}(\hat{t},\hat{u}) \int_{t'=0}^{t'=\hat{t}} \| (|L\mathcal{P}^{[1,\Ntop]} \Psi| + |\bX \mathcal{P}^{[1,\Ntop]} \Psi|) \mathcal{P}^{[1,\Ntop]} \mathfrak G \|_{L^1(\Sigma_{t'}^u)}\, dt'  \notag\\
&\: + C  \iota_F^{-1}(\hat{t},\hat{u}) \int_{t'=0}^{t'=\hat{t}} \upmu_{\star}^{-\f 32}(t') \left\{\int_{s=0}^{s=t'} \|\mathcal{P}^{[1,\Ntop]}\mathfrak G\|_{L^2(\Sigma_s^u)} \, ds\right\}^2 \, dt', \label{eq:appendix.F.basic} 
	\\
G(\hat{t},\hat{u}) \leq &\: C \epd^2 + \upbeta_1 F(t,u) + \upbeta_2 H(t,u) 
	\notag\\
&\: + C \iota_G^{-1}(\hat{t},\hat{u}) \int_{t'=0}^{t'=\hat{t}} \| (|L\mathcal{P}^{[1,\Ntop]} \Psi| + |\bX \mathcal{P}^{[1,\Ntop]} \Psi|) \mathcal{P}^{[1,\Ntop]} \mathfrak G \|_{L^1(\Sigma_{t'}^u)}\, dt'  \notag\\
&\: + C  \iota_G^{-1}(\hat{t},\hat{u}) \int_{t'=0}^{t'=\hat{t}} \upmu_{\star}^{-\f 32}(t') \left\{\int_{s=0}^{s=t'} \|\mathcal{P}^{[1,\Ntop]}\mathfrak G\|_{L^2(\Sigma_s^u)} \, ds\right\}^2 \, dt', \label{eq:appendix.G.basic} 
	\\
H(\hat{t},\hat{u}) \leq &\: C\epd^2 + \upgamma_1 F(t,u) + \upgamma_2 H(t,u) \notag\\
&\:  + C \iota_H^{-1}(\hat{t},\hat{u}) \int_{t'=0}^{t'=\hat{t}} \| (|L\mathcal{P}^{[1,\Ntop-1]} \Psi| + 
|\bX \mathcal{P}^{[1,\Ntop-1]} \Psi|) \mathcal{P}^{[1,\Ntop-1]} \mathfrak G \|_{L^1(\Sigma_{t'}^u)}\, dt', 
	\label{eq:appendix.H.basic}
\end{align}
where $C > 0$ is a constant,
while
$\upalpha_1$, $\upalpha_2$, $\upalpha_3$, 
$\upbeta_1$, $\upbeta_2$,
$\upgamma_1$ and $\upgamma_2$ are constants that obey the following smallness conditions
(as long as $\toprate \in \mathbb N$ and $c>0$ are sufficiently large): 
\begin{equation}\label{eq:funny.smallness}
\upalpha_1 
	+
	4 \upalpha_2 \upgamma_1  
	+ 
	\upalpha_3 \upbeta_1
	+
	4 \upalpha_3 \upbeta_2 \upgamma_1 < 1, 
\quad \quad \upgamma_2 < 3/4.
\end{equation}
At this point we fix $c>0$ and $\toprate \in \mathbb N$. From now on, 
we allow the general constants $C>0$ to depend on these particular fixed choices 
of $c$ and $\toprate$. 

For each of the three integrals on RHSs~\eqref{eq:appendix.F.basic}--\eqref{eq:appendix.H.basic}, 
we absorb $\iota_1^c(\hat{t})\iota_2^c(\hat{t}) e^{c \hat{t}} e^{c \hat{u}}$ into the general constant $C$, 
and then take the supremum with respect to $\hat{t}$. 
For instance, for the first integral on RHS \eqref{eq:appendix.F.basic},
we deduce that for $(\hat{t},\hat{u}) \in [0,t]\times [0,u]$, 
we have:
\begin{equation*}
\begin{split}
&\: \iota_F^{-1}(\hat{t},\hat{u}) \int_{t'=0}^{t'=\hat{t}} \| (|L\mathcal{P}^{[1,\Ntop]} \Psi| + |\bX \mathcal{P}^{[1,\Ntop]} \Psi|) \mathcal{P}^{[1,\Ntop]} \mathfrak G \|_{L^1(\Sigma_{t'}^u)}\, dt' \\
\ls &\: \sup_{\hat{t}' \in [0,t]} \upmu_{\star}^{2\toprate - 1.8}(\hat{t}') \int_{t'=0}^{t'=\hat{t}'}   \| (|L\mathcal{P}^{[1,\Ntop]} \Psi| + |\bX \mathcal{P}^{[1,\Ntop]} \Psi|) \mathcal{P}^{[1,\Ntop]} \mathfrak G \|_{L^1(\Sigma_{t'}^u)} \, dt'.
\end{split}
\end{equation*}
We perform the same operation on the other integrals. Since we have taken a supremum, the RHSs are independent of $(\hat{t},\hat{u})$. We then take supremum over $(\hat{t},\hat{u}) \in [0,t]\times [0,u]$ on the LHSs of 
\eqref{eq:appendix.F.basic}--\eqref{eq:appendix.H.basic} to obtain, 
with the same constants $\upalpha_1$, $\upalpha_2$, $\upalpha_3$, 
$\upbeta_1$, $\upbeta_2$,
$\upgamma_1$ and $\upgamma_2$,  
but with a different constant $C$, the following inequalities:
\begin{align}
F(t,u) \leq &\: C \epd^2 + \upalpha_1 F(t,u) + \upalpha_2 H(t,u) + \upalpha_3 G(t,u) \notag \\
&\: + C \sup_{\hat{t} \in [0,t]} \upmu_{\star}^{2\toprate - 1.8}(\hat{t}) \int_{t'=0}^{t'=\hat{t}}   \| (|L\mathcal{P}^{[1,\Ntop]} \Psi| + |\bX \mathcal{P}^{[1,\Ntop]} \Psi|) \mathcal{P}^{[1,\Ntop]} \mathfrak G \|_{L^1(\Sigma_{t'}^u)} \, dt'  \notag \\
&\: + C \sup_{\hat{t} \in [0,t]} \upmu_{\star}^{2\toprate - 1.8}(\hat{t})  \int_{t'=0}^{t'=\hat{t}}  \upmu_{\star}^{-\f 32}(t')  \left\{\int_{s=0}^{s=t'} \|\mathcal{P}^{[1,\Ntop]}\mathfrak G\|_{L^2(\Sigma_s^u)} \, ds\right\}^2 \, dt', 
\label{eq:appendix.F.again} \\
G(t,u) \leq &\: C \epd^2 + \upbeta_1 F(t,u) + \upbeta_2 H(t,u)  \notag \\
&\: + C \sup_{\hat{t} \in [0,t]} \upmu_{\star}^{2\toprate - 1.8}(\hat{t}) \int_{t'=0}^{t'=\hat{t}}   \| (|L\mathcal{P}^{[1,\Ntop]} \Psi| + |\bX \mathcal{P}^{[1,\Ntop]} \Psi|) \mathcal{P}^{[1,\Ntop]} \mathfrak G \|_{L^1(\Sigma_{t'}^u)} \, dt'  \notag \\
&\: + C \sup_{\hat{t} \in [0,t]} \upmu_{\star}^{2\toprate - 1.8}(\hat{t})  \int_{t'=0}^{t'=\hat{t}}  \upmu_{\star}^{-\f 32}(t')  \left\{\int_{s=0}^{s=t'} \|\mathcal{P}^{[1,\Ntop]}\mathfrak G\|_{L^2(\Sigma_s^u)} \, ds\right\}^2 \, dt',
\label{eq:appendix.G.again} \\
H(t,u) \leq &\: C\epd^2 + \upgamma_1 F(t,u) + \upgamma_2 H(t,u) \notag \\
&\: + C \sup_{\hat{t} \in [0,t]} \upmu_{\star}^{2\toprate - 3.8}(\hat{t}) \int_{t'=0}^{t'=\hat{t}}  
\| (|L\mathcal{P}^{[1,\Ntop-1]} \Psi| + |\bX \mathcal{P}^{[1,\Ntop-1]} \Psi|) \mathcal{P}^{[1,\Ntop-1]} \mathfrak G \|_{L^1(\Sigma_{t'}^u)} \, dt'. \label{eq:appendix.H.again}
\end{align}

The main point is the smallness conditions \eqref{eq:funny.smallness} on the constants 
$\upalpha_1,\cdots,\upgamma_2$ allow us to solve
the inequalities \eqref{eq:appendix.F.again}--\eqref{eq:appendix.H.again}
using a reductive approach.
More precisely, using that $\upgamma_2 < 3/4$, we soak the $\upgamma_2 H(t,u)$
term on RHS~\eqref{eq:appendix.H.again} back into the LHS
to isolate $H(t,u)$, at the expense of enlarging $C$
and replacing $\upgamma_1$ with $4 \upgamma_1$.
We then insert this estimate for $H(t,u)$ into RHS~\eqref{eq:appendix.G.again} 
to obtain an estimate for $G(t,u)$,
and then insert these estimates for $H(t,u)$ and $G(t,u)$ into RHS~\eqref{eq:appendix.F.again}
to obtain the following inequality:
\begin{align}
\begin{split}
F(t,u) \leq &\: 
C \epd^2 
+ 
\left\lbrace
	\upalpha_1 
	+
	4 \upalpha_2 \upgamma_1  
	+ 
	\upalpha_3 \upbeta_1
	+
	4 \upalpha_3 \upbeta_2 \upgamma_1
\right\rbrace
	F(t,u)
	\label{eq:appendix.F.FINAL}  
	\\
&\: + C \sup_{\hat{t} \in [0,t]} \upmu_{\star}^{2\toprate - 1.8}(\hat{t}) \int_{t'=0}^{t'=\hat{t}}   \| (|L\mathcal{P}^{[1,\Ntop]} \Psi| + |\bX \mathcal{P}^{[1,\Ntop]} \Psi|) \mathcal{P}^{[1,\Ntop]} \mathfrak G \|_{L^1(\Sigma_{t'}^u)} \, dt' \\
&\: + C \sup_{\hat{t} \in [0,t]} \upmu_{\star}^{2\toprate - 1.8}(\hat{t})  \int_{t'=0}^{t'=\hat{t}}  \upmu_{\star}^{-\f 32}(t')  \left\{\int_{s=0}^{s=t'} \|\mathcal{P}^{[1,\Ntop]}\mathfrak G\|_{L^2(\Sigma_s^u)} \, ds\right\}^2 \, dt' \\
&\: + C \sup_{\hat{t} \in [0,t]} \upmu_{\star}^{2\toprate - 3.8}(\hat{t}) \int_{t'=0}^{t'=\hat{t}}  \| (|L\mathcal{P}^{[1,\Ntop-1]} \Psi| + |\bX \mathcal{P}^{[1,\Ntop-1]} \Psi|) \mathcal{P}^{[1,\Ntop-1]} \mathfrak G \|_{L^1(\Sigma_{t'}^u)} \, dt'.
\end{split}
\end{align}
From the smallness condition 
$\upalpha_1 
	+
	4 \upalpha_2 \upgamma_1  
	+ 
	\upalpha_3 \upbeta_1
	+
	4 \upalpha_3 \upbeta_2 \upgamma_1
< 1$
featured in \eqref{eq:funny.smallness},
it follows that we can soak the terms
$
\left\lbrace
	\upalpha_1 
	+
	4 \upalpha_2 \upgamma_1  
	+ 
	\upalpha_3 \upbeta_1
	+
	4 \upalpha_3 \upbeta_2 \upgamma_1
\right\rbrace
	F(t,u)
$
on RHS~\eqref{eq:appendix.F.FINAL} back into LHS~\eqref{eq:appendix.F.FINAL} to isolate $F(t,u)$,
at the expense of increasing the constant $C$.
We therefore deduce the following inequality:
\begin{equation}\label{eq:appendix.F.G.H.final}
\begin{split}
&\: F(t,u)
	\\
\ls &\: \epd^2 + \sup_{\hat{t} \in [0,t]} \upmu_{\star}^{2\toprate - 1.8}(\hat{t}) \int_{t'=0}^{t'=\hat{t}}   \| (|L\mathcal{P}^{[1,\Ntop]} \Psi| + |\bX \mathcal{P}^{[1,\Ntop]} \Psi|) \mathcal{P}^{[1,\Ntop]} \mathfrak G \|_{L^1(\Sigma_{t'}^u)} \, dt' \\
&\: + \sup_{\hat{t} \in [0,t]} \upmu_{\star}^{2\toprate - 1.8}(\hat{t})  \int_{t'=0}^{t'=\hat{t}}  \upmu_{\star}^{-\f 32}(t')  \left\{\int_{s=0}^{s=t'} \|\mathcal{P}^{[1,\Ntop]}\mathfrak G\|_{L^2(\Sigma_s^u)} \, ds\right\}^2 \, dt' \\
&\: + \sup_{\hat{t} \in [0,t]} \upmu_{\star}^{2\toprate - 3.8}(\hat{t}) \int_{t'=0}^{t'=\hat{t}}  \| (|L\mathcal{P}^{[1,\Ntop-1]} \Psi| + |\bX \mathcal{P}^{[1,\Ntop-1]} \Psi|) \mathcal{P}^{[1,\Ntop-1]} \mathfrak G \|_{L^1(\Sigma_{t'}^u)} \, dt'.
\end{split}
\end{equation}
Then from \eqref{eq:appendix.F.G.H.final} and the arguments given above, 
we deduce that $G(t,u)$ and $H(t,u)$ are also bounded by $\leq$ 
RHS~\eqref{eq:appendix.F.G.H.final} (where we enlarge $C$ if necessary).

Recalling the definitions of $F$,$G$, and $H$ in 
\eqref{eq:F.def.in.appendix}--\eqref{eq:H.def.in.appendix}, 
we see that \eqref{eq:appendix.F.G.H.final} and the similar bounds for $G(t,u)$ and $H(t,u)$
collectively imply \eqref{eq:main.EE.prop.top}.

\pfstep{Step~2: The lower orders (Proof of \eqref{eq:main.EE.prop.low})} 
To prove the lower-order energy estimates, 
we start by considering the energy inequality 
given by 
the below-top-order estimate from Proposition~\ref{P:JAREDimproved.critical.constant},
i.e., the estimate \eqref{E:JAREDBELOWTOPORDERINTEGRALINEQUALITY},
which features the additional term \eqref{E:NEWBELOWTOPORDER}
compared to \cite[(14.6)]{LS}.

Observe that on RHS~\eqref{E:JAREDBELOWTOPORDERINTEGRALINEQUALITY}, 
except for $\int_{t'=0}^{t'=t} \f{\mathbb{Q}^{\f 12}_{[1,N-1]}(t',u) }{\upmu_{\star}^{\f 12}(t',u)} \left\{ \int_{s=0}^{s=t'} \f{\mathbb{Q}^{\f 12}_{[1,N]}(s,u) }{\upmu_{\star}^{\f 12}(s,u)}\, ds \right\}\, dt'$, 
every other term can be treated directly by Gr\"onwall's inequality (using Proposition~\ref{prop:mus.int}),
as in \cite{LS}. 
It thus follows that:
\begin{equation}\label{eq:energy.appendix.losing.derivatives}
\begin{split}
&\: \sup_{t'\in [0,t]} \max\{\mathbb{Q}_{[1,N-1]}(t',u), \, \mathbb K_{[1,N-1]}(t',u)\} \\
\leq &\: C \epd^2 + C \int_{t'=0}^{t'=t} \f{\mathbb{Q}^{\f 12}_{[1,N-1]}(t',u) }{\upmu_{\star}^{\f 12}(t',u)} \left\{ \int_{s=0}^{s=t'} \f{\mathbb{Q}^{\f 12}_{[1,N]}(s,u) }{\upmu_{\star}^{\f 12}(s,u)}\, ds \right\}\, dt' \\
&\: + C \| (|L\mathcal{P}^{[1,N-1]} \Psi| + |\bX \mathcal{P}^{[1,N-1]} \Psi|) |\mathcal{P}^{[1,N-1]} \mathfrak G| \|_{L^1(\Mtu)}.
\end{split}
\end{equation}

To proceed, we analyze the double time-integral term on RHS~\eqref{eq:energy.appendix.losing.derivatives}.
For any $\varsigma >0$, we have:
\begin{equation}\label{eq:energy.appendix.losing.derivatives.2}
\begin{split}
&\: \int_{t'=0}^{t'=t} \f{\mathbb{Q}^{\f 12}_{[1,N-1]}(t',u) }{\upmu_{\star}^{\f 12}(t',u)} \left\{ \int_{s=0}^{s=t'} \f{\mathbb{Q}^{\f 12}_{[1,N]}(s,u) }{\upmu_{\star}^{\f 12}(s,u)}\, ds \right\}\, dt' \\
\leq &\:  \left(\sup_{t'\in [0,t]} \mathbb{Q}^{\f 12}_{[1,N-1]}(t') \right)
	\times \left(\sup_{s\in [0,t]} \min\{1,\upmu_{\star}^{\toprate-\Ntop+N-0.9}(s)\} \mathbb{Q}^{\f 12}_{[1,N]}(s) \right) \\
&\: \times \int_{t'=0}^{t'=t} \f{1}{\upmu_{\star}^{\f 12}(t')} \left\{ \int_{s=0}^{s=t'}  \f{\max\{1,\upmu_{\star}^{-\toprate+\Ntop-N+0.9}(s)\}}{\upmu_{\star}^{\f 12}(s)}\, ds \right\}\, dt' \\
\leq &\: \varsigma \sup_{t'\in [0,t]} \mathbb{Q}_{[1,N-1]}(t') \\
&\: + C \varsigma^{-1} \max\{1, \upmu_{\star}^{-2\toprate+2\Ntop-2N+3.8}(t) \} (\sup_{s\in [0,t]} 
\min\{1,\upmu_{\star}^{2\toprate-2\Ntop+2N-1.8}(s)\} \mathbb{Q}_{[1,N]}(s)),
\end{split}
\end{equation}
where to obtain the last inequality, 
we have used Young's inequality and the following estimate,
which follows from Proposition~\ref{prop:mus.int}:
\begin{equation*}
\begin{split}
&\: \int_{t'=0}^{t'=t} \f{1}{\upmu_{\star}^{\f 12}(t')} \left\{ \int_{s=0}^{s=t'}  \f{\max\{1,\upmu_{\star}^{-\toprate+\Ntop-N+0.9}(s)\}}{\upmu_{\star}^{\f 12}(s)}\, ds \right\}\, dt' \\
\ls &\: \int_{t'=0}^{t'=t} \f{\max\{ 1, \upmu_{\star}^{-M+\Ntop-N+1.4}(t')\} }{\upmu_{\star}^{\f 12}(t')}  \, dt' \ls \max\{1, \upmu_{\star}^{-\toprate+\Ntop-N+1.9}(t) \}.
\end{split}
\end{equation*}

Inserting \eqref{eq:energy.appendix.losing.derivatives.2} into \eqref{eq:energy.appendix.losing.derivatives} and fixing
$\varsigma>0$ to be sufficiently small, 
we can absorb the term $C\varsigma (\sup_{t'\in [0,t]} \mathbb{Q}_{[1,N-1]}(t'))$ 
back into LHS~\eqref{eq:energy.appendix.losing.derivatives}. Thus, for this fixed value of 
$\varsigma$, we obtain:
\begin{equation*}
\begin{split}
&\: \sup_{t'\in [0,t]} \max\{\mathbb{Q}_{[1,N-1]}(t',u), \, \mathbb K_{[1,N-1]}(t',u)\} \\
\ls &\: \epd^2 + \max\{1, \upmu_{\star}^{-2M+2\Ntop-2N+3.8}(t) \} 
\left(\sup_{s\in [0,t]} \min\{1,\upmu_{\star}^{2\toprate-2\Ntop+2N-1.8}(s)\} \mathbb{Q}_{[1,N]}(s) \right) \\
&\: + \| (|L\mathcal{P}^{[1,N-1]} \Psi| + |\bX \mathcal{P}^{[1,N-1]} \Psi|) |\mathcal{P}^{[1,N-1]} \mathfrak G| \|_{L^1(\Mtu)}.
\end{split}
\end{equation*}
After changing the index $N$ to $N+1$, we conclude the estimate \eqref{eq:main.EE.prop.low}. \qedhere
\end{proof}

\bibliographystyle{amsalpha}
\bibliography{JBib}

\end{document}